\documentclass[
		twoside,openright,titlepage,numbers=noenddot,headinclude,
	 	footinclude=true,cleardoublepage=empty,
		dottedtoc, manychapters,
		BCOR=5mm,paper=a4,fontsize=13pt, 
		ngerman,american,french, 
		]{scrreprt}

\PassOptionsToPackage{utf8}{inputenc} 
\usepackage{inputenc}

\PassOptionsToPackage{linedheaders, eulerchapternumbers,listings, pdfspacing, subfig,beramono,eulermath,parts,manychapters}{classicthesis}

\newcommand{\myTitle}{Colouring digraphs}

\newcommand{\myName}{Guillaume Aubian\xspace}

\newcommand{\myFaculty}{Put data here\xspace}

\newcommand{\myUni}{Université Paris-Cité\xspace}

\newcommand{\myVersion}{version 4.2\xspace}

\newcounter{dummy} 
\providecommand{\mLyX}{L\kern-.1667em\lower.25em\hbox{Y}\kern-.125emX\@}

\usepackage{lipsum} 

\usepackage{babel}

\usepackage{csquotes}
\PassOptionsToPackage{%
backend=bibtex8,
bibencoding=ascii,%
language=auto,%
style=numeric-comp,%
sorting=nyt, 
maxbibnames=10, 
natbib=true, 
firstinits,
url=false
}{biblatex}
\usepackage{biblatex}
 
\PassOptionsToPackage{fleqn}{amsmath} 
 \usepackage{amsmath}
 
\PassOptionsToPackage{T1}{fontenc} 
\usepackage{fontenc}

\usepackage{textcomp} 

\usepackage{scrhack} 

\usepackage{xspace} 

\usepackage{mparhack} 

\usepackage{fixltx2e} 

\PassOptionsToPackage{smaller}{acronym} 
\usepackage{acronym} 

\PassOptionsToPackage{pdftex}{graphicx}
\usepackage{graphicx} 

\usepackage{tabularx} 
\usepackage{caption}
\usepackage{subfig}  

\usepackage{listings} 
\PassOptionsToPackage{pdftex,hyperfootnotes=false,pdfpagelabels}{hyperref}
\usepackage{hyperref}  
\hypersetup{
colorlinks=true, linktocpage=true, pdfstartpage=13, pdfstartview=FitV,
breaklinks=true, pdfpagemode=UseNone, pageanchor=true, pdfpagemode=UseOutlines,%
plainpages=false, bookmarksnumbered, bookmarksopen=true, bookmarksopenlevel=1,%
hypertexnames=true, pdfhighlight=/O,
urlcolor=webbrown, linkcolor=RoyalBlue, citecolor=webgreen, 
pdftitle={\myTitle},
pdfauthor={\textcopyright\ \myName, \myUni, \myFaculty},
pdfsubject={},
pdfkeywords={},
pdfcreator={pdfLaTeX},
pdfproducer={LaTeX with hyperref and classicthesis}
}

\makeatletter
\@ifpackageloaded{babel}
{
\addto\extrasamerican{

}
\addto\extrasngerman{

}
}{\relax}
\makeatother

\usepackage{classicthesis} 

\usepackage[a4paper,outer=3cm]{geometry}

\DeclareOldFontCommand{\bf}{\normalfont\bfseries}{\mathbf}
\usepackage{multirow}

\usepackage[utf8]{inputenc}
\usepackage[T1]{fontenc}
\usepackage{times}

\usepackage{xspace}

\usepackage{longtable}
\usepackage{amsmath,amsthm,amssymb,amsthm, stmaryrd}
\usepackage{enumerate}

\usepackage{algorithm}
\usepackage{algpseudocode}
\algnewcommand\algorithmicinput{\textbf{Input:}}
\algnewcommand\algorithmicoutput{\textbf{Output:}}
\algnewcommand\Input{\item[\algorithmicinput]}
\algnewcommand\Output{\item[\algorithmicoutput]}

\usepackage{pgf,tikz}
\usepackage{svg}
\usetikzlibrary{arrows}
\usetikzlibrary{decorations.markings}

\usepackage{pdftricks}
\begin{psinputs}
	\usepackage{pstricks,pst-plot,pst-node,pst-func}
\end{psinputs}
\usepackage{graphics,graphicx}

\renewcommand{\tilde}{\widetilde}

\tikzstyle{vertex}=[circle, draw, inner sep=0pt, minimum size=6pt]
\newcommand{\vertex}{\node[vertex]}

\usetikzlibrary{decorations.pathreplacing}

\tikzset{->-/.style={decoration={
  markings,
  mark=at position .5 with {\arrow{>}}},postaction={decorate}}}

 \usepackage{marginnote}

\newcommand{\m}[1]{}

\usepackage[nothm]{thmbox}
\usepackage{amsthm}
\usepackage{thmtools}
\usetikzlibrary{patterns}
\usetikzlibrary{calc}
\usepackage{setspace}

\declaretheorem[name=Théorème, thmbox=M]{theoreme}
\declaretheorem[name=Conjecture, thmbox=M]{conjecturefr}
\declaretheorem[parent=section,thmbox=M]{theorem}

\declaretheorem[numberlike=theorem,thmbox=M]{proposition}
\declaretheorem[numberlike=theorem,thmbox=M]{corollary}
\declaretheorem[numberlike=theorem,thmbox=M]{conjecture}

\declaretheorem[numberlike=theorem,thmbox=M]{observation}\declaretheorem[numberlike=theorem,thmbox=M]{property}
\declaretheorem[numberlike=theorem,thmbox=M]{definition}

\newtheorem{claim}{Claim}[theorem]

\declaretheorem[numberlike=theorem]{lemma}
\declaretheorem[numberlike=theorem]{remark}

\newcommand{\Gya}{Gy\'arf\'as\xspace}

\newcommand{\Erd}{Erd\H os\xspace}

\newcommand{\Chu}{Chudnovsky\xspace}

\newcommand{\Lov}{Lov\'asz\xspace}

\newcommand{\Hag}{H\"aggkvist\xspace}

\newcommand{\eN}{\mathbb{N}}

\newcommand{\sm}{\setminus}
\newcommand{\mc}{\mathcal}

\newcommand{\mC}{\mathcal{C}}

\newcommand{\ra}{\rightarrow}
\newcommand{\Ra}{\Rightarrow}
\newcommand{\la}{\leftarrow}

\newcommand{\olra}{\overleftrightarrow}
\DeclareRobustCommand*{\ora}{\overrightarrow}
\newcommand{\ova}[1]{\overline{#1}}
\newcommand{\ovlra}{\overleftrightarrow}

\newcommand{\dmax}{\Delta_{max}}
\newcommand{\dmin}{\Delta_{min}}
\newcommand{\bid}{\overset \leftrightarrow}

\DeclareMathOperator{\dic}{\ora \chi}

\newcommand{\F}{Forb_{ind}}

\newcommand{\overbar}[1]{\mkern 1.7mu\overline{\mkern-1.7mu#1\mkern-1.7mu}\mkern 1.7mu}

\tikzstyle{vertex}=[circle,draw, top color=gray!5, 
	    bottom color=gray!30, minimum size=12pt, scale=1, inner sep=0.5pt]
\tikzstyle{arc}=[->, > = latex',  thick]
\tikzstyle{edge}=[thick, blue]

\def\centerarc[#1](#2)(#3:#4:#5) {\draw[#1] ($(#2)+({#5*cos(#3)},{#5*sin(#3)})$) arc (#3:#4:#5); }

\newenvironment{proofclaim}
	{\noindent {\bf Proof of Claim
	:}}
	{\hfill $\square$ \par\vspace{11pt}}

\newenvironment{subproof}
	{\noindent {\bf Proof of Claim
	:}}
	{\hfill $\square$ \par\vspace{11pt}}

\newcommand{\omg}{oriented complete multipartite graph\xspace}
\newcommand{\omgs}{oriented complete multipartite graphs\xspace}
\newcommand{\ocg}{oriented chordal graph\xspace}
\newcommand{\ocgs}{oriented chordal graphs\xspace}

\newenvironment{maintheorem}{%
    \medskip
  \thmbox[M]{\textbf{Theorem \ref{thm:struct}}}%
  \hspace*{-1.5em}\slshape\ignorespaces%
  }
  {%
  \endthmbox\vspace*{.75ex}%
  }

\newenvironment{inroundtheorem}{%
    \medskip
  \thmbox[M]{\textbf{Theorem \ref{thm:in-round}}}%
  \hspace*{-1.5em}\slshape\ignorespaces%
  }
  {%
  \endthmbox\vspace*{.75ex}%
}

\newcommand{\SD}{S^+_2}

\newcommand{\dP}[1]{\ora{P}_{\hspace{-0.1em}#1}}
\newcommand{\dC}[1]{\ora{C}_{\hspace{-0.1em}#1}}
\newcommand{\dK}[1]{\ora{K}_{\hspace{-0.1em}#1}}
\newcommand{\TT}[1]{\ora{TT}_{\hspace{-0.1em}#1}}
\newcommand{\ob}[1]{\overbar{#1}}

\preto\fullcite{\AtNextCite{\defcounter{maxnames}{99}}}

\begin{document}

\setlength{\emergencystretch}{2cm}

\frenchspacing 

\raggedbottom 

\selectlanguage{american} 


\pagenumbering{roman} 

\pagestyle{plain} 



\begin{titlepage}

\vspace*{-4.5cm}\hspace*{12.7cm}\includegraphics[scale=0.17]{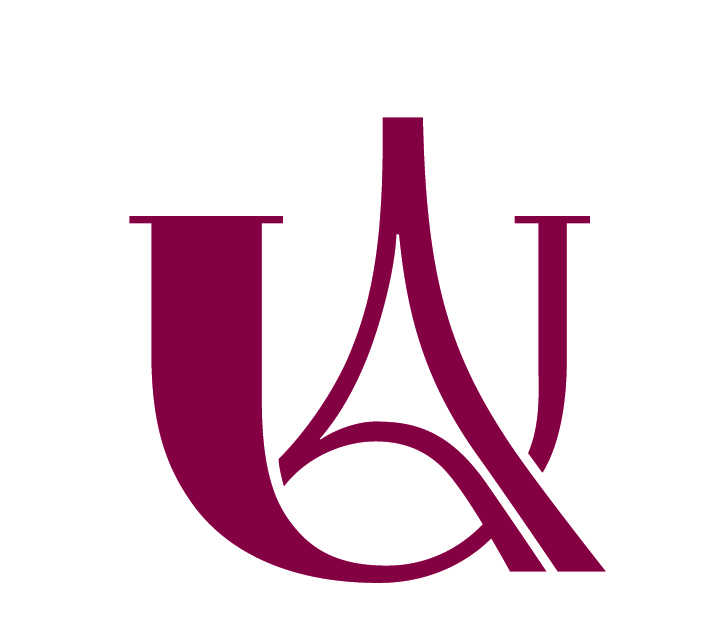}
\vspace*{1cm}

\begin{addmargin}[-1cm]{-2cm}

\begin{center}
\large

\vspace{-3cm}

{\large Université Paris-Cité}

École Doctorale de Sciences Mathématiques \\
de Paris Centre\\
{\small ED 386}\\

\vspace{0.5cm}
Préparée à l'\textbf{I}nstitut de \textbf{R}echerche \\
en \textbf{I}nformatique \textbf{F}ondamentale et\\
au \textbf{D}épartement d'\textbf{I}nformatique de\\
l'\textbf{É}cole \textbf{N}ormale \textbf{S}upérieure\\

\vspace{0.8cm}
\begingroup
\huge\color{Maroon}\spacedallcaps{\myTitle} \\  
\endgroup

\vspace{0.8cm}

Par \spacedlowsmallcaps{\myName}

\vspace{0.8cm}
    
{\Large Thèse de doctorat en Informatique}

\vspace{0.8cm}

Dirigée par \spacedlowsmallcaps{Pierre Charbit}

\vspace{0.5cm}

Présentée et soutenue publiquement\\
le \spacedlowsmallcaps{20/06/2023}




\end{center}

\vspace{0.7cm}

Devant un jury composé de :

\vspace{-0.2cm}

\begin{tabular}{p{13cm}p{5cm}}
        \multirow{2}{*}{\spacedlowsmallcaps{Pierre Charbit}} & \multirow{2}{*}{\emph{Directeur de thèse}} \\
       {\footnotesize Maître de conférences, Université Paris-Cité} & \\[-0.4cm]
        \multirow{2}{*}{\spacedlowsmallcaps{Pierre Aboulker}} & \multirow{2}{*}{\emph{Co-encadrant}} \\
       {\footnotesize Maître de conférences, École Normale Supérieure} & \\[-0.4cm]
        \multirow{2}{*}{\spacedlowsmallcaps{Stéphane Bessy}} & \multirow{2}{*}{\emph{Rapporteur}} \\
       {\footnotesize Professeur, Université de Montpellier} & \\[-0.4cm]
        \multirow{2}{*}{\spacedlowsmallcaps{Alexander Scott}} & \multirow{2}{*}{\emph{Rapporteur}} \\
       {\footnotesize Professeur, Université d'Oxford} & \\[-0.4cm]
        \multirow{2}{*}{\spacedlowsmallcaps{Marthe Bonamy}} & \multirow{2}{*}{\emph{Examinatrice}} \\
       {\footnotesize Chercheuse, CNRS, Université de Bordeaux} & \\[-0.4cm]
        \multirow{2}{*}{\spacedlowsmallcaps{Vincent Cohen-Addad}} & \multirow{2}{*}{\emph{Examinateur}} \\
       {\footnotesize Chercheur, Google Research} & \\[-0.4cm]
        \multirow{2}{*}{\spacedlowsmallcaps{Mat\v{e}j Stehlík}} & \multirow{2}{*}{\emph{Examinateur}} \\
       {\footnotesize Professeur, Université Paris-Cité} & \\[-0.4cm]
    
\end{tabular}
\end{addmargin}

\end{titlepage} 

\cleardoublepage
\thispagestyle{empty}

\vspace*{8cm}Cette thèse a été préparée principalement à l'Institut de Recherche en Informatique Fondamentale (IRIF) sous la direction de Pierre Charbit, et partiellement au Département d'Informatique de l'École Normale Supérieure (DIENS), sous la supervision de Pierre Aboulker.

\vspace{1cm}

This thesis was mainly prepared at the Institut de Recherche en Informatique Fondamentale (IRIF) under the supervision of Pierre Charbit, and partly in the Département d'Informatique de l'École Normale Supérieure (DIENS), under the supervision of Pierre Aboulker.

\vspace{1cm}

\includegraphics[scale=0.17]{upccropped.png}\hspace*{6.3cm}\includegraphics[scale=0.09]{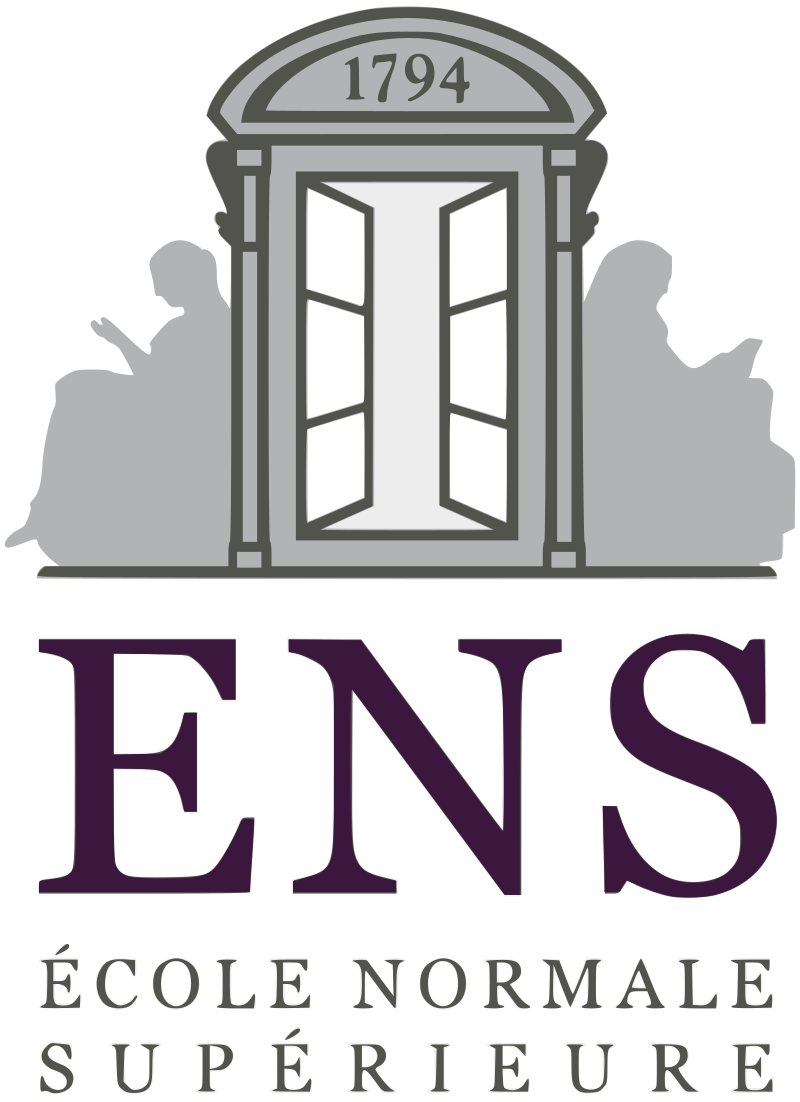}

\cleardoublepage\pdfbookmark[1]{Résumé}{Résumé} 

\begingroup
\let\clearpage\relax
\let\cleardoublepage\relax
\let\cleardoublepage\relax

\chapter*{Résumé}

\section*{Des réseaux aux graphes}

Au cours des dernières décennies, les réseaux sont devenus omniprésents dans notre vie quotidienne. Parmi ces réseaux se trouvent des réseaux sociaux, des réseaux de neurones ou des réseaux informatiques, dont Internet, mais les réseaux apparaissent également dans des domaines à première vue éloignés de l'informatique, par exemple pour ce qui est des réseaux routiers, des réseaux électriques ou des réseaux métaboliques.

En informatique et en mathématiques, lorsqu'ils sont étudiés d'un point de vue théorique, les réseaux sont plutôt appelés des graphes, et leur étude la théorie des graphes. Si les applications mentionnées précédemment sont plutôt récentes, la théorie des graphes les précède de plusieurs siècles. En effet, c'est en 1736 que, étudiant quel chemin emprunter pour traverser chaque pont de Königsberg (désormais Kaliningrad), Euler comprît que la topologie précise des lieux n'avait aucune importance, et seule importait quelles îles reliaient les différents ponts. La théorie des graphes était née, même si le terme ne fut inventé que 142 ans plus tard par Sylvester.

\medskip

En 1852, Guthrie découvrit que la carte des comtés anglais pouvait être colorée en utilisant uniquement quatre couleurs de telle sorte que deux comtés avec une frontière commune ne reçoivent pas la même couleur. Il se demanda alors si la propriété se généralisait à toutes les cartes. Rapidement, ette conjecture s'est révélée être équivalente à la conjecture correspondante sur les graphes, la fameuse conjecture des quatre couleurs, résolue plus d'un siècle plus tard, non sans l'aide d'ordinateurs.

\begin{theoreme}[Théorème des quatre couleurs (Appel, Haken, 1977, \cite{AH77})]\label{thm:four_colours_fr}

Tout graphe planaire peut être coloré avec quatre couleurs de telle sorte que deux sommets adjacents ne reçoivent pas la même couleur.

\end{theoreme}

Cela conduisit à l'étude du nombre chromatique des graphes, c'est-à-dire le nombre minimum de couleurs nécessaires pour colorier un graphe de telle sorte que deux sommets adjacents ne reçoivent pas la même couleur. En particulier, une question de premier ordre, encore aujourd'hui, est d'étudier comment le nombre chromatique est lié à la structure d'un graphe donné.

\medskip

L'un des premiers exemples en la matière est la construction de graphes sans triangle avec un grand nombre chromatique par Tutte (sous le pseudonyme de Blanche Descartes) dans \cite{D54}. Comme un graphe complet à $n$ sommets a un nombre chromatique de $n$, le nombre chromatique d'un graphe est toujours au moins égal à son clique number. Par sa construction, Tutte montra que le nombre chromatique n'est pas borné supérieurement par une fonction du clique number. Cependant, il reste intéressant d'étudier la structure des graphes dans lesquels les deux notions sont liées, voire égales.

\smallskip
Un graphe dont le nombre chromatique est égal à son clique number, et dont tous les sous-graphes induits satisfont cette propriété, est dit parfait. Dans \cite{L72}, Lov\'asz prouva que si un graphe est parfait, son complémentaire l'est aussi. Puisque les cycles impairs sur au moins cinq sommets ne sont pas parfaits, un graphe parfait n'a ni trou impair ni anti-trou impair. Berge conjectura alors, dans sa célèbre Conjecture Forte des Graphes Parfaits, que la réciproque est vraie. Cette conjecture fut finalement prouvée par Chudnovsky, Robertson, Seymour et Thomas (voir \cite{Perfect}).

Un autre cas intéressant est celui des graphes dont le nombre chromatique est borné par une fonction de leur clique number. De telles classes de graphes sont dites $\chi$-bornées. \Gya et Sumner ont conjecturé que, pour toute forêt $F$, la classe des graphes sans copie induite de $F$ est $\chi$-bornée. Cette conjecture est toujours ouverte.

\section*{Coloration de graphes dirigés}

En 1982, Neumann-Lara introduisit dans \cite{NL82} le pendant dirigé de cette notion de coloration : colorier un graphe dirigé $D$, c'est partitionner ses sommets en ensembles induisant des graphes orientés acycliques et, comme pour les graphes non orientés, le nombre dichromatique de $D$, noté $\dic(D)$, est la taille minimale d'une telle partition. Remarquons que cette notion est naturelle en ceci que remplacer chaque arête d'un graphe non orienté $G$ par deux arcs opposés donne un graphe orienté $\overleftrightarrow{G}$ vérifiant $\dic(\overleftrightarrow{G}) = \chi(G)$.

En 2001, cette notion a été réintroduite par Mohar dans \cite{M03}, qui a prouvé de nombreux résultats concernant le nombre dichromatique, parmi lesquels le suivant : étant donné un graphe dirigé $D$ avec une matrice d'adjacence $M_D$, $\dic(D)$ est majoré par un plus la valeur absolue maximale d'une valeur propre de $M_D$. Comme le même résultat existe pour les graphes non orientés, cela corrobore l'idée selon laquelle cette définition du nombre dichromatique est la bonne notion pour les graphes dirigés, et on peut espérer généraliser les résultats sur le nombre chromatique des graphes non orientés aux graphes dirigés via le nombre dichromatique.

De tels résultats ont été trouvés dans différents domaines de la théorie des graphes, comme la théorie des graphes extrémaux~\cite{BBSS20, HK15, KS20}, la théorie algébrique des graphes~\cite{M10}, les sous-structures forcées par un grand nombre dichromatique~\cite{AAC21, ACN21, ACL19, hero, GSS20, HLNT19, S21}, le nombre dichromatique avec liste~\cite{BHL18, HM11}, la dicoloration des graphes orientés sur des surfaces~\cite{AHKR21, LM17, S19}, la théorie des flots~\cite{H17, KV12} etc.

En particulier, une version du Théorème~\ref{thm:four_colours_fr} a été conjecturée par Neumann-Lara :

\begin{conjecturefr}[Conjecture des deux couleurs]
Tout graphe orienté planaire a nombre dichromatique au plus 2.
\end{conjecturefr}

Cette conjecture est toujours ouverte.





\medskip

Cette thèse est consacrée à l'étude du nombre dichromatique des graphes dirigés. La question centrale à laquelle j'ai cherché à répondre est de savoir comment la structure d'un graphe dirigé affecte son nombre dichromatique, en s'inspirant du cas non-dirigé, où des analyses similaires ont conduit à l'étude du nombre chromatique.

\section*{Plan de la thèse}

Dans la première partie de ce manuscrit, nous avons étudié quelques métriques classiques et leur impact sur le nombre dichromatique, en particulier en cherchant à borner le nombre dichromatique par une fonction de ces métriques. Une des métriques clefs que nous avons considérée est le degré maximal. Sur les graphes non-dirigés, le célèbre théorème de Brooks \cite{B41} établit que le nombre chromatique d'un graphe connexe est au plus égal à son degré maximal plus un, et qu'il n'y a égalité que pour les graphes complets et les cycles impairs. En cherchant des résultats similaires dans le cas dirigé, un problème se pose : il existe plusieurs notions qui peuvent faire office de degré pour les graphes dirigés.

Dans le chapitre~\ref{chpt:brooks}, nous avons approfondi la relation entre les différentes notions de degrés et le nombre dichromatique. Nous avons commencé par examiner $\Delta_{max}$, et bien qu'une caractérisation des cas d'égalité avait déjà été trouvée par Harutyunyan et Mohar dans \cite{HMGallai}, nous fournissons plusieurs preuves de ce résultat. Nous nous intéressont ensuite à $\Delta_{min}$ et obtenons un résultat d'impossibilité pour une caractérisation simple des cas d'égalité pour cette métrique. Ce travail, réalisé en collaboration avec Pierre Aboulker, a été publié dans \cite{AA22}.

\medskip
Il y a des raisons de penser que notre approche est, sinon la bonne, au moins naturelle : en effet, le théorème de Brooks est central dans le domaine de la coloration de graphes, et il a donné lieu à de nombreuses généralisations utilisant diverses notions de coloration, voir par exemple \cite{RPJ75} pour les hypergraphes, ou des métriques plus fortes que le degré. Un tel exemple de métrique est l'arête-connectivité locale maximale, que l'on définit comme le nombre maximum de chemins arêtes-disjoints entre deux sommets distincts donnés. Il peut être prouvé que le nombre chromatique d'un graphe est au plus égal à son arête-connectivité locale maximale plus un. Cette borne est une amélioration par rapport à la borne sur le degré, et il reste à caractériser les cas extrêmaux. Cela a été fait par Aboulker, Brettell, Havet, Marx et Trotignon \cite{ABHMT17} lorsque l'arête-connectivité locale maximale est au plus trois, et par Stiebitz et Toft \cite{ST16} pour des arête-connectivités locales maximales plus grandes. Un résultat similaire pour les hypergraphes a été trouvé par Schweser, Stiebitz et Toft \cite{SST22} lorsque l'arête-connectivité locale maximale est d'au moins trois. Dans le chapitre \ref{chpt:lambda_brooks}, nous prouvons un résultat similaire pour les graphes dirigés lorsque l'arc-connectivité locale maximale est d'au moins trois, et nous décrivons un algorithme polynomial pour reconnaître les cas extrêmes. Il s'agit d'un travail conjoint avec mes directeurs de thèse, Pierre Aboulker et Pierre Charbit, et dont un préprint est accessible sur arXiv~\cite{AAC23}.

\bigskip
La deuxième partie de cette thèse examine l'impact de la restriction de la structure des graphes dirigés en interdisant certains sous-graphes induits orientés. Nous nous demandons quels ensembles finis de graphes dirigés ont la propriété que les graphes dirigés ne les contenant pas ont un nombre dichromatique borné. Dans le cas non orienté, \Gya et Sumner ont conjecturé qu'il est nécessaire et suffisant pour de tels ensembles de contenir un graphe complet et une forêt. Aboulker, Charbit et Naserasr ont proposé un analogue dirigé de cette conjecture dans \cite{ACN21}, qui est plus complexe à énoncer. Dans le chapitre \ref{chpt:gyarfas}, nous résumons leur travail.

\medskip
Dans le chapitre \ref{chpt:quasi_transitive}, nous résolvons un cas de cette conjecture en caractérisant les héros dans les graphes orientés quasi-transitifs. Ce résultat est obtenu en utilisant un théorème de décomposition pour les graphes orientés quasi-transitifs. Ce travail conjoint avec mes directeurs de thèse Pierre Aboulker et Pierre Charbit, fait partie d'un article soumis, et une prépublication est disponible sur arXiv \cite{AACmulti}.

\medskip

Dans le chapitre~\ref{chpt:multipartite}, nous résolvons un autre cas de cette conjecture, qui équivaut à caractériser les graphes dirigés tels que les orientations de graphes multipartis complets qui ne les contiennent pas ont un nombre dichromatique borné. Nous obtenons une caractérisation complète de ces graphes dirigés, en utilisant entre autres une réduction à un problème sur les graphes ordonnés. Ce travail conjoint avec mes directeurs de thèse Pierre Aboulker et Pierre Charbit est le résultat principal d'un article dont le préprint est disponible sur arXiv \cite{AACmulti}.

\medskip
Dans le chapitre~\ref{chpt:semiround}, nous abordons un autre cas de la conjecture en prouvant que les graphes orientés out-transitifs ont un nombre dichromatique borné. Nous y parvenons en utilisant un théorème de décomposition pour les graphes orientés out-transitifs, puis l'utilisons pour prouver un cas spécial de la conjecture de Caccetta-\Hag. Nous prouvons également un théorème de décomposition similaire pour les graphes dirigés localement semi-complets, et l'appliquons pour prouver des résultats mineurs sur cette classe. Ce travail, réalisé en collaboration avec mes directeurs de thèse Pierre Aboulker et Pierre Charbit, est publié dans la revue European Journal of Combinatorics \cite{AAC21}.

\medskip
Dans le chapitre~\ref{chpt:psix}, nous prouvons que les graphes orientés sans clique à trois sommets et sans chemin dirigé induit sur six sommets ont un nombre dichromatique borné en nous appuyant sur les ensembles dipolaires, un outil utile pour obtenir des bornes supérieures sur le nombre dichromatique. Ce travail, réalisé en collaboration avec Pierre Aboulker, Pierre Charbit et Stéphan Thomassé a été soumis et un préprint est disponible sur arXiv \cite{AACT22}.

\medskip
Le chapitre~\ref{chpt:chordal} est entièrement consacré à la caractérisation des graphes orientés tels que les orientations des graphes cordaux ne les contenant pas ont un nombre dichromatique borné. Nous décrivons d'abord deux constructions d'orientations de graphes cordaux avec un nombre dichromatique non borné, et observons que seuls quelques rares graphes orientés apparaissent dans les deux constructions. Nous prouvons ensuite que les orientations de graphes cordaux ne contenant pas ces graphes orientés ont un nombre dichromatique borné, obtenant ainsi une caractérisation. Ce travail, réalisé en collaboration avec Pierre Aboulker et Raphael Steiner, est publié dans le Journal of Discrete Mathematics de la SIAM \cite{AAS22}.

\bigskip

La dernière partie de cette thèse diffère du reste de ce manuscrit en ceci qu'elle ne concerne pas les graphes orientés. Dans le chapitre~\ref{chpt:defective}, nous considérons le problème consistant à colorer les arêtes d'un multigraphe de sorte que pour tout sommet et toute couleur, au plus $d$ arêtes incidentes à ce sommet utilisent cette couleur, avec $d$ un entier fixé. Si $d = 1$, cela revient au problème classique de coloration d'arêtes. Nous prouvons des bornes optimales sur le nombre de couleurs nécessaires, relativement à $d$ et au degré maximum, généralisant ainsi pour toute valeur de $d$ la borne correspondante trouvée par Shannon \cite{S49} quand $d=1$. Nous considérons ensuite ce problème sur les graphes simples, et pour chaque valeur de $d$ et du degré maximum, soit nous prouvons que le problème est NP-complet, soit nous proposons un algorithme polynomial. Ce travail, réalisé en collaboration avec Pierre Aboulker et Chien-Chung Huang, est publié dans le Journal of Electronic Combinatorics \cite{AAH22}.

\endgroup

\vfill

\cleardoublepage

\thispagestyle{empty}

\newgeometry{top=0.8in,bottom=0.8in}

\noindent\textbf{Titre :} Coloration de graphes dirigés

\vspace{0.5cm}

\noindent\textbf{Résumé :}

\vspace{0.4cm}
{\setstretch{0.2}\small

Les réseaux sont devenus omniprésents dans notre vie quotidienne, qu'il s'agisse de réseaux sociaux, de réseaux de neurones ou de réseaux routiers. Pourtant, les graphes, leur pendant théorique, sont utilisés depuis des siècles pour modéliser des problèmes pratiques. Un graphe est un ensemble de sommets reliés par des arêtes. Si on considère des arêtes orientées, on parlera plutôt de digraphes. L'un des concepts les plus féconds de la théorie des graphes, appliqué aussi bien à des problèmes d'allocation de registres qu'à l'attribution de fréquences radio, est la coloration de graphes, qui consiste à attribuer des couleurs aux sommets de manière à ce que deux sommets adjacents aient des couleurs distinctes. Le nombre chromatique d'un graphe est alors le nombre minimum de couleurs nécessaires. Cette thèse s'intéresse au nombre dichromatique, une métrique introduite en 1982 par Neumann-Lara comme équivalent du nombre chromatique, mais pour les digraphes. Colorer un digraphe, c'est attribuer une couleur à chacun de ses sommets de sorte qu'aucun cycle dirigé ne soit monochromatique, et le nombre dichromatique d'un digraphe est le nombre minimum de couleurs nécessaires. Des résultats récents suggèrent que cette métrique est la bonne notion de coloration dans le cas dirigé. Le but de cette thèse est d'étudier comment la structure d'un digraphe affecte son nombre dichromatique.

\vspace{0.2cm}

Dans la première partie de ce travail, nous examinons comment le nombre dichromatique interagit avec d'autres métriques. Tout d'abord, nous considérons le degré, c'est-à-dire le nombre maximum de voisins d'un sommet. Dans le cas non dirigé, cela correspond au théorème de Brooks, un théorème célèbre avec de nombreuses variations et généralisations. Dans le cas des digraphes, il n'existe pas de métrique naturelle correspondant au degré maximal. Nous étudions donc comment différentes notions de degré conduisent soit à des théorèmes de type Brooks, soit à des résultats d'impossibilité. Nous étudions également l'arc-connectivité maximale, une métrique plus générale, fournissons un théorème semblable au théorème de Brooks pour cette métrique ainsi qu'un algorithme polynomial pour reconnaître les cas extrêmaux.

\vspace{0.2cm}
La deuxième partie de ce manuscrit se concentre sur un analogue dirigé de la conjecture de \Gya-Sumner, qui essaie de caractériser les ensembles S de graphes tels que les graphes ayant un nombre chromatique suffisamment grand contiennent  un  graphe  de  S.  Cette  conjecture  reste  largement  ouverte.  Pour  les  digraphes,  une  conjecture correspondante a été proposée par Aboulker, Charbit et Naserasr. Nous prouvons plusieurs cas de cette conjecture, principalement en démontrant que certaines classes de digraphes ont un nombre dichromatique borné. Par exemple, nous prouvons que les graphes orientés quasi-transitifs et localement out-transitifs ont un petit nombre dichromatique. Nous caractérisons également les digraphes qui doivent apparaître dans les orientations des graphes multipartis complets avec un nombre dichromatique suffisamment grand et, ce faisant, nous découvrons un contre-exemple à la conjecture initiale d'Aboulker, Charbit et Naserasr. Nous obtenons des résultats similaires pour les digraphes sans triangle et sans chemins dirigés sur six sommets, ainsi que pour les orientations des graphes cordaux. 

\vspace{0.2cm}
Dans la dernière partie de cette thèse, nous abordons le problème de l'arête-coloriage d-défectueux, qui consiste à colorer les arêtes d'un multigraphe de telle sorte que, pour tout sommet, aucune couleur n'apparaisse sur plus de d de ses arêtes incidentes. Lorsque d est égal à un, cela correspond au problème de l'arête-coloration. Shannon a trouvé une borne stricte sur le nombre de couleurs nécessaires par rapport au degré maximal lorsque d est égal à un, et nous étendons ce résultat à toute valeur de d. Nous explorons également ce problème sur des graphes simples et prouvons des résultats qui étendent le théorème de Vizing à toute valeur de $d$.

}

\vspace{0.6cm}

\noindent\textbf{Mots-clefs :} digraphe, graphe, dirigé, coloration, dichromatique, réseaux

\restoregeometry
\newpage
\thispagestyle{empty}
\newgeometry{top=0.8in,bottom=0.8in}

\noindent\textbf{Title:} Colouring digraphs

\vspace{0.5cm}

\noindent\textbf{Abstract:}

\vspace{0.4cm}
{\setstretch{0.2}\small Networks are ubiquitous in our daily life, whether they are social networks, neural networks, road networks or computer networks. Yet, graphs, their theoretical pendant, have been used for centuries to model real-life problems. A graph is a set of vertices with edges connecting them. In many applications, it is useful to give edges a direction, thus obtaining a digraph (short for directed graph). One of the most fertile concepts of graph theory (applied in a wide range of practical problems, from register allocation to mobile radio frequency assignment) is graph colouring, that consists in assigning colours to vertices so that adjacent vertices get distinct colours. The chromatic number of a graph is then the minimum number of colours required. This thesis examines the dichromatic number, a metric introduced in 1982 by Neumann-Lara as a counterpart to the chromatic number for digraphs. Colouring a digraph consists in assigning a colour to each of its vertices so that no directed cycle is monochromatic, and the dichromatic number of a digraph is the minimum number of colours needed for such a colouring. Recent results suggest that this metric is the appropriate analogue for the corresponding metric on undirected graphs. The aim of this thesis is to investigate how the structure of a digraph affects its dichromatic number and to extend various results on undirected colouring to digraphs.

\vspace{0.2cm}
In the first part of this work, we examine how the dichromatic number interacts with other metrics. First, we consider the degree, which is the maximum number of neighbours of a vertex. In the undirected case, this corresponds to Brooks' theorem, a celebrated theorem with multiple variations and generalizations. In the directed case, there is no natural metric corresponding to the maximum degree, so we explore how different notions of maximum directed degree lead to either Brooks-like theorems or impossibility results. We also investigate the maximum local-arc connectivity, a metric that encompasses several degree-like metrics. We demonstrate that the dichromatic number of a digraph is upper-bounded by one plus its maximum local-arc connectivity, characterize extremal digraphs, and provide a polynomial algorithm to recognize them.

\vspace{0.2cm}
The second part of this manuscript focuses on a directed analogue of the \Gya-Sumner conjecture. The \Gya-Sumner conjecture tries to characterize sets S of undirected graphs such that graphs with large enough chromatic number must contain a graph of S. This conjecture is still largely open. On digraphs, a corresponding conjecture was proposed by Aboulker, Charbit, and Naserasr. We prove several subcases of this conjecture, mainly demonstrating that certain classes of digraphs have bounded dichromatic number. For instance, we prove that quasi-transitive and locally out-transitive oriented graphs have a small dichromatic number. We also characterize digraphs that must appear in orientations of complete multipartite graphs with large enough dichromatic number and, in doing so, discover a counterexample to the initial conjecture of Aboulker, Charbit, and Naserasr. We obtain similar results for digraphs with no triangle and no directed paths on six vertices, as well as for orientations of chordal graphs.

\vspace{0.2cm}
In the last part of this thesis, we address the d-edge-defective-colouring problem, which involves colouring edges of a multigraph such that, for any vertex, no colour appears on more than d of its incident edges. When d equals one, this corresponds to the infamous edge-colouring problem. Shannon established a tight bound on the number of colours needed relative to the maximum degree when d equals one, and we extend this result to any value of d. We also explore this problem on simple graphs and prove results that extend Vizing's theorem to any value of d.

}

\vspace{0.6cm}

\noindent\textbf{Keywords:} digraph,  directed,  graph,  colouring,  coloring,  dichromatic, networks\\

\restoregeometry












\pdfbookmark[1]{Acknowledgements}{Acknowledgements} 

\begin{flushright}{\slshape    
Travail terminé !} \\ \medskip
--- Péon, \emph{Warcraft}
\end{flushright}
\smallskip


\begingroup

\let\clearpage\relax
\let\cleardoublepage\relax
\let\cleardoublepage\relax

\chapter*{Remerciements}

\bigskip

Mes premiers remerciements vont à Stéphane Bessy et Alexander Scott qui ont accepté la tâche ingrate et fastidieuse de relire attentivement ce manuscrit. Je voudrais aussi remercier Marthe Bonamy, Vincent Cohen-Addad et Matěj Stehlík pour avoir accepté de faire partie de mon jury.

\medskip

Évidemment, rien de tout ceci n'aurait été possible sans l'encadrement de Pierre, mais aussi de Pierre. Les conseils avisés de Pierre (par exemple, « rentre plus dans les détails », ou « ne travaille pas sur d'autres problèmes avant d'avoir écrit ce que tu as trouvé ») et de Pierre (par exemple, « rentre moins dans les détails », ou « ne travaille pas sur d'autres problèmes avant d'avoir écrit ce que tu as trouvé ») m'ont permis de progresser en tant que chercheur, et je pense avoir réussi à les faire miens (à part peut-être pour ce qui est d'écrire mes résultats avant de m'attaquer à d'autres problèmes). Pas un instant durant ces trois années je n'ai regretté d'avoir quitté mon poste d'ingénieur pour une thèse, et vous y êtes pour beaucoup. Je vous remercie pour nos discussions scientifiques, vos conseils avisés, mais aussi pour les bons moments — moins scientifiques — que l'on a partagés !

\medskip

J'ai eu la chance de rencontrer une (infime) partie de la communauté graphes, en France et à l'international, et je voulais exprimer ma gratitude pour toutes les collaborations plus ou moins fructueuses dont j'ai pu profiter. Il y a ceux que j'ai rencontrés dans l'Hérault ou dans les Cévennes : les jeunes (Amadeus, Clément, Colin, Florian, Fred, Hugo, Julien, Quentin et Yann) et les moins jeunes (Emeric, Nicolas de Lyon, Nicolas de Nisse, Stéphan et Ararat — et son terrifiant cognac arménien). Mais aussi celles et ceux avec qui j'ai travaillé en présentiel et en distanciel : Aline, Cléophée, Eun-Jung, Hélène, Nicolas, Raphael et Sarah.

\smallskip
Je voulais aussi remercier toutes les personnes de l'IRIF et du DIENS avec qui j'ai pu partager de précieux moments : Chien-Chung, Gabriel, Garance — qui m'a rendu addict à \emph{Fantasy Realms} —, Michel — qui a partagé mon addiction à \emph{Fantasy Realms} —, Mickaël, Mónika, Olivier, Reza, Tatiana et Yann.

\bigskip

Je souhaiterais conclure en remerciant des personnes qui me sont chères sans être directement impliquées dans mon travail.

\medskip

Celles et ceux qui me connaissent savent que les compétitions d'algorithmique ont occupé une grande partie de mon temps libre et moins libre. Grâce à celles-ci, j'ai voyagé avec des personnes formidables qui partagent ma passion, dans des lieux exotiques comme Milan, Porto ou Gif-Sur-Yvette.
Il y a bien sûr le club algo de l'ENS (Jill-Jenn le coach, Christoph le coach-en-chef, Clément le co-co-coach, Hugo le peut-être-futur-coach, Clémence, Cyril, Émile et nos débats acharnés, Étienne, Garance, Lê Thành Dũng — qui a sauvé ma L3 en m'apprenant les $\lambda$-choses —, Lucas, Marin le troll, Olivier, Rémi, Shendan, Stéphane et Victor) et de manière générale la communauté algorithmique française (Pierre avec ses problèmes trop durs, Augustin avec ses problèmes vraiment trop durs, Noé avec ses problèmes beaucoup trop durs, Christophe que j'aimerais bien rencontrer un jour, Cup avec son clavier de l'enfer, Clément qui m'a initié aux joies de la comptabilité, et évidemment Mathis,  Gaëtan, et Rédouane, les compères milanais).

\smallskip

Je tiens aussi à exprimer ma reconnaissance envers la clique des médeux : Aaron et ses goûts cinématographiques sûrs, Antoine et son humour délicat, Corentin, sobre parmi les sobres, Emma et ses choix de carrière assurés, Henriette l'hyperactive, Quentin le séducteur,  Manon la taciturne et Amanda la bavarde, Vincent le jeune, Zéfyr le désordonné, Xavier qui m'a tiré au labo à bout de bras et enfin, Elyes, avec qui j'ai passé tant d'après-midi à travailler assidûment. J'ose à peine imaginer comment j'aurais vécu les moments de doute qui accompagnent inévitablement une thèse — d'autant plus en période de COVID-19 — sans chacun d'entre vous.

\smallskip

Je souhaite également exprimer ma gratitude envers ma mère, mon père ainsi que mes trois frère et sœurs.

\smallskip
Je conclurais en remerciant celle qui m'a supporté toutes ces années, qui m'a encouragé dans mes succès et réconforté dans mes échecs, avec qui je partage mes goûts de vieux et une tisane devant \emph{Questions pour un champion} ou \emph{Only Connect}.

Camille, j'espère partager ta vie encore longtemps.

\bigskip
Vous êtes tous plus chers à mes yeux que vous ne pouvez l'imaginer, et il y a un peu de vous tous dans cette thèse. Merci beaucoup.

\endgroup 






\cleardoublepage

\pdfbookmark[1]{Publications}{Publications} 

\chapter*{Publications} 



\noindent For the coherence of this thesis, not all of my works were included in this manuscript. All the results obtained during my PhD are listed below.

\section*{Published papers}

\begin{itemize}
\item \fullcite{AA22}

\item \fullcite{AAH22}

\item \fullcite{AAS22}

\item \fullcite{AAC21}
\end{itemize}

\section*{Preprints}

\begin{itemize}

\item \fullcite{AACT22}

\item \fullcite{AACmulti}

\item \fullcite{AHHKNRV22}

\item \fullcite{AAC23}
\end{itemize}




\pagestyle{scrheadings} 

\cleardoublepage

\refstepcounter{dummy}

\pdfbookmark[1]{\contentsname}{tableofcontents} 

\setcounter{tocdepth}{2} 

\setcounter{secnumdepth}{3} 

\manualmark
\markboth{\spacedlowsmallcaps{\contentsname}}{\spacedlowsmallcaps{\contentsname}}
\tableofcontents 
\automark[section]{chapter}
\renewcommand{\chaptermark}[1]{\markboth{\spacedlowsmallcaps{#1}}{\spacedlowsmallcaps{#1}}}
\renewcommand{\sectionmark}[1]{\markright{\thesection\enspace\spacedlowsmallcaps{#1}}}

\clearpage


\cleardoublepage

\pagenumbering{arabic} 
\setcounter{page}{21}

\cleardoublepage 



\part{Prolegomena} 
\chapter{Introduction}

\section{From networks to graphs}

In the last decades, networks have become ubiquitous in our everyday life. Of course, most of this could only happen, for theoretical aspects of computer science were studied and developed in the mean times, but also long before. These networks may be social networks, neural networks or computer networks, such as the Internet, but networks can also be found in seemingly unrelated domains, with networks such as road networks, powerline networks or metabolic networks.
\smallskip

In computer science and mathematics, when studied from a theoretical point of view, networks tend to be called graphs, and their study is called graph theory. While previously mentioned applications are rather recent, graph theory predates them by centuries. In 1736, Euler figured out that when studying how to cross every bridge of K\"onigsberg (now Kaliningrad) exactly once, only mattered which islands were linked with a bridge. Graph theory was born, even though the term was only coined 142 years later by Sylvester.

\medskip
 
In 1852, Guthrie figured out that the map of English counties could be coloured using only four distinct colours so that no two counties with a common border received the same colour, and wondered whether it was true of all maps. This conjecture was quickly found to be equivalent to the corresponding conjecture on graphs, solved more than a century later with a computer-assisted proof.

\begin{theorem}[Four colours theorem (Appel, Haken, 1977, \cite{AH77})]\label{thm:four_colours}
Every planar graph can be coloured with four colours so that no two adjacent vertices receive the same colour.
\end{theorem}

This gave rise to the study of the chromatic number of graphs, that is the minimum number of colours needed to colour a graph so that no two adjacent vertices receive the same colour. In particular, an important question is to study how the chromatic number relates to the structure of a given graph.

\medskip

One of the first examples is the construction of graphs with large chromatic number but no triangle by Tutte (writing as Blanche Descartes) in \cite{D54}. As a complete graph on $n$ vertices has chromatic number $n$, the chromatic number of a graph is always at least its clique number. With its construction, Tutte showed on the other hand that the chromatic number of graphs is not upper-bounded by a function of their clique number. However, it remains interesting to study the structure of graphs in which the two notions are linked, or even in which they are equal.

\smallskip
A graph whose chromatic number is equal to its clique number, and whose every induced subgraphs satisfy this property, is said to be perfect. In \cite{L72}, Lov\'asz proved that if a graph is perfect, so is its complementary. Since odd cycles on at least five vertices are not perfect, a perfect graph has no odd hole nor odd antihole. Berge conjectured in its infamous Strong Perfect Graph Conjecture that the reciprocal holds. It was finally proven true by Chudnovsky, Robertson, Seymour and Thomas in \cite{Perfect}.

Outside of this tight case, an interesting case is that of graphs whose chromatic number is bounded by a function of their clique number. Such classes of graphs are said to be $\chi$-bounded. \Gya and Sumner independently conjectured that, for any forest $F$, the class of graphs with no induced copy of $F$ is $\chi$-bounded. This conjecture still remains largely open.

\section{Colouring digraphs}

In 1982, Neumann-Lara introduced in \cite{NL82} a directed analogue of the usual colouring of graphs. In this setting, colouring a directed graph $D$ consists in partitioning its vertices into sets inducing acyclic directed graphs and, similarly to graphs, the dichromatic number of $D$, denoted $\dic(D)$, is the minimum size of such a partition. Note that this notion could be considered natural for replacing each edge of an undirected graph $G$ by two opposite arcs yields a directed graph $\overleftrightarrow{G}$ which satisfies $\dic(\overleftrightarrow{G}) = \chi(G)$.

In 2001, this notion was re-introduced by Mohar in \cite{M03}. He then went on to prove many results regarding dichromatic number among which the following: given a digraph $D$ with adjacency matrix $M_D$, $\dic(D)$ is upper-bounded by one plus the largest modulus of an eigenvalue of $M_D$. As the exact same result exists for undirected graphs, this corroborates that this definition of the dichromatic number may be the right notion, and we can hope to generalize results on the chromatic number of undirected graphs to directed graphs via the dichromatic number.

Such results have then been found in various areas of graph theory such as extremal graph theory~\cite{BBSS20, HK15, KS20}, algebraic graph theory~\cite{M10}, substructure forced by large dichromatic number~\cite{AAC21, ACN21, ACL19, hero,GSS20, HLNT19, S21}, list dichromatic number~\cite{BHL18, HM11}, dicolouring digraphs on surfaces~\cite{AHKR21, LM17, S19}, flow theory~\cite{H17, KV12}, links between dichromatic number and girth~\cite{HM12, S20}.
\medskip

In particular, an analogue of Theorem~\ref{thm:four_colours} has been conjectured by Neumann-Lara:

\begin{conjecture}[Two colours conjecture]
Oriented planar graphs have dichromatic number at most $2$.
\end{conjecture}

This conjecture is still open.





\medskip

This thesis is devoted to the study of the dichromatic number of directed graphs. The central question we aim to answer is how the structure of a directed graph impacts its dichromatic number. We draw inspiration from the undirected case, where similar analyses have driven the study of the chromatic number.

\section{Outline of this thesis}

\bigskip
In the first part of this manuscript, we study classical metrics and how they affect the dichromatic number. In particular, we aim to upper-bound the dichromatic number by a function of these metrics. One of the key metrics we consider is the maximum degree. On undirected graphs, the infamous Brooks' theorem \cite{B41} states that the chromatic number of a connected graph is at most one plus its maximum degree and that equality is only obtained on complete graphs and odd cycles. When looking for similar results in the directed case, a problem arises: there are multiple notions corresponding to the degree of directed graphs.

\medskip
In Chapter~\ref{chpt:brooks}, we delve into the relationship between different notions of degrees and the dichromatic number. We begin by examining the maximum maxdegree, and while a characterization of tight cases was already found by Harutyunyan and Mohar in \cite{HMGallai}, we provide multiple proofs for the same result. We then turn our attention to the maximum mindegree and obtain an impossibility result for a simple characterization of tight cases for this metric. This work, done in collaboration with Pierre Aboulker, was published in \cite{AA22}.

\medskip
We argue that our approach is natural since Brooks' theorem is central to graph colouring and has given rise to multiple generalizations using different notions of colouring, such as \cite{RPJ75} for hypergraphs, and stronger metrics than the degree metric. One such metric is the maximum local edge-connectivity, which is defined as the maximum number of edge-disjoint paths between two distinct vertices. It can be proven that the chromatic number of a graph is at most one plus its maximum local edge-connectivity. This bound is an improvement over the corresponding degree-based bound, and it remains to characterize the tight cases. This was done by Aboulker, Brettell, Havet, Marx and Trotignon \cite{ABHMT17} when maximum local edge-connectivity is at most three, and by Stiebitz and Toft \cite{ST16} for larger maximum local edge-connectivities. A similar result for hypergraphs was found by Schweser, Stiebitz and Toft \cite{SST22} when the maximum local edge-connectivity is at least three. In Chapter~\ref{chpt:lambda_brooks}, we prove a similar result for directed graphs when the maximum local edge-connectivity is at least three, and provide a polynomial algorithm to recognize tight cases. This is a joint work with my advisors Pierre Aboulker and Pierre Charbit has been submitted, and a preprint is available on arXiv~\cite{AAC23}.

\bigskip
The second part of this thesis examines the impact of restricting the structure of directed graphs by forbidding certain induced subdigraphs. We ask ourselves which finite sets of directed graphs are such that directed graphs not containing them have bounded dichromatic number. In the undirected case, \Gya and Sumner conjectured that it is necessary and sufficient for such sets to contain a complete graph and a forest. Aboulker, Charbit and Naserasr have proposed a directed analogue of this conjecture in \cite{ACN21}, which is more involved in its statement. In Chapter~\ref{chpt:gyarfas}, we summarize their work.

\medskip
In Chapter~\ref{chpt:quasi_transitive}, we solve one case of this conjecture by characterizing heroes in quasi-transitive oriented graphs. This result is obtained using a decomposition theorem for quasi-transitive oriented graphs. This joint work with my advisors Pierre Aboulker and Pierre Charbit is part of a submitted paper and a preprint is available on arXiv \cite{AACmulti}.

\medskip
In Chapter~\ref{chpt:multipartite}, we solve another case of this Conjecture, which is equivalent to characterizing which directed graphs are such that oriented complete multipartite graphs not containing them have bounded dichromatic number. We obtain a complete characterization of such directed graphs, using among other things a reduction to a problem on ordered graphs. This joint work with my advisors Pierre Aboulker and Pierre Charbit is the main result of a submitted paper, whose preprint is available on arXiv \cite{AACmulti}.

\medskip
In Chapter~\ref{chpt:semiround}, we address another case of the conjecture by proving that out-transitive oriented graphs have  bounded dichromatic number. We accomplish this by using a decomposition theorem for out-transitive oriented graphs and use this decomposition theorem to prove a special case of Caccetta-\Hag Conjecture. We also prove a similar decomposition theorem for locally semicomplete directed graphs and apply it to prove minor results on this class. This work, done in collaboration with my advisors Pierre Aboulker and Pierre Charbit, is published in the European Journal of Combinatorics \cite{AAC21}.

\medskip
In Chapter~\ref{chpt:psix}, we prove that oriented graphs with no clique on three vertices and no induced directed path on six vertices have a bounded dichromatic number by relying on dipolar sets, a useful tool for obtaining upper bounds on dichromatic numbers. This work, done in collaboration with Pierre Aboulker, Pierre Charbit, and Stéphan Thomassé, has been submitted and a preprint is available on arXiv \cite{AACT22}.

\medskip
In Chapter~\ref{chpt:chordal}, we characterize which directed graphs are such that orientations of chordal graphs not containing them have bounded dichromatic number. We first describe two constructions of oriented chordal graphs with unbounded dichromatic number, and observe that only a handful of directed graphs appear in both constructions. We then prove that oriented chordal graphs not containing these directed graphs have bounded dichromatic number, thus obtaining a characterization. This work, done in collaboration with Pierre Aboulker and Raphael Steiner, is published in the SIAM Journal of Discrete Mathematics \cite{AAS22}.

\bigskip
The last part of this thesis is different in that it does not concern directed graphs. In Chapter~\ref{chpt:defective}, we consider the problem of colouring edges of a multigraph so that for any vertex and any colour, at most $d$ edges incident with this vertex share that colour, with $d$ a fixed integer. If $d = 1$, this amounts to the classical edge-colouring problem. We prove tight bounds on the number of colours needed, relative to $d$ and the maximum degree, thus generalizing to any value of $d$ the corresponding bound of Shannon \cite{S49} when $d=1$. We then consider the corresponding problem on simple graphs, and for each value of $d$ and the maximum degree, we either prove that the problem is NP-complete or provide a polynomial algorithm. This work, done in collaboration with Pierre Aboulker and Chien-Chung Huang, is published in the Electronic Journal of Combinatorics \cite{AAH22}. 
\chapter{Definitions and Notations}

\emph{In this chapter, we give some definitions and notations that we will use along the document. Most of them follow from classical textbooks such as \cite{BG18}, \cite{BM08} or \cite{D05}.}

\section{Graphs}

If $V$ is a set and $k$ a positive integer, we note $V \choose k$ the set of subsets of exactly $k$ elements of $V$.
A \emph{graph}, or \emph{unoriented graph}, is a pair $G=(V,E)$ of finite sets such that $E$ is a subset of  $V \choose 2$. If $E$ is instead a multisubset of $V \choose 2$, $G = (V,E)$ is called a \emph{multigraph}.

\medskip

In the rest of this chapter, we let $G$ be a graph.

\medskip

\textbf{Vertices and edges.}
Elements of $V$ are the \emph{vertices} of $G$, elements of $E$ are its \emph{edges}. 
For notational simplicity, we write $uv$ for the unordered pair $\{u,v\}$.
The vertex set of a graph $G$ is referred to as $V(G)$, its edge set as $E(G)$. 
We refer to the number of vertices of a graph as the \emph{order} of the graph, and to the number of edges of a graph as the \emph{size} of the graph.

\medskip

\textbf{Adjacency.}
Let $u$ and $v$ two distinct vertices of $G$ and $X$ a subset of $V(G)$.
If $uv$ is an edge of $G$, we say that $u$ is \emph{adjacent} to $v$ or that $u$ is a \emph{neighbour} of $v$. 
If $uv \notin E(G)$, we say that $u$ is \emph{non-adjacent} to $v$, or that $u$ is a \emph{non-neighbour} of $v$. 
We denote by $N(v)$  the \emph{neighbourhood} of $v$, that is the set of neighbours of $v$ in $G$. 
We denote by $N(X)$ the set of vertices of $V(G) \setminus X$ that see at least one vertex in $X$, and $N[X]=N(X) \cup X$. 

\medskip

\textbf{Notations.}
Let $X \subseteq V(G)$, we denote $G \setminus X = (V(G) \setminus X, E(G) \setminus \{uv \mid u \in X, v \in N(u)\})$. If $X = \{x\}$, then we rather use $G \setminus x$ instead of the cumbersome $G \setminus \{x\}$.
Let $X \subseteq E(G)$, we denote $G - X = (V(G), E(G) \setminus X)$.
If $X = \{x\}$, then we rather use $G - x$ instead of the cumbersome $G - \{x\}$.

\medskip

\textbf{Degree.}
The \emph{degree} of $v$ in $G$, denoted by $d(v)$ is the number of neighbours of $v$ in $G$. The maximum degree of a vertex in $G$ is denoted by $\Delta(G)$. $G$ is said to be \emph{$k$-regular} if for every vertex $v$, $d(v)=k$.

\subsection{Subgraphs and induced subgraphs}

In this subsection, let $F$ be a graph. 

\medskip

\textbf{Subgraph.}
We say that $G$ and $F$ are \emph{isomorphic} if there exists a bijection $\varphi : V(G) \ra V(F)$ such that $uv \in E(G) \Leftrightarrow \varphi(x)\varphi(y) \in E(F)$ for all $u$, $v$ in $V(G)$.  
We do not distinguish between isomorphic graphs and write $G=F$ if $G$ and $F$ are isomorphic. 
If $V(F) \subseteq V(G)$ and $E(F) \subseteq E(G)$, then $F$  is a \emph{subgraph} of $G$. If $F$ is a subgraph of $G$ and $F \neq G$, then $F$ is a \emph{proper subgraph} of $G$.

\medskip

\textbf{Induced subgraph.}
If $X$ is a subset of $V(G)$, we denote by $G[X]$ the graph that has $X$ as vertex set and ${X \choose 2} \cap E(G)$ as edge set. 
We say that $G[X]$ is the \emph{subgraph of $G$ induced by $X$}. 
If there exists  $X  \subseteq V(G)$ such that $G[X]$ is isomorphic $F$, we say that $F$ is \textit{an induced subgraph of $G$}.
If $F$ is a subgraph (resp.\ an induced subgraph) of $G$, we say that $G$ \textit{contains} (or \textit{admits}) $F$ \textit{as a subgraph} (resp.\ \textit{as an induced subgraph}). 

\medskip

\textbf{Hereditary.}
In this document, we say that $G$ is  \emph{$F$-free} if $G$ does not contain $F$ as an induced subgraph.
Let $\mc{F}$ a class of graphs. We say that $G$ is $\mc{F}$\emph{-free} if for any graph $F \in \mc F$, $G$ is $F$-free.
We denote by $\F(\mc{F})$ the set of all $\mc{F}$-free graphs, and by $\F(F)$ the set of all $F$-free graphs.
 A class of graphs $\cal C$ is \textit{hereditary} if for any graph $G$ in $\cal C$, every induced subgraph $H$ of $G$ belongs to $\cal C$. 
It is clear that a class of graphs defined by forbidding subgraphs or induced subgraphs is hereditary.

\subsection{Connectivity}

\textbf{Path.}
$P$ is a {\em path of $G$} if it is a sequence of distinct vertices $x_1x_2\ldots x_k$,
$k\geq 1$, such that $x_ix_{i+1} \in E(G)$ for all $1\leq i < k$.
Edges $x_ix_{i+1}$,  for  $1\leq i < k$, are called the {\em edges of $P$}.
The \emph{length} of a path is the number of its \emph{edges}.
Vertices $x_1$ and $x_k$ are the {\em endvertices} of $P$, and 
$x_2\ldots x_{k-1}$ is the {\em interior} of $P$.
$P$ is referred to as a {\em $p_1p_k$-path}.

\medskip

\textbf{Local (edge-)connectivity.}
Two paths $P_1$ and $P_2$ that share their endvertices are said to be \emph{internally edge-disjoint} if their edges are disjoint. They are said to be \emph{internally disjoint} if their interior are disjoint. For two distinct vertices $u$ and $v$, the local connectivity between $u$ and $v$, denoted $\kappa(u,v)$, is the maximum number of mutually internally disjoint paths.  The \emph{local edge-connectivity between $u$ and $v$}, denoted $\lambda(u,v)$, is the maximum number of mutually internally edge-disjoint paths. The \emph{maximum local edge-connectivity of $G$}, denoted $\lambda(G)$, is the maximum of $\lambda(u,v)$ for all pairs of distinct vertices $u$ and $v$.


\medskip

\textbf{k-connected.}
Let $k$ be an integer. $G$ is \emph{$k$-connected} if for every distinct vertices $u, v \in V(G)$, $\kappa(u,v) \geq k$. If $k = 1$, $G$ is also said to be \emph{connected}. If $k = 2$, $G$ is also said to be \emph{biconnected}.
A \emph{$k$-connected component} of $G$ is a maximal $k$-connected subgraph of $G$. $1$-connected components are also called \emph{connected components}. $2$-connected components are also called \emph{blocks}.

\subsection{Basic graph classes}

\textbf{Girth.}
$C$ is a \emph{cycle of $G$} if it is a sequence of vertices $p_1p_2\ldots p_kp_1$, $k \geq 3$, such that $p_1\ldots p_k$ is a path of $G$ and $p_1p_k \in E(G)$.
Edges $p_ip_{i+1}$,  for  $1\leq i <k$, and edge $p_1p_k$ are called the {\em edges of $C$}.
The {\em length} of a cycle of $G$ is the number of its edges. The \emph{girth of $G$} is the maximum length of a cycle of $G$. If $G$ is a forest, its girth is $+\infty$. If all induced subgraphs of $G$ have girth $3$ or $+\infty$, $G$ is said to be \emph{chordal}.

\medskip

\textbf{Star.}
A graph with no cycle is a \emph{forest}. A connected forest is a \emph{tree}. A \emph{star} is a tree with no path of length three. A \emph{forest of star} is a forest whose all connected components are stars.

\medskip

\textbf{Path and cycle graphs.}
Let $k$ be an integer. A \emph{path graph}, denoted $P_k$, is a graph isomorphic to $(\{x_1,\dots, x_k\}, \{x_i, x_{i+1} \mid 1 \leq i < k \})$. If $k \geq 3$, a \emph{cycle graph}, denoted $C_k$, is a graph isomorphic to $(\{x_1,\dots, x_k\}, \{x_i, x_{i+1} \mid 1 \leq i < k \} \cup \{x_k,x_1\})$.

\medskip

\textbf{Clique and stable.}
If $E(G) = \emptyset$, $G$ is said to be a \emph{stable graph}, and denoted $\ova{K}_n$ where $n$ is order of $G$. An \emph{independant} or a \emph{stable} of $G$ is a subgraph of $G$ isomorphic to a stable graph. The \emph{independance number} of $G$, denoted $\alpha(G)$, is the maximum order of a stable of $G$.
$G$ is said to be a \emph{complete graph} if $E(G) = V(G) \choose 2$. A clique of $G$ is a subgraph of $G$ isomorphic to a complete graph. The clique number of $G$, denoted $\omega(G)$, is the maximum order of a clique of $G$.

\medskip

\textbf{Complete k-partite.}
$G$ is said to be a {\em complete k-partite graph} if $V(G)$ can be partitioned into $k$ non-empty subsets $A_1, \dots, A_k$ such that, for $i=1, \dots, k$, $A_i$ is a stable set and, for any $\{i,j\} \subseteq \{1, \dots, k\}$, there are all possible edges between $A_i$ and $A_j$.

\medskip

\textbf{Wheel.}
$G$ is said to be a \emph{wheel} if there exists a universal vertex $v$ such that $G[V(G) \setminus \{v\}]$ is a cycle graph. If $G[V(G) \setminus \{v\}]$ is an odd cycle, $G$ is said to be an \emph{odd wheel}.


\medskip

\textbf{Chromatic number.}
Let $k$ be an integer. A \emph{$k$-colouring} of $G$ is a partition of $V(G)$ into $k$ subsets inducing stable graphs. Equivalently, a $k$-colouring of an undirected graph is a function $\varphi : V(G) \ra [1,k]$ such that $uv \in E(G) \Leftrightarrow \varphi(u) \neq \varphi(v)$. The chromatic number of $G$, denoted $\chi(G)$, is the least $k$ such that $G$ admits a $k$-colouring. If $\chi(G) = 2$, then $G$ is said to be \emph{bipartite}.





\section{Directed and oriented graphs}

A \emph{directed graph}, or \emph{digraph}, is a pair $D = (V,A)$ of finite sets such that $A$ is a subset of $(V \times V) \setminus \{(v,v) \mid v \in V\}$.

\medskip

In the rest of this chapter, we let $D$ be a digraph.

\medskip

\textbf{Adjacency.}
The elements of $V$ are the \textit{vertices} of $G$, the element of $A$ are its \textit{arcs}. 
For notational simplicity, we write $uv$ or $u \ra v$ for the ordered pair $(u,v)$.
The vertex set of a digraph $D$ is referred to as $V(D)$, its arc set as $A(D)$.
Let $D$ be a digraph, $u$ and $v$ two distinct vertices of $D$ and $X$ a subset of $V(D)$.
If $uv \in A(D)$, we say that $u$ \emph{sees} $v$, that $v$ \emph{is seen} by $u$, that $v$ is an \emph{out-neighbour} of $u$ or that $u$ is an \emph{in-neighbour} of $v$. 
If $uv, vu \notin A(D)$, we say that $u$ is \emph{non-adjacent} to $v$, or that $u$ is a \emph{non-neighbour} of $v$.
The sets of \emph{out-neighbours} and \emph{in-neighbours} of a vertex $v$ are denoted as $v^+$ and $v^-$, respectively. $N(v)$ denotes the set of neighbours of $v$, which is the union of $v^+$ and $v^-$. 

\medskip

\textbf{Degree.}
We use $d^+(v)$ and $d^-(v)$ to denote the number of \emph{out-neighbours} and \emph{in-neighbours} of a vertex $v$, respectively. The \emph{min-degree} of a vertex $v$, denoted as $d_{min}(v)$ is the minimum of $d^{-}(v)$ and $d^{+}(v)$. The \emph{max-degree} of a vertex $v$, denoted as $d_{max}(v)$ is the maximum of $d^{-}(v)$ and $d^{+}(v)$. The maximum of $d_{min}(v)$ over all vertices of $D$ is denoted $\Delta_{min}(v)$. The maximum of $d_{max}(v)$ over all vertices of $D$ is denoted $\Delta_{max}(v)$.  
A digraph is said to be \emph{$k$-regular} if for every vertex $v$, $d^+(v)=d^-(v)=k$.

\medskip

\textbf{Notations.}
If $X \subseteq V(D)$, we denote $D \setminus X = (V(D) \setminus X, \{uv \in A(D) \mid u,v \notin X \})$. If $X = \{x\}$, then we rather use $D \setminus x$ instead of the cumbersome $D \setminus \{x\}$.
If $X \subseteq A(D)$, we denote $D - X = (V(D), A(D) \setminus X)$.
If $X = \{x\}$, then we rather use $D - x$ instead of the cumbersome $D - \{x\}$.

\medskip

\textbf{Oriented graph.}
The \emph{underlying multigraph} of $D$ is the multigraph $(V(D),\{\{u,v\} \mid uv \in A(D)\})$. It is denoted as $\tilde D$. If for any arc $uv$ of $D$, $vu \notin A(D)$, $D$ is said to be an \emph{oriented graph}. Note that $D$ is an oriented graph if and only if its underlying multigraph is a graph.

\subsection{Subdigraphs and induced subdigraphs}

Let $D'$ be a digraph.

\medskip

\textbf{Subdigraph.}
We say that $D$ and $D'$ are \emph{isomorphic} if there exists a bijection $\varphi : V(D) \ra V(D')$ such that $uv \in A(D) \Leftrightarrow \varphi(x)\varphi(y) \in A(D')$ for all $u$, $v$ in $V(D)$.  
We do not distinguish between isomorphic digraphs and write $D=D'$ if $D$ and $D'$ are isomorphic. 
If $V(D') \subseteq V(D)$ and $A(D') \subseteq A(D)$, then $D'$  is a \emph{subdigraph} of $D$. If $D'$ is a subdigraph of $D$ and $D \neq D'$, then $D'$ is a \emph{proper subdigraph} of $D$.

\medskip

\textbf{Induced subdigraph.}
If $X$ is a subset of $V(D)$, we denote by $D[X]$ the graph that has $X$ as vertex set and ${X \times X} \cap A(G)$ as arc set. 
We say that $D[X]$ is the \emph{subdigraph of $D$ induced by $X$}. 
If there exists  $X  \subseteq V(D)$ such that $D[X]$ is isomorphic to $D'$, we say that $D'$ is \textit{an induced subdigraph of $D$}.
If $D'$ is a subdigraph (resp.\ an induced subdigraph) of $G$, we say that $D$ \textit{contains} (or \textit{admits}) $D'$ \textit{as a subdigraph} (resp.\ \textit{as an induced subdigraph}). 

\medskip

\textbf{Hereditary.}
In this document, we say that $D$ is  \emph{$D'$-free} if $D$ does not contain $D'$ as an induced subdigraph.
Let $\mc{F}$ a class of digraphs. We say that $D$ is $\mc{F}$\emph{-free} if for any digraph $F \in \mc F$, $D$ is $F$-free.
We denote by $\F(\mc{F})$ the set of all $\mc{F}$-free graphs, and by $\F(D')$ the set of all $D'$-free graphs.
 A class of digraphs $\cal C$ is \textit{hereditary} if for any digraph $D$ in $\cal C$, every induced subdigraph $H$ of $D$ belongs to $\cal C$. 

\subsection{Connectivity}

\textbf{Directed paths.}
$P$ is {\em directed path}, or \emph{dipath}, of $D$ if it is a sequence of distinct vertices $p_1p_2 \ldots p_k$,
$k\geq 1$, such that $p_i p_{i+1} \in A(D)$ for all $1\leq i < k$.
Arcs $p_i p_{i+1}$,  for  $1\leq i < k$, are called the {\em arcs of $P$}.
$P$ is referred to as a {\em $p_1p_k$-dipath}. 

\medskip

\textbf{Local arc-connectivity.}
Two dipaths $P_1$ and $P_2$ that share their endvertices are said to be \emph{internally arc-disjoint} if their arc sets are disjoint.
For two distinct vertices $u$ and $v$, the \emph{local arc-connectivity from $u$ to $v$}, denoted $\lambda(u,v)$, is the maximum number of mutually internally edge-disjoint paths. The \emph{maximum local arc-connectivity of $D$}, denoted $\lambda(D)$, is the maximum of $\lambda(u,v)$ for all pairs of distinct vertices $u$ and $v$.


\medskip

\textbf{(Strong) connectivity.}
A directed graph is called \emph{strongly connected} or \emph{strong} if there is a directed path between any pair of its vertices. It is said to be \emph{(weakly) connected} (resp. \emph{biconnected}) if its underlying multigraph is connected (resp. biconnected). The \emph{blocks} of $D$ are the maximal biconnected subdigraphs of $D$. We denote as $D + D'$ the disjoint union of $D$ and $D'$, that is the digraph $(V(D) \cup V(D'), A(D) \cup A(D'))$.

\subsection{Basic digraph classes}

\textbf{Orientations.}
$\olra{G}$ is the digraph isomorphic to $(V(G), \bigcup_{uv \in E(G)} \{uv, vu\})$. A graph of the form $\olra{G}$ for some graph $G$ is said to be \emph{symmetric}. $D$ is an \emph{orientation} of $G$ if $G$ is isomorphic to the underlying multigraph of $D$.Let $\mc{G}$ be a class of graphs (like complete, multipartite complete, chordal graphs).  $D$ is {\em symmetric $\mc{G}$} if it is isomorphic to $\ovlra{G}$ for some graph $G \in \mc{G}$. $D$ is {\em oriented $\mc{G}$} if it is an orientation of some graph $G \in \mc{G}$.

\medskip

\textbf{Tournaments.}
Oriented complete graphs are called \emph{tournaments}. If for any pair of vertices $u$ and $v$, $uv \in A(D)$ or $vu \in A(D)$, $D$ is said to be \emph{semicomplete}. Note that in a semicomplete digraph, there can be both arcs $uv$ and $vu$, which is forbidden in a tournament.

\medskip

\textbf{Oriented stars.}
A case that will draw our attention is the case of oriented stars on $k+1$ vertices in which only one vertex has positive outdegree, which we call \emph{out-stars} and denote $S_2^+$. Oriented stars on $k+1$ vertices in which only one vertex has positive indegree are called \emph{in-stars} and denoted $S_2^-$.

\medskip

\textbf{Digons.}
$C$ is a \emph{directed cycle}, or \emph{dicycle}, of $D$ if it is a sequence of vertices $c_1c_2\ldots c_k$, $k \geq 2$, such that $c_1\ldots c_k$ is a directed path of $D$ and $c_1c_k \in A(D)$.
Arcs $c_ic_{i+1}$,  for  $1\leq i <k$, and arc $c_1c_k$ are called the {\em arcs of $C$}.
The {\em length} of a dicycle is the number of its arcs. A directed cycle of length $2$ is called a \emph{digon}, and is denoted $[c_1,c_2]$.

\medskip

\textbf{Transitive tournaments.}
The \emph{directed girth}, or \emph{digirth}, of $D$ is the maximum length of directed cycle of $D$. $D$ is \emph{acyclic} if it has no directed cycle. If $D$ is acyclic, its digirth is $+\infty$. A set $X \subseteq V(D)$ is \emph{acyclic} if $D[X]$ is acyclic. The only acyclic tournament on $n$ vertices is called the \emph{transitive tournament}, and is denoted $TT_{n}$. When it is clear from context, we will write $n$ instead of $TT_n$, and $K_1$ instead of $TT_1$.  Given a transitive tournament $T$ on $n$ vertices $\{v_1, \dots, v_n\}$, we say that $v_1, \dots, v_n$ is the \emph{topological ordering} of $T$ if, for all $1 \le i<j \le n$, we have $v_iv_j \in A(T)$.

\medskip

\textbf{Dipaths and dicycles.}
Let $k$ be an integer. A \emph{directed path}, or \emph{dipath}, denoted $\ora{P_k}$, is an oriented graph isomorphic to $(\{x_1,\dots, x_k\}, \{x_i, x_{i+1} \mid 1 \leq i < k \})$. If $k \geq 3$, a \emph{directed cycle}, or \emph{dicycle}, denoted $\ora{C_k}$, is a digraph isomorphic to $(\{x_1,\dots, x_k\}, \{x_i, x_{i+1} \mid 1 \leq i < k \} \cup \{x_k,x_1\})$.

\subsection{Dicolouring}

\textbf{Dichromatic number.}
A \emph{dicolouring} of a digraph $D$ is a partition of $V(D)$ into acyclic subsets. A \emph{$k$-dicolouring} is a dicolouring using $k$ acyclic subsets. The \emph{dichromatic number} of $D$, denoted $\dic(G)$, is the minimum $k$ such that $D$ admits a $k$-dicolouring.
It is known that for any undirected graph $G$, the symmetric digraph $\olra G$ satisfies $\chi(G) = \dic(\olra G)$. We will sometimes extend $\dic$ to subsets of vertices, using $\dic(X)$ to mean $\dic(D[X])$ where $X \subseteq V(D)$.

\medskip

\textbf{k-dicritical.}
$D$ is \emph{$k$-dicritical} if $\dic(D) = k$ and for any proper subdigraph $D'$ of $D$, $\dic(D') < k$.
$D$ is \emph{$k$-vertex-dicritical} if $\dic(D) = k$ and for any proper induced subdigraph $D'$ of $D$, $\dic(D') < k$.

\medskip

\textbf{Heroes.}
Let $\mc{D}$ be a class of digraphs, then $\dic(\mc{D}) = \max_{D \in \mc{D}} \dic(D)$. Note that it may not be finite.
$\mc{D}$ is said to be \emph{heroic} if $\dic(\F(\mc{D}))$ is finite. A digraph $D$ is a \emph{hero in $\mc{D}$} if $\dic(\F(D) \cap \mc{D})$ is finite. Heroes in tournaments are simply called \emph{heroes}.

\cleardoublepage 


\ctparttext{\centering In which we delve into the relationship between different metrics and the dichromatic number, exhibit upper bounds on the dichromatic number in terms of these metrics and characterize tight cases.} 
\part{Maximum degree and local arc-connectivity} 

\chapter{A directed analogue of Brooks' theorem}\label{chpt:brooks}

\begin{flushright}{\slshape    
This chapter is built upon a joint work \\
with Pierre Aboulker, published in \cite{AA22}}. \\ \medskip
\end{flushright}

\emph{In this chapter, we extend Brooks' theorem to digraphs, discuss how different choices of maximum degree bounds the dichromatic number and how we can characterize tight cases for each of them.}

\section{Introduction}
It is an easy observation that for every graph $G$, $\chi(G) \leq \Delta(G)+1$. The following classical result of Brooks characterizes the (very few) graphs for which equality holds. 
\begin{theorem}[Brooks' Theorem, \cite{B41}]
A graph $G$ satisfies $\chi(G) = \Delta(G) + 1$ if and only if $G$ is an odd cycle or a complete graph. 
\end{theorem}

Many proofs of Brooks' Theorem have been found, and the different proofs generalize and extend in many directions. See~\cite{CR14} for a particularly nice survey on this subject. 
But Brooks' Theorem has also been extended to other notions of colouring, among which the colouring of digraphs via the notion of dicolouring. 
The aim of this chapter is to give four new proofs of the directed version, each of them adapted from a proof of the undirected version, along with an NP-completeness result. 
\medskip

The following easily holds (see subsection~\ref{subsec:def} for a proof): for every digraph $G$, $\dic(G) \leq \dmin(G) +1 \leq \dmax(G) +1$. 
Recall that a \emph{symmetric cycle} (resp. \emph{symmetric complete graph}) is a digraph obtained from a cycle (resp. from a complete graph), by replacing each edge by a digon.

We can then state the directed version of Brooks' Theorem. 
It was first proved by Mohar in~\cite{M10}, but we discovered that the proof is incomplete, see Section~\ref{sec:lovasz} for more details. Anyway, in~\cite{HMGallai}, Harutyunyan and Mohar generalised Gallai's Theorem  (a strengthening of Brooks' Theorem for list colourings) to digraphs, which gave an alternative and correct proof. 

\begin{theorem}[\cite{M10, HMGallai}]\label{brooks_max_oriented}
Let $G$ be a connected digraph, then $\dic(G) \leq \Delta_{max}(G) + 1$ and equality holds if and only if one of the following occurs:
\begin{itemize}
\item[(a)] $G$ is a directed cycle or,
\item[(b)] $G$ is a symmetric cycle of odd length or,
\item[(c)] $G$ is a symmetric complete graph on at least 4 vertices.
\end{itemize}
\end{theorem}

The next four sections are devoted to four new proofs of the directed Brooks' Theorem. 
In the last section, we show that  it is NP-complete to decide if $\dic(G) = \Delta_{min}(G) +1$, thus a simple characterization of digraphs satisfying $\dic(G) = \dmin(G) +1$ is very unlikely. 


\subsection{Definitions and preliminaries} \label{subsec:def}

We denote by $\mc B_{1}$ the set of directed cycles, $\mc B_{2}$ the set of symmetric odd cycles and, for $k \geq 3$, $\mc B_{k} = \{\olra K_{k+1}\}$ where $\olra K_{k+1}$ is the symmetric complete graph on $k + 1$ vertices.
Observe that the directed version of Brooks' Theorem is equivalent to the following statement: \emph{A digraph $G$ has dichromatic number at most $\dmax(G) +1$ and equality occurs if and only if $G$ contains a connected component isomorphic to a member of $\mc B_{\dmax(G)}$}. 
We sometimes call the members of $\mc B_k$ \emph{exceptions}.
\medskip

Given a digraph $G$ and an ordering $(v_1, \dots, v_n)$ of its vertices, 
to \textit{colour greedily}  $G$ is to colour $v_1, \dots, v_n$ in this order by giving to $v_i$ the minimum between the smallest colour not used in $N^+(V) \cap \{v_1, \dots, v_{i-1}\}$ and the smallest colour not used in $N^-(V) \cap \{v_1, \dots, v_{i-1}\}$. 
It is easy to see that any ordering leads to a dicolouring with at most $\Delta_{min}(G)+1$ colours. And since we clearly have $\dmin(G) \leq \dmax(G)$, we have:
$$\dic(G) \leq \Delta_{min}(G) +1 \leq \dmax(G) +1$$

The following easy lemma will be used in the four proofs of the directed Brooks' Theorem. Note that it does not hold if one replaces $\dmax(G)$ by $\dmin(G)$, implicit examples are given in Section \ref{sec:dmin}. 

\begin{lemma}\label{lem:reg}
If $G$ is a connected non-regular digraph, then  $\dic(G) \leq \Delta_{max}(G)$.
\end{lemma}

\begin{proof}
Since $G$ is non-regular, it has a vertex $u_1$ such that $d_{min}(u_1) < \Delta_{max}(G)$. Let $u_1, \dots, u_n$ be a vertex ordering output by a BFS on $\tilde G$ starting at $u_1$. By greedily colouring $G$ with respect to the ordering $u_n, \dots, u_1$, we get a dicolouring with at most  $\Delta_{max}(G)$ colours.
\end{proof}

If $\Delta_{max}(G) = 1$, then every vertex has at most one in-neighbour and at most one out-neighbour so $G$ is a directed cycle or a path. Hence, $\dic(G) = 2$ if and only if $G$ is a directed cycle. This proves Theorem \ref{brooks_max_oriented} for $\Delta_{max}(G) = 1$. So we only need to prove the directed Brooks' Theorem for digraphs with $\dmax(G) \geq 2$,  and we have the base case when we want to proceed by induction on the value of $\dmax(G)$. 


\section{Lov\'asz' proof: greedy dicolouring}\label{sec:lovasz}

In this section, we adapt the proof of Brooks' Theorem given by Lov\'asz in~\cite{L75}. 
The idea is the following: when we greedily colour the vertices of a connected digraph $G$ using the reverse order output by a BFS of $\tilde G$, each vertex except (possibly) the last one receives a colour from $\{1, \dots, \Delta_{max}(G)\}$. Indeed, the fact that $G$ is connected ensures that each vertex (except possibly the last one) has at most $\dmax(G) -1$ in-neighbours or out-neighbours already coloured. 
The goal of the proof is then to find an ordering of the vertices such that the last vertex can also be coloured with colour from $\{1, \dots, \Delta_{max}(G)\}$. 

The first version of the directed Brooks' Theorem appeared in~\cite{M10} and the  given proof is based on Lov\'asz' idea, but appears to be incomplete.  
To explain why, let us dive a little deeper into the proof. 
The goal is to find a vertex $v$ with two out- (or two in-) neighbours $v_1$, $v_2$ such that $v_1$ and $v_2$ are not linked by a digon and such that $G \sm \{v_1, v_2\}$ is connected. 
You can then choose an ordering of the vertices that starts with $v_1$ and $v_2$ and continue with the reverse order output by a BFS of $\tilde G$ starting at $v$ (so the ordering ends with $v$). A greedy dicolouring gives colour $1$ to $v_1$ and $v_2$, and thus there will be an available colour from $\{1, \dots, \dmax(G)\}$ to colour $v$ (the last vertex of the ordering). 
In~\cite{M10}, a vertex $v$ with two in- or two out-neighbours $v_1$ and $v_2$ not linked by a digon is found, but the fact that $G \sm \{v_1, v_2\}$ is connected is not checked and reveals to be non-trivial to prove.  
We now give full proof based on this idea.

\begin{theorem}
A connected digraph $G$ has dichromatic number at most $\dmax(G) +1$ and equality occurs if and only if it is a member of $\mc B_{\dmax(G)}$. 
\end{theorem}

\begin{proof}
Let $G$ be a counter-example, that is $G$ is connected, $\dic(G) = \Delta_{max}(G) +1$ and $G$ is not a member of $\mc B_{\dmax(G)}$. Set $k = \Delta_{max}(G) \geq 2$ and recall that $\tilde{G}$ denotes the underlying graph of $G$.
By Lemma~\ref{lem:reg}, $G$ is $k$-regular.

\begin{claim}\label{claim:2con}
 $\tilde{G}$ is $2$-connected
\end{claim}

\begin{proofclaim}
Assume for contradiction that $\tilde G$ has a cutvertex $u$ and let $C_1$ be a connected component of $G - u$, and $C_2$ the union of the other connected components.  
Set $G_{i} = G[C_{i} \cup \{u\}]$ for $i=1,2$. By Lemma~\ref{lem:reg}, $G_1$ and $G_2$ are $k$-dicolourable. 
Up to permuting colours, we may assume that the $k$-dicolourings of $G_1$ and $G_2$ agree on $u$, which gives a $k$-dicolouring of $G$, a contradiction. 
\end{proofclaim}

\begin{claim}\label{claim:2edgeCut}
  $\tilde G$ has no edge-cut of size $2$.
\end{claim}

\begin{proofclaim} 
Assume by contradiction that $G$ has an edge cutset $\{e_1, e_2\}$. 
Let $G_{1}$ and $G_2$ be the two connected components of $G - \{e_1, e_2\}$. 
Both $G_1$ and $G_2$ are $k$-colourable by Lemma \ref{lem:reg}. 
A $k$-dicolouring of $G_1$ and $G_2$ give a $k$-dicolouring of $G$ as soon as the extremities of $e_1$ and $e_2$ use at least two distinct colours. Permuting colours in $G_1$ if necessary, we get a $k$-dicolouring of $G$. 
\end{proofclaim}

\begin{claim}\label{claim:2cut}
If $\{u,v\} \subseteq V(G)$ is a cutset of $\tilde{G}$, then $\{u,v\}$ is a stable set. 
\end{claim}

\begin{proofclaim}
Let $\{u,v\} \subseteq V(G)$ be a cutset of $\tilde{G}$ and assume for contradiction and without loss of generality, that $uv$ is an arc of $G$. 
Let $C_{1}$ be a connected component of $\tilde G \sm \{u,v\}$ and $C_2$ the union of the other connected components.  Set $G_{i} = G[C_{i} \cup \{u, v\}]$ for $i=1,2$.

Since $\tilde{G}$ is $2$-connected, both $u$ and $v$ have some neighbours in both $C_1$ and $C_2$ and thus $G_1$ and $G_2$ are $k$-dicolourable by Lemma~\ref{lem:reg}.
 If both $G_1$ and $G_2$ admit a $k$-dicolouring in which $u$ and $v$ receive distinct (resp. same) colours, then we get a $k$-dicolouring of $G$, a contradiction (because no induced cycle can intersect both $C_1$ and $C_2$). 
 So we may assume without loss of generality that $u$ and $v$ receive the same colour (resp. distinct colours) in every $k$-dicolouring of $G_1$ (resp. in every $k$-dicolouring of $G_2$). 

If $u$ has an out-neighbour in $C_2$, then $d_{G_1}^+(u) \leq k-1$. We can $k$-dicolour $G_1 - \{u\}$, and extend the $k$-dicolouring to $u$ with a colour not appearing in the out-neighbourhood of $u$, so, in particular, distinct from the colour of $v$, a contradiction.
So $u$ has no out-neighbour in $C_2$ and similarly, $v$ has no in-neighbour in $C_2$. 

Suppose $u$ has in-degree at least $2$ in $G_{2}$. Then $d_{G_1}^-(u) \leq k-2$ and thus we can $k$-dicolour $G_1 - \{u\}$ and extend this dicolouring to $G_1$ by giving to $u$ a colour not used in its in-neighbour and distinct from $v$, a contradiction. 
So $u$ has exactly one in-neighbour in $G_2$, and similarly, $v$ has exactly one out-neighbour in $G_2$ which gives us an edge cutset of size $2$, a contradiction with (\ref{claim:2edgeCut}).
\end{proofclaim}

\begin{claim}\label{claim:lovasz}
Let $x$ be a vertex of $G$ and $u$ and $v$ two out-neighbours of $x$. Then either $\{u,v\}$ induces a digon, or $\{u,v\}$ is a cutset. The same holds if $u$ and $v$ are in-neighbours of $x$. 
\end{claim}

\begin{proofclaim}
Assume for contradiction that $\{u, v\}$ does not induce a digon and is not a cutset of $G$. 
Let $G' = G - \{u, v\}$ and $\tilde{G}'$  the underlying graph of $G'$. Since $\tilde{G}'$ is connected, there is a BFS ordering $(x = u_{1}, u_{2}, \dots, u_{n-2})$ of $\tilde{G}'$. Set $u_{n-1} = u$ and $u_{n} = v$. 
We now greedily dicolour $G'$ with respect to the order $(u_n, u_{n-1}, \dots, u_1)$.  Since $G[\{u_{n}, u_{n-1}\}]$ is not a digon, $u_{n}$ and  $u_{n-1}$ both receive colour $1$. For $i=n-2, \dots 2$, $u_i$ has at least one neighbour in $G[\{u_1, \dots u_{i-1}\}]$, and thus $u_i$ has at most $k-1$ in- or out-neighbours in $G[u_n, \dots, u_i]$ and hence we can assign a colour from $\{1, \dots, k\}$ to it.  
Finally, since $u_n$ and $u_{n-1}$ receive colour $1$ and are both in the out-neighbourhood of $u_1$, the out-neighbourhood of $u_1$ is coloured with at most $k-1$ distinct colours and thus $u_1$ receive colour from $\{1, \dots, k\}$, a contradiction. 
The proof is the same when $u$ and $v$ are in-neighbours of $x$.
\end{proofclaim}
\medskip

Observe that $G$ cannot be a symmetric digraph because of the undirected Brooks' Theorem. So there exists $u, v \in V(G)$ such that $uv \in A(G)$ and $vu \notin A(G)$. 
By (\ref{claim:2cut}), $\{u, v\}$ is not a cutset.

\begin{claim}\label{neighboursofuarecut}
For every $a \in u^+ \setminus \{v\}$, $\{a, v\}$ is a cutset. 
\end{claim}

\begin{proofclaim}
Suppose $\{a, v\}$ is not a cutset. By (\ref{claim:lovasz}) $\{a,v\}$ induces a digon and thus $u$ and $v$ are in-neighbours of $a$. But $\{u,v\}$ is not a cutset by (\ref{claim:2cut}) and does not induce a digon, a contradiction to (\ref{claim:lovasz}).
\end{proofclaim}
\medskip

Let $H = G - v$ and let $a \in u^+ \setminus \{v\}$. 
By (\ref{neighboursofuarecut}) $a$ is a cutvertex of $H$, so $H$ has at least two blocks (where a \textit{block} is a maximal 2-connected subgraph of $\tilde G$).
Since $\tilde G$ is $2$-connected, $v$ has a neighbour in each leaf block of the block decomposition of $\tilde H$. 

We now break the proof into two parts with respect to the value of $k$. 
Suppose first that $k = 2$. 
If the two out-neighbours (resp. the two in-neighbours) of $v$ belong to distinct blocks of $\tilde H$, then $v^+$ does not induce a digon, nor a cutset of $G$, a contradiction to~(\ref{claim:lovasz}). 
Hence $v^+$ is included in a leaf block of $H$ and  $u^{-}$ in another one.
Now, dicolour $H$ with $2$ colours (it is possible by Lemma~\ref{lem:reg}). 
Let $w$ be a cutvertex of $H$ separating the leaf blocks containing the neighbours of $v$. Observe that every cycle containing $v$ must go through $w$. Hence we can extend the $2$-dicolouring of $H$ by giving to $v$ a colour distinct from the one received by $w$  to get a $2$-dicolouring of $G$, a contradiction.

Assume now that $k \geq 3$. So there exists $b \in u^+ \setminus \{a,v\}$. By (\ref{neighboursofuarecut}), both $a$ and $b$ are cutvertices of $H$. Since $uv \in A(G)$,  $u$ is not a cutvertex of $H$ by (\ref{claim:2cut}). 
Let $U$ be the block of $H$ containing $u$ (which is unique because $u$ is not a cutvertex of $H$). Since $u$ sees both $a$ and $b$, $U$ is not a leaf block of $H$. 
Let $U_{1}$ and $U_{2}$ be two distinct leaf blocks of $H$. Since $\tilde{G}$ is $2$-connected, $v$ must have neighbours in $U_1$ and $U_2$. Let $u_1 \in U_1$ and $u_2 \in U_2$ be two neighbours of $v$. 
So $u$, $u_1$, $u_2$ are in pairwise distinct blocks of $H$ which implies that for every $\{x,y\} \subseteq \{u,u_1,u_2\}$, $\{x,y\}$ does not induce a digon and is not a cutset of $\tilde{G}$. 
Now, since $u$, $u_1$, $u_2$ are neighbours of $v$, two of them are included in the in-neighbourhood or in the out-neighbourhood of $v$, a contradiction to (\ref{claim:lovasz}).
\end{proof}


\section{Acyclic subdigraph and induction}

The proof of this section is an adaptation of a proof of Rabern~\cite{R14}, see also Section 3 of \cite{CR14}. 
Here is a sketch of the proof. 
Let $G$ be a digraph with $\dmax(G) = k$. We do induction on $k$. 
We first choose a maximal induced acyclic subdigraph $M$ of $G$ and prove that $G-M$ must have dichromatic number $k-1$ and thus must contain a connected component $T$ isomorphic to a member of $\mc B_{k-1}$ by induction. We then show that a $k$-dicolouring of $G-T$ can be extended to $G$.

\begin{theorem}
Let $G$ a digraph such that $\dic(G) = \Delta_{max}(G) + 1$. Then $G$ contains a connected component isomorphic to a member of $\mc B_{\Delta_{max}(G)}$. 
\end{theorem}
\begin{proof}
The theorem is true for digraphs $G$ with $\dmax(G) = 1$. Let $k \geq 2$ and assume the theorem holds for digraph with maximum maxdegree at most $k-1$. By means of contradiction, assume there exists a digraph $G$ with $\dmax(G) = k$ violating the theorem. 
We choose such a $G$ with the minimum number of vertices. 
By Lemma~\ref{lem:reg}, $G$ is $k$-regular. 

We now prove two technical claims. 

\begin{claim}
If $k \geq 3$, $G$ cannot contain $\olra{K}_{k+1}$ less an arc, or less a digon, as an induced subdigraph. 
\end{claim}

\begin{proofclaim}
Suppose $G$ contains a subdigraph $K$ isomorphic to $\olra{K}_{k + 1}$ less a digon $\{uv,vu\}$.  
Observe that $u$ and $v$ both have exactly one in-neighbour and one out-neighbour outside of $K$, and that all other vertices of $K$ have no neighbour outside of $K$. Now, by Lemma~\ref{lem:reg}, $G - K$ can be $k$-dicoloured and we can extend this $k$-dicolouring to $G$ as follows: at most one colour is forbidden for $u$ and one for $v$, hence, since $k \geq 3$, we can give the same colour to $u$ and $v$, and then assign the $k-1$ remaining colours to $V(K)\sm \{u,v\}$. We thus get a $k$-dicolouring of $G$, a contradiction.  The same reasoning holds when an arc is missing instead of a digon. 
\end{proofclaim}

\begin{claim}
If $k = 2$, $G$ cannot contain a symmetric odd cycle less an arc, or less a digon, as an induced subdigraph. 
\end{claim}

\begin{proofclaim}

Let $\ell \geq 1$.
Assume for contradiction that $G$ contains a subdigraph $C$ isomorphic to $\overleftrightarrow{C}_{2\ell + 1}$ less an arc $uv$. Let us consider a $2$-dicolouring of $G - C$ and assume without loss of generality that the out-neighbour of $u$ not in $C$ is coloured $1$. We can colour $u$ and $v$ with colour $2$, and greedily colour $C - u - v$ to obtain a $2$-dicolouring of $G$, a contradiction.

Suppose now that $G$ contains a subdigraph $C$ isomorphic to $\overleftrightarrow{C}_{2l + 1}$ less a digon $\{uv,vu\}$. 
Let us name $F = (G - (C - \{u,v\})) / uv$. Either $F$ is $2$-dicolourable, in which case there exists a $2$-dicolouring of $G - \{C-\{u,v\}\}$ in which $u$ and $v$ receive the same colour and we can extend this dicolouring to $C$ or, as $\Delta_{max}(F) \leq 2$ and $|V(F)| < |V(G)|$, $F$ is a symmetric odd cycle, which implies $G$ is a symmetric odd cycle as well, a contradiction.
\end{proofclaim}

Let $M$ be a maximal directed acyclic subdigraph of $G$. By maximality of $M$, every vertex in $G-M$ must have at least one in-neighbour and one out-neighbour in $M$, so $\Delta_{max}(G - M) \leq k-1$. 
Moreover, $\dic(G - M) = k$,  as otherwise we could $(k - 1)$-dicolour $G - M$ and use a $k^{th}$ colour for $M$. So $G-M$ has a connected component $T$ isomorphic to a member of $\mathcal B_{k-1}$ by induction. 

Suppose first that there exists $u \in V(T)$ whose in-neighbour $x$ and out-neighbour $y$ in $G - T$ are distinct. 
Let $H = G - T$ to which is added the arc $xy$ if   $xy \notin A(G)$. 
Observe that $\Delta_{max}(H) \leq k$. 
Then $H$ does not contain any element of $\mathcal B_{k}$ (as $G$ does not contain an element of $\mathcal B_{k}$ less an arc) which, by minimality of $G$, implies that $H$ is $k$-dicolourable. Thus there is a $k$-dicolouring of $G - T$ with no monochromatic path from $y$ to $x$.

We are now going to show that such a dicolouring can be extended to $T$. We break the proof into two parts with respect to the value of $k$. 

Assume first that $k \geq 3$. Then $T$ induces $\olra K_{k}$. 
Observe that each vertex of $T$ has precisely one in-neighbour and one out-neighbour outside of $T$. So we can greedily extend the $k$-dicolouring of $G-T$ to $G-u$. We can now greedily extend this dicolouring to $u$. This is possible because there is no monochromatic path from $y$ to $x$ in $G-T$. 


Assume now that $k=2$. Then $T$ induces a directed cycle. 
If $\cup_{v \in T} N(v) \setminus V(T)$  is monochromatic of colour $c$, we can assign colour $c$ to $u$ and the other colour to vertices of $T-\{u\}$ to obtain a proper $2$-dicolouring of $G$. If not,  there must exist a vertex $z$ in $T$ such that, naming $z'$ its out-neighbour in $T$, $z'^{+} \cup z'^{-} \cup z^{+} \setminus V(T)$ is not monochromatic. Let $c$ be the colour of the out-neighbour of $z$ not in $T$. We can then safely assign colour $c$ to $z'$ and then greedily extend the dicolouring to $T \setminus \{z\}$.  Now, since the two out-neighbours of $z$ are coloured $c$, we can safely assign the other colour to $z$ to obtain a proper $2$-dicolouring of $G$. 
\medskip

We can now assume that each vertex $u$ of $T$ is linked to $G - T$ via a digon. 
If there is a vertex $x$ in $G-T$ linked to all vertices of $T$, then $T$ has at most $k$ vertices and thus must be isomorphic to $\olra K_{k}$. Hence $T \cup \{x\}$ induces $\olra K_{k+1}$, a contradiction. 

So, there exist two distinct vertices $x,y$ in $G-T$ linked via a digon to two (distinct) vertices of $T$. 
Let $H = G - T$ to which is added arcs $xy$ and $yx$ (if not existing). Then $H$ does not contain any element of $\mathcal B_{k}$ (as $G$ does not contain an element of $\mc B_{k}$ less a digon or an arc) and thus, by minimality of $G$, $H$ is $k$-dicolourable. 
Thus, $G - T$ admits a $k$-dicolouring in which $x$ and $y$ receive distinct colours. We can easily extend this $k$-dicolouring to a $k$-dicolouring of $G$ since each vertex of $T$ has a set of $k-1$ available colours and some pair  of vertices  in $T$ (the neighbours of $x$ and $y$) get distinct sets.
\end{proof}


\section{$k$-trees}

The proof presented in this section is an adaptation of a proof of Tverberg~\cite{T83}, see also section 4 of \cite{CR14}. 

A digraph $G$ is a \emph{direct composition} of digraphs $G_{1}$ and $G_{2}$ on vertices $v_{1} \in V(G_{1})$ and $v_{2} \in V(G_{2})$ if it can be obtained from the disjoint union of $G_1$ and $G_2$ by adding exactly one arc between $v_1$ and $v_2$ (either $v_1v_2$ or $v_2v_1$). 
A digraph $G$ is a \emph{cyclic composition} of digraphs $G_{1}, \dots, G_{\ell}$ ($\ell \geq 2$) on vertices $v_{1} \in V(G_{1}), \dots, v_{\ell} \in V(G_{\ell})$ if it can be obtained from the disjoint union of the $G_{i}$ by adding  the arcs $v_iv_{i+1}$ for $i=1, \dots, \ell- 1$ and $v_{\ell}v_1$

A digraph $G$ is a \emph{$k$-tree} if $\dmax(G) \leq k$ and it can be constructed as follows:
\begin{itemize}
    \item the digraphs in $\mathcal B_{k-1}$ are $k$-trees;
    \item a  direct or cyclic composition of $k$-trees is a $k$-tree;
\end{itemize}


Let $G$ be a digraph. 
A \emph{direct $k$-leaf} of $G$ is an induced subdigraph $T$ of $G$ such that $T$ belongs to $\mathcal B_{k-1}$ and $G$ is a direct composition of $T$ and $G-T$. 
If $G$ cannot be obtained from a cyclic composition of members of $\mathcal B_{k-1}$, an induced subdigraph $T$ of $G$ is a  \emph{cyclic $k$-leaf} of $G$ if $T$ can be obtained from $\ell \geq 1$ disjoint  $T_1, \dots, T_{\ell}$ belonging to $\mathcal B_{k-1}$ by adding $\ell -1 $ arcs $v_iv_{i+1}$ for $i=1, \dots, \ell-1$ where $v_i \in V(T_i)$, and $G$ is a cyclic composition of $G - T$ and $T_1, \dots, T_{\ell}$. See Figure~\ref{fig_tree}. 

A \emph{$k$-leaf} of $G$ is either a direct $k$-leaf or a cyclic $k$-leaf of $G$, or $G$ itself if $G$ is a member of $\mathcal B_{k-1}$ or $G$ is obtained from a cyclic composition of members of $\mathcal B_{k-1}$.
Observe that two distinct $k$-leaves of a digraph $G$ are always vertex disjoint and that a $k$-tree has at least two $k$-leaves except if it is a member of $\mathcal B_{k-1}$ or if  it can be obtained by a cyclic composition of members of $\mathcal B_{k-1}$. 

A \emph{$k$-path} is a digraph obtained by taking the disjoint union of $l \geq 2$ members $T_1, \dots, T_{\ell}$ of $\mathcal B_{k-1}$ and adding arcs $v_iv_{i+1}$ for $i=1, \dots, \ell-1$ where $v_i \in V(T_i)$.


\medskip

    \begin{figure}[ht]
    
    \centering
    
	\begin{tikzpicture}[line cap=round,line join=round,>=triangle 45, scale=1.2] 
	
        \begin{scope}[xshift=0cm,yshift=0cm,scale=1]
            \draw[blue] [->-] (-1,0) to[bend left = 18] (0,-1);
            \draw[blue] [->-] (0,-1) to[bend left = 18] (-1,0);
            \draw[blue] [->-] (-1,0) to[bend left = 18] (0,1);
            \draw[blue] [->-] (0,1) to[bend left = 18] (-1,0);
            \draw[blue] [->-] (-1,0) to[bend left = 18] (1,0);
            \draw[blue] [->-] (1,0) to[bend left = 18] (-1,0);
            \draw[blue] [->-] (0,-1) to[bend left = 18] (0,1);
            \draw[blue] [->-] (0,1) to[bend left = 18] (0,-1);
            \draw[blue] [->-] (0,-1) to[bend left = 18] (1,0);
            \draw[blue] [->-] (1,0) to[bend left = 18] (0,-1);
            \draw[blue] [->-] (0,1) to[bend left = 18] (1,0);
            \draw[blue] [->-] (1,0) to[bend left = 18] (0,1);
            \fill[blue] (-1,0) circle (2pt);
            \fill[blue] (0,-1) circle (2pt);
            \fill[blue] (0,1) circle (2pt);
            \fill[blue] (1,0) circle (2pt);
        \end{scope}

        \begin{scope}[xshift=3cm,yshift=0cm,scale=1]
            \draw [->-] (-1,0) to[bend left = 18] (0,-1);
            \draw [->-] (0,-1) to[bend left = 18] (-1,0);
            \draw [->-] (-1,0) to[bend left = 18] (0,1);
            \draw [->-] (0,1) to[bend left = 18] (-1,0);
            \draw [->-] (-1,0) to[bend left = 18] (1,0);
            \draw [->-] (1,0) to[bend left = 18] (-1,0);
            \draw [->-] (0,-1) to[bend left = 18] (0,1);
            \draw [->-] (0,1) to[bend left = 18] (0,-1);
            \draw [->-] (0,-1) to[bend left = 18] (1,0);
            \draw [->-] (1,0) to[bend left = 18] (0,-1);
            \draw [->-] (0,1) to[bend left = 18] (1,0);
            \draw [->-] (1,0) to[bend left = 18] (0,1);
            \fill (-1,0) circle (2pt);
            \fill (0,-1) circle (2pt);
            \fill (0,1) circle (2pt);
            \fill (1,0) circle (2pt);
        \end{scope}

        \begin{scope}[xshift=3cm,yshift=-3cm,scale=1]
            \draw [->-] (-1,0) to[bend left = 18] (0,-1);
            \draw [->-] (0,-1) to[bend left = 18] (-1,0);
            \draw [->-] (-1,0) to[bend left = 18] (0,1);
            \draw [->-] (0,1) to[bend left = 18] (-1,0);
            \draw [->-] (-1,0) to[bend left = 18] (1,0);
            \draw [->-] (1,0) to[bend left = 18] (-1,0);
            \draw [->-] (0,-1) to[bend left = 18] (0,1);
            \draw [->-] (0,1) to[bend left = 18] (0,-1);
            \draw [->-] (0,-1) to[bend left = 18] (1,0);
            \draw [->-] (1,0) to[bend left = 18] (0,-1);
            \draw [->-] (0,1) to[bend left = 18] (1,0);
            \draw [->-] (1,0) to[bend left = 18] (0,1);
            \fill (-1,0) circle (2pt);
            \fill (0,-1) circle (2pt);
            \fill (0,1) circle (2pt);
            \fill (1,0) circle (2pt);
        \end{scope}

        \begin{scope}[xshift=6cm,yshift=0cm,scale=1]
            \draw [->-] (-1,0) to[bend left = 18] (0,-1);
            \draw [->-] (0,-1) to[bend left = 18] (-1,0);
            \draw [->-] (-1,0) to[bend left = 18] (0,1);
            \draw [->-] (0,1) to[bend left = 18] (-1,0);
            \draw [->-] (-1,0) to[bend left = 18] (1,0);
            \draw [->-] (1,0) to[bend left = 18] (-1,0);
            \draw [->-] (0,-1) to[bend left = 18] (0,1);
            \draw [->-] (0,1) to[bend left = 18] (0,-1);
            \draw [->-] (0,-1) to[bend left = 18] (1,0);
            \draw [->-] (1,0) to[bend left = 18] (0,-1);
            \draw [->-] (0,1) to[bend left = 18] (1,0);
            \draw [->-] (1,0) to[bend left = 18] (0,1);
            \fill (-1,0) circle (2pt);
            \fill (0,-1) circle (2pt);
            \fill (0,1) circle (2pt);
            \fill (1,0) circle (2pt);
        \end{scope}

        \begin{scope}[xshift=7.7cm,yshift=1.7cm,scale=1]
            \draw[red] [->-] (-1,0) to[bend left = 18] (0,-1);
            \draw[red] [->-] (0,-1) to[bend left = 18] (-1,0);
            \draw[red] [->-] (-1,0) to[bend left = 18] (0,1);
            \draw[red] [->-] (0,1) to[bend left = 18] (-1,0);
            \draw[red] [->-] (-1,0) to[bend left = 18] (1,0);
            \draw[red] [->-] (1,0) to[bend left = 18] (-1,0);
            \draw[red] [->-] (0,-1) to[bend left = 18] (0,1);
            \draw[red] [->-] (0,1) to[bend left = 18] (0,-1);
            \draw[red] [->-] (0,-1) to[bend left = 18] (1,0);
            \draw[red] [->-] (1,0) to[bend left = 18] (0,-1);
            \draw[red] [->-] (0,1) to[bend left = 18] (1,0);
            \draw[red] [->-] (1,0) to[bend left = 18] (0,1);
            \fill[red] (-1,0) circle (2pt);
            \fill[red] (0,-1) circle (2pt);
            \fill[red] (0,1) circle (2pt);
            \fill[red] (1,0) circle (2pt);
        \end{scope}

        \begin{scope}[xshift=7.7cm,yshift=-1.7cm,scale=1]
            \draw[red] [->-] (-1,0) to[bend left = 18] (0,-1);
            \draw[red] [->-] (0,-1) to[bend left = 18] (-1,0);
            \draw[red] [->-] (-1,0) to[bend left = 18] (0,1);
            \draw[red] [->-] (0,1) to[bend left = 18] (-1,0);
            \draw[red] [->-] (-1,0) to[bend left = 18] (1,0);
            \draw[red] [->-] (1,0) to[bend left = 18] (-1,0);
            \draw[red] [->-] (0,-1) to[bend left = 18] (0,1);
            \draw[red] [->-] (0,1) to[bend left = 18] (0,-1);
            \draw[red] [->-] (0,-1) to[bend left = 18] (1,0);
            \draw[red] [->-] (1,0) to[bend left = 18] (0,-1);
            \draw[red] [->-] (0,1) to[bend left = 18] (1,0);
            \draw[red] [->-] (1,0) to[bend left = 18] (0,1);
            \fill[red] (-1,0) circle (2pt);
            \fill[red] (0,-1) circle (2pt);
            \fill[red] (0,1) circle (2pt);
            \fill[red] (1,0) circle (2pt);
        \end{scope}

        \begin{scope}[xshift=9.4cm,yshift=0cm,scale=1]
            \draw[red] [->-] (-1,0) to[bend left = 18] (0,-1);
            \draw[red] [->-] (0,-1) to[bend left = 18] (-1,0);
            \draw[red] [->-] (-1,0) to[bend left = 18] (0,1);
            \draw[red] [->-] (0,1) to[bend left = 18] (-1,0);
            \draw[red] [->-] (-1,0) to[bend left = 18] (1,0);
            \draw[red] [->-] (1,0) to[bend left = 18] (-1,0);
            \draw[red] [->-] (0,-1) to[bend left = 18] (0,1);
            \draw[red] [->-] (0,1) to[bend left = 18] (0,-1);
            \draw[red] [->-] (0,-1) to[bend left = 18] (1,0);
            \draw[red] [->-] (1,0) to[bend left = 18] (0,-1);
            \draw[red] [->-] (0,1) to[bend left = 18] (1,0);
            \draw[red] [->-] (1,0) to[bend left = 18] (0,1);
            \fill[red] (-1,0) circle (2pt);
            \fill[red] (0,-1) circle (2pt);
            \fill[red] (0,1) circle (2pt);
            \fill[red] (1,0) circle (2pt);
        \end{scope}

        \draw [->-] (1,0) -- (2,0);
        \draw [->-] (3,-1) -- (3,-2);
        \draw [->-] (4,0) to[bend left = 18] (5,0);
        \draw [->-] (5,0) to[bend left = 18] (4,0);
        \draw [->-] (7,0) -- (7.7,0.7);
        \draw [red][->-] (7.7,0.7) -- (8.4,0);
        \draw [red][->-] (8.4,0) -- (7.7,-0.7);
        \draw [->-] (7.7,-0.7) -- (7,0);
        
    \draw[red] [decorate,decoration={brace,mirror,amplitude=10pt}]
    (6.5,-3) -- (10.5,-3) node [midway,yshift=-1cm] {\textcolor{red}{A cyclic leaf}};

    \draw[blue] [decorate,decoration={brace,amplitude=10pt}]
    (-1,1.3) -- (1,1.3) node [midway,yshift=1cm] {\textcolor{blue}{A direct leaf}};

	\end{tikzpicture}
    \caption{A 4-tree} \label{fig_tree}
    \end{figure}
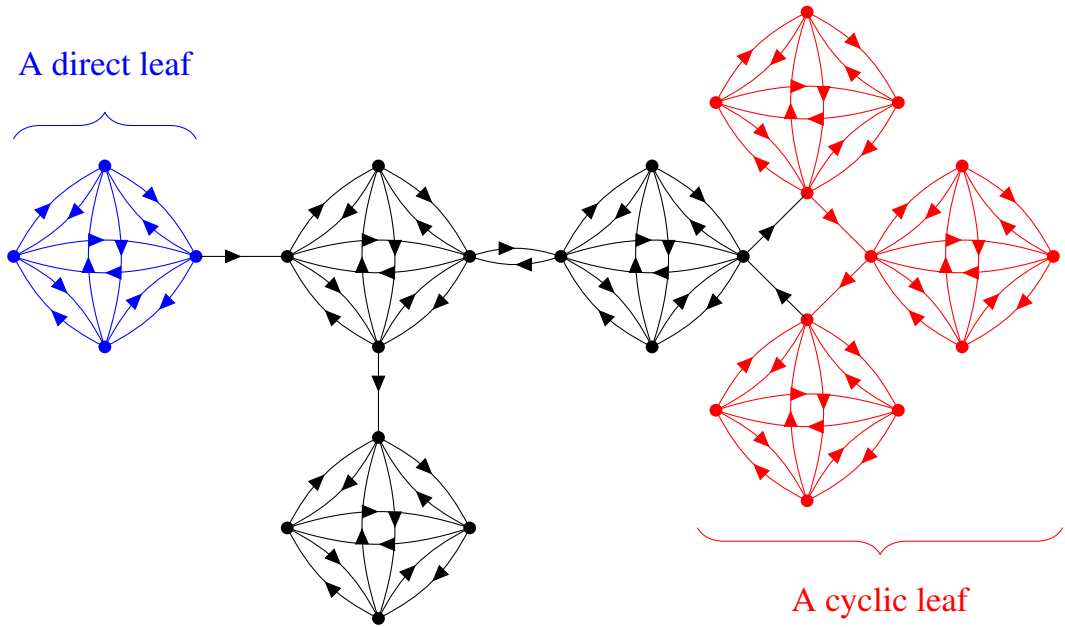

The following easy observation will be useful during the proof. 
\begin{observation}\label{obs_1}
Let $G$ be a $k$-tree. 
Then all vertices of $G$ have mindegree at least $k-1$. Moreover, $G$ has at least $k+1$ vertices of mindegree $k-1$, except if $G = \olra{K}_{k}$ or if it is a symmetric path of odd length (and thus $k=2$). 
\end{observation}

The main ingredient of the proof is the following lemma:

\begin{lemma}\label{lem:ktree}
Let $G$ be a connected digraph and $k = \Delta_{max}(G) \geq 2$. 
Then either $G$ is a member of $\mathcal B_k$, or $G$ is a $k$-tree, or there exists $v \in V(G)$ such that $d_{max}(v) = k$ and no connected component of $G - \{v\}$ is a $k$-tree. 
\end{lemma}

\begin{proof}
Let $G$ be a digraph with $\dmax(G) = k$ and assume that $G$ is not a member of $\mathcal B_k$ nor a $k$-tree. 

\begin{claim}\label{clm:no_leaf}
$G$ has no $k$-leaf.
\end{claim}

\begin{proofclaim}
Assume first that $G$ has a direct $k$-leaf $T$, and let $v$ be the unique vertex of $T$ having a neighbour outside of $T$. Recall that $T$ belongs to $\mathcal B_{k-1}$ by definition of a direct $k$-leaf. 
Then $d_{max}(v) = k$ and $G-\{v\}$  has two connected components, $T-\{v\}$ and $G -T$. $G-T$ is not a $k$-tree otherwise $G$ is too, and $T-\{v\}$ is clearly not a $k$-tree, so we are done. 

Assume now that $G$ has a cyclic $k$-leaf $T$ made of $\ell$ members $T_1, \dots, T_{\ell}$ of $\mathcal B_{k-1}$ and let $v_1, \dots, v_{\ell}$ be as in the definition of cyclic $k$-leaf. Then $d_{max}(v_1) = k$ and $G - \{v_1\}$ has two connected components, $T_1-\{v_1\}$ and $G-T_1$. As in the previous case, none of them is a $k$-tree. 
\end{proofclaim}

We say that a vertex $v$ of $G$ is \textit{special} if it is contained in an induced subdigraph  of $G$ isomorphic to a member of $\mc B_{k-1}$ and $d_{max}(v) = k$. 
For each special vertex $x$, choose arbitrarily an induced subdigraph of $G$ isomorphic to a member of $\mathcal B_{k-1}$ that we name $T_x$. Moreover, we name $H_x$ the connected component of $G-x$ containing $T_x - x$. Note that in the case where $G- x$ is connected, we have $H_{x} = G- x$. 

If no induced subdigraph of $G$ is isomorphic to a member of $\mathcal B_{k-1}$, then any vertex $v$ with maxdegree $k$ is such that no component of $G - \{v\}$ is a $k$-tree. 
Moreover, if $G$ has an induced subdigraph $H$ isomorphic to a member of $\mc B_{k-1}$, then at least one of its vertices must have a maxdegree equal to $k$, otherwise $G = H$ is a $k$-tree, a contradiction. 
Hence, $G$ must contain some special vertices, and every subdigraph of $G$ isomorphic to a member of $\mathcal B_{k-1}$ contains a special vertex. 
\medskip

Assume there exists a special vertex $v$ such that $H_v$ is not a $k$-tree. 
If $G-v$ is connected, then $v$ is such that $d_{max}(v) = k$ and no component of $G - \{ v \}$ is a $k$-tree. 
So we can assume $G-v$ is not connected. 

Assume first $v$ has only one neighbour $a$ in $G-H_v$. Let $G_a$ be the connected component of $G-v$ containing $a$. We may assume  $G_a$ is a $k$-tree, otherwise $v$ is such that $d_{max}(v) = k$ and no component of $G - \{ v \}$ is a $k$-tree.  
If $G_a$ is isomorphic to a member of $\mathcal B_{k-1}$, then $G_a$ is a $k$-leaf of $G$ (direct of cyclic depending if $a$ and $v$ are linked by a single arc of a digon), if $G_a$ is a cyclic composition of members of $\mathcal B_{k-1}$, then $G$ contains a cyclic $k$-leaf, and otherwise $G_a$ has at least two $k$-leaves, one of the two does not contain $a$ and is thus a $k$-leaf of $G$. Each case contradicts (\ref{clm:no_leaf}).

So $v$ has at least two neighbours $a$ and $b$ in $G-H_v$, and $a \neq b$. 
If $a$ and $b$ are in two distinct connected component $G_a$ and $G_b$ of $G-v$, then one of $G_a$ or $G_b$ must be a $k$-tree, for otherwise $v$ is such that $d_{max}(v) = k$ and no component of $G - \{ v \}$ is a $k$-tree, and we find a $k$-leaf as in the previous case.

So we may assume that $G-H_v$ is connected. 
Moreover, $G-H_v$ must be a $k$-tree, for otherwise $v$ is such that $d_{max}(v) = k$ and no component of $G - \{ v \}$ is a $k$-tree.
If $G-H_v$ has a $k$-leaf disjoint from $\{a,b\}$, then it is a $k$-leaf of $G$, a contradiction to (\ref{clm:no_leaf}). So $G-H_v$ is isomorphic to a member of $\mathcal B_{k-1}$ or is a cyclic composition of members of $\mathcal B_{k-1}$ or has exactly two leaves, $T_a$ and $T_b$ containing respectively $a$ and $b$. 

If $G-H_v$ is a member of $\mathcal B_{k-1}$, then $G-a$ is connected and is not a $k$-tree, so $a$ is such that $d_{max}(a) = k$ and no component of $G - \{ a \}$ is a $k$-tree 
If $G-H_v$ is a cyclic composition of members of $\mathcal B_{k-1}$, then $a$ cannot be a cutvertex of $G-H_v$ (otherwise $d_{max}(a)>k$), and thus $G-a$ is connected and is not a $k$-tree, so again $a$ is such that $d_{max}(a) = k$ and no component of $G - \{ a \}$ is a $k$-tree. 

So $H_v$ has exactly two leaves $T_a$, $T_b$ as explained above. 
Observe that the only vertex of $T_a$ with maxdegree $k$ in $G$ is $a$, for otherwise $G-a$ is connected and is not a $k$-tree, so $a$ satisfies the theorem. 
The same holds for $T_b$ and $b$. 
Let $T$ be an induced subdigraph of $G-H_v$ isomorphic to a member of $\mathcal B_{k-1}$ that does not contain $a$ nor $b$. 
If $T$ has at least $3$ vertices of maxdegree $k$, then $G-H_v$ contains a $k$-leaf disjoint from $\{a,b\}$, a contradiction to (\ref{clm:no_leaf}). 
If $T$ has exactly two vertices of maxdegree $k$, then deleting one leads to a connected digraph which is not a $k$-tree and we are done. 
So we may assume that each subdigraph of $G-H_v$ isomorphic to a member of $\mathcal B_{k-1}$ contains exactly one vertex of maxdegree $k$. 
It implies that $G-H_v$ is a $k$-path and that $G$ is a cyclic composition of members of $\mathcal B_{k-1}$ and thus a $k$-tree, a contradiction. 
\medskip

We may now assume that for every special vertex $v$, $H_v$ is a $k$-tree. 
Let $x$ be a special vertex and assume without loss of generality that $d^-(x) = k$. 
Let $S$ be the set of vertices in $T_x$ with in-degree $k$. 
If $T_x -S$ is non-empty, there must exists an arc $st$ where $s \in S$ and $t \in T_x - S$ (because $T_x$ is strongly connected).  
Since $H_s$ is a $k$-tree, $t$ must have in-degree at least $k-1$ in $G-s$, and thus has in-degree $k$ in $G$, a contradiction. So every vertex of $T_x$ has in-degree $k$. 
Let $y$ be an in-neighbour of $x$ in $T_x$. 
As $H_x$ is a $k$-tree, $y$ has out-degree at least $k-1$ in $H_x$, and thus has out-degree $k$ in $G$. Now, by the same reasoning as above, we get that every vertex of $T_x$ has out-degree $k$. 
This proves that for every special vertex $v$, every vertex $u$ in $T_v$ has in- and out-degree $k$.  

Let $x$ be a special vertex. We know that $H_x$ is a $k$-tree. So every vertex of $H_x$ is contained in a subdigraph isomorphic to a member of $\mc B_{k-1}$, and thus has in- and out-degree $k$ in $G$. 
Hence, every vertex of $H_x$ has in- and out-degree $k$ in $H_x$ except the neighbours of $v$. So $H_v$ has at most $k$ vertices of mindegree $k-1$. 
If $k \geq 3$, it implies that $H_x$ is isomorphic to $\olra{K_k}$ and thus $G = \olra K_{k+1}$, a contradiction. And if $k=2$, it implies that $H_x$ is a symmetric path of odd length (obtained by doing a sequence of cyclic compositions of digons) and thus $G$ is a symmetric cycle of odd length, a contradiction.

\end{proof}

\begin{theorem}
Let $G$ be a connected digraph with $\Delta_{max}(G) = k$.
Then $\dic(G) \leq k+1$
and equality occurs if and only if $G$ is a member of $\mathcal B_k$. 
\end{theorem}

\begin{proof} 
We proceed by induction on $k$, so we may assume $k \geq 2$. 
If $G$ is a member of $\mathcal B_k$, then we are done.  If $G$ is a $k$-tree, then it is $k$-dicolourable because members of $\mathcal B_{k-1}$ are $k$-dicolourable, and compositions preserve $k$-dicolourability. 

So, by Lemma \ref{lem:ktree}, $G$ has a vertex $v_1$ with $d_{max}(v_1) = k$ and such that no connected component of $G-\{v_1\}$ is a $k$-tree. Let $G_2, \dots, G_r$ be the connected components of $G - \{v_1\}$. 
Observe first that each $G_i$ has a vertex with mindegree at most $k-1$, so it is not a member of $\mathcal B_k$. 
For each $G_i$, either $\dmax(G_i) \leq k-1 $ and since $G_i$ is not a $k$-tree, it is $k-1$-dicolourable by induction, or, by Lemma~\ref{lem:ktree}, $G_i$ contains a vertex $v_i$ such that the maxdegree of $v_i$ in $G_i$ equal $k$ and no connected component of $G_i - \{v_i\}$ is a $k$-tree. In the latter case, we choose such a vertex $v_i$ and continue this procedure on the connected components of $G_i \setminus \{v_i\}$ and so on. 

We obtain a set of ordered vertices $v_1, \dots, v_s$ (we apply the procedure level by level, putting an arbitrary order inside each level) such that $v_i$ has either no in-neighbour or no out-neighbour in $\{v_1, \dots, v_{i-1}\}$ (because maxdegree of $v_i$ in $G_i$ is $k = \dmax(G)$). So the digraph induced by $\{v_1, \dots, v_s\}$ is acyclic. 
Moreover, $G-\{v_1, \dots, v_s\}$ is made of vertex disjoint $(k-1)$-dicolourable induced subgraph of $G$. Hence, $G$ is $k$-dicolourable.  
\end{proof}


\section{Partitioned dicolouring}
In this section, we adapt a proof of Brooks' Theorem based on a specific partition of the vertices introduced by Lov\'asz in~\cite{L66}.  See section 5 of \cite{CR14} for the undirected version of the proof as well as a short history of the involved methods. The same kind of method has been recently used in~\cite{BSS21} to prove a generalisation of the directed Brooks' Theorem. 

Let $G = (V,A)$ be a digraph. 
We say that $G$ is \emph{$r$-special} if for every vertex $v \in V$, either $d_{min}(v) < r$ or $d_{min}(v) = d_{max}(v) = r$ (note that last equality is equivalent to $d^+(v)= d^-(v) = r$). 
Let $r_1$ and $r_2$ be two positive integers. A partition $\mc P = (V_{1}, V_{2})$ of $V(G)$ is \emph{($r_1, r_2$)-normal} if it minimizes $r_{2}|A(G[V_1])| + r_{1}|A(G[V_2)]|$.

Next observation is used frequently in the proof and is a basic property of $(r_1,r_2)$-normal partition. 
\begin{observation}\label{obs:special}
Let $G$ be a digraph. If $\mc P$ is a $(r_1, r_2)$-normal partition of $G$ with $r_1 + r_2 \geq \dmax(G) \geq 1$, then $G[V_{1}]$ is $r_{1}$-special and $G[V_{2}]$ is $r_{2}$-special. 
\end{observation}

\begin{proof}
Assume for contradiction and without loss of generality that $G[V_1]$ is not $r_1$-special. Then there is $v_1 \in V_1$ such that $d_{min}(v_1) \geq r_1$ and $d_{max}(v_1) \geq r_1 + 1$ in $G[V_1]$. Assume without loss of generality that $d^-_{G[V_1]}(v_1) \geq r_1$ and $d^+_{G[V_1]}(v_1) \geq r_1 + 1$. 

Set $V'_1 = V_1 \setminus \{v_1\}$ and $V'_2 = V_2 \cup \{v_1\}$ and let us prove that the partition $(V'_1, V'_2)$ contradicts the fact that $(V_1, V_2)$ is $(r_1, r_2)$-normal. 
Since $r_1 + r_2 \geq \dmax(G)$, we have that $d^+_{G[V'_2]}(v_1) \leq r_2 - 1$ and $d^-_{G[V'_2]}(v_1) \leq r_2$. Hence:
$$(r_{2}|A(G[V_1])| + r_{1}|A(G[V_2)]|) - (r_{2}|A(G[V'_1])| + r_{1}|A(G[V'_2)]|)$$
$$ \leq -(2r_1+1)r_2 +r_1(2r_2-1) = -r_1-r_2<0$$

a contradiction. 

\end{proof}

Let $G$ be a digraph, and  $\mc P$  a $(r_1,r_2)$-normal partition of $G$ with $r_1 + r_2 \geq \dmax(G)$. 
We define the \emph{$\mc P$-components} of $G$ as the connected components of $G[V_{1}]$ and $G[V_{2}]$. 
A $\mc P$-component is an \emph{obstruction} if it is a member of $\mathcal B_{r_{1}}$ in $G[V_{1}]$ or a member of $\mathcal B_{r_{2}}$ in $G[V_{2}]$. 
A path $v_1 \dots v_k$  in the underlying graph of $G$ is \emph{$\mc P$-acceptable} if $v_{1}$ is in an obstruction and  vertices of $\mc P$ are in pairwise distinct $\mc P$-components.
We say that a  $\mc P$-acceptable path is \emph{maximal} if every neighbour of $v_{k}$ is in the same $\mc P$-component as some vertex in the path. 
Given a partition $\mc P$, to \emph{move a vertex $u$} is to move it to the other part of $\mc P$. 

Observation~\ref{obs:special} together with the fact that digraphs in $\mathcal B_k$ are $k$-regular easily implies the following facts that will be used routinely during the proof:
\begin{itemize}
    \item If a $\mc P$-component contains an obstruction, then the obstruction is the whole $\mc P$-component. 
    \item If a vertex $u$ is in an obstruction, then the partition created by moving $u$ is again $(r_1, r_2)$-normal.
\end{itemize}

\begin{lemma} \label{lem:partition}
Let $k \geq 2$. 
Let $G = (V, A)$ be a $k$-regular connected digraph not in $\mathcal B_{k}$ and let $r_{1}, r_{2} \geq 1$ such that $r_{1} + r_{2} = k$. There exists a $(r_1, r_2)$-normal partition $(V_{1},V_{2})$ such that, for $i \in \{1,2\}$, $G[V_{i}]$ is $r_i$-special and has no obstruction.
\end{lemma}

\begin{proof}
By Observation~\ref{obs:special}, for every $(r_1,r_2)$-normal partition $(V_1, V_2)$, $G[V_i]$ is $r_i$-special for $i=1, 2$. 

Suppose that the lemma is false and let $G$ be a counterexample.  Among the $(r_1, r_2)$-normal partitions of $G$ with the minimum number of obstructions, let $\mc P = (V_1, V_2)$ be one with the shortest maximal $\mc P$-acceptable path $v_{1}\dots v_{\ell}$. 
We refer to the minimality of the number of obstructions by saying ``by minimality of $\mc P$", and to the minimality of the $\mc P$-acceptable path by saying ``by minimality of $\ell$".

Throughout the proof, we often move some vertex $u$ that belongs to an  obstruction $A$. 
Since this destroys $A$ and results in a $(r_1,r_2)$-normal partition, the minimality of $\mc P$  implies that the move creates a new obstruction and thus the obtained partition has the same number of obstructions as $\mc P$. 
Moreover, this new obstruction  contains $u$  and  the neighbours of $u$ in the other part. This implies that the neighbours of $u$ in the other part are contained in a single $\mc P$-component $C$ (because obstructions do not have cut-vertex),   and that $C \cup u$ is an obstruction. Finally, note that an obstruction containing a digon is a symmetric digraph. 
These facts are constantly used in the proof. 

Let $A$ and $B$ be the $\mc P$-components containing $v_{1}$ and $v_{\ell}$ respectively. 
Let $X=N_{A}(v_{\ell})$. 
\medskip 

Assume  $X= \emptyset$. 
Moving $v_{1}$ creates a new $(r_1, r_2)$-normal partition $\mc P'$. Since $v_{1}$ is adjacent to $v_{2}$, the new obstruction contains $v_{2}$. Moreover, $A \setminus v_{1}$ is not an obstruction. So $v_{2}, v_{3} \dots v_{\ell}$ is a maximal $\mc P'$-acceptable path, violating the minimality of $\ell$. Hence $|X| \geq 1$.
\medskip 

\begin{figure}[ht]
    \centering
    \begin{tikzpicture}
   
        \node (V1) at (-5,0) {$V_1$};
        \node (V2) at (5,0) {$V_2$};
    
        \node (A) at (-4,4) {$A$};
        \node (B) at (3.5,-3.5) {$B$};
        \node (X) at (-3,0.5) {$X$};
        
        \node (down) at (0,-4) {};
        \node (up) at (0,4) {};
        \draw (down) -- (up);
        
        \node (ul) at (-3,3) {};
        \node (ur) at (3,3) {};
        \node (dl) at (-3,-3) {};
        \node (dr) at (3,-3) {};

        \draw (ul) circle (2);
        \filldraw[pattern=north west lines, pattern color=red] (ul) circle (2);
        
        \draw (ur) circle (1);
        \draw (dl) circle (1);
        
        \draw (dr) circle (1);
    
        \draw[thick] (-2,1.3) arc (0:180:1);

        \node (ul_v) at (-2,3) [vertex] {$v_1$};
        \node (ur_v) at (2.5,3) [vertex] {$v_2$};
        \node (dl_v) at (-2.5,-3) [vertex] {$v_3$};
        \node (dr_v) at (2.5,-3) [vertex] {$v_4$};

        \draw[gray] (dr_v) -- (-3, 1.3);
        
        \draw (ul_v) -- (ur_v);
        \draw (ur_v) -- (dl_v);
        \draw (dl_v) -- (dr_v);
    
    \end{tikzpicture}

    \caption{A partition $\mc P = (V_1, V_2)$, a maximal $\mc P$-acceptable path along with $X$, $A$ and $B$ as defined in the proof of Lemma~\ref{lem:partition}. Red colour indicates an obstruction. $G[V_1]$ is $r_1$-normal and $G[V_2]$ is $r_2$-normal.} 
    \label{fig:my_label}
\end{figure}
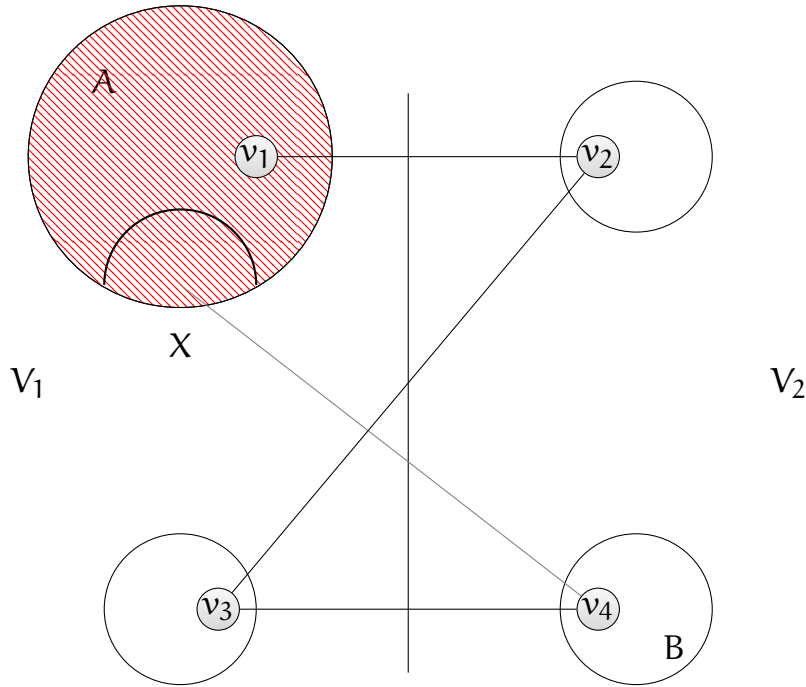

Assume now that $|X| \geq 2$ and let $x_1$, $x_2$ be two vertices in $X$. 
Let us first prove that $G[X \cup v_{\ell}$ is a symmetric complete graph. 
Assume that $x_1v_{\ell} \in A(G)$ (the case $v_{\ell}x_1 \in A(G)$ is similar).
As explained above (in the second paragraph of the proof), $B \cup x_1$ and $B \cup x_2$ are obstructions, which implies that $x_2v_{\ell} \in A(G)$. This is because $x_1v_{\ell}$ is an arc and obstructions are regular. 
By moving $x_1$ and then $v_{\ell}$, we get that $(A \sm x_1) \cup v_{\ell}$ is an obstruction, so $x_2x_1 \in A(G)$ (again because obstructions are regular). 
Similarly, $(A \sm x_2) \cup v_{\ell}$ is an obstruction and thus $x_1x_2 \in A(G)$. So $x_1$ and $x_2$ are linked by a digon, which implies that $v_{\ell}$ is linked to $x_1$ and $x_2$ by digons (this is again because obstructions are regular and $(A \sm x_1) \cup v_{\ell}$ and $(A \sm x_2) \cup v_{\ell}$ are obstructions).  
We deduce that $G[X \cup v_{\ell}]$ is a symmetric complete graph  

Let us now prove that $G[A\cup v_{\ell}]$ is a symmetric complete digraph. 
Since $A$ is an obstruction and $x_1$ and $x_2$ are linked by a digon, $A$  induces a symmetric digraph. 
If $A = X$ we are done, so we may assume that $A$ has at least three vertices. 
Since $(A \sm x_1) \cup v_{\ell}$ is an obstruction, $v_{\ell}$ has at least two neighbours in  $A \sm x_1$ and thus $|X| \geq 3$. 
Since $X$ induces a complete symmetric digraph, $A$ contains a symmetric triangle and thus must be a symmetric complete digraph. This implies that $G[A\cup v_{\ell}]$ is a symmetric complete digraph as announced. 

Let us now prove that $B$ and $A \cup B$ also induce a symmetric complete graph. 
Since $G[A \cup v_{\ell}]$ induces a complete symmetric digraph, for every $a \in A$, $B \cup a$ is an obstruction. This implies that each vertex of $A$ shares the same neighbourhood in $B$ and that $B$ induces a symmetric digraph.  
If $B = \{v_{\ell}\}$ we are done, so $B$ has at least two vertices. Let $a \in A$. 
Since $B \cup a$ is an obstruction, $B \cup a$ contains a symmetric triangle, and thus $B$ is a symmetric complete digraph. Finally, it implies that for every $a \in A$, $B \cup \{a\} \sm \{v_{\ell}\}$ induces a complete symmetric digraph, and so $A \cup B$ induces a complete symmetric digraph. 

All together, this proves that $G[A] = \overleftrightarrow K_{r_1+1}$, $G[B] = \overleftrightarrow K_{r_2}$ (because for every $a \in A$, $G[A]$ is an obstruction i.e. is a member of $\mathcal B_{r_2}$, and is a complete symmetric digraph). So $A \cup B$ induces $\overleftrightarrow K_{r_1 + r_2 + 1} = \overleftrightarrow K_{k+1}$ and since $G$ is $k$-regular, $G = \overleftrightarrow K_{k+1}$, a contradiction with the hypothesis that $G$ is not a member of $\mc B_{k}$. 
\medskip 

We may assume from now on that $|X| = 1$. 
Assume first that $X= \{v_1\}$. 
Moving $v_1$ creates an obstruction containing both $v_2$ and $v_{\ell}$, so $\ell = 2$. 
Since the path $v_1v_2$ is a maximal $\mc P$-acceptable path, $v_2 = v_{\ell}$ has no neighbour in the other part besides $v_1$. 
Hence, after moving $v_1$ and $v_2$, $v_2$ is the only vertex in its component, and thus cannot be in an obstruction, a contradiction. 

So instead $X = \{x\}$ and $x \neq v_1$. 
Let us prove that $A = \{x,v_1\}$, that $A$ induces a digon, and that $v_{\ell}$ and $x$ are linked by a digon.  
In order to do so, move each $v_1, v_2, \dots, v_{\ell}$ in turns.  
Moving $v_1$ destroys  $A$ and thus  creates a new obstruction containing $v_2$.  
For $1 \leq i \leq \ell-2$, moving $v_i$ creates a new obstruction containing $v_{i+1}$, which in turns is destroyed by the move of $v_{i+1}$, creating a new obstruction containing $v_{i+2}$. 
Finally, after the move of $v_{\ell}$, $v_{\ell}$ is in an obstruction containing $x$ and  since $|X| = 1$, this new obstruction only contains $x$ and $v_{\ell}$, and thus is a digon. This also implied that $A = \{v_1, x\}$ and thus induces a digon. Moreover, it implies that $r_1 = 1$.   

Moving $v_1$ creates an obstruction containing $v_2$. By minimality of $\ell$, in the new partition $\mc P'$ obtained after moving $v_1$, the path $v_2v_3 \dots v_{\ell}x$ is a maximal $\mc P'$-acceptable path. 
So the obstruction containing $v_2$ (the first obstruction of a maximal acceptable path) must be a $\overleftrightarrow K_2$ (for the same reason $A$ is a $\overleftrightarrow K_2$), so $v_1$ and $v_2$ are linked by a digon and $r_2 = 1$. 
Now, moving $v_1$ and then $v_2$, the same argument can be applied to the path $v_3 \dots v_{\ell}xv_1$ implying that $v_2$ is linked to $v_3$ by a digon. 
Similarly, each $v_i$ for $i=2, \dots, \ell-1$ 
is linked to $v_{i+1}$ by a digon. 
This implies that $G$ contains a symmetric cycle of odd length (namely $v_1v_2 \dots v_{\ell}xv_1$), and since $G$ is $k$-regular and we clearly have $r_1 = r_2 =1$, $G$ is equal to this symmetric odd cycle, a contradiction. 
\end{proof}

\begin{theorem}
A connected digraph $G$ has dichromatic number at most $\dmax(G) +1$ and equality occurs if and only it is a member of $\mc B_{\dmax(G)}$. 
\end{theorem}

\begin{proof}
We proceed by induction on $\Delta_{max}$. 
Let $G$ be a connected digraph with $\dmax(G) = k \geq 2$. As usual, we may assume that $G$ is $k$-regular. 
If $G$ is a member of $\mc B_k$, then we are done, so we may assume that it is not and we need to prove that $G$ is $k$-dicolourable. 
Hence, by Lemma~\ref{lem:partition}, there exists a $(1, k-1)$-normal partition $(V_1, V_2)$ such that, for $i=1, 2$,  $G[V_i]$ is $r_i$-special and has no obstruction. 
Set $G_i = G[V_i]$ for $i=1, 2$. An obstruction in $G_1$ is a directed cycle, so $G_1$ is acyclic. 
We are now going to prove that $G_2$ is $k-1$-dicolourable. 
Let $S \subseteq V(G_2)$ be the set of vertices with maxdegree $k$ in $G_2$. 
Hence, every vertex in $V(G_2) \setminus S$ has maxdegree $k-1$ (in $G_2$) and has no $\overleftrightarrow K_{k-1}$ (because $G_2$ has no obstruction) so, by minimality of $k$, $G_2 \setminus S$ is $(k-1)$-dicolourable. 
Since $G_2$ is $(k-1)$-special, vertices in $S$ have mindegree at most $k-2$ in $G_2$. Hence, we can greedily extend a $k-1$-dicolouring of $G_2 \setminus S$ to $G_2$. Using one more colour for $V_1$, we get a $k$-dicolouring of $G$. 
\end{proof}


\section{No Brooks' analogue for $\dmin$} \label{sec:dmin}

As explained in the introduction, every digraph $G$ can be dicoloured with $\dmin(G) + 1$ colours. 
In this section, we prove that given a digraph $G$, deciding if it is $\dmin(G)$-dicolourable is $NP$-complete. 
It is thus unlikely that digraphs satisfying $\dic(G) = \dmin(G)+1$ admit a simple characterization, contrary to the digraphs satisfying $\dic(G) = \dmax(G) +1$.  

It is known that for all $k \geq 2$, $k$-\textsc{dicolourability} is NP-complete~\cite{Bo04}, where $k$-\textsc{dicolourability} is the following problem:\\
\underline{Input}: A digraph $G$.\\
\underline{Question}: Is $G$ $k$-dicolourable?

\begin{theorem}
For all $k \geq 2$, $k$-\textsc{dicolourability} is $NP$-complete even when restricted to digraph $G$ with $\dmin(G) = k$. 
\end{theorem}

\begin{proof}
Let $k \geq 2$ be a fixed integer. 
As is customary, membership to $NP$ is clear.
Given a digraph $G$, we are going to construct a digraph $G'$ such that $\dmin(G') \leq k$ and $G$ is $k$-dicolourable if and only if $G'$ is $k$-dicolourable.

Let $G=(V, A)$ be a digraph. We construct $G'$ as follows:
\begin{itemize}
    \item For every vertex $u$ of $G$, put $k+1$ vertices in $G'$: $u^-$,  $u^+$, $u_1$, \dots, $u_{k-1}$.
    \item for each vertex $u$, $G'[\{u^-, u_1, \dots, u_{k-1}\}]$ and $G'[\{u^+, u_1, \dots, u_{k-1}\}]$ are complete symmetric digraphs, and $u^-u^+ \in A(G')$. 
    \item For every $uv \in A(G)$, $u^+v^- \in A(G')$. 
\end{itemize}

For every vertex $u \in V(G)$, we have $d_{G'}^-(u^+) =d_{G'}^+(u^-) = k$ and for $i=1, \dots, k-1$, $d^+_{G'}(u_i) = k$. Hence, $\dmin(G') \leq k$. 









\begin{claim} 
If $G$ is $k$-dicolourable, then $G'$ is too. 
\end{claim}

\begin{proofclaim}
Let $\varphi$ be a $k$-dicolouring of $G$. For every vertex $u$, assign to $u^-$ and $u^+$ the colour $\varphi(u)$, and the $k-1$ other colours to $\{u_1, \dots, u_{k-1}\}$. We claim this is a proper $k$-dicolouring of $G'$. 
Suppose it is not. Let $C$ be a monochromatic directed cycle in $G'$. It cannot use any vertex $u_i$ as these vertices have a colour distinct from all of their neighbours. Thus $C$ only uses arcs of the form $u^-u^+$ or $u^+v^-$ which easily implies the existence of a monochromatic directed cycle in $G$, a contradiction.  
\end{proofclaim}

\begin{claim}
If $G'$ is $k$-dicolourable, then $G$ is too. 
\end{claim}

\begin{proofclaim}

Let $\varphi$ be a $k$-dicolouring of $G'$. 
For every vertex $u \in V(G)$, for $i=1, \dots, k-1$, the vertices $u_i$ receive pairwise distinct colours. So, $\varphi(u^+)=\varphi(u^-)$. Hence, for every vertex $u \in V(G)$, assigning the colour $\varphi(u^+)$ to $u$ gives a valid $k$-dicolouring of $G$. 
\end{proofclaim}
\end{proof}


\chapter{A Brooks' theorem for local arc-connectivity} \label{chpt:lambda_brooks}

\begin{flushright}{\slshape    
This chapter is built upon a joint \\
work with Pierre Aboulker and \\
Pierre Charbit, published in \cite{AAC23}}. \\ \medskip
\end{flushright}

\emph{In this chapter, we extend directed Brooks' theorem to local arc-connectivity, characterize extremal digraphs and describe an algorithm to recognize them.}

\newcommand{\htk}{\vec{E\mc{HT}}_{k}}
\newcommand{\hk}{\vec{\mc{H}}_{k}}
\newcommand{\htthree}{\vec{E\mc{HT}}_{3}}
\newcommand{\hthree}{\vec{\mc{H}}_{3}}

\newcommand{\figg}[1]{\input{figs/#1}}

\section{Introduction}


\subsection{The undirected case}



As discussed in Chapter~\ref{chpt:brooks}, it is an easy observation that, for every graph $G$, $\chi(G) \leq  \Delta(G) + 1$, where $\Delta(G)$ is the maximum degree of $G$. 
Moreover, equality holds for odd cycles and complete graphs, and the chromatic number of a graph equals the maximum chromatic number of its connected components. This leads to a full characterization of graphs $G$ for which $\chi(G) = \Delta(G) + 1$, famously known as Brooks' Theorem. 
Let  $\mathcal B_2$ be the set of odd cycles and for $k \neq 2$, let  $\mathcal B_k = \{ K_{k+1}\}$.

\begin{theorem}[Brooks' Theorem~\cite{B41}]
    Let $G$ be a graph. Then $\chi(G) = \Delta(G) + 1=k+1$ if and only if one of the connected components of $G$ is in $\mathbb B_k$. 
\end{theorem}

We recall that, given two vertices $u,v$ of $G$, $\lambda(u,v)$ is the maximum number of edge-disjoint paths linking $u$ and $v$, and $\lambda(G)$ is the maximum local edge connectivity of $G$, that is  $\max_{u\neq v} \lambda(u,v)$.  Mader~\cite{M73} proved that for every graph $G$, $\chi(G) \leq \lambda(G) + 1$. 
Moreover, it is clear that  $\lambda(G) \leq \Delta(G)$.
Thus, for every graph $G$, 
$$ \chi(G) \leq \lambda(G) + 1 \leq \Delta(G) +1$$

Hence one can ask for graphs $G$ for which 


\begin{equation}\label{eq:nono_extremal}
  \chi(G) = \lambda(G) + 1  
\end{equation}

Exception of Brooks' Theorem of course satisfies~\eqref{eq:nono_extremal}, but it turns out that there are more. 

To describe them, we need a famous construction first introduced by Haj\'os~\cite{H61} to construct an infinite family of $k$-critical graphs. Let $G_1$ and $G_2$ be two graphs, with $uv_1 \in E(G_1)$ and  $v_2w \in E(G_2)$. 
The \emph{Haj\'os join} of $G_1$ and $G_2$ with respect to $(uv_1, v_2w)$ is the graph $G$ obtained from the disjoint union of $G_1 - uv_1$ and $G_2 - v_2w$, by identifying $v_1$ and $v_2$ to a new vertex $v$, and adding the edge $uw$. 

Recall that an  \emph{odd wheel} is a graph obtained from an odd cycle by adding a vertex adjacent to every vertex of the odd cycle. Note that $K_4$ is an odd wheel and that odd wheels satisfy~\eqref{eq:nono_extremal}. 

One can prove that the Haj\'os join $G$ of two graphs $G_1$ and $G_2$ satisfies \eqref{eq:nono_extremal} if and only if both $G_1$ and $G_2$ satisfies it. 
Moreover,  the maximum local edge-connectivity of a graph equals the maximum local edge connectivity of its blocks. 

This leads to the characterization of graphs $G$ satisfying~\eqref{eq:nono_extremal}, proven by Aboulker et al.~\cite{ABHMT17} for graphs $G$ with $\chi(G) \leq 4$, and by Stiebitz and Toft~\cite{ST16} for $\chi(G) \geq 5$.  

Let $\mathcal H_k = \mathbb B_k$ when $k \leq 2$, let $\mathcal H_3$ be the smallest class containing all odd wheels and closed under taking Haj\'os join, and for $k \geq 4$, let $\mathcal H_k$ be the smallest class of graphs containing $K_k$ and closed under taking Haj\'os join. 
\begin{theorem}[\cite{ST16}]\label{thm:nono_version}
    Let $G$ be a graph. Then $\chi(G) = \lambda(G) + 1 = k+1$ if and only if a block of $G$ is in $\mathcal H_k$. 
\end{theorem}

The goal of this chapter is to generalize this result to digraphs. 
Note that this result has already been generalized for hypergraphs. Our result also generalizes this case in a certain sense, we explain that in Section~\ref{sec:hypergraph}.


\subsection{Our result: the directed case}\label{sec:our_result}


Brooks' Theorem has been generalized to digraphs by Mohar~\cite{M10}.
In Chapter~\ref{chpt:brooks}, we discussed how there are several notions of maximum degree for a digraph, and found out that the most suitable one to generalize Brooks' theorem is the following: given a digraph $D$, let $\Delta_{max}(D)$ be the maximum over the vertices of $G$ of the maximum of their in-degree and their out-degree. 


This leads to the directed Brooks' Theorem.
Let  $\vec{\mathcal B}_1$ be the set of directed cycles,  let $\vec{\mathcal B}_2$ be the set of symmetric odd cycles and, for $k=0$ and $k \geq 3$, let  $\vec{\mathcal B}_k = \{\bid K_{k+1}\}$. 

\begin{theorem}[Directed Brooks' theorem~\cite{M10}, see also~\cite{AA22}]\label{thm:dir_brooks}
Let $D$ be a digraph. Then $\dic(D) = \Delta_{max}(D) + 1=k+1$ if and only if a strong component of $D$ is in  $\vec{\mathcal B}_k$
\end{theorem}

Let $D$ be a digraph. Recall that, given a pair of ordered vertices $(u,v)$, we denote by $\lambda(u,v)$ the maximum number of arc-disjoint directed paths from $u$ to $v$, and by $\lambda(D)$ the \emph{maximum local arc connectivity} of $D$, that is $max_{u\neq v}\lambda(u,v)$. 
Neumann-Lara~\cite{NL82} proved that for every digraph $D$, $\dic(D) \leq \lambda(D) + 1$. 
Since we clearly have that $\lambda(D) + 1 \leq \Delta_{max}(D) + 1$, we get that for every digraph $D$: 
$$ \dic(D) \leq \lambda(D) + 1 \leq \Delta_{max}(D)  + 1 $$

The main result of this chapter is  a full characterization of digraphs $D$ for which 
\begin{equation}\label{eq:extremal}
\dic(D) = \lambda(D) + 1    
\end{equation}
 when $\dic(D) \geq 4$. This generalizes directed Brooks' Theorem and Theorem~\ref{thm:nono_version} for digraphs with a dichromatic number at least $4$. 
\medskip 

There are two easy observations one can make about digraphs that satisfy~\ref{eq:extremal}. First, a digraph satisfies this property if and only if one of its strong components satisfies it. Indeed, for both $\chi$ and $\lambda$, the value for a digraph is the maximum of the values of its strong components, which implies our claim because $\dic(G)\leq \lambda(D)+1$ for every digraph. Second, and for the same exact reason, if a digraph $D$ is strongly connected but has cutvertices, then $D$ satisfies~\ref{eq:extremal} if and only if one of its blocks satisfies~\ref{eq:extremal}. Note that these blocks also induce strongly connected digraphs.\\

Hence our main theorem will be a structural characterization of the class of \emph{$k$-extremal digraphs} (for $k\geq 3$ and $k=1$) where a digraph $D$ is $k$-extremal if $D$ is strongly connected, its underlying graph is $2$-connected, and $\dic(D) = k+1 = \lambda(D) + 1 $.\\

Characterizing $1$-extremal digraphs is rather easy, we will prove in Section~\ref{sec:kextr} (Theorem~\ref{thm:1extremal}) that they are exactly the class of directed circuits. Studying $k$-extremal digraphs for larger values of $k$ requires more engineering. Mimicking the construction appearing in Theorem~\ref{thm:nono_version}, we need to come up with an analogue of Haj\'os joins for digraphs. 
As a matter of fact, we will need a wild generalization of Haj\'os join, giving a new way to construct $k$-dicritical  digraphs that is interesting on its own.

The most natural way to generalize Haj\'os join is to take the same definition, and replace edges by digons. Given two vertices $u$ and $v$, we set $[u,v] = \{uv, vu\}$.

\begin{definition}[Bidirected Haj\'os join]\label{def:bidirected_HJ}
  Let $D_1$ and $D_2$ be two digraphs, with $[u,v_1] \subseteq A(D_1)$ and  $[v_2,w] \subseteq A(D_2)$. 
The \emph{bidirected Haj\'os join} of $D_1$ and $D_2$ with respect to $([u,v_1], [w, v_2])$ is the digraph $D$ obtained from the disjoint union of $D_1- [u,v_1]$ and $D_2 - [w, v_2]$, by identifying $v_1$ and $v_2$ to a new vertex $v$, and adding the digon $[u,w]$.   
\end{definition}

Bidirected Haj\'os joins were first introduced and studied in~\cite{BBSS20}. 

Still inspired by the Haj\'os joins, the so-called \emph{directed Haj\'os join}, first introduced in~\cite{HK15} and also studied in~\cite{BBSS20} is  defined as follows.  See for example Figure~\ref{fig:directed_join}.

\begin{definition}[Directed Haj\'os join]\label{def:directed_HJ}
Let $D_1$ and $D_2$ be two digraphs, with $uv_1 \in A(D_1)$ and  $v_2w \in A(D_2)$. 
The \emph{directed Haj\'os join} of $D_1$ and $D_2$ with respect to $(uv_1, v_2w)$ is the digraph $D$ obtained from the disjoint union of $D_1- uv_1$ and $D_2 - v_2w$, by identifying $v_1$ and $v_2$ to a new vertex $v$, and adding the arc $uw$.   
\end{definition}

    \begin{figure}[!hbtp]
    \begin{center}
        \begin{tikzpicture}[scale=0.5]

            \begin{scope}[yshift=2cm]
            \begin{scope}[rotate = -60]
                \vertex (x) at (0,0) {$v$};
                \vertex (v1) at (5,0) {$w$};
                \draw (2.5,0) ellipse (3cm and 2cm) {};
            \end{scope}

            \begin{scope}[rotate = -120]
                \vertex (v2) at (5,0) {$u$};
                \draw (2.5,0) ellipse (3cm and 2cm) {};
            \end{scope}
            \end{scope}

            \draw[->-] (v2) to (v1);

            \begin{scope}[shift=(v2)]
                \begin{scope}[xshift=-18cm,rotate = 60]
                    \vertex (l) at (0,0) {$u$};
                    \vertex (v) at (5,0) {$v_1$};
                    \draw[->-] (l) to (v);
                    \draw (2.5,0) ellipse (3cm and 2cm) {};
                    \node () at (-2,0) {$D_1$};
                \end{scope}
    
                \begin{scope}[xshift=-10cm, rotate = 120]
                    \vertex (l) at (0,0) {$w$};
                    \vertex (v) at (5,0) {$v_2$};
                    \draw[->-] (v) to (l);
                    \draw (2.5,0) ellipse (3cm and 2cm) {};
                    \node () at (-2,0) {$D_2$};
                \end{scope}
            \end{scope}

            \node () at (0,-4) {$D$};
           
        \end{tikzpicture}
        \end{center}
           \caption{$D$ is a directed Haj\'os join of $D_1$ and $D_2$.}
           \label{fig:directed_join}
        \end{figure}
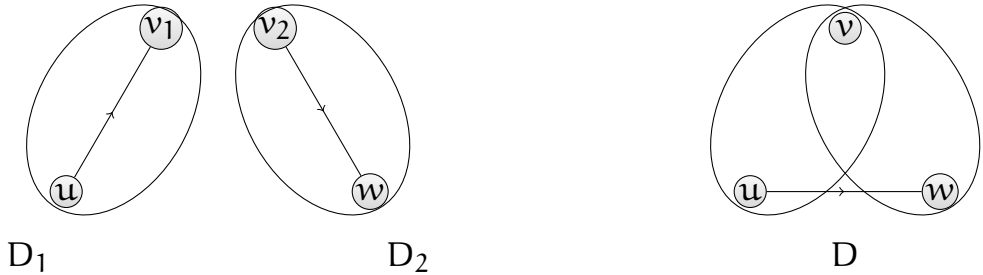

These two operations are particularly interesting because one can prove that the bidirected (directed) Haj\'os join 
$D$ of two digraphs $D_1$ and $D_2$ is $k$-dicritical if and only if both $D_1$ and $D_2$ are $k$-dicritical. They, therefore, provide a way to construct an infinite family of $k$-dicritical digraphs. They are also primordial for us because $D$ satisfies~\eqref{eq:extremal} if and only if both $D_1$ and $D_2$ do. Thus, they also provide a way to construct an infinite family of digraphs satisfying~\eqref{eq:extremal}.    

But these two joins are not enough to capture all digraphs satisfying~\eqref{eq:extremal}. In order to do so, we need to define Haj\'os tree joins, that can be seen as a generalization of  bidirected Haj\'os join. See Figure~\ref{fig:HTcomplete}.

\begin{definition}[Haj\'os tree join and Haj\'os star join]\label{def:HTJ}
Given
\begin{itemize}
    \item a tree $T$ embedded in the plane with edges $\{u_1v_1, \dots, u_nv_n\}$, $n \geq 2$, 
    \item A circular ordering $C=(x_1, \dots, x_{\ell})$ of the leaves of $T$, taken following the natural ordering given by the embedding of $T$,  and
    \item for $i=1, \dots, n$, $D_i$, a digraph such that $V(D_i) \cap V(T) = \{u_i,v_i\}$, $[u_i,v_i] \subseteq A(D_i)$, 
    \item For $1 \leq i \neq j \leq n$, $V(D_i) \setminus \{u_i, v_i\} \cap V(D_j) \setminus \{u_j, v_j\} = \emptyset$ 
\end{itemize} 

we define the \emph{Haj\'os} tree join $T(D_1, \dots, D_n; C)$ to be the digraph obtained from the $D_i$ by adding the directed cycle  $C = x_1 \ra x_2 \ra \dots \ra x_{\ell} \ra x_1$. 

$C$ is called the  \emph{peripheral cycle} of $D$ and  vertices $u_1, v_1, \dots, u_n,v_n$ are the \emph{junction vertices} of $D$ (note that there are $n-1$ of them). 

When $T$ is a star, we call it \emph{Haj\'os star join}. 
\end{definition}

\input{figs/figHTcomplete}

Note that, when $T$ is the path on three vertices, we recover a bidirected haj\'os join. 
  

The basic idea of the Haj\'os tree join is the following: if each $D_i$ is $k$-dicritical, then any $(k-1)$-dicolouring of $D_i-[a_i,b_i]$ give the same colour to $a_i$ and $b_i$, which implies that in any $(k-1)$-dicolouring of $D-A(C)$, all the junction vertices receive the same colour, and thus $D$ is not $(k-1)$-dicolourable. 
Actually, we can prove that $D$ satisfies~\eqref{eq:extremal} if and only each of the $D_i$ does. Hence, Haj\'os tree joins also provides a way to construct an infinite family  of digraphs satisfying \eqref{eq:extremal}.



\begin{definition}   
Let $\vec{\mathcal H}_3$ be the smallest class of digraphs containing all bidirected odd wheels and closed under taking directed Haj\'os joins and Haj\'os tree joins, and for $k \geq 4$, let  $\hk$ be the smallest class of digraphs containing $\bid{K}_{k+1}$ and closed under taking directed Haj\'os joins and Haj\'os tree joins.
\end{definition}

\begin{theorem}\label{thm:main_HT}
    Let $k\geq 3$. Let $D$ be a digraph. Then $\dic(D) = \lambda(D) + 1=k+1$ if and only if a strong biconnected component of $D$ is in $\hk$
\end{theorem}

In Section~\ref{sec:algo}, we also give an algorithm that decides in polynomial time if a digraph $D$ belongs to $\vec {\mathcal H}_k$. 

The rest of the chapter is organized as follows. In Section~\ref{sec:kextr} we prove several important structural properties of $k$-extremal digraphs. In Section~\ref{sec:decthm} we prove a first step towards the main theorem by giving a decomposition theorem for the class, and in the following section we give the final proof of Theorem~\ref{thm:main_HT}. In Section~\ref{sec:algo} we give a polynomial time algorithm to recognize $k$-extremal digraphs, and in the last section, we discuss the case $k=2$ (open).

\section{Tools}\label{sec:tools}

In \cite{M27}, Menger proved the following fundamental result connecting dicuts and arc-disjoint dipaths:

\begin{theorem}[Menger Theorem \cite{M27}]\label{thm:menger}
    Let $D$ be a directed multidigraph and let $u, v \in V(D)$ be a pair of distinct vertices. Then $\lambda(u,v)=\partial^+(U)$ where $(U, \overline{U})$ is a minimum  dicut separating $u$ from $v$. 
\end{theorem}



\Lov~\cite{L73} proved the following result (note that \Lov proved it for digraphs, but the proof also works for multidigraphs):
\begin{theorem}[\Lov~\cite{L73}]\label{thm:lovasz_lambda-reg}
Let $D$ be a multidigraph in which $\lambda(x,y) = \lambda(y,x)$ for any $x, y \in V(D)$. Then $D$ is Eulerian.
\end{theorem}

The next lemma is crucial as it describes the structure of minimal cuts in $k$-extremal digraphs.

\begin{lemma}\label{lem:mono_vs_rainbow}
    Let $D$ be a digraph such that $\dic(D) > k$. If $(X,\overline{X})$ is a dicut of $D$ of size at most $k$ such that $D[X]$ and $D[\overline X]$ are both $k$-dicolourable, then there exists a dicut $(M,R)$ (for monochromatic and rainbow) of $D$ such that 
    \begin{itemize}
    \item $(M,R)=(X,\overline{X})$ or $(M,R)=(\overline X,X)$,
    \item In every $k$-dicolouring of $D[M]$, the vertices in $M\cap N(R)$ all receive the same colour,
    \item In every $k$-dicolouring of $D[R]$, all $k$ colours must appear in $R\cap N(M)$.
    \end{itemize} 
    In particular, the cut has a size of exactly $k$ and for every $k$-colouring of $R$, there is exactly one arc in each direction between $M$ and each of the $k$ colour classes of $R$.
\end{lemma}

\begin{proof}
Assume $\varphi_X$ and $\varphi_{\overline X}$ are  $k$-dicolouring of respectively $D[X]$, and $D[\overline X]$. 
    Let $B$ be the bipartite graph with parts $U = \{1,\dots,k\}$ and $V = \{1,\dots,k\}$, and edge-set $\{\varphi_{X}(u)\varphi_{\overline{X}}(v) \mid uv \in A(D), u \in X , v \in \overline{X})\}$.
    Note that by construction, there is an injection from $E(B)$ to the set of arcs from $X$ to $\overline X$, so that $|E(B)|\leq k$. Also if a vertex of $B$ has degree $0$ it means the corresponding colour is not used in the dicolouring of $X$ or $\overline X$.
      
     Let $H$ be the complement of $B$, that is $V(B) = V(H)$ and $E(H) = \{uv \mid u \in U, v \in V, uv \notin E(B)\}$. 
    
    Suppose first that $H$ has a perfect matching $M$.
     Let $\varphi:V(D) \leftarrow [k]$ be defined as follows:
    for every $x \in \overline{X}$, $\varphi(x)=\varphi_{\overline{X}}(x)$ and, for every $x \in X$, if $\varphi_{X}(x)=i$, then $\varphi(x)=j$ where $ij \in M$.  
    There is no monochromatic dicycle in $D[X]$, since $\varphi$ is a permutation of $\varphi_X$ on $X$, and there is no monochromatic arc from $X$ to $\overline{X}$. Thus $\varphi$ is a  $k$-dicolouring of $D$, a contradiction.

    Thus, $H$ has no perfect matching. By Hall's marriage theorem, there is $Z \subseteq U$ such that $|N_{H}(Z)| < |Z|$. Thus, there are all possible edges between $Z$ and $N_H(Z)$ in $B$ and by counting the number of these edges we get:
    \begin{align*}
        k \geq |E(B)| & \geq |Z| (k-|N_H(Z)|) \\
                            & \geq |Z|(k-(|Z|-1))
    \end{align*}
    Hence  $(k - |Z|)(|Z| - 1) \leq 0$. 
    But as $1 \leq |Z| \leq k$,  this implies $|Z| = 1$ or $|Z| = k$. In both cases, we have $\Delta(B) = k$. but since the dicut $(X,\overline X)$ contains at most $k$ arcs, it means the digraph $B$ is simply a star with $k$ leaves, which is exactly what we claim.
\end{proof}


\begin{corollary}\label{coro:small_cut}
Let $D$ be a digraph and $V=V_1\cup V_2$ be a partition of its vertex. If the dicut $(V_1,V_2)$ contains strictly less than $k$ arcs, and $\dic(D[V_i])\leq k$ for $i=1,2$, then $\dic(D)\leq k$. As a consequence for any digraph $\dic(D)\leq \lambda(D)+1$. 
\end{corollary}

When proving that some digraphs have small maximum local edge connectivity, we will often use the following Lemma:

\begin{lemma}\label{lem:lambda_diminishing}
    Let $D$ be a digraph, $u \neq v \in V(D)$ and $P$ a $uv$-dipath. Then $\lambda(D + uv - A(P)) \leq \lambda(D)$.
\end{lemma}

\begin{proof}
    Let  $H = D + uv \setminus A(P)$. 
    Assume for contradiction that there exists $x,y \in V(D)$ that are linked by $\lambda(D)+1$ arc-disjoint $xy$-dipaths.  Since these dipaths cannot all exist in $D$, one of them contains the arc $uv$. Then, by replacing $uv$ by $P$, we obtain $\lambda(D)+1$ arc-disjoint $xy$-dipaths in $D$, a contradiction. 
\end{proof}


\section{First properties of $k$-extremal digraphs}\label{sec:kextr}

Recall that a digraph $D$ is \emph{$k$-extremal} if it is biconnected, strong and $\dic(D) = \lambda(D) + 1 = k+1$. 
The following lemma proves easy but fundamental properties of $k$-extremal digraphs that will be used constantly in the proofs. 

\begin{lemma}\label{lem:prop_k-extremal}
    Let $k \geq 1$, and let $D$ be a $k$-extremal multidigraph. Then $D$ is Eulerian, $(k+1)$-dicritical and $\lambda(x,y)=k$ for every pair of distinct vertices $x$ and $y$. 
    In particular, all minimum dicuts $(X, \overline{X})$ of $D$ satisfies $\partial^+(X) = \partial^-(X) =k$
\end{lemma}

\begin{proof}
    Let $D$ be a $k$-extremal multidigraph, and assume $D$ is a minimal counter-example. 

    We first prove that $D$ is $(k+1)$-vertex-dicritical. 
    We proceed by contradiction. Let $X \subsetneq V(D)$ be minimal such that $\dic(D[X]) =k+1$. By minimality of $X$, $D[X]$ is biconnected and strong. 
    Moreover, since  $k+1 = \dic(D[X]) \leq \lambda(D[X]) + 1 \leq \lambda(D)+1 \leq  k+1$, we have $\lambda(D[X]) = k$. 
    So $D[X]$ is $k$-extremal and thus, by minimality of $D$, $\lambda_{D[X]}(u,v)=k$ for every pair of distinct vertices $u$, $v$ in $X$.    
    
    Let $x \in X$ such that $x$ has an out-neighbour in $\overline{X}$ (it exists because $D$ is strong). Let $R^+(x)$ (resp. $R^-(x)$) be the set of vertices $y \in V(D) \setminus X$ such that there is a $xy$-dipath (resp. a $yx$-dipath) with vertices in $\overline{X} \cup \{x\}$. 
    Let $y \in R^+(x)$. Since $D$ is strong, there exists a shortest dipath $P$ from $y$ to $X$. Let $x' \in X$ be the last vertex of $P$. If $x \neq x'$, then $\lambda_D(x,x') \geq \lambda_{D[X]}(x,x')+1 = k+1$, a contradiction. So $x=x'$ and thus $y \in R^-(x)$.  Hence, $R^+(x) \subseteq R^-(x)$ and similarly $R^-(x) \subseteq R^+(x)$. So $R^+(x) = R^-(x)$ and we set $R(x)= R^+(x)$. 
    Since $D$ is biconnected, there exists a shortest path $P=x_1\dots x_{\ell}$ (in the underlying graph of $D$) linking $X\setminus \{x\}$ with $R(x)$ with $x_1 \in X$ and $x_{\ell} \in R(x)$.
    If $\ell \geq 3$, then $x_{\ell-1} \in V(D) \setminus (X \cup R(x))$. But then if $x_{\ell - 1}x_{\ell} \in A(D)$, then $x_{\ell - 1} \in R^-(x)$ and if $x_{\ell} x_{\ell - 1} \in A(D)$, then $x_{\ell - 1} \in R^+(x)$, and thus $x_{\ell-1} \in R(x)$ in both cases, a contradiction. So $\ell =2$. 
    But then $\lambda_D(x_1,x) = k+1$ or $\lambda_D(x,x_1)=k+1$, a contradiction. 
    This proves that $D$ is $(k+1)$-vertex-dicritical. 
    \medskip

    Let $x,y \in V(D)$ and assume for contradiction that $\lambda_D(x,y) \leq k-1$. Then, by Menger Theorem~\ref{thm:menger}, $D$ has a dicut $(X, \overline{X})$ of size at most $k-1$ with $x \in X$ and $y \in \overline{X}$.  
    We have that $\dic(X) \leq k$ and $\dic(\overline{X}) \leq k$ and thus, by Corollary~\ref{coro:small_cut}, $\dic(D) \leq k$, a contradiction. 

    Let $xy \in A(D)$, and let $H= D - xy$. Since $\lambda_D(x,y) = k$, $\lambda_H(x,y) =k-1$ and, as above, $\dic(H) \leq k$. So $D$ is $k+1$-dicritical. 
    
    Finally, by Theorem~\ref{thm:lovasz_lambda-reg}, $D$ is Eulerian. 
\end{proof}

As a direct consequence, we can prove the characterization of $1$-extremal digraphs.

\begin{theorem}\label{thm:1extremal}
    A digraph is $1$-extremal if and only if it is a directed cycle.
\end{theorem}

\begin{proof}
    It is clear that all directed cycles are $1$-extremal. Conversely, let $D$ be a $1$-extremal digraph. Then $\dic(D) = 2$, thus $D$ admits an induced directed cycle on vertex set $X \subseteq V(D)$. By Lemma~\ref{lem:prop_k-extremal}, $D$ is $2$-dicritical. But $\dic(D[C]) = 2$, thus $X = V(D)$. Hence, $D$ is a directed cycle. 
\end{proof}

Since a $k$-extremal digraph is $k+1$-dicritical, for any arc $uv$ there exists a $k$-dicolouring of $D-uv$, and if we put back the arc, then all monochromatic cycles must go through $uv$, and thus correspond to a monochromatic $vu$-dipath. In the case of a digon, we can say more.
\begin{lemma}\label{lem:extremdigon}
    If $D$ is $k$-extremal and $[uv] \subseteq A(D)$, there exists a $k$-dicolouring of $D-[uv]$ such that there is no monochromatic $uv$-dipath nor any monochromatic $vu$-dipath.
\end{lemma}
In other words, there exists an assignment of $k$ colours to the vertices of $D$ such that the only monochromatic cycle is the digon $[uv]$.
\begin{proof}
Consider a digon $[uv]$ in a $k$-extremal digraph $D$. There must be a minimal cut with $k$ arcs in each direction separating $u$ from $v$, so let $(M,R)$ be such a cut as in the statement of Lemma~\ref{lem:mono_vs_rainbow}. Assume without loss of generality that $u\in M$ and $v\in R$. Since $D$ is vertex critical, we can give a proper dicolouring to both $D[M]$ and $D[R]$ and up to permuting the colours we assign colour $1$ to $u$ and $v$. By the conclusion of Lemma~\ref{lem:mono_vs_rainbow} regarding the structure of the cut (each colour class in $
R$ has exactly one arc going to $M$ and one arc coming from $M$), the only possible monchromatic cycle is the digon $uv$
\end{proof}

In this chapter, we will sometimes need to contract one side of a minimum dicut and apply induction on the obtained digraph. For this to work properly, we need to ensure that the dicut does not isolate any vertex, so that the obtained digraph is strictly smaller than the original digraph. To prove that we can always find such a dicut, we use a method derived from \cite{R14} (see also Section 3 of \cite{AA22} for its use in proving Brooks' theorem for digraphs).

\begin{lemma}\label{lem:all_cuts_isolate_vertices}
Let $k \geq 4$. If all minimum dicuts of a $k$-extremal digraph $D$ isolate a vertex, then $D = \bid K_{k+1}$.
\end{lemma}

\begin{proof}
    Let $D$ be a $k$-extremal digraph in which every minimum dicut isolates a vertex. If every vertex of $D$ has indegree and outdegree at most $k$, then $D = \bid K_{k+1}$ by Theorem~\ref{thm:dir_brooks}. Otherwise, since every vertex of $D$ has indegree and outdegree at least $k$, $D$ has a vertex with outdegree strictly greater than $k$ or a vertex with indegree strictly greater than $k$. 

    If there are two distinct vertices $u$ and $v$ with $d_{max}(u) \geq k+1$ and $d_{max}(v) \geq k+1$, then a minimum $uv$-dicut does not isolate $u$ nor $v$ (because $D$ is Eulerian by Lemma~\ref{lem:prop_k-extremal}).   
    So there is a unique vertex $u$ with $d_{max}(u) \geq k+1$. 
    
    
    Let $M \subseteq V(D)$ be a maximal set of vertices such that $u \in M$ and $D[M]$ is acyclic. Then every vertex of $V(D) \setminus M$ has outdegree and indegree at most $k-1$ in $V(D) \setminus M$, i.e. $\Delta_{max}(D[V(D) \setminus M]) \leq k - 1$. 
    As $\dic(D[M]) = 1$ and $\dic(D) = k+1$, we have that $\dic(D[V(D) \setminus M]) \geq k \geq \Delta_{max}(D[V(D) \setminus M]) +1$. 
    By Theorem~\ref{thm:dir_brooks} applied on $D[V(D) \setminus M]$, there exists $K \subseteq V(D) \setminus M$ such that $D[K] = \bid{K_{k}}$. As every vertex of $K$ has in- and outdegree exactly $k$ in $D$ and $k-1$ in $D[K]$, $\partial^+(K) = \partial^-(K)= k$. 
    Thus $(K, \overline K)$ is a minimum dicut and thus $V(D) \setminus K = \{u\}$ by hypothesis, a contradiction to the fact that $d_{max}(u) \geq k+1$. 
\end{proof}

When looking for a similar result for $3$-extremal digraphs, we rather use a method derived from the one in \cite{L66} (see also Section 5 of \cite{AA22} for its use in proving Brooks' theorem for digraphs).

\begin{lemma}\label{lem:all_cuts_isolate_vertices_three}
    If all minimum dicuts of a $3$-extremal digraph $D$ isolate a vertex, then $D = \bid W_{2\ell+1}$ for some $\ell \geq 1$ or $D$ is a directed Haj\'os join or a bidirected Haj\'os join of two digraphs.
\end{lemma}

\begin{proof}
    Let $D$ be a $3$-extremal digraph in which every minimum dicut isolates a vertex, and assume for contradiction that $D$ is not a symmetric odd wheel nor a directed Haj\'os join.
    Similarly to the proof of Lemma~\ref{lem:all_cuts_isolate_vertices}, we can prove that there is a unique vertex $u$ with $d^+(u) = d^-(u) \geq 4$ and for every $v \in V(D) \setminus u$, $d^+(v) =d^-(v) = 3$. 
    

    Let $P = (X,Y)$ be a partition of $V(D)$. We say that $P$ is  a \emph{special partition} if $u \in X$, $D[X]$ is acyclic and $X$ has maximum size among all sets $X'$ such that $u \in X'$ and $D[X']$ is acyclic. Note that if $P=(X,Y)$ is a special partition, then every vertex of $Y$ has at least one in-neighbour and one out-neighbour in $X$, and thus has in- and outdegree at most $2$ in $D[Y]$. 
    
    An \emph{obstruction of $P$} is a connected component of $D[Y]$ isomorphic to a symmetric odd cycle. Since obstructions are $2$-regular, a connected component of $D[Y]$ contains an obstruction if and only if it isomorphic to an obstruction. Note also that every special partition has at least one obstruction, for otherwise, by Theorem~\ref{thm:dir_brooks}, $\dic(D[Y]) \leq 2$ and thus $\dic(D) \leq 3$, a contradiction.
    
    $P$ is said to be a \emph{super-special partition} if it is special and it minimizes the number of obstructions among all special partitions. 
    
    We call the following operation \emph{switching $x$ and $y$}.
    
    \begin{claim}\label{clm:switch}
        Let $P = (X,Y)$ be a super-special partition. Let $y$ be a vertex in an obstruction of $P$, and $x \in X \setminus \{u\}$ be a neighbour of $y$. Then $P' = (X \cup \{y\} \setminus \{x\}, Y \cup \{x\} \setminus \{y\})$ is a super-special partition, and $x$ is in an obstruction of $P'$.
    \end{claim}
    
    \begin{proofclaim}
         Suppose without loss of generality that $xy \in A(D)$. Let $Z \subseteq Y$ be the vertex set of the obstruction containing $y$. As $d^-(y) = 3$ and $y$ has $2$ in-neighbours in $Z$, $y$ has no in-neighbour in $X \setminus \{x\}$. Thus $D[X \cup \{y\} \setminus \{x\}]$ is acyclic. As $x \neq u$,  $P'$ is special. 
         Since removing any vertex of a symmetric odd cycle yields a digraph that is not a symmetric odd cycle, $D[Z \setminus \{y\}]$ is not a symmetric odd cycle. Thus $x$ is in an obstruction of $P'$ and $P'$ is super-special. 
    \end{proofclaim}

    The switching operation is particularly useful thanks to the following claim:

    \begin{claim}\label{clm:two_neighbours_obstruction}
     Let $P = (X,Y)$ be a super-special partition, and $Z$ the vertex set of an obstruction of $P$. Vertices in $X \setminus u$ have at most one neighbour in $Z$. 
    \end{claim}

    \begin{proofclaim} 
        Let $Z=\{v_1, \dots, v_s\}$ and $v_i$ and $v_{i+1}$ are linked by a digon for $i=1, \dots, s$ (subscript are taken modulo $s$). Suppose for contradiction that there is
        $x \in X \setminus \{u\}$ such that $x$ is a neighbour of $v_i$ and $v_j$ for some $i \neq j$. 
        By claim~\ref{clm:switch}, we can switch $x$ and $v_i$ to obtain the super special partition $P' = (X \cup \{v_i\} \setminus \{x\} , Y \cup \{x\} \setminus \{v_i\} )$. Since $x$ is a neighbour of $v_j$, the obstruction of $P'$  containing $x$ is $D[Z \cup \{x\} \setminus \{a\}]$, i.e. $D[Z \cup \{x\} \setminus \{a\}]$ is a symmetric odd cycle. Hence, $x$ is linked by a digon to $v_{i-1}$ and $v_{i+1}$. Now, by switching $x$ and $v_{i+1}$ in $P$, we get that $x$ is also linked by a digon to $v_{i+2}$ and thus is linked by a digon to every vertex of $Z$. 
        This implies that $s=3$, and thus the dicut $(V(Z),V(D) \setminus V(Z))$ has size $3$, so it is a minimum dicut that does not isolate a vertex, a contradiction. 
    \end{proofclaim}

    Let $P_1=(X_1, Y_1)$ be a super-special partition of $D$ and let $Z_1$ an obstruction of $P_1$. If no vertex of $Z_1$ has a neighbour in $X_1 \setminus \{u\}$, then $D[Z \cup \{u\}]$ is a symmetric odd wheel and we are done. So there exist $x_1 \in X_1 \setminus \{u\}$ and $y_1 \in Z_1$ such that $x_1$ and $y_1$ are adjacent. Set $Q_1 = Z_1 \setminus \{y_1\}$. 
    
    By claim~\ref{clm:switch}, $P_2=(X_2, Y_2)$ with $X_2= X_1 \cup \{y_1\} \setminus \{x_1\}$ and $Y_2 = Y_1 \cup \{x_1\} \setminus \{y_1\}$ is a super-special partition and $x_1$ is in an obstruction $Z_2$ of $P_2$. Let $Q_2 = Z_2 \setminus \{x_1\}$, so $Q_2$ is a symmetric path and is a connected component of $D[Y_1]$. 
    Observe that no vertex of $Q_2$ is adjacent with $y_1$. 
    If $y_1$ is the only vertex in $X_2 \setminus \{u\}$ with a neighbour in $V(Z_2)$, then either $x_1$ and $y_1$ are linked by a digon and $D$ is a bidirected Haj\'os join (by Lemma~\ref{lem:HB_sufficient}, because deleting $u$ and $[x_1, y_1]$ separates $Z_2$ from the rest of the digraph), or $D$ is a directed Haj\'os join (by Lemma~\ref{lem:HJsufficient}, because deleting $u$ and the arc linking $x_1$ and $y_1$ separates $Z_2$ from the rest of the digraph). A contradiction in both cases. 
    Hence, there is $x_2 \in X_2 \setminus \{u, y_1\}$ such that $x_2$ has a neighbour $y_2 \in V(Z_2)$. 
    
    Let $P_3=(X_3, Y_3)$ where $X_3 = X_2 \cup \{y_2\} \setminus \{x_2\}$ and $Y_3 = Y_2 \cup \{x_2\} \setminus \{y_2\}$. By claim~\ref{clm:switch}, $P_3$ is a super-special partition and $x_2$ is in an obstruction $Z_3$ of $P_3$.  
    Note that, claim~\ref{clm:two_neighbours_obstruction},  $x_2$ has at most one neighbour in $Z_1$ and in $Z_2$, but it has two neighbours in $Z_3$, this implies that  $Z_3$ is disjoint from $Z_1$ and $Z_2$. As previously, if $y_2$ is the only vertex in $X_3 \setminus u$ with a neighbour in $Z_3$, then $D$ is a directed Haj\'os join or a bidirected Haj\'os join, a contradiction. 
    So there exists $x_3 \in X_3 \setminus \{u, y_1, y_2\}$ such that $x_3$ has a neighbour $y_3 \in V(Z_3)$. 

    This process can be continued and never stop, a contradiction.

\end{proof}

\begin{lemma}\label{lem:contraction_extremal}
    Let $k \geq 1$. Let $D$ be a $k$-extremal digraph and let $(A,\overline{A})$ be a minimum dicut of 
    $D$. Then $D/A$ or $D/\overline{A}$ is $k$-extremal.
\end{lemma}

\begin{proof}
    Set $H=D/A$ and let $a$ be the vertex into which $A$ is contracted in $H$.
    Since $D$ is strong, so is $H$.
    
    Let us first prove that $\lambda(H) \leq k$. 
     By Lemma~\ref{lem:prop_k-extremal}, $d^+(a), d^-(a) =  k$.  
    Let $u,v \in H$, and let us prove that $\lambda(u,v) \leq k$. Since $d^+(a), d^-(a) \leq  k$, the result holds if $a \in \{u,v\}$. 
    Let $(B, \overline{B})$ be a minimum $uv$-dicut in $D$. 

    $|\partial_D|$ is submodular, i.e. it satisfies that $\forall S, T \subseteq V(D), \partial_D(S) + \partial_D(T) \geq \partial_D(S \cup T) + \partial_D(S \cap T)$. 
    
    By Lemma~\ref{lem:prop_k-extremal}, the local arc-connectivity of any pair of vertices of $D$ is $k$, so, given $X \subset V(D)$  distinct from $\emptyset$ and $V(D)$, we have $\partial_D(X) \geq 2k$ and equality holds if and only if $|\partial_D^+(X)| = |\partial_D^+(\overline{X})| = k$. 
    
    By submodularity of $|\partial_D|$, $4k = |\partial_D(A)| + |\partial_D(B)| \geq |\partial_D(A \cap B)| + |\partial_D(A \cup B)|$. 
    Moreover, since $v \in \overline{B}$, $A \cap B \neq \emptyset$ and $A\cup B \neq V(D)$, and it is clear that $A \cap B \neq V(D)$ and $A \cup B \neq \emptyset$. 
    Thus $|\partial_D(A \cup B)| = 2k$, which implies that $|\partial_{H}^+(B \setminus A \cup \{a\})| \leq k$, i.e. $(B \setminus A \cup \{a\}, \overline B \setminus A)$ is a $uv$-dicut of $H$ of size at most $k$. 

    Let us now show that $H$ is biconnected. 
    For every $x \in \overline A$, $H \setminus x = (D \setminus x)/A$ is connected because $D \setminus x$ is connected. So it suffices to show that $H \setminus a$ is connected. 
    Let $X,Y$ be two connected components of $H \setminus a$. As $H$ is strong, there must be at least one arc from $a$ to $X$. But as $(A, \overline{A})$ is a minimum dicut of $D$, $a$ has outdegree $k$ in $H$, and thus there are at most $k-1$ arcs from $a$ to $Y$. Thus $(\overline{Y},Y)$ is a dicut of $D$ of size at most $k-1$, a contradiction.  So $H$ is biconnected. 
     
    As $D$ is Eulerian, $(\overline{A},A)$ is also a minimum dicut of $D$. 
    Thus $D/\overline{A}$ is both strong and biconnected, and $\lambda(D/\overline{A}) = k$.

    We now show that either $\dic(H) \geq  k + 1$ or $\dic(D/\overline{A}) \geq  k+1$.
    Let $\varphi_{A}$ be a $k$-dicolouring of $D[A]$ and $\varphi_{\overline{A}}$  a $k$-dicolouring of $D[\overline{A}]$.
    By Lemma~\ref{lem:mono_vs_rainbow}, we may assume without loss of generality that every vertex $N(\overline{A}) \cap A$ are coloured $1$. , and that every colour is used by vertices in $N(A) \cap \overline{A}$.


    Let us prove that  $\dic(H) \geq k+1$. 
    Suppose for contradiction that $H$ admits a proper $k$-dicolouring $\varphi_H$, chosen, up to permuting colours, so that $\varphi_H(a) = 1$.
    Consider $\varphi : V(D) \to [1,k]$ such that $\varphi(x) = \varphi_A(x)$ if $x \in A$ and $\varphi(x) = \varphi_{H}(x)$ if $x \in \overline{A}$. 
    Since $N(\overline{A}) \cap A$ are coloured $1$ with respect to $\varphi_A$, it is easy to see that $\varphi$ is a $k$-dicolouring of $D$, a contradiction.
\end{proof}


\section{Haj\'os joins - A First Decomposition Theorem}\label{sec:decthm}
Our main theorem presented in the introduction (Theorem~\ref{thm:main_HT}) is a {\em structure theorem} for the class of $k$-extremal digraphs, in the sense that it is an "if and only if". The goal of this section is to prove an intermediate result that is a {\em decomposition theorem} for this class (an "only if" theorem). It involves the notion of Haj\'os bijoin described just after the statement.

\begin{theorem}\label{thm:decHJHBJ}
Let $k\geq 3$. If $D$ is $k$-extremal, then:
    \begin{itemize}
        \item either $D=\bid{K}_{k+1}$
        \item or $D$ is a symmetric odd wheel (only in the case $k=3$),
        \item or $D$ is a directed Haj\'os join of two $k$-extremal digraphs,
        \item or $D$ is a Haj\'os bijoin of two $k$ extremal-digraphs.
    \end{itemize}
\end{theorem}

\begin{definition}[Haj\'os bijoin and degenerated Haj\'os bijoin]\label{def:bijoin}
    Let $D_1$ and $D_2$ be two digraphs. Let $ta_1, a_1w \in A(D_1)$ ($t=w$ is possible) and $t$ and $w$ are in the same connected component of $D_1 \setminus a_1$. Let $va_2,a_2u \in A(D_2)$ ($u=v$ is possible) and $u$ and $v$ are in the same connected component of $D_2 \setminus a_2$. 
    The \emph{Haj\'os bijoin} of $D_1$ and $D_2$ with respect to $\big((t,a_1,w), (v,a_2, u)\big)$ is the digraph $D$ obtained from the disjoint union of $D_1 - \{ta_1, a_1w\}$ and $D_2 - \{va_2, a_2u\}$ by identifying $a_1$ and $a_2$ into a new vertex $a$, and adding the arcs $tu$ and $vw$. \\
    If $t=w$ and $u \neq v$ (or $u=v$ and $t \neq w$), we say it is a \emph{degenerated Haj\'os bijoin}.\\
    If $t=w$ and $u=v$, then it is the bidirected Haj\'os join of $D_1$ and $D_2$ with respect to $([t,a_1], [u,a_2])$. See Figure~\ref{fig:bijoin}.
\end{definition}

    \begin{figure}[!hbtp]
    \begin{center}
        \begin{tikzpicture}[scale=1]
            \begin{scope}[xshift = -1cm]
            \begin{scope}
                \vertex (r) at (1.2,0) {};
                \vertex (u) at (0,1) {$t$};
                \vertex (d) at (0,-1) {$w$};
                \draw (0,0) circle (1.5);
            \end{scope}
        
            \begin{scope}[xshift=2.4cm,xscale=-1]
                \vertex (r2) at (1.2,0) {$a$};
                \vertex (u2) at (0,1) {$u$};
                \vertex (d2) at (0,-1) {$v$};
                \draw (0,0) circle (1.5);
            \end{scope}

            \node () at (1.2,-2) {$D$};
            \end{scope}
            
            \draw[->-, bend left=15] (u) to (u2);
            \draw[->-, bend left=15] (d2) to (d);

            \begin{scope}[xshift=-9cm]
                \vertex (r) at (1.2,0) {$a_1$};
                \vertex (u) at (0,1) {$t$};
                \vertex (d) at (0,-1) {$w$};
                \draw[->-, bend left=15] (u) to (r);
                \draw[->-, bend left=15] (r) to (d);
                \draw (0,0) circle (1.5);
                \node () at (0,-2) {$D_1$};
            \end{scope}
        
            \begin{scope}[xshift=-5.5cm,xscale=-1]
                \vertex (r2) at (1.2,0) {$a_2$};
                \vertex (u2) at (0,1) {$u$};
                \vertex (d2) at (0,-1) {$v$};
                \draw[->-, bend right=15] (d2) to (r2);
                \draw[->-, bend right=15] (r2) to (u2);
                \draw (0,0) circle (1.5);
                \node () at (0,-2) {$D_2$};
            \end{scope}
        
        \end{tikzpicture}
        \end{center}
           \caption{$D$ is the Haj\'os bijoin of $D_1$ and $D_2$ with respect to $(t,a_1, w), (u,a_2, v)$.}
           \label{fig:bijoin}
        \end{figure}
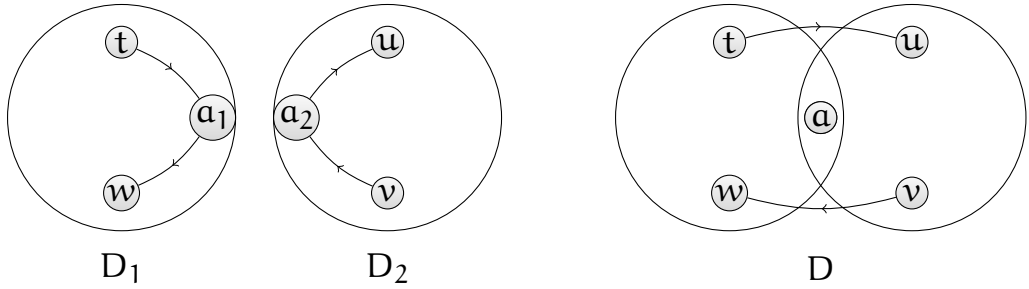

Note that if $t=w$ and $u=v$, then we get what we called earlier a bidirected Haj\'os join.

\subsection{Properties of Haj\'os join and bijoins}

For the definition of directed Haj\'os join, see Definition~\ref{def:directed_HJ}. We first prove an essential result about $k$-extremal digraphs and directed Haj\'os joins

\begin{lemma}\label{lem:extremal_directed_Haj}
    Let $k \geq 1$. 
    Let $D$ be the directed Haj\'os join of two digraphs $D_1$ and $D_2$. Then $D$ is $k$-extremal if and only if  both $D_1$ and $D_2$ are.
\end{lemma}

\begin{proof}
    Suppose $D$ is a directed Haj\'os join of two digraphs $D_1$ and $D_2$ with respect to $(uv_1, v_2w)$, i.e. there is $uv_1 \in A(D_1)$ and $v_2u \in A(D_2)$ such that $D$ is obtained from disjoint copies of $D_1-uv_1$ and $D_2-v_2w$ by identifying the vertices $v_1$ and $v_2$ to a new vertex $v$ and adding the edge $uw$. 

    \begin{claim}[Theorem 2 in~\cite{BBSS20}]\label{clm:kawa_DHJ}
        $D$ is $k+1$-dicritical if and only if both $D_1$ and $D_2$ are. 
    \end{claim}
    
    Let us first suppose that $D_1$ and $D_2$ are $k$-extremal.  
    By Claim~\ref{clm:kawa_DHJ}, $D$ is $k+1$-dicritical, so it is also biconnected and strong. Since the maximum local connectivity of a digraph equal the maximum maximum local connectivity of its blocks, we have that  $\lambda(D-\{uw\} + \{uv_1, v_2w\}) =\lambda(D_1) = \lambda(D_2) = k$, and by Lemma~\ref{lem:lambda_diminishing}, $\lambda(D) \leq \lambda(D-\{uw\} + \{uv_1, v_2w\} =  k$. Thus $k +1 = \dic(D) \leq \lambda(D) + 1 = k+1 $, so $\lambda(D) = k$ and $D$ is $k$-extremal. 
    

    Suppose now that $D$ is $k$-extremal. 
    By claim~\ref{clm:kawa_DHJ}, both $D_1$ and $D_2$ are $(k+1)$-dicritical and thus are also strong and biconnected. 
    Since $D$ is strong, it has a $wv$-dipath, and this dipath uses only arcs in the copy of $D_2 -wv_2$. Let $P$ be such a dipath to which we add the arc $uw$ at the beginning. 
    Then $\lambda(D + uv - A(P)) \leq \lambda(D) = k$, and since $D_1$ is a subdigraph of $D + uv - A(P)$, we have that $\lambda(D_1) \leq k$. Thus  $k +1 = \dic(D_1) \leq \lambda(D_1) + 1 = k+1 $, so $\lambda(D) = k$ and $D_1$ is $k$-extremal. Similarly, $D_2$ is $k$-extremal. 
\end{proof}

 If $D$ is a directed Haj\'os join of two digraphs, then there exists an arc ($uw$ in Definition~\ref{def:directed_HJ}), such that $D-uw$ has a cutvertex. The following lemma asserts that if $D$ is $k$-extremal, then the converse holds. This is sometimes useful to prove that $D$ admits a directed Haj\'os join. 
 \begin{lemma}\label{lem:HJsufficient}
     If $D$ is $k$-extremal and there exists an arc $a\in A(D)$, such that $D-a$ is not biconnected, then $D$ is a directed Haj\'os join.
 \end{lemma}
 \begin{proof}
     Denote $a=uw$ and $v$ be the cutvertex of 
     $D-a$. We need to prove that $uv\not\in A(D)$ and $vw\not\in A(D)$. We only prove it for $uv$, the argument is identical for $vw$ by directional duality. Assume by contradiction $uv\in A(D)$. Let $D_1$ be the digraph induced by the block of $D-a$ that contains $u$. Since $D$ is vertex critical by Lemma~\ref{lem:prop_k-extremal}, $D_1$ is $k$-dicolourable, but since $uv$ is an arc, it means the dicolouring contains no monochromatic $vu$-dipath. If we now consider $D_2$ to be the union of all the other blocks of $D-a$, then again it admits a $k$-dicolouring. Up to permuting the colours, we can assume the $2$ $k$-dicolourings agree on the colour of $v$, and thus get a $k$-dicolouring of $D$ that is proper: indeed any dicycle in $D$ is either contained in $D_1$ or $D_2$ or contains $a$ and goes through $v$, so must contain a $vu$-dipath. In all  cases it cannot be monchromatic. 
 \end{proof}

Here is the analogue of Lemma~\ref{lem:extremal_directed_Haj} for bijoins, note that we have only one direction here.

\begin{lemma}\label{lem:extremal_bijoin_Haj}
    Let $k \geq 3$. Let $D$ be a Haj\'os bijoin of two digraphs $D_1$ and $D_2$. If $D$ is $k$-extremal, then both $D_1$ and $D_2$ are $k$-extremal.
\end{lemma}

\begin{proof} 
    Let $D$ be a $k$-extremal digraph, and $D$ is a Haj\'os bijoin of two digraphs $D_1$ and $D_2$ with respect to $\big((t,a_1,w), (v,a_2, u)\big)$, i.e. there exists $tu,vw \in A(D)$ and $a \in V(D)$ such that $D\setminus \{a\} - \{tu,vw\}$ has two connected components with vertex sets $V'_1$ and $V'_2$ such that $D_1=D[V'_1 \cup a] + \{ta,aw\}$, and $D_2= D[V'_2 \cup a] + \{va,au\}$ and $t$ and $w$ are in the same connected component of $D_1\setminus a$ and  $u$ and $v$ are in the same connected component of $D_2\setminus a$. 
    \smallskip

    Let us first prove that $D_1$ is biconnected. Assume for contradiction that $D_1$ has a cutvertex $x$. Observe that $\{t,w,a_1\} \setminus x$ are in the same connected component of $D_1 \setminus x$. Indeed, if $x=a$ it is by hypothesis, and otherwise it is because $ta, aw \in A(D_1)$. Hence, $D \setminus x$ has a connected component disjoint from $\{t,a,w\}$, and thus $x$ is a cutvertex of $D$, a contradiction.  
    \smallskip 

    As $D$ is Eulerian, so is $D_1$ by construction. And since an Eulerian connected digraph is strong, $D_1$ is strong. 
    \smallskip 
    
    Let $x,y \in V(D_1)$ and let us prove that $\lambda_{D_1}(x,y) \leq k$. 
    Let $(X,\overline{X})$ be a minimum $xy$-dicut in $D$, i.e. $x \in X$ and $y \in \overline{X}$. Since $D$ is $k$-extremal, $(\overline{X}, X)$ is a minimum $yx$-dicut, and thus, up to permuting $y$ and $x$, we may assume without loss of generality that $a \in \overline{X}$. 
    Let $X_{D_1} =X \cap V(D_1)$ and  consider the $xa$-dicut of $D$ $(X_{D_1}, \overline{X}_{D_1})$.  If $t \in X$, then $\partial_{D_1}^+(X) = \partial_D^+(X) -tu + ta$, and otherwise $\partial_{D_1}^+(X) = \partial_D^+(X)$. Hence $|\partial_{D_1}^+(X)| = |\partial_D^+(X)| = k$, so $\lambda_{D_1}(u,v) \leq k$ and thus $\lambda(D_1) \leq k$. 
    \smallskip 
    
    Let us now prove that $\dic(D_1) \geq k + 1$. Suppose $D_1$ admits a $k$-dicolouring $\varphi_1$. Let $\varphi_2$ be a $k$-dicolouring of $D_2 - \{va,au\}$  
    such that $\varphi_1(a) = \varphi_2(a)$ and, if $\varphi_1(a) \neq \varphi_1(t)$, such that $\varphi_2(u) \neq \varphi_1(t)$ (this can always be done because $k \geq 3$). 
    Consider $\varphi : V(D) \to [1,k]$ such that $\varphi(x) = \varphi_1(x)$ if $x \in V(D_1) \cup a$ and $\varphi(x) = \varphi_2(x)$ if $x \in V(D_2)$. 

    Since $\dic(D) = k+1$, $D$ has a monochromatic directed cycle $C$ with respect to $\varphi$. 
    By construction of $\varphi$, $C$ goes through $tu$, or $vw$ or both. 
    If $C$ goes through   $tu$ but not  $vw$, then $C$ contains an $at$-dipath, which is not monochromatic since $ta \in A(D_1)$ and $\varphi_1$ is a $k$-dicolouring of $D_1$, a contradiction.
    Similarly, we get a contradiction if $C$ goes through   $vw$ but not  $tu$. 
    We may thus assume that $C$ uses both $tu$ and $vw$. In particular, $C$ contains an $uv$-dipath $P_{uv}$ included in $D_2$, and a $wt$-dipath $P_{wt}$ included in $D_1$. 
    If $\varphi_1(a) = \varphi_1(t)$, then $P_{wt}$  plus the arcs $ta$ and $aw$ form a monochromatic dicyle of $D_1$, a contradiction. 
    Thus $\varphi_1(a) \neq \varphi_1(t)$ and thus $\varphi(t) \neq \varphi(u)$ by construction of $\varphi$, so $C$ is not monochromatic, a contradiction. This finishes the proof that  $\dic(D_1) \geq k+1$. 
    \smallskip 
    
    Now, since $k+1 \leq \dic(D_1) \leq \lambda(D_1) + 1 \leq k+1$, we have $\dic(D_1) = \lambda(D_1) + 1 = k+1$. Altogether we get that $D_1$ is $k$-extremal. Similarly, $D_2$ is $k$-extremal. 
\end{proof}

Note that the reciprocal of lemma~\ref{lem:extremal_bijoin_Haj} does not hold, observe for example that the Haj\'os bijoin of two $\bid{K}_{k+1}$ is $k$-dicolourable, see Figure~\ref{fig:bijoin_not_reciprocal}.

    \begin{figure}[!hbtp]
    \begin{center}
        \begin{tikzpicture}[scale = 1.5]

            \begin{scope}
                \vertex[color=green] (r) at (0,0) {};
                \vertex[color=blue] (l) at (-2,0) {};
                \vertex[color=red] (u) at (-1,1) {};
                \vertex[color=green] (d) at (-1,-1) {};
                \draw[->-, bend right=15] (d) to (r);
                \draw[->-, bend right=15] (r) to (u);
                
                \draw[->-, bend left=15] (r) to (l);
                \draw[->-, bend left=15] (l) to (r);
                \draw[->-, bend left=15] (u) to (d);
                \draw[->-, bend left=15] (d) to (u);
                \draw[->-, bend left=15] (u) to (l);
                \draw[->-, bend left=15] (l) to (u);
                \draw[->-, bend left=15] (d) to (l);
                \draw[->-, bend left=15] (l) to (d);
            \end{scope}
        
            \begin{scope}[shift=(r),xscale=-1,yscale=-1]
                \vertex[color=blue] (l) at (-2,0) {};
                \vertex[color=red] (uu) at (-1,1) {};
                \vertex[color=green] (dd) at (-1,-1) {};
                \draw[->-, bend right=15] (dd) to (r);
                \draw[->-, bend right=15] (r) to (uu);
                \draw[->-, bend left=15] (r) to (l);
                \draw[->-, bend left=15] (l) to (r);
                \draw[->-, bend left=15] (uu) to (dd);
                \draw[->-, bend left=15] (dd) to (uu);
                \draw[->-, bend left=15] (uu) to (l);
                \draw[->-, bend left=15] (l) to (uu);
                \draw[->-, bend left=15] (dd) to (l);
                \draw[->-, bend left=15] (l) to (dd);
            \end{scope}

            \draw[->-, bend left=15] (u) to (dd);
            \draw[->-, bend left=15] (uu) to (d);

        \end{tikzpicture}
        \end{center}
           \caption{A $k$-dicolouring of a bijoin of two $\bid{K}_{4}$.}
           \label{fig:bijoin_not_reciprocal}
    \end{figure}

However, we can still prove the following holds:

\begin{lemma}\label{lem:extremal_bijoin_Haj_only_if}
    Let $k \geq 3$. Let $D$ be a Haj\'os bijoin of two digraphs $D_1$ and $D_2$. If, for $i=1,2$,  $D_i$ is biconnected, strong, Eulerian and $\lambda(D_i) \leq k$, then $D$ is biconnected, strong,  Eulerian and $\lambda(D) \leq k$. 
\end{lemma}

\begin{proof} 
    Let, for $i=1,2$,  $D_i$ be a biconnected, strong, Eulerian digraph with $\lambda(D_i) \leq k$.
    Let $D$ be a Haj\'os bijoin of $D_1$ and $D_2$ with respect to $\big((t,a_1,w), (v,a_2, u)\big)$, i.e. there exists $tu,vw \in A(D)$ and $a \in V(D)$ such that $D\setminus \{a\} - \{tu,vw\}$ has two connected components with vertex set $V'_1$ and $V'_2$ such that $D_1=D[V'_1 \cup a] + \{ta,aw\}$, and $D_2= D[V'_2 \cup a] + \{va,au\}$ and $t$ and $w$ are in the same connected component of $D_1\setminus a$ and  $u$ and $v$ are in the same connected component of $D_2\setminus a$. 
    Set $V_1=V'_1 \cup a$ and $V_2 = V'_2 \cup a$.  
    \smallskip 

    Let us first prove that $D$ is biconnected. 
    Assume for contradiction that $D$ has a cutvertex $x$. 
    Since for $i=1, 2$ $D_i$ is biconnected, $D[V'_i]=D_i\setminus a$ is connected, and since moreover there is an (actually two) arc between $D[V'_1]$ and $D[V'_2]$, $D \setminus a$ is connected. Thus $x \neq a$. 
    Assume without loss of generality that $x \in V'_1$ and let $C$ be the connected component of $D \setminus x$ containing $a$. Then $V_2$ is included in $C$, and since there is an arc between $u$ and $V_2$ and between $w$ and $V_2$, $u$ and $w$ are also in $C$. Thus $a$, $u$ and $w$ are in the same connected component of $D \setminus x$, which implies that $x$ is a cutvertex of $D_1$, a contradiction. 
    

    \smallskip

    As $D_1$ and $D_2$ are Eulerian,  so is $D$ by construction. And since an Eulerian biconnected digraph is strong, $D_1$ is strong. 

    \smallskip 
    
    Let $D' = D - \{tu,vw\} + \{ta,au,va,aw\}$. As the blocks of $D'$ are $D_1$ and $D_2$, $\lambda(D') \leq k$. Hence, by applying twice Lemma~\ref{lem:lambda_diminishing}, we get that $\lambda(D) \leq \lambda(D')  \leq k$.  
\end{proof}

The following lemma is an analogue of Lemma~\ref{lem:HJsufficient} in the case of Haj\'os bijoins.
\begin{lemma}\label{lem:HB_sufficient}
    Let $k \geq 3$. Let $D$ be a $k$-extremal digraph. Suppose there exists $t,u,w \in V(D)$ such that $D - \{tu,uw\}$ is not biconnected. Then $D$ is a directed Haj\'os join or a Haj\'os bijoin. 
\end{lemma} 

\begin{proof} 
    Let $a$ be a cut-vertex of $D - \{tu,uw\}$. If $t$ and $w$ are not in the same connected component of $D \setminus a - \{tu, uw\}$, then $D \setminus a - tu$ is not connected, and thus $D$ is a Haj\'os join. Hence, we can suppose that $t$ and $w$ are in the same connected component of $D \setminus a - \{tu,uw\}$.

    To prove that $D$ is a Haj\'os bijoin, it remains to prove that $ta,aw,ua,au \notin A(D)$. We will prove something a bit stronger.

    \begin{claim}
        Let $k \geq 3$. Let $D'$ be a $k$-extremal digraph. Suppose there exists $a, t,u,v,w \in V(D')$ such that $D' - \{tu,vw\}$ has two connected components $D_1, D_2$ with $t,w \in V(D_1)$ and $u,v \in V(D_2)$. Then $ta,au,va,aw \notin A(D')$. 
    \end{claim}

    \begin{proofclaim}
        Suppose that $ta \in A(D')$. Let us first prove that $D_1 + aw$ is biconnected. Assume for contradiction that $D_1 + aw$ has a cutvertex $x$. Observe that $\{t,w,a\} \setminus x$ are in the same connected component of $D_1 + aw \setminus x$. Indeed, if $x=a$ it is by hypothesis, and otherwise it is because $ta, aw \in A(D_1)$. Hence, $D' \setminus x$ has a connected component disjoint from $\{t,a,w\}$, and thus $x$ is a cutvertex of $D'$, a contradiction.
        \smallskip 
    
        Since for any two vertices $u,v$ of $D'$, $\lambda_{D'}(u,v) = k$, we have that $\lambda_{D' + aw - \{tu,vw\}}(u,v) \geq k - 2 \geq 1$, and thus $D' + aw - \{tu,vw\}$ is strong. As $D_1 + aw$ is a biconnected component of $D' + aw - \{tu,vw\}$, $D_1 + aw$ is strongly connected.
    
        Let $P_{av}$ be any $av$-dipath in $D_2$, which must exist for $\lambda_{D' - \{tu,vw\}}(a,v) \geq \lambda_{D'}(a,v) - 2 \geq k - 2 \geq 1$. Then, by Lemma~\ref{lem:lambda_diminishing}, $\lambda(D' + aw - A(P_{av}) - vw) \leq \lambda(D') = k$, and since $D_1 + aw$ is a subgraph of $D' + aw - A(P_{av}) - vw$, $\lambda(D_1 + aw) \leq k$.
    
        Yet $D_1 + aw$ is not Eulerian, for $D'$ is Eulerian, and the indegree of $t$ does not change in $D_1$ but its outdegree decreases by $1$. Thus, by Lemma~\ref{lem:prop_k-extremal}, $D_1 + aw$ is not $k$-extremal, hence $\dic(D_1 + aw) \leq k$. Let $\varphi_2$ be a $k$-dicolouring of $D_2$. Let $\varphi_1$ be a $k$-dicolouring of $D_1 + aw$ chosen so that $\varphi_1(a) = \varphi_2(a)$ and so that if $\varphi_1(a) \neq \varphi_1(t)$, then $\varphi_1(t) \neq \varphi_2(u)$ (which is always possible since $k \geq 3$).
    
        Consider $\varphi:V(D') \rightarrow [k]$ such that           
        $$
        \varphi(y) = \left\{
            \begin{array}{ll}
                \varphi_1(x_1) & \text{ if } y=x \\
                \varphi_1(y)   & \text{ if } y \in V(D_1)\\ 
                \varphi_2(y)   & \text{ if } y \in V(D_2)
            \end{array}
        \right.
        $$  
        Since $\dic(D') = k+1$, $\varphi$ contains a monochromatic dicycle $C$. Since $\varphi_1$ and $\varphi_2$ are $k$-dicolourings of respectively $D_1$ and $D_2$, $C$ intersects both $V(D_1) \setminus a$ and $V(D_2) \setminus a$. 
        If $C$ contains $tu$ but not $vw$, then $C$ goes through $a$, and thus there is a monochromatic $at$-dipath in $D_1$ with respect to $\varphi_1$, a contradiction to the choice of $\varphi_1$. If $C$ contains $vw$ but not $tu$, then $C$ goes through $a$, and thus there is a monochromatic $wa$-dipath in $D_1$ with respect to $\varphi_1$, a contradiction to the choice of $\varphi_1$. If $C$ contains both $tu$ and $vw$, then there is a monochromatic $wa$-dipath in $D_1$ with respect to $\varphi$. Since $ta, aw \in A(D_1 + aw)$, this implies that $\varphi_1(a) \neq \varphi(t)$. But then $\varphi(t) \neq \varphi(u)$, a contradiction.

        Thus we have proven that $ta \notin A(D')$. By symmetry, this implies that $ta,au,va,aw \notin A(D')$.
    \end{proofclaim}

    Applying this claim with $u = v$, this proves $ta, au, ua, aw \notin A(D')$ and thus that $D$ is a Haj\'os bijoin.

\end{proof}


\subsection{Proof of Theorem~\ref{thm:decHJHBJ}}
    
    Let $k \geq 3$. 
    We prove the theorem by induction on the number of vertices.  Let $D$ be $k$-extremal, and assume by contradiction that $D$ is neither a symmetric complete graph,  a symmetric odd wheel, a directed Haj\'os  join, nor a Haj\'os bijoin.

    Given a digraph  $D$, a \emph{flower} of $D$ is an induced subdigraph $F$ of $D$ isomorphic to a symmetric path $P$ with an even number of vertices plus a vertex $x$ linked to each vertex of $P$ by a digon, and such that no internal vertex of $P$ has a neighbour outside $F$, while other (that is $x$ and the two extremities of $P$) have exactly one in-neighbour and one out-neighbour outside $F$. The  vertex $x$ is called the \emph{center} of $F$.

    \begin{claim}
    $D$ has a minimum dicut $(X, \overline{X})$ such that either $k \geq 4$ and $D[\overline{X}] = \bid{K}_k$, or $k=3$ and $D[\overline{X}]$ is a flower of $D$.
    \end{claim}

    \begin{proofclaim}    
        By  Lemma~\ref{lem:all_cuts_isolate_vertices} and Lemma~\ref{lem:all_cuts_isolate_vertices_three}, there is a minimum dicut $(A,\overline{A})$  such that $|A|, |\overline{A}| > 1$. 
        By Lemma~\ref{lem:contraction_extremal}, up to permuting $A$ and $\overline{A}$, 
        we may assume that $D / A$ is $k$-extremal and thus by induction is either a symmetric complete graph on $k+1$ vertices, a symmetric complete wheel, a Haj\'os bijoin of two $k$-extremal digraphs or a directed Haj\'os  join of two $k$-extremal digraphs.   
        Let $a$ be the vertex into which $A$ is contracted in $D / A$.
        
        Suppose there exist two $k$-extremal digraphs $D_1$ and $D_2$ such that $D / A$ is either a Haj\'os bijoin of $D_1$ and $D_2$ with respect to $\big((t,b_1,w), (v,b_2, u)\big)$ or a directed Haj\'os  join of $D_1$ and $D_2$ with respect to $(ub_1, b_2w)$, and let $b$ the vertex into which $b_1$ and $b_2$ are identified in $D$. Then, as $b_1$ and $b_2$ both have outdegree at least $k$ in respectively $D_1$ and $D_2$, $b$ has outdegree at least $2k-2 > k$ in $k$, and thus $b \neq a$. When one un-contracts the vertex $a$ to get the original digraph $D$, it is clear that the structure of the directed Haj\'os join or Haj\'os bijoin is preserved, which yields a contradiction.
        
        Otherwise, $D/A$ is a symmetric complete graph or a symmetric odd wheel in which $a$ is a vertex of outdegree $k$. In both cases, taking $X = A$ yields the desired property. 
    \end{proofclaim}

        
    We distinguish between two cases depending on whether or not there is a vertex of $\overline{X}$ that has a distinct in- and out-neighbour in $X$.
    \smallskip 

    \noindent\textbf{Case 1}: There exists $v \in \overline{X}$ and $u,w \in X$ with $uv, vw \in A(D)$ and $u \neq w$
        
    Let $D' = D - \{uv,vw\} + \{uw\}$. Due to Lemma~\ref{lem:lambda_diminishing}, $\lambda(D') \leq \lambda(D) = k$. Let us now show that $\dic(D') \geq k+1$. Suppose for contradiction that  $D'$ admits a $k$-dicolouring $\varphi$. 
    As $\dic(D) = k+1$, $\varphi$ is not a $k$-dicolouring of $D= D' + \{uv, vw\} -\{uw\}$, and thus there is a monochromatic dicycle $C$ in $D' + \{uv, vw\} -\{uw\}$ containing $uv$ or $vw$ or both. 
    Since $v$ is linked via a digon to all its neighbours except for $u$ and $w$, $C$ contains both $uv$ and $vw$. By replacing the arcs $uv$ and $vw$ by $uw$ in $C$, we get a monochromatic dicycle in $D'$, a contradiction. 
    Thus $k+1 \leq \dic(D') \leq \lambda(D') + 1 \leq k+1$ and hence $\dic(D') = \lambda(D') + 1 = k+1$. 

    Since there are $k$ arc-disjoint dipaths between any pair of vertices in $D$, there are at least $k-2 \geq 1$ arc-disjoint dipaths between any pair of vertices in $D'$, thus $D'$ is strong.
    But $D'$ is not $k$-extremal, since $\partial_{D'}^{+}(X) = k-1$ contradicts Lemma~\ref{lem:prop_k-extremal}.  Thus $D'$ is not biconnected. Hence $D - \{uv,vw\}$ is not biconnected. By Lemma~\ref{lem:HB_sufficient}, this implies that $D$ is a Haj\'os bijoin or a Haj\'os join, a contradiction.


    \medskip 
    
    \noindent\textbf{Case 2:} There are only digons between $X$ and $\overline{X}$ 

     Assume $D[X]$ is not strong and let $C_t$ be a terminal component of $D[X]$. We have  $\partial^+(C_t) \subseteq \overline X$ and $\partial^+(C_t) \geq k$, so the digons linking $X$ and $\overline{X}$ are all incident with some vertex of $C_t$. Thus $D$ is not strong, a contradiction. Hence $D[X]$ is strong.

    \noindent\textbf{Case 2a:} $D[X]$ is not biconnected.
    
    Consider $B = (B_1, \dots, B_n)$ a longest path of blocks in $D[X]$.
    
    Let $V(B_1) \cap V(B_{2}) = \{a\}$. 
    There is a digon between $V(B_1) \setminus a$ and $\overline{X}$, for otherwise $a$ is a cutvertex of $D$. 
    Suppose there is only one digon $[b,y]$ between $V(B_1) \setminus a$ and $\overline{X}$, with $b \in V(B_1) \setminus a$ and $y \in \overline{X}$. 
    Then $D \setminus a - [b,y]$ is not connected, thus by Lemma~\ref{lem:HB_sufficient}, $D$ is a directed Haj\'os join or a Haj\'os bijoin, a contradiction.
    
        
    Thus we can assume that there are at least two digons between $\overline{X}$ and $V(B_1) \setminus V(B_2)$, say $[a_1,y_1]$ and $[b_1,y_1']$ with $a_1, b_1 \in V(B_1) \setminus V(B_2)$ and $y_1,  y_1' \in \overline{X}$. 
    Similarly, there are two digons between $\overline{X}$ and $V(B_n) \setminus V(B_{n-1})$, say $[a_n,y_n]$ and $[b_n,y_n']$ with $a_n, b_n \in V(B_n) \setminus V(B_{n-1})$. 
    Note that $a_1 = b_1$ and $a_n = b_n$ are possible. As we have found $4$ digons between $X$ and $\overline{X}$, this implies that $k \geq 4$.

    Set $H= B + [a_1, a_n]$ and let us prove that $\lambda(H) = k$. 
    
    By Lemma~\ref{lem:mono_vs_rainbow}, $a_1$ and $a_n$ receive the same colour in all $k$-dicolouring of $D[X]$, and since any $k$-dicolouring of $B$ can easily be extended to a $k$-dicolouring of $D[X]$, the same holds for any $k$-dicolouring of $B$, and thus  
    $\dic(H) \geq k+1$. Let $ P_{a_1a_n} = a_1 \rightarrow y_1 \rightarrow y_n \rightarrow a_n$ and  $P_{a_na_1} = a_n \rightarrow y_n \rightarrow y_1 \rightarrow a_1$. 
    By Lemma~\ref{lem:lambda_diminishing},  $\lambda(D+[a_1, a_n] - A(P_1) - A(P_2)) \leq \lambda(D) = k$, and thus $\lambda(H) \leq k$. 
    Thus $k+1 \leq \dic(H) \leq  \lambda(H) + 1 \leq k+1$, so $\lambda(H) = k+1$. 
        
    Hence there are $k$ arc-disjoint $b_1 b_n$-dipaths in $H$. But replacing any potential use of $a_1a_n$ by $P_{a_1a_n}$ and any potential use of $a_na_1$ by $P_{a_na_1}$, and considering the dipath $b_1 \rightarrow y'_1 \rightarrow y'_n \rightarrow b_n$, we get $k+1$ arc-disjoint  $b_1b_n$-dipaths in $D$, a contradiction.
    \medskip 
    
    \noindent\textbf{Case 2b:} $D[X]$ is biconnected.

    If a vertex $a \in X$ is adjacent to every vertex in $\overline{X}$, then $a$ is a cutvertex of $D$, a contradiction.

    

    Let $a, b \in X$ such that there exist $a',b' \in \overline{X}$ with $[a,a'], [b,b'] \subseteq A(D)$. If $k=3$, let them be chosen so that neither $a'$ nor $b'$ is the center of $D[X]$.
    Let $D' = D[X] + [a,b]$.
    Since $D[X]$ is strong and biconnected, so is $D'$.
    By Lemma~\ref{lem:mono_vs_rainbow}, $\dic(D') \geq k + 1$ for in every $k$-dicolouring $\varphi$ of $D[X]$, $\varphi(a) = \varphi(b)$.
    By Lemma~\ref{lem:lambda_diminishing}, $\lambda(D') \leq \lambda(D) = k$.
    Thus $D'$ is $k$-extremal.

    By induction, $D'$ is either a symmetric odd wheel, a symmetric complete graph, a directed Haj\'os join of two digraphs $D_1$ and $D_2$ or a Haj\'os bijoin of two digraphs $D_1$ and $D_2$.

    If $D'$ is a directed Haj\'os join of two digraphs $D_1$ and $D_2$, then either $a, b \in V(D_1)$ or $a, b \in V(D_2)$, and thus $D$ is a directed Haj\'os join, a contradiction.

    If $D'$ is a Haj\'os bijoin of two digraphs $D_1$ and $D_2$ with respect to $((t,x_1,w),(v,x_2,u))$, then if either $a, b \in V(D_1)$ or $a,b \in V(D_2)$, $D$ is itself a Haj\'os bijoin, a contradiction. Otherwise, this means that $t=w$, $v = u$ and $\{t,u\} = \{a,b\}$. But then, this contradicts that $D[X]$ is biconnected.

    Thus $D'$ is symmetric.

    Suppose there is no vertex in $X \setminus \{a,b\}$ with a neighbour in $\overline{X}$, then note that $a$ and $b$ must both have at least two neighbours in $\overline{X}$, for otherwise either $D \setminus [a,a']$ or $D \setminus [b,b']$ is not biconnected, and thus by Lemma~\ref{lem:HB_sufficient}, $D$ is a directed Haj\'os join or a directed bijoin, a contradiction. Note that this implies that $k \geq 4$. Let $a',a'' \in \overline{X} \cap N(a)$ and $b', b'' \in \overline{X} \cap N(b)$. Then, since $\lambda_{D'}(a,b) = k$, we have that $\lambda_{D[X]}(a,b) \geq k-1$, that is there exist $k-1$ arc-disjoint $ab$-dipaths in $D[X]$. But then, since $a \rightarrow a' \rightarrow b' \rightarrow b$ and $a \rightarrow a'' \rightarrow b'' \rightarrow b$ are $ab$-dipath that do not use any arc of $D[X]$, we have that $\lambda_{D}(a,b) \geq k + 1$, a contradiction.

    Thus, there exists a vertex $c \in X \setminus \{a,b\}$ with a neighbour $c' \in \overline{X}$. $c$ cannot be a neighbour of $a$ nor $b$, for either $[a,c] \ in A(D)$ or $[b,c] \in A(D)$ and in any $k$-dicolouring $\varphi$ of $D[X]$, $\varphi(a) = \varphi(b) = \varphi(c)$ by Lemma~\ref{lem:mono_vs_rainbow}, a contradiction. This implies that $D' \neq \bid K_{k+1}$, and thus $k=3$ and $D'$ is a symmetric odd wheel. Let $x$ be the universal vertex of $D'$. Since $c$ is neither a neighbour of $a$ nor of $b$, we have that $x \notin \{a,b,c\}$. Let $d,e$ be the two other neighbours of $c$ in $D$. Then, $x \rightarrow d \rightarrow c$, $x \rightarrow e \rightarrow c$, $x \rightarrow a \rightarrow a' \rightarrow c' \rightarrow c$ and $x \rightarrow c$ are $4 > k$ arc-disjoint $xc$-dipaths, a contradiction.

\section{Haj\'os trees - Structure Theorems}\label{sec:mainthm}
At the end of this section, we will prove the main result of this chapter: $k$-extremal digraphs are exactly the digraphs in $\hk$, which we recall is a class built from  $\bid K_{k+1}$ (for $k\geq 4$) or symmetric odd wheels (for $k = 3$) using directed Haj\'os  joins and Haj\'os tree joins. In order to simplify our arguments, we will prove in fact an equivalence with another class called $\htk$ that is based on a single operation called extended Haj\'os tree join (which generalizes Haj\'os trees) defined below.

An \emph{Euler tour} is a closed trail that traverses each arc exactly once. 
Let $T$ be a tree. An \emph{Eulerian list} $C$ of $T$  is a circular list of the vertices of $T$ encountered following an Eulerian tour of $\bid T$. A \emph{partial Eulerian list} $C'$ of $T$ is a circular sublist of an Eulerian list of $T$  with the following properties: 
 each leaf of $T$ is in $C'$, and each non-leaf vertex of $T$ appears at most once in $C'$.  

\begin{definition}[Extended Haj\'os tree join, Figure~\ref{fig:extended_hajos_tree_join}]\label{def:EHTJ}
Given 
\begin{itemize}
    \item a tree $T$ with edges $\{u_1v_1, \dots, u_nv_n\}$, $n \geq 2$, 
    \item a  partial Eulerian list $C=(x_1, \dots, x_{\ell})$ of $T$,  and
    \item for $i=1, \dots, n$, $D_i$ is a digraph such that $V(D_i) \cap V(T) = \{u_i,v_i\}$, and $[u_i,v_i] \subseteq A(D_i)$
    \item For $1 \leq i \neq j \leq n$, $V(D_i) \setminus \{u_i, v_i\} \cap V(D_j) \setminus \{u_j, v_j\} = \emptyset$ 
\end{itemize}
We define the \emph{extended Haj\'os tree join} $T(D_1, \dots, D_n;C)$ to be the digraph $D$ obtained from the $D_i$ by adding the dicycle $C = x_1 \ra x_2 \ra \dots \ra x_{\ell} \ra x_1$. 

We say that $D$ is the \emph{extended Haj\'os tree join} of $(T, D_1, \dots, D_n)$ with respect to $C$. 

$C$ is called the  \emph{peripheral cycle} of $D$  and vertices $u_1, v_1, \dots, u_n,v_n$ are the \emph{junction vertices} of $D$ (note that there are $n-1$ of them). \\

\end{definition}

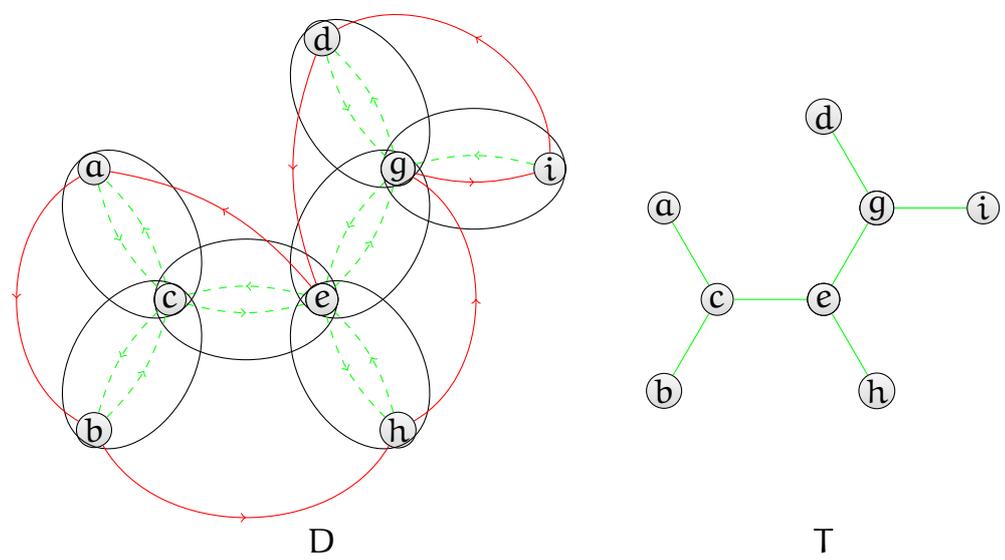
\begin{figure}[!hbtp]
    \begin{center}
        \begin{tikzpicture}[scale=0.4]

            \begin{scope}[xshift=18cm,scale = 0.7]              
                \begin{scope}
                    \vertex (c) at (0,0) {$c$};
                    \vertex (e) at (5,0) {$e$};
                    \draw[green] (c) -- (e);
                \end{scope}
    
                \begin{scope}[shift=(c), rotate = 120]
                    \vertex (a) at (5,0) {$a$};
                    \draw[green] (a) -- (c);
                \end{scope}
    
                \begin{scope}[shift=(c), rotate = -120]
                    \vertex (b) at (5,0) {$b$};
                    \draw[green] (b) -- (c);
                \end{scope}
    
                \begin{scope}[shift=(e), rotate = 60]
                    \vertex (e) at (0,0) {$e$};
                    \vertex (g) at (5,0) {$g$};
                    \draw[green] (e) -- (g);
                \end{scope}
    
                \begin{scope}[shift=(e), rotate = -60]
                    \vertex (e) at (0,0) {$e$};
                    \vertex (h) at (5,0) {$h$};
                    \draw[green] (e) -- (h);
                \end{scope}
    
                \begin{scope}[shift = (g)]
                    \vertex (g) at (0,0) {$g$};
                    \vertex (i) at (5,0) {$i$};
                    \draw[green] (i) -- (g);
                \end{scope}
                
                \begin{scope}[shift = (g), rotate=120]
                    \vertex (g) at (0,0) {$g$};
                    \vertex (d) at (5,0) {$d$};
                    \draw[green] (d) -- (g);
                \end{scope}

            \end{scope}

            \node () at (e |- 0, -8) {$T$};

            \begin{scope}
                \vertex (c) at (0,0) {$c$};
                \vertex (e) at (5,0) {$e$};
                \draw[->-, bend right = 15, green, dashed] (c) to (e);
                \draw[->-, bend right = 15, green, dashed] (e) to (c);
                \draw (2.5,0) ellipse (3cm and 2cm) {};
            \end{scope}

            \begin{scope}[shift=(c), rotate = 120]
                \vertex (a) at (5,0) {$a$};
                \draw[->-, bend right = 15, green, dashed] (c) to (a);
                \draw[->-, bend right = 15, green, dashed] (a) to (c);
                \draw (2.5,0) ellipse (3cm and 2cm) {};
            \end{scope}

            \begin{scope}[shift=(c), rotate = -120]
                \vertex (b) at (5,0) {$b$};
                \draw[->-, bend right = 15, green, dashed] (c) to (b);
                \draw[->-, bend right = 15, green, dashed] (b) to (c);
                \draw (2.5,0) ellipse (3cm and 2cm) {};
            \end{scope}

            \begin{scope}[shift=(e), rotate = 60]
                \vertex (e) at (0,0) {$e$};
                \vertex (g) at (5,0) {$g$};
                \draw[->-, bend right = 15, green, dashed] (e) to (g);
                \draw[->-, bend right = 15, green, dashed] (g) to (e);
                \draw (2.5,0) ellipse (3cm and 2cm) {};
            \end{scope}

            \begin{scope}[shift=(e), rotate = -60]
                \vertex (e) at (0,0) {$e$};
                \vertex (h) at (5,0) {$h$};
                \draw[->-, bend right = 15, green, dashed] (e) to (h);
                \draw[->-, bend right = 15, green, dashed] (h) to (e);
                \draw (2.5,0) ellipse (3cm and 2cm) {};
            \end{scope}

            \begin{scope}[shift = (g)]
                \vertex (g) at (0,0) {$g$};
                \vertex (i) at (5,0) {$i$};
                \draw[->-, bend right = 15, green, dashed] (i) to (g);
                \draw (2.5,0) ellipse (3cm and 2cm) {};
            \end{scope}
            
            \begin{scope}[shift = (g), rotate=120]
                \vertex (g) at (0,0) {$g$};
                \vertex (d) at (5,0) {$d$};
                \draw[->-, bend right = 15, green, dashed] (g) to (d);
                \draw[->-, bend right = 15, green, dashed] (d) to (g);
                \draw (2.5,0) ellipse (3cm and 2cm) {};
            \end{scope}

            \node () at (e |- 0, -8) {$D$};

            \draw[->-, bend right = 60, red] (a) to (b);
            \draw[->-, bend right = 60, red] (b) to (h);
            \draw[->-, bend right = 60, red] (h) to (g);
            \draw[->-, bend right = 15, red] (g) to (i);
            \draw[->-, bend right = 60, red] (i) to (d);
            \draw[->-, bend right = 20, red] (d) to (e);
            \draw[->-, bend right = 20, red] (e) to (a);

        \end{tikzpicture}
        \end{center}
           \caption{A cartoonish drawing of an extended Haj\'os tree join $D$. Its peripheral cycle is in red. Removed digons are in dashed green. $T$ is the corresponding tree.}
           \label{fig:extended_hajos_tree_join}
        \end{figure}

Observe that an extended Haj\'os tree join in which the partial Eulerian list only uses leaves of $T$ is a Haj\'os tree join. Observe also that the peripheral may be a digon, in the case where $T$ is a path and the dicycle only goes through the extremities of $T$. 

\begin{definition}
For $k \geq 4$, let $\htk$ be the smallest class of digraphs that contains $\bid K_{k+1}$ and is closed under taking extended Haj\'os tree joins. Let $\htthree$ be the smallest class of digraphs that contains $\bid W_{2\ell+1}$, for every integer $\ell \geq 1$, and is closed under taking extended Haj\'os tree joins.
\end{definition}


The two following lemmata imply (by induction on the number of vertices) that $\htk\subset \hk$.

\begin{lemma}\label{lem:HTK=HJorHT}
Let $D\in \htk$ Then $D$ is either a directed Haj\'os join or a Haj\'os tree join of digraphs in $\htk$.
\end{lemma}
\begin{proof}
Denote $D=T(D_1, \dots, D_n;C)$ as in Definition~\ref{def:EHTJ}. If $C$ does not use any internal vertex of $T$, then $D$ is a Haj\'os tree join, so we can assume $C$ uses an internal vertex $v$ of $T$. Let $X$ be the connected component of $T \setminus v$ that contains the out-neighbour of $v$ in $C$. 
    Let $\mathcal L_1$ be the list of digraphs corresponding to the edges of $T[X \cup v]$, and $\mathcal L_2 = \{D_1, \dots, D_n\}\setminus \mathcal L_1$. As $C$ is a partial Eulerian list, which is thus obtained from an Eulerian tour, it is of the form $C = vP_{X}P_{\overline{X}}$ 
    where $P_X$ is the portion of $C$ contained in $X$, and $P_{\overline X}$ is the portion of $C$ contained in $V(T)-(X\cup v)$. 
    Note that, since $v$ is a cutvertex of $T$, $P_{X}$ and $P_{\overline{X}}$ are non-empty

    Let $x$ be the last element of $P_{X}$ and $y$ be the first element of $P_{\overline{X}}$. Then $D \setminus v - xy$ is disconnected, and thus $D$ is a directed Haj\'os join of 
    $D'_1 = D[\bigcup_{G \in \mathcal L_1} V(G)] + xv$, 
    and $D'_2 = D[\bigcup_{G \in \mathcal L_2} V(G)] + vy$. 


\end{proof}

\begin{lemma}\label{lem:HJ_preserve_htk}
Any directed Haj\'os join of digraphs in $\htk$ is in $\htk$.
\end{lemma}

\begin{proof}
    Let $D$, $D'$ be two digraphs, $uv_1 \in A(D)$, $v_2w \in A(D')$ and let $H$ be the directed Haj\'os join of $D$ and $D'$ with respect to $(uv_1, v_2w)$. We call $v$ the vertex obtained after identifying $v_1$ and $v_2$. 
     We want to prove that $H \in \htk$. We distinguish four cases.

    \noindent\textbf{Case 1:} $D=D'=\bid K_{k+1}$ or $D = D' = \bid W_{2\ell + 1}$ for some $\ell \geq 1$\\ 
    Let $T = (\{u,v,w\}, \{uv,vw\})$ be the path of length $2$ and let $L=(u,v,w,u)$ be a partial Eulerian list of $T$. Then it is easy to check that $T(D, D'; L)$ is the directed Haj\'os join of $D$ and $D'$, see Figure~\ref{fig:Hajos_K_4}. This proves Case 1. \smallskip 

    \begin{figure}[!hbtp]
    \begin{center}
        \begin{tikzpicture}[scale=1.3]
            \vertex (l1) at (-1,0) {u};
            \vertex (l2) at (-1.2,1.2) {};
            \vertex (l3) at (-2,1.5) {};
            \vertex (l4) at (0,2) {v};
        
            \vertex (r1) at (1,0) {w};
            \vertex (r2) at (1.2,1.2) {};
            \vertex (r3) at (2,1.5) {};
            
            \draw[arc, bend right=15] (l1) to (l2);
            \draw[arc, bend right=15] (l2) to (l1);
            \draw[arc, bend right=15] (l3) to (l2);
            \draw[arc, bend right=15] (l2) to (l3);
            \draw[arc, bend left = 13] (l1) to (l3);
            \draw[arc, bend right] (l3) to (l1);

            \draw[arc, bend right=15] (l2) to (l4);
            \draw[arc, bend right=15] (l4) to (l2);
            \draw[arc, bend left = 13] (l3) to (l4);
            \draw[arc, bend right] (l4) to (l3);
            
            \draw[arc, bend right=15] (r2) to (l4);
            \draw[arc, bend right=15] (l4) to (r2);
            \draw[arc, bend right = 13] (r3) to (l4);
            \draw[arc, bend left] (l4) to (r3);

            \draw[arc] (l4) to (l1);
            \draw[arc] (r1) to (l4);
            \draw[arc] (l1) to (r1);
            
            \draw[arc, bend right=15] (r1) to (r2);
            \draw[arc, bend right=15] (r2) to (r1);
            \draw[arc, bend right=15] (r3) to (r2);
            \draw[arc, bend right=15] (r2) to (r3);
            \draw[arc, bend right = 13] (r1) to (r3);
            \draw[arc, bend left] (r3) to (r1);
        \end{tikzpicture}
    \end{center}
    \caption{\label{fig:Hajos_K_4} The directed Haj\'os join of two $\bid K_4$.}
    \end{figure}

    From now on, we may assume that there exist a tree $T$, a partial Eulerian list $L=(x_1, \dots, x_{\ell})$ of $T$ and $n$ digraphs $D_1, \dots, D_n$ such that $D = T(D_1, \dots, D_n; C)$ where $C=x_1 \rightarrow \dots \rightarrow x_{\ell} \rightarrow x_1$.

    \noindent\textbf{Case 2:} $uv_1 \notin A(C)$\\     
    There is $i \in [n]$ such that $uv_1 \in A(D_i)$. 
    By induction, there exists $H' \in \htk$ such that $H'$ is the directed Haj\'os join of $D_i$ and $D'$ with respect to $(uv_1, v_2w)$. Then $H=T(D_1, \dots, D_{i-1}, H', D_{i+1}, \dots, D_{n}; C) \in \htk$. \smallskip
    
    \noindent\textbf{Case 3:} $uv_1 \in A(C)$ and $D'$ is a symmetric complete graph or a symmetric odd wheel\\
    In particular, $u$ and $v_1$ are vertices of $T$. 
    Then  $H=T'(D_1, \dots, D_n, D'; C') \in \htk$, where $T'$ is obtained from $T$ by adding the vertex $w$ and the edge $v_1w$, and $C'$ is obtained from the partial Eulerian list $L'$ obtained from $L$ by adding $w$ between $u$ and $v_1$ (in other words the peripheral cycle $C'$ is obtained from $C$ by deleting $uv_1$ and adding $uw$ and $wv_1$). 
    \smallskip 

    \noindent\textbf{Case 4:} $uv_1 \in A(C)$ and $D'$ is neither a symmetric complete graph nor a symmetric odd wheel\\
    Then $D' = T'(D'_1, \dots, D'_{n'};C')$ for some tree $T'$, some digraphs $D'_1, \dots, D'_{n'}$ and a peripheral cycle $C'$ built from a partial Eulerian list $L'$ of $T$.  

    If $v_2w \notin A(C')$, then the result follows from Case 2. So we may assume that $v_2w \in A(C')$. 

    Since $uv_1 \in A(C)$ and $v_2w \in A(C')$, we have $L=(u,v_1, L_1, u)$ and $L'=(v_2,w,L'_1, v_2)$ for some lists $L_1$ and $L'_1$. 
    
    Let $T_H$ be the tree obtained from $T$ and $T'$ by identifying $v_1$ and $v_2$ to a new vertex $v$. 
    Then $H=T_H(D_1, \dots, D_n,D'_1, \dots, D'_{n'}: C_H)$, where $C_H$ is obtained from the partial Eulerian list $L_H=(v, L_1, u, w, L_2, v)$. 
\end{proof}


The following result is crucial, as it will allow us to use Haj\'os bijoins given by Theorem~\ref{thm:decHJHBJ} and preserve the fact of being in $\htk$ (this is the main reason why extended Haj\'os tree joins are more convenient to use that Haj\'os tree join combined with directed Haj\'os join).

\begin{lemma}\label{lem:P3_in_peripheral}
    Let $k \geq 3$ and let $D \in \htk$ and $uv, vw \in A(D)$ with $u \neq w$ such that in all $k$-dicolourings of $D \setminus \{uv, vw\}$, there is a monochromatic $wu$-dipath. 
    Then $D=T(D_1, \dots, D_n;C)$ for some tree $T$ and digraphs $D_1, \dots, D_n$ and peripheral cycle $C$ such that $uv, vw \in E(C)$.  
\end{lemma}

\begin{proof}
    Let $D \in \htk$ and $uv, vw \in A(D)$ with $u \neq w$ such that in all $k$-dicolourings of $D \setminus \{uv, vw\}$, there is a monochromatic $wu$-dipath. Assume that the result holds for every digraph in $\htk$ with strictly less vertices than $D$.

    If $D$ is symmetric, then let $\varphi$ be a $k$-dicolouring of $D - uv$ (which must exist for $D$ is dicritical): $w$ has a colour distinct from the colours of all its neighbours with respect to $\varphi$, and thus $\varphi$ is a $k$-dicolouring of $D - \{uv,vw\}$ in which there is no $wu$-dipath. So we may assume that $D$ is not a symmetric complete graph nor a symmetric odd wheel.
    

    Thus, $D=T\big((D_1, [u_1,v_1]), \dots, (D_n, [u_b,v_b]);C\big)$. 
    Assume for contradiction that for any choice of $D_1, u_1, v_1, \dots D_n, u_n, v_n, C$, $uv \notin E(C)$ or $vw \notin E(C)$. 

     \noindent\textbf{Case 1:} $u \in V(D_i) \setminus \{u_i,v_i\}$ for some $i \in[n]$, $v=v_i$, and $w \notin V(D_i) \setminus \{u_i\}$. \\ 
     In this case, we will find a $k$-dicolouring of $D\setminus \{uv,vw\}$ with no monochromatic $wu$-dipath, thus obtaining a contradiction. 
    Since $D$ is $(k+1)$-dicritical, there is a $k$-dicolouring $\varphi$  of $D \setminus uv_i$. Observe first that $\varphi(u) = \varphi(v_i)$, for otherwise we would get a $k$-dicolouring of $D$, a contradiction. 
    Moreover, $\varphi(u_i) \neq \varphi(v_i)$, for otherwise all junction vertices receive the same colour and the peripheral cycle is monochromatic 

    Since $\varphi$ is, in particular, a $k$-dicolouring of $D\setminus \{uv_i,v_iw\}$, it contains a monochromatic $wu$-dipath $P$ by hypothesis. 
    So $\varphi(w)=\varphi(u) \neq \varphi(u_i)$ and thus $w \neq u_i$. 
    If $P$ goes through $v_i$, the arc $v_iw$ yields a monochromatic directed cycle, a contradiction. So $P$ does not contain $v_i$. But since $u$ and $w$ are in two distinct connected components of  $D\setminus \{u_i,v_i\}$, $P$ contains $u_i$, contradicting the fact that $\varphi(u_i) \neq \varphi(u)$.  
    \medskip 

    \noindent\textbf{Case 2:} $u,v,w \in V(D_i)$ for some $i$, and  $uv,vw \notin A(C)$. \\
    Note that, since $uv,vw \notin A(C)$, $\{uv,vw\} \cap \{u_iv_i, v_iu_i\}=\emptyset$.

    Suppose first that $D_i = T'\big(D'_1, \dots, D'_m;C'\big)$ with $uv, vw \in E(C')$. As $|V(C)| \geq 3$, $C$ is not a digon, and thus there exists $j$ such that $[u_i, v_i] \in D'_j$. But then:
    
    $$D = T'\big(D'_1, \dots, D_{j-1}, T\big(D_1, \dots, D_{i-1}, D'_{j}, D_{i+1}, \dots, D_n;C\big)  , D_{j+1}, \dots, D'_m;C'\big)$$

    contradicting that for any choice of $D_1, u_1, v_1, \dots D_n, u_n, v_n, C$, $uv \notin E(C)$ or $vw \notin E(C)$. 

    Thus by induction, as $D_i$ has strictly less vertices than $
    D$, there is a $k$-dicolouring $\varphi_i$ of $D_i\setminus \{uv,vw\}$ such that there is no monochromatic $wu$-dipath in $D_i$. Observe that $\varphi_i(u_i) \neq \varphi_i(v_i)$ since $u_i$ and $v_i$ are linked by a digon in  $D_i\setminus \{uv,vw\}$. 

    Now, let $\varphi$ be a $k$-dicolouring of $D \setminus V(D_i) \cup \{u_i,v_i\}$.
    If $\varphi(u_i) = \varphi(v_i)$, then all junction vertices receive this same colour, and $C$ is monochromatic, a contradiction. So $\varphi(u_i) \neq \varphi(v_i)$. Now, we may assume without loss of generality that $\varphi(u_i) = \varphi_i(u_i)$ and $\varphi(v_i) = \varphi_i(v_i)$, and obtain a $k$-dicolouring of $D \setminus \{uv,vw\}$ with no monochromatic $wu$-dipath. Indeed, a $wu$-dipath is either included in $D_i-[u_i,v_i]$, or contains both $u_i$ and $v_i$.  \smallskip

    Let us now explain why these two cases cover all possible cases. 
    Since $uv \notin E(C)$ or $vw \notin E(C)$, we may assume that at least one vertex of $\{u,v,w\}$ is not a junction vertex, for an arc linking two junction vertices is an arc of $C$. 

    \begin{itemize}
        \item If none of $\{u,v,w\}$ is a junction vertex, we are in case 2. 
        \item  If $v$ is a junction vertex and $u$ is not. Then $u \in V(D_i)\setminus \{u_i,v_i\}$ for some $i \in [n]$. Then either $w \notin V(D_i) \setminus \{u_i\}$, and we are in case 1, or $w \in V(D_i) \setminus u_i$, and we are in case 2. 
        \item By directional duality, the previous case is the same as the case where $v$ is a junction vertex and $w$ is not. 
        \item If $v$ is not a junction vertex, then $v \in V(D_i) \setminus \{u_i,v_i\}$ for some $i \in [n]$, and thus $u,w \in V(D_i)$, and we are in case 2.  
    \end{itemize}
    
      \end{proof}

From the previous lemma, we deduce an analogous of Lemma~\ref{lem:HJ_preserve_htk} for bijoins.

\begin{lemma}\label{lem:HBJ_preservehtk}
   Let $k\geq 3$. If $D$ is not $k$-dicolourable and $D$ is the bijoin of two digraphs $D_1\in \htk$ and $D_2\in \htk$, then $D\in \htk$. 
\end{lemma}

 \begin{proof}
        Let $ta_1, a_1w \in A(D_1)$, $t \neq w$,  and $t$ and $w$ are in the same connected component of $D_1 \setminus a_1$. Let $va_2, a_2u \in A(D_2)$, $u \neq v$,  and $u$ and $v$ are in the same connected component of $D_2 \setminus a_2$.
        Assume that $D$ is obtained from the disjoint union of $D_1 - \{ta_1, a_1w\}$ and $D_2 - \{va_2,a_2u\}$ by identifying $a_1$ and $a_2$ into a new vertex $a$, and adding the arcs $tu$ and $vw$, i.e. $D$ is the bijoin of $D_1$ and $D_2$ with respect to $((t,a_1, w),(u, a_2, v))$.


        
        Suppose first that $D_1 -\{ta_1, a_1w\}$ admits a $k$-dicolouring $\varphi_1$ with no monochromatic $wt$-dipath.
        Either there is no monochromatic $wa_1$-dipath, or no monochromatic $a_1t$-dipath. Without loss of generality, suppose there is no monochromatic $wa_1$-dipath.
        As $D_2$ is dicritical, $D_2 - va_2$ is $k$-dicolourable, and thus there is a $k$-dicolouring $\varphi_2$ of $D_2- \{va_2, a_2u\}$ with no monochromatic $ua_2$-dipath. Up to permuting colours, we may assume that $\varphi_1(a_1) = \varphi_2(a_2)$. 
        Consider $\varphi:V(D) \rightarrow [k]$ such that           
        $$
        \varphi(x) = \left\{
            \begin{array}{ll}
                \varphi_1(a_1) & \text{ if } x=a \\
                \varphi_1(x)   & \text{ if } x \in V(D_1)\\ 
                \varphi_2(x)   & \text{ if } x \in V(D_2)
            \end{array}
        \right.
        $$  
        Since $\dic(D) \geq k+1$, $\varphi$ contains a monochromatic dicycle $C$. Since $\varphi_1$ and $\varphi_2$ are $k$-dicolourings of respectively $D_1$ and $D_2$, $C$ intersects both $V(D_1) \setminus a_1$ and $V(D_2) \setminus a_2$. Thus it contains $tu$, or $vw$ or both. 
        If $C$ contains $tu$ but not $vw$, then $C$ goes through $a$, and thus there is a monochromatic $ua_2$-dipath in $D_2$ with respect to $\varphi_2$, a contradiction to the choice of $\varphi_2$. 
        Similarly, we get a contradiction if $C$ contains $vw$ but not $tu$. 
        Hence $C$ contains both $tu$ and $vw$ and thus there is a monochromatic $wt$-dipath in $D_1$, a contradiction to the choice of $\varphi_1$.

        Hence, all $k$-dicolourings of $D_1-\{ta_1,a_1w\}$ admit a monochromatic $wt$-dipath. Similarly, all $k$-dicolourings of $D_2-\{va_2,a_2u\}$ admit a monochromatic $uv$-dipath.
        
        By Lemma~\ref{lem:P3_in_peripheral}, $D_1=T^1(D^1_1, \dots, D^1_n;C^1)$ for a tree $T^1$,   digraphs $D^1_1, \dots, D^1_n$ and peripheral cycle $C^1$ such that $ta_1, a_1w \in A(C^1)$. Similarly, $D_2=T^2(D^2_1, \dots, D^2_m;C^2)$ for a tree $T^2$,   digraphs $D^2_1, \dots, D^2_m$ and peripheral cycle $C^2$ such that $va_2, a_2u \in A(C^2)$. 
        Set $T$ to be the tree obtained from $T_1$ and $T_2$ by identifying $a_1$ and $a_2$ to a vertex $a$. Now, $D = T(D^1_1, \dots, D^1_n, D^2_1, \dots, D^2_m;C)$ where $C$ is obtained from $C_1$ and $C_2$ after deleting arcs $ta_1$, $a_1w$, $va_2$ and $a_2w$, and adding $tu$ and $vw$. 

    \end{proof}
    
For the proof of our main theorem, we need to prove that digraphs in $\hk$ are indeed $k$-extremal. We know directed Haj\'os joins preserve $k$-extremality (Lemma~\ref{lem:extremal_directed_Haj}), we do it now for Haj\'os tree joins (we prove an if and only if for the purpose of the recognition algorithm of the next section)


\begin{lemma}\label{lem:HT_iff_ext}
 Let $k \geq 2$. Let $D, D_1, \dots, D_n$ be digraphs such that $D$ is a Haj\'os tree join of the $D_i$. Then $D$ is $k$-extremal if and only if all digraphs $D_1, \dots, D_n$ are $k$-extremal,
\end{lemma}
 \begin{proof}
Let $D=T(D_1, \dots, D_n;C)$ where $T$, $C$, $D_1, \dots, D_n$ are as in Definition~\ref{def:HTJ}. For each $D_i$, we also denote by $u_iv_i$ the pair of vertices in $V(T)\cap V(D_i)$ such that the digon $[u_i,v_i]$ is in $A(D_i)$ but was removed in the construction of $D$. 

Let  $D'$ the digraph obtained from $D$ by putting back all digons $[u_i,v_i]$ between vertices of $T$, and by removing the arcs in the peripheral cycle $C$. $D'$ is a digraph whose $2$-blocks are exactly the $D_i$. One can easily observe that $\lambda(D')=\max_i \lambda(Di)$ and $\dic(D')=\max_i \dic(Di)$ . For every arc $uv\in A(C)$, let $P_{uv}$ be the unique $uv$-dipath that uses only arcs between vertices of $T$ (arcs from the digons that were removed to construct $D$). It is easy to notice that all $P_{uv}$ are pairwise arc-disjoint (each cycle $uv+P_{uv}$ corresponds to one face of the planar graph $T+C$). Therefore one can go from $D'$ to $D$ by applying successive operations where one replaces the $uv$-dipath $P_{uv}$ by the arc $uv$ for each arc $uv\in A(C)$. By Lemma~\ref{lem:lambda_diminishing}, we thus obtain 
$\lambda(D) \leq \lambda(D')=\max_i \lambda(Di)$.

Assume $D$ is $k$-dicolourable. Then in any $k$-dicolouring the vertices of $T$ cannot all get the same colour (otherwise $C$ would me monochromatic). So there must be a digraph $D_i$ such that the vertices $u_i$ and $v_i$ get distinct colours. But this provides a proper $k$-dicolouring of the corresponding $D_i$. Hence  
$\min_i \dic(Di) \leq k$.

With the two previous paragraphs, we can already prove that if each $D_i$ is $k$-extremal then $D$ is $k$-extremal. Indeed we have $\dic(D_i)=k+1=\lambda(D_i)+1$ for every $i$ and since every digraph satisfies $\dic(D)\leq \lambda(D)+1$, we have $\dic(D)=k+1=\lambda(D)+1$.  If $D$ admits a cutvertex, then by construction of the Haj\'os tree, this cutvertex must be a cutvertex in some $D_i$, which contradicts the fact that $D_i$ is biconnected. Now observe that since every $D_i$ is Eulerian (for they are $k$-extremal), $D$ is also eulerian and is thus strong.\\

Now assume $D$ is $k$-extremal. First observe that every $D_i$ is connected: if not, then $u_i$ and $v_i$ must be in the same component, but then any vertex not in this component would still be disconnected from $u_i$ in $D$. Similarly, we get that every $D_i$ is biconnected. Now because $D$ is Eulerian, every $D_i$ is also Eulerian and every $D_i$ is strong.

Let us prove that $\lambda(D_i)=k$ for every $i$. Assume that there exist in $D_i$ two vertices $u$ and $v$ with $p$ arc-disjoint $uv$-dipaths, then either these dipaths do not use any arc in the digon $[u_i,v_i]$, in which case these dipaths are still present in $D$. Or they do use one of the arcs, but we can assume they don't use both, for otherwise we could reroute the dipaths to obtain a collection of $uv$-dipaths that do not use any arc in the digon $[u_i,v_i]$. So assume without loss of generality that these paths use $u_iv_i$. But then we can replace this arc with a $u_iv_i$-dipath using only arcs in some other $D_j$ for $j\neq i$ and some peripheral arcs of $C$. Hence we still get $p$ pairwise arc disjoint $uv$-dipaths. Therefore $\lambda(D) \geq\max_i \lambda(Di)$, so $\lambda(D_i)\leq k$ for every $i$.

Now let us finish by proving that no $D_i$ is $k$-dicolourable. Observe that since $D$ is $k$-extremal it is vertex critical (by Lemma~\ref{lem:prop_k-extremal}), so any digraph $D_i-[u_i,v_i]$ is $k$-dicolourable, and thus every $D_i$ is $k+1$-dicolourable. We first prove the following claim. 
\begin{claim}
    If $T=(V(T),A(T))$ is a tree on at least three vertices given with a strict subset $A'\subsetneq  A(T)$, there exists $\varphi : V(T)\to \{1,2,3\}$ such that the endpoints of edges in $A'$ receive the same colour, the endpoints of edges not in $A'$ receive distinct colours, and the leaves of $T$ do not all receive the same colour.
\end{claim}
\begin{proofclaim}
Let $uv$ be an edge not in $A'$ and consider the two connected components $T_u$ and $T_v$ of $T-uv$. If some connected component only contains edges in $A'$ we colour it with one single colour. If not we apply induction. Up to permuting the colours we can do so that $u$ and $v$ receive distinct colours. If we applied induction to either $T_u$ or $T_v$, then all leaves do not get the same colour, and if not it means all edges of $T$ except $uv$ were in $A'$, but in that case since the colour of $u$ is distinct from the colour of $v$, the leaves in $T_u$ and $T_v$ must get distinct colours.
\end{proofclaim}

Let now $A'$ be the set of edges $u_iv_i$ of $T$ such that $\dic(D_i)=k+1$. If all edges are in $A'$, then we have our result, we assume by contradiction this is not the case and apply the claim above. Observe now that if $u_iv_i\in A'$, then the digraph $D_i$ is $k$-extremal so we can apply Lemma~\ref{lem:extremdigon} to get a $k$-diclouring of $D_i-[u_i,v_i]$ in which $u_i$ and $v_i$ receive the same colour (that we choose to be the one given by the claim) but such that there is no monochromatic dipath between $u_i$ and $v_i$. And if $u_iv_i\not\in A'$, then  $\dic(D_i)=k$ and we just take a valid $k$-dicolouring of $D_i$ (it must give distinct colours to $u_i$ and $v_i$, which we can again choose to be the one given by the claim). Note that in all cases, we obtain a dicolouring of the vertices of $D$ that is proper on each $D_i-[u_i,v_i]$, such that there is no monochromatic dipath between any pair of vertices of $T$, and such that $C$ is not monochromatic. This is a valid $k$-dicolouring of $D$, our final contradiction.


 
 

\end{proof}

We are now ready to prove our main theorem.
\begin{theorem} Let $D$ be a digraph. The three following statements are equivalent
\begin{description}
\item[i)] $D$ is $k$-extremal
\item[ii)] $D\in \htk$
\item[iii)] $D \in \hk$
\end{description}
\end{theorem}

\begin{proof}
We prove the statement by induction on the number of vertices of $D$.

\begin{description}
\item[i$\Ra$ ii]

Let $D$ be a $k$-extremal digraph. By Theorem~\ref{thm:decHJHBJ}, $D$ is either a directed Haj\'os join or a Haj\'os bijoin of two $k$-extremal digraphs.

Assume first that $D$ is the directed Haj\'os join of two digraphs $D_1$ and $D_2$. By Lemma~\ref{lem:extremal_directed_Haj}, both $D_1$ and $D_2$ are $k$-extremal. Thus by induction $D_1$ and $D_2$ belong to $\htk$ and since directed Haj\'os joins preserve the fact of being in this class (Lemma~\ref{lem:HJ_preserve_htk}), we have that $D$ is in $\htk$.

So we can assume that $D$ is not a directed Haj\'os join but is a Haj\'os bijoin of two $k$-extremal digraphs $D_1$ and $D_2$. By Lemma~\ref{lem:extremal_bijoin_Haj}, both $D_1$ and $D_2$ are extremal. So by the induction hypothesis, they both belong to $\htk$. By Lemma~\ref{lem:HBJ_preservehtk}, $D$ is in $\htk$. 


\item[ii$\Ra$ iii] By Lemma~\ref{lem:HTK=HJorHT} a digraph $D$ in $\htk$ is either a directed Haj\'os join or a Haj\'os tree. In both cases the digraph $D_i$ used for the join are in $\htk$ and thus in $\hk$ by induction. Hence $D$ is in $\hk$.

\item[iii$\Ra$ i] This is guaranteed by the fact that both directed Haj\'os joins and Haj\'os tree joins preserve the fact of being $k$-extremal (Lemmata~\ref{lem:extremal_directed_Haj} and~\ref{lem:HT_iff_ext}).
\end{description}
\end{proof}

\section{Recognition algorithm}\label{sec:algo}

In this section, we give a polynomial time algorithm to decide if a given digraph $D$ satisfies $\dic(D) = \lambda(D) + 1$. For algorithmic reasons, we need to avoid Haj\'os trees, so we need to devise another characterization, using the notion of parallel Haj\'os joins.




\begin{definition}[Parallel Haj\'os join]\label{def:parallel_join}
    Let $D_B$ be a digraph, set $B=V(D_B)$ and let $[a,b] \subseteq A(D_{B})$. \\  
    Let $D_{AC}$ be a digraph with $V(D_{AC}) = A \cup C$, $A \cap C= \{x\}$, let $t,w \in A \setminus x$ such that $t,w$ are in the same connected component of $D_{AC} \setminus x$, and let $u,v \in V(C)$ such that $u$ and $v$ are in the same connected component of $D_{AC}[C] \setminus x$.\\ 
    The \emph{parallel Haj\'os join} $D$ of $D_{AC}$ and $D_B$ with respect to $(t,u,v,w, [a,b])$ is the digraph obtained from disjoint copies of $D_B-[a,b]$, $D_{AC}[A]$ and $D_{AC}[C]$, by identifying the copy of $x$ in $D_{AC}[A]$ to $a$, and the copy of $x$ in $D_{AC}[C]$ to $b$
    See Figure~\ref{fig:parallel_join}.
\end{definition}

    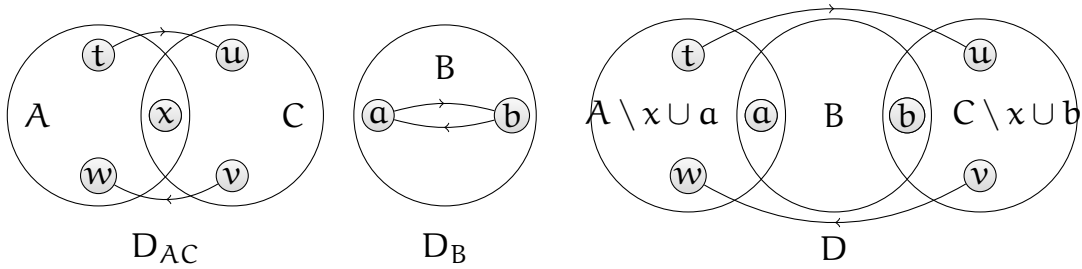
\begin{figure}[!hbtp]
    \begin{center}
        \begin{tikzpicture}[scale=0.8]

            \begin{scope}[xshift=0cm]
                \vertex (r) at (1.2,0) {$a$};
                \vertex (u) at (0,1) {$t$};
                \vertex (d) at (0,-1) {$w$};
                \draw (0,0) circle (1.6);
                \node () at (-0.6,0) {$A \setminus x \cup a$};
            \end{scope}

            \begin{scope}[xshift=2.4cm]
                \vertex (lc) at (1.2,0) {$b$};
                \vertex (rc) at (-1.2,0) {$a$};
                \draw (0,0) circle (1.6);
                \node () at (0,0) {$B$};
                \node () at (0,-2.2) {$D$};
            \end{scope}
        
            \begin{scope}[xshift=4.8cm,xscale=-1]
                \vertex (r2) at (1.2,0) {$b$};
                \vertex (u2) at (0,1) {$u$};
                \vertex (d2) at (0,-1) {$v$};
                \node () at (-0.6,0) {$C\setminus x \cup b$};
                \draw (0,0) circle (1.6);
            \end{scope}

            \draw[->-, bend left=30] (u) to (u2);
            \draw[->-, bend left=30] (d2) to (d);

            \begin{scope}[xshift=-9.7cm]
                \vertex (r) at (1.1,0) {$x$};
                \vertex (u) at (0,1) {$t$};
                \vertex (d) at (0,-1) {$w$};
                \node () at (-1,0) {$A$};
                \draw (0,0) circle (1.5);
            \end{scope}

            \begin{scope}[xshift=-7.5cm,xscale=-1]
                \vertex (r2) at (1.1,0) {$x$};
                \vertex (u2) at (0,1) {$u$};
                \vertex (d2) at (0,-1) {$v$};
                \node () at (-1,0) {$C$};
                \draw (0,0) circle (1.5);
            \end{scope}

            \draw[->-, bend left=30] (u) to (u2);
            \draw[->-, bend left=30] (d2) to (d);

            \node () at (-8.6,-2.2) {$D_{AC}$};

            \begin{scope}[xshift=-4cm]
                \vertex (lc) at (1.1,0) {$b$};
                \vertex (rc) at (-1.1,0) {$a$};
                \draw[->-, bend left=15] (lc) to (rc);
                \draw[->-, bend left=15] (rc) to (lc);
                \node () at (0,0.8) {$B$};
                \draw (0,0) circle (1.5);
                \node () at (0,-2.2) {$D_B$};
            \end{scope}

        \end{tikzpicture}
        \end{center}
           \caption{$D$ is a parallel Haj\'os join of $D_{AC}$ and $D_{B}$ with respect to $(t,u,v,w)$.}
           \label{fig:parallel_join}
        \end{figure}

Let us say an informal word on the intuition behind parallel Haj\'os join. 
Let $D=T(D_1, \dots, D_n;C)$ be a Haj\'os tree join and assume $u_iv_i \in E(T)$ is such that  both $u_i$ and $v_i$ are interior vertices of $T$. Then $D$ is the Haj\'os parallel join of $D_{i}$ and the Haj\'os tree join obtained after contracting $D_i$. 


As before, we need to prove that this operation preserves extremality.

\begin{lemma}\label{lem:extremal_parallel_Haj}
    Let $k \geq 3$. A parallel Haj\'os join of two digraphs $D_{AC}$ and $D_C$ is $k$-extremal if and only if both $D_{AC}$ and $D_{B}$ are $k$-extremal.
\end{lemma}

\begin{proof} 
    Let $D$ be the parallel Haj\'os join of $D_{AC}$ and $D_{B}$ as in Definition~\ref{def:parallel_join}. 
    
    Let $D_A = D[A] + \{ta,aw\}$, $D_{C} = D[C] + \{vb,bu\}$ and $D_{BC} = D[V(B) \cup V(C)] + \{va,au\}$. See Figure~\ref{fig:DA_DB_DC_DBC}. 

    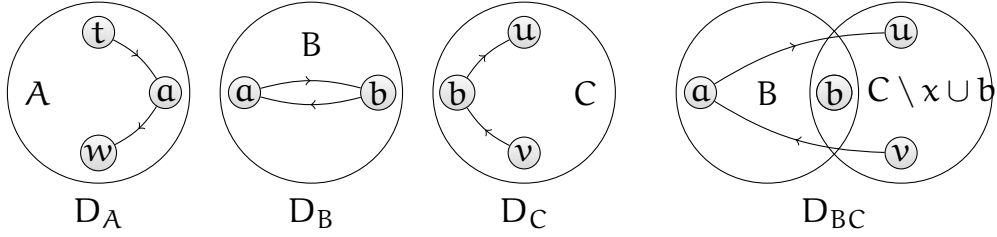
\begin{figure}[!hbtp]
    \begin{center}
        \begin{tikzpicture}[scale=0.8]
            \begin{scope}[xshift=-11cm]
                \vertex (r) at (1.1,0) {$a$};
                \vertex (u) at (0,1) {$t$};
                \vertex (d) at (0,-1) {$w$};
                \draw[->-, bend left=15] (u) to (r);
                \draw[->-, bend left=15] (r) to (d);
                \node () at (-1,0) {$A$};
                \draw (0,0) circle (1.5);
                \node () at (0,-2) {$D_A$};
            \end{scope}

            \begin{scope}[xshift=-7.5cm]
                \vertex (lc) at (1.1,0) {$b$};
                \vertex (rc) at (-1.1,0) {$a$};
                \draw[->-, bend left=15] (lc) to (rc);
                \draw[->-, bend left=15] (rc) to (lc);
                
                \node () at (0,0.8) {$B$};
                \draw (0,0) circle (1.5);
                \node () at (0,-2) {$D_B$};
            \end{scope}
        
            \begin{scope}[xshift=-4cm,xscale=-1]
                \vertex (r2) at (1.1,0) {$b$};
                \vertex (u2) at (0,1) {$u$};
                \vertex (d2) at (0,-1) {$v$};
                \draw[->-, bend right=15] (d2) to (r2);
                \draw[->-, bend right=15] (r2) to (u2);
                \draw (0,0) circle (1.5);
                \node () at (-1,0) {$C$};
                \node () at (0,-2) {$D_C$};
            \end{scope}

            \begin{scope}[xshift=0cm]
                \vertex (ll) at (1.1,0) {$b$};
                \vertex (rr) at (-1.1,0) {$a$};
                \node () at (0,0) {$B$};
                \draw (0,0) circle (1.5);
            \end{scope}
        
            \begin{scope}[xshift=2.2cm,xscale=-1]
                \vertex (r2) at (1.1,0) {$b$};
                \vertex (u2) at (0,1) {$u$};
                \vertex (d2) at (0,-1) {$v$};
                
                \node () at (-0.5,0) {$C\setminus x \cup b$};
                \draw (0,0) circle (1.5);
            \end{scope}

            \draw[->-, bend left=15] (d2) to (rr);
            \draw[->-, bend left=15] (rr) to (u2);

            \node () at (1.1,-2) {$D_{BC}$};

        \end{tikzpicture}
        \end{center}
           \caption{$D_A$, $D_B$, $D_C$ and $D_{BC}$.}\label{fig:DA_DB_DC_DBC}
        \end{figure}
    
    
    Observe that 
    \begin{itemize}
        \item $D$ is a degenerated Haj\'os bijoin of $D_{A}$ and $D_{BC}$, 
        \item $D_{BC}$ is a degenerated Haj\'os bijoin of $D_{B}$ and $D_{C}$, and
        \item $D_{AC}$ is a Haj\'os bijoin of $D_{A}$ and $D_{C}$.  
    \end{itemize}



    
    
    Suppose first that $D$ is $k$-extremal. 

    Since $D$ is a Haj\'os bijoin of $D_{A}$ and $D_{BC}$, and $D_{BC}$ is a degenerated Haj\'os bijoin of $D_{B}$ and $D_{C}$, we get that $D_A$ and $D_B$ and $D_C$ are $k$-extremal. 
    So it remains to prove that $D_{AC}$ is $k$-extremal. 
    
    Since $D_A$ and $D_C$ are  $k$-extremal, and since $D_{AC}$ is the Haj\'os bijoin of $D_A$ and $D_C$, by Lemma~\ref{lem:extremal_bijoin_Haj_only_if} we have that $D_{AC}$ is strong, biconnected and $\lambda(D_{AC}) \leq k$.

    Let us now prove that $\dic(D_{AC}) \geq k+1$. Suppose $D_{AC}$ admits a $k$-dicolouring $\varphi_{AC}$. Since $D_B$ is $k$-extremal, $D_B - [a,b]$ admits a $k$-dicolouring $\varphi_B$ with $\varphi_B(a) = \varphi_B(b)$. 
    Up to permuting colours, we may assume that $\varphi_B(a) = \varphi_{AC}(x)$. Let $\varphi: V(D) \to [1,k]$ be such that $\varphi(y) = \varphi_{B}(y)$ if $y \in B$ and $\varphi(y) = \varphi_{AC}(y)$ if $y \in A \cup C \setminus x$. As any dicycle of $D$  is either included in $D_B$, or contains vertices that form a dicycle in $D/B = D_{AC}$, $\varphi$ is a $k$-dicolouring of $D$, a contradiction. 

    Thus $k+1 \leq \dic(D_{AC}) \leq \lambda(D) + 1 \leq k+1$. Hence $\dic(D_{AC}) = \lambda(D_{AC})=k+1$, which ends the proof that $D_{AC}$ is $k$-extremal. 
    \bigskip 

    Suppose now that $D_{AC}$ and $D_{B}$ are $k$-extremal and let us prove that $D$ is $k$-extremal. 

    Since $D_{AC}$ is $k$-extremal and is a Haj\'os bijoin of $D_A$ and $D_C$, both $D_A$ and $D_C$ are $k$-extremal by Lemma~\ref{lem:extremal_bijoin_Haj}. 
    
    Since $D_{BC}$ is the Haj\'os bijoin of $D_B$ and $D_C$ and $D_B$ and $D_C$ are both $k$-extremal, $D_{BC}$ is biconnected, strong, Eulerian and $\lambda(D_{BC}) \leq k$  by Lemma~\ref{lem:extremal_bijoin_Haj_only_if}. 
    Thus by Lemma~\ref{lem:extremal_bijoin_Haj_only_if}, $D$ is biconnected, strong, Eulerian and $\lambda(D) \leq k$. 
    
    Let us now prove that $\dic(D) \geq k + 1$. Suppose that $D$ admits a $k$-dicolouring $\varphi_D$.
    Then, as $D[B] = D_{B} - [a,b]$, and because $D_B$ is $k$-extremal, $\varphi_D(a) = \varphi_D(b)$.  
    We are going to split the proof into two cases, in each case we prove that $\dic(D_{AD}) \leq k$, a contradiction. 
    
    \smallskip 
    
    \noindent\textbf{Case 1:} $\varphi_D(t) \neq \varphi_D(u)$

    Since $\varphi_D(a) = \varphi_D(b)$, either $\varphi_D(t) \neq \varphi_D(a)$ or $\varphi_D(u) \neq \varphi_D(b)$. Suppose without loss of generality that $\varphi_D(t) \neq \varphi_D(a)$. 
    Let $\varphi_C$ be a $k$-dicolouring of $D_{C} - bu = D[C] +vb$ and, up to permuting colours, assume  that $\varphi_C(b) = \varphi_D(b) (= \varphi_D(a))$. There is no monochromatic $bv$-dipath in $\varphi_{C}$ as $vb \in A(D_C - bu)$. Also, $\varphi_{C}(b) = \varphi_{C}(u)$ for $D_C$ is $k$-extremal, and thus $\varphi_{C}(u) \neq \varphi_D(t)$. 
    
    Now, let $\varphi_{AC} : V(D_{AC}) \to [1,k]$ be such that 
    
$$
\varphi_{AC}(y) = \left\{
    \begin{array}{ll}
        \varphi_D(a) & \text{ if } y=x \\
        \varphi_{C}(y) & \text{ if } y \in C \setminus b \\
        \varphi_D(y) &  \text{ if } y \in A \setminus a
    \end{array}
\right.
$$


    Observe that, since there is no monochromatic $bv$-dipath with respect to $\varphi_C$, there is no monochromatic $xv$-dipath with respect to $\varphi_{AC}$. 
    Since any dicycle of $D_{AC}$ is either included in $A \cup x$ or in $C \cup x$,  or goes through $tu$, or contains an  $xv$-dipath, $\varphi_{AC}$ is a $k$-dicolouring of $D_{AC}$, a contradiction.

    \smallskip

    \noindent\textbf{Case 2:} $\varphi_D(t) = \varphi_D(u)$. 

    
    
    There is not both a monochromatic $wt$-dipath and a monochromatic $uv$-dipath with respect to $\varphi_D$. Without loss of generality, suppose there is no monochromatic $uv$-dipath. Then, either there is no monochromatic $ub$-dipath or no monochromatic $bv$-dipath.
    Suppose without loss of generality that there is no monochromatic $bv$-dipath. Since $D_{A}$ is $k$-extremal, $D_{A} - aw$ admits a $k$-dicolouring $\varphi_{A}$. 
    Up to permuting colours, we may assume that $\varphi_{A}(a) = \varphi_D(a)$. 
    Note that $ta \in A(D_{A} - aw)$, and thus there is no monochromatic $at$-dipath in $\varphi_{A}(a)$.

    Let $\varphi_{AC} : V(D_{AC}) \to [1,k]$ be such that
    $$
\varphi_{AC}(y) = \left\{
    \begin{array}{ll}
         \varphi_D(a) & \text{ if } y=x \\
        \varphi_{A}(y) & \text{ if } y \in A \setminus a \\
        \varphi(y) &  \text{ if }y \in A \setminus a
    \end{array}
\right.
$$


    Observe that: 
    \begin{itemize}
        \item since there is no monochromatic $at$-dipath with respect to $\varphi_{A}$, there is no monochromatic $xt$-dipath with respect to $\varphi_{AD}$, 
        \item since there is no monochromatic $uv$-dipath with respect to $\varphi_D$, there is no monochromatic $uv$-dipath with respect to $\varphi_{AC}$, and
        \item since there is no monochromatic $bv$-dipath with respect to $\varphi_{D}$, there is no monochromatic $xv$-dipath with respect to $\varphi_{AD}$. 
    \end{itemize}

    Finally, observe that dicycle of $D_{AC}$  is either included in $D_{AC}[A \cup x]$ or $D_{AC}[C \cup x]$, or contains a $uv$-dipath, or a $xv$-dipath or a $ww$-dipath. Hence, $\varphi_{AC}$ is a $k$-dicolouring of $D_{AC}$, a contradiction.

    \medskip
    Thus $D$ is $k$-extremal.

\end{proof}

The algorithm will use the following third decomposition theorem for $\hk$.
\begin{theorem}\label{lem:htk_implies_other_joins}
    Let $k \geq 3$. If $D$ is $k$-extremal, then one of the following holds:
    \begin{itemize}
        \item $D$ is a symmetric odd wheel or
        \item $D = \ovlra{K}_k$, or
        \item $D$ is a directed Haj\'os join of two digraphs $D_1$ and $D_2$, or
        \item $D$ is a parallel Haj\'os join of two digraphs $D_1$ and $D_2$, or
        \item $D$ is a Haj\'os star join of $n$ digraphs $D_1, \dots D_n$, or
    \end{itemize}
\end{theorem}
   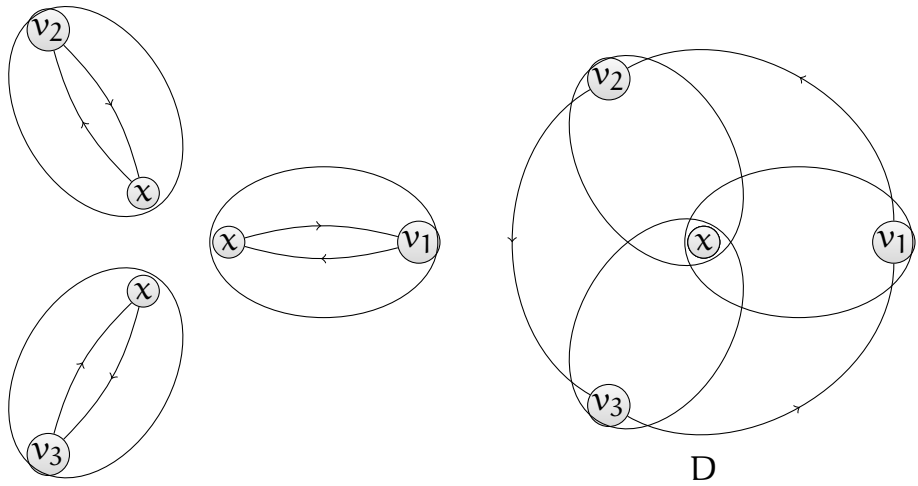
\begin{figure}[!hbtp]
    \begin{center}
        \begin{tikzpicture}[scale=0.5]

            \begin{scope}
                \vertex (lc) at (0,0) {$x$};
                \vertex (v1) at (5,0) {$v_1$};
                \draw (2.5,0) ellipse (3cm and 2cm) {};
            \end{scope}

            \begin{scope}[rotate = 120]
                \vertex (lc) at (0,0) {$x$};
                \vertex (v2) at (5,0) {$v_2$};
                \draw (2.5,0) ellipse (3cm and 2cm) {};
            \end{scope}

            \begin{scope}[rotate = -120]
                \vertex (lc) at (0,0) {$x$};
                \vertex (v3) at (5,0) {$v_3$};
                \draw (2.5,0) ellipse (3cm and 2cm) {};
            \end{scope}

            \draw[->-, bend right=60] (v1) to (v2);
            \draw[->-, bend right=60] (v2) to (v3);
            \draw[->-, bend right=60] (v3) to (v1);

            \node () at (0,-6) {$D$};

            \begin{scope}[xshift = -14cm]
                
                \begin{scope}[rotate=0]
                    \begin{scope}[xshift=1.5cm]
                        \vertex (l) at (0,0) {$x$};
                        \vertex (v) at (5,0) {$v_1$};
                        \draw[->-, bend left=15] (l) to (v);
                        \draw[->-, bend left=15] (v) to (l);
                        \draw (2.5,0) ellipse (3cm and 2cm) {};
                    \end{scope}
                    
                \end{scope}
                
                \begin{scope}[rotate=120]
                    \begin{scope}[xshift=1.5cm]
                        \vertex (l) at (0,0) {$x$};
                        \vertex (v) at (5,0) {$v_2$};
                        \draw[->-, bend left=15] (l) to (v);
                        \draw[->-, bend left=15] (v) to (l);
                        \draw (2.5,0) ellipse (3cm and 2cm) {};
                    \end{scope}
                \end{scope}
    
                \begin{scope}[rotate=-120]
                    \begin{scope}[xshift=1.5cm]
                        \vertex (l) at (0,0) {$x$};
                        \vertex (v) at (5,0) {$v_3$};
                        \draw[->-, bend left=15] (l) to (v);
                        \draw[->-, bend left=15] (v) to (l);
                        \draw (2.5,0) ellipse (3cm and 2cm) {};
                    \end{scope}
                \end{scope}
            \end{scope}
           
        \end{tikzpicture}
        \end{center}
           \caption{$D$ is a Haj\'os star join of $D_1$, $D_2$ and $D_3$.}
           \label{fig:cyclic_join}
        \end{figure}
\begin{proof}
    Suppose $D \neq \ovlra{K}_k$ and $D$ is not a symmetric odd wheel. Since it is one of the possible outputs of this theorem, we can assume that $D$ is not a directed Haj\'os join.  Thus by our main Theorem $D$ is a Haj\'os tree join:  there exists a tree $T$ with edges $\{u_1v_1, \dots, u_nv_n\}$, some  digraphs $D_1, \dots, D_n$ with $[u_i,v_i] \subseteq A(D_i)$ for $i=1, \dots, n$ such that $D=T(D_1, \dots, D_n;C)$, where $C$ is a directed cycle going through  the leaves of $T$.
    \smallskip


    


    If $T$ is a star, then $D$ is a Haj\'os star join of $D_1, \dots, D_n$ and we are done.
    \smallskip 
    
    Hence, there is $u_i,v_i \in E(T)$ such that $u_i$ and $v_i$ are both interior vertices of $T$.  
    Let $T_{u_i}$ and $T_{v_i}$ be the two connected component of $T-u_iv_i$ containing respectively $u_i$ and $v_i$. 
    By definition of a partial Eulerian list, the dicycle $C$ is the concatenation of two vertex disjoint dipaths $P_{u_i}$ and $P_{v_i}$ such that $V(P_{u_i})$ are the leaves of $T$ contained in $T_{u_i}$ and $V(P_{v_i})$ are the leaves of $T$ contained in $T_{v_i}$.
    There is $x_{u_i}, y_{a} \in V(P_{u_i})$ and $x_{v_i}, y_{v_i} \in V(P_{v_i})$ such that $x_{u_i}y_{v_i}, x_{v_i}y_{u_i} \in A(C)$. 

    Let $D'$ be the digraph obtained from $D$ by deleting $V(D_{i})\setminus \{u_i,v_i\}$ and identifying $u_i$ and $v_i$ to a new vertex $x$. 
    Then $D$ is the Haj\'os parallel join of $D'$ and $D_{ab}$ with respect to $(x_{u_i}, y_{v_i}, x_{v_i}, y_{u_i})$. To see this, look at Definition~\ref{def:parallel_join} and observe that: 
    \begin{itemize}
        \item $D_{i}$ plays the role of $D_B$,
        \item $D'$ plays the role of $D_{AC}$,
        \item $A=\bigcup_{u_jv_j \in T_{u_i}} V(D_{j})$ and $C=\cup_{u_jv_j \in T_{v_i}} V(D_{j})$,
        \item $x_{u_i}, y_{v_i}, x_{v_i}, y_{u_i}$ plays the role of respectively $t,u,v,w$
    \end{itemize} 
    Finally, observe that $x_{u_i}, y_{u_i}$ are in the same connected component of $D[A]\setminus x$ because of $P_{u_i}$ and $x_{v_i}, y_{v_i}$ are in the same connected component of $D[C] \setminus x$ because of $P_{v_i}$.

\end{proof}


\begin{theorem}
    Let $k \geq 3$. There is an algorithm that decides if a given digraph $D$ is $k$-extremal in time $\mathcal O(n^{10})$. 
\end{theorem}

\begin{proof} 
    Our algorithm is based on Theorem~\ref{lem:htk_implies_other_joins} together with Lemmata~\ref{lem:extremal_directed_Haj},~\ref{lem:extremal_parallel_Haj} and~\ref{lem:HT_iff_ext}
    
    Let $D$ be a digraph on $n$ vertices. Checking if $D$ is strong and biconnected can be done in time $\mc{O}(n^2)$. It takes time $\mc{O}(n^2)$ to check if $D = \ovlra{K}_k$ or $D$ is a symmetric odd wheel. If $D = \ovlra{K}_k$, then our algorithm outputs that $D$ is $k$-extremal. 
    We may now assume that $D$ is strong, biconnected and distinct from $\bid K_k$ and symmetric odd wheels. 

    \begin{claim}\label{clm:alg_DHJ}
        We can decide in time $\mathcal O(n^5)$ that either $D$ is not  the directed Haj\'os join of two digraphs, or $D$ is the directed Haj\'os join of two digraphs $D_1$ and $D_2$ and compute $D_1$ and $D_2$. 
    \end{claim}

    \begin{proofclaim}
        Checking if $D$ is a directed Haj\'os join of two digraphs $D_1$ and $D_2$ can be done by testing for all triples of vertices $(u, v, w)$ if $uw \in A(D)$, $D \setminus v - uw$ is not connected and $u$ and $w$ are in distinct components of $D \setminus v - uw$. 
        If $(u, v, w)$ is such a triple, Let $R_u$ (resp. $R_w$) be the connected component of $D \setminus v - uw$ containing $u$ (resp. containing $w$).
        Then $D$ is the directed Haj\'os join of $D[R_u \cup v] + uv$ and $D[R_w \cup v] + vw$. 
        This can be done in time $\mc{O}(n^5)$. 
    \end{proofclaim}

    \begin{claim}\label{clm:alg_parallelJ}
        We can decide in time $\mathcal O(n^8)$ that either $D$ is not the parallel Haj\'os join of two digraphs, or $D$ is the parallel Haj\'os join of two digraphs $D_1$ and $D_2$ and compute $D_1$ and $D_2$. 
    \end{claim}

    \begin{proofclaim}
        Checking if $D$ is a directed Haj\'os join of two digraphs $D_1$ and $D_2$ can be done by testing for all $6$-tuples of vertices $(t, u, v, w, a, b)$ if $tu, vw \in A(D)$, $D \setminus \{a,b\} - \{tu, vw\}$ has a connected component $A$ containing both $t$ and $u$, a connected component $C$ containing both $u$ and $v$, and some other connected components union of which we name $B$. Then $D$ is the parallel Haj\'os join of  the digraphs obtained from $D[A \cup a]$ and $D[C \cup b]$ by deleting $tu, vw$, identifying $a$ and $b$ into a new vertex $x$, and adding arcs $tx,xw,vx,xu$, and $D[B] + [a,b]$.         
        This can be done in time $\mc{O}(n^8)$.
    
    \end{proofclaim}

    \begin{claim}\label{clm:alg_HSJ}
        We can decide in time $\mc{O}(n^5)$ that either $D$ is not the Haj\'os star join of some digraphs, or $D$ is a Haj\'os star join of some digraphs  $D_1, \dots, D_{\ell}$ and compute $D_1, \dots, D_{\ell}$. 
    \end{claim}
    
    \begin{proofclaim}
        Observe that $D$ is a Haj\'os star join of $\ell$ digraphs if and only if it has $\ell +1$ vertices $x, v_1, \dots, v_{\ell}$ such that $C= v_1 \ra \dots \ra v_{\ell} \ra v_1$ and $D\setminus x -A(C)$ has exactly $\ell$ connected component $R_1, \dots, R_{\ell}$ such that $v_i \in R_i$ for $i=1, \dots, \ell$.  
        Indeed, if it is the case then $D=T(D_1, \dots, D_{\ell}, C)$ where $T$ is the tree with edges $\{xv_1, \dots, xv_{\ell}\}$,  and $D_i = D[R_i \cup x] + [x,v_i]$, and the "only if" part is straightforward by definition of a Haj\'os star join.  
        
        Hence, given $\ell+1$ vertices $y, p_1, \dots, p_{\ell}$, we can decide if they can play the role of respectively $x, v_1, \dots, v_{\ell}$ in time $\mc O(n^2)$. 
        But this is not enough to conclude because  $\ell$ can be large. 

        Anyway, we are going to show that given a triple of vertices $(y,p_{\ell}, p_1)$, we can guess in time $\mathcal O(n^2)$ if there exists $p_2, \dots, p_{\ell-1}$ such that $y, p_1, \dots, p_{\ell}$ can play the role of respectively $x, v_1, \dots, v_{\ell}$. In this case, we say that $(y, p_{\ell}, p_1)$ is a \emph{good guess}. 

        Let $(y, p_{\ell}, p_1)$ be a triple of vertices of $D$ such that $p_{\ell}p_1 \in A(G)$. 
        Compute the list of bridges $\mathcal B$ of $D \setminus y - p_{\ell}p_1$. This can be done in $\mc O(n^2)$ time. Observe that if $\mathcal B$ does not contain a $p_1p_{\ell}$-dipath, then our guess is wrong. 
        Assume otherwise, and observe that $\mathcal B$ induces a forest, so it has a unique $p_1p_{\ell}$-dipath, say $P=p_1 \ra p_2 \ra \dots \ra p_{\ell}$. 
        Now, $(y, p_{\ell}, p_1)$ is a good guess if and only  $p_2, p_3, \dots, p_{\ell-1}$ can play the role of respectively $v_2, \dots, v_{\ell-1}$, so we are done. 

        Altogether, it takes $\mc O(n^2)$ to decide if a triple of vertices is a good guess, so the total time is $\mathcal O(n^5)$.
           
        
        
    \end{proofclaim}
    
    Thus in time $\mc{O}(n^8)$, we can check whether $D$ is a directed Haj\'os join or a parallel Haj\'os join of two digraphs $D_1$ and $D_2$ and compute $D_1$ and $D_2$, or a Haj\'os star join of $\ell$ digraphs $D_1, \dots D_{\ell}$ and compute $D_1, \dots, D_{\ell}$.
    If all these checks fail, then by Theorem~\ref{lem:htk_implies_other_joins} $D$ is not $k$-extremal.
    
    If $D$ is a directed Haj\'os join of two digraphs $D_1$ and $D_2$, we can then recursively check whether $D_1$ and $D_2$ are $k$-extremal, and our algorithm can return that $D$ is $k$-extremal if and only if they both are $k$-extremal,  by Lemma~\ref{lem:extremal_directed_Haj}. 
    We do the same if $D$ is a parallel join (Lemma~\ref{lem:extremal_parallel_Haj}) or a star Haj\'os join (Lemma~\ref{lem:HT_iff_ext}). 

    \medskip
    
    Let us now prove that our algorithm has time complexity $\mc{O}(n^{10})$. First, note that in each case, in time $\mc{O}(n^{8})$, either we conclude that $D$ is not $k$-extremal, or we make $\ell \geq 2$ recursive calls on digraphs $(D_i)_{i \in [1, \ell]}$. We have that $\sum_{i \in [1, \ell]} |V(D_i)| - 1 \leq |V(D)| - 1$ and for $i \in [1, \ell]$, that $2 \leq |V(D_i)| < |V(D)|$. Let us consider $T$, the rooted tree of recursive calls of our algorithm, with each node $v$ labelled with the digraph $D_v$ of the corresponding recursive call. Let $depth$ be the function which associates to a node its depth in $T$. Then, 
    $$\sum_{v \mid depth(v) = 0} |V(D_{v})| - 1 = V(D) - 1 \leq n$$ and, for $i \geq 1$, $$\sum_{v \mid depth(v) \leq i} |V(D_{v})| - 1 \leq \sum_{v \mid depth(v) = i - 1} |V(D_{v})| - 1.$$ Thus we can recursively prove for any $k \in \mathbb{N}$ that $\sum_{v \mid depth(v) \leq k} |V(D_{v})| - 1 \leq n$. As every $D_v$ has $|V(D_v)| \geq 2$, this implies that there are at most $n$ calls at any depth. But, since $T$ has depth at most $n$, this means there are at most $n^2$ recursive calls. Each of these recursive calls takes time at most $\mc{O}(n^8)$, and thus our algorithm has time complexity $\mc{O}(n^{10})$.
\end{proof}


\section{The hypergraph case} \label{sec:hypergraph}

As mentioned in the introduction, Theorem~\ref{thm:nono_version} has already been generalized to hypergraph with chromatic number at least $4$ by  Schweser,  Stiebitz and  Toft~\cite{SST19}. Their result is closely related to ours as we explain below. 

Let $H$ be a hypergraph. 
Its chromatic number $\chi(H)$  is the least integer $k$ such that the vertices of $H$ can be coloured in such a way that no hyperedge is monochromatic. 
A \emph{$uv$-hyperpath} in $H$ is a sequence 
$(u_1, e_1, u_2, e_2, \dots , e_{q-1}, u_q)$ of distinct vertices 
$u_1, u_2, \dots , u_q$ of $H$ and distinct hyperedges $e_1, e_2,\dots , e_{q-1}$ of $H$ such that $u = u_1, v = u_q$ and $\{u_i, u_{i+1}\} \subseteq e_i$ for $i \in \{1, 2,\dots  , q - 1\}$. 
The local connectivity $\lambda(u,v)$ of two vertices $u$ and $v$  is the maximum number of hyperedge-disjoint $uv$-hyperpaths linking $u$ and $v$ and the maximum local connectivity of $H$ is $\lambda(H) = max_{u\neq v}\lambda(u,v)$. 

Let $H_1$ and $H_2$ be two hypergraphs and, for $i=1,2$, let $e_i \in E(H_i)$ and $v_i\in e_i$. The \emph{Haj\'os hyperjoin} of $H_1$ and $H_2$ with respect to $((e_1,v_1), (e_2, v_2))$ is the hypergraph $H$ obtained from $H_1$ and $H_2$ by identifying $v_1$ and $v_2$ into a new vertex $v$, deleting $e_1$ and $e_2$ and adding a new edge $e$ where $e=e_1\cup e_2 \setminus \{v_1,v_2\}$ or $e=e_1\cup e_2 \cup \{v\} \setminus \{v_1,v_2\}$. 

Let $\mathcal H_3$ be the smallest class of hypergraphs that contains all odd wheels and is closed under taking Haj\'os hyperjoins, and for $k \geq 4$, 
$\mathcal H_k$  is the smallest class of hypergraphs that contains $K_{k+1}$ and is closed under taking Haj\'os hyperjoins. 

\begin{theorem}[\cite{SST19}]\label{thm:main_hypergraph}
    Let  $H$ be a hypergraph with $\chi(H) = k+1 \geq 4$. Then $\chi(H) = \lambda(H) + 1$ if and only if a block of $H$ is in $\mathcal H_k$. 
\end{theorem}

Given a digraph $D$, let $H_D$ be the hypergraph with vertex set $V(D)$, and $e \subseteq  V(D)$ is a hyperedge of $H_D$ if and only if it induces a directed cycle in $D$. We clearly have that $\dic(D) = \chi(H_D)$. 
Hence, one could suspect that our result is actually implied by the result of Schweser,  Stiebitz and  Toft. But this is not the case because a dipath of $D$ does not need to translate into a hyperpath of $H_D$, and thus the maximum local arc-connectivity of $D$ does not need to be equal to the maximum local edge-connectivity of $H_D$.  Actually, we can prove that the class of extremal digraphs strictly contains the class of extremal hypergraphs in the following sense: 

\begin{lemma}\label{lem:digraphs>hypergraphs}
Let $k \geq 3$. 
    \begin{itemize}
    \item[(i)] For every hypergraph $H \in  \mathcal H_k$, there exists a digraph $D \in \hk$ such that  $H_D= H$. 
    \item[(ii)] There exist  (an infinite family of) digraphs $D$ such that $D \in \hk$ and  $H_D \notin  \mathcal H_k$.  
\end{itemize}
\end{lemma}

Let us first prove an important property of hypergraphs in $\mathcal H_k$. 

\begin{property}\label{prop:hypergraph_nul}    
Let $H \in \mathcal H_k$. Then for every $e,e' \in E(H)$, $|e\cap e'| \leq 1$.
\end{property}

\begin{proof}
    The result holds for complete graphs and odd wheels, and it is easy to see that if it holds for two hypergraphs $H_1$ and $H_2$, then it also holds for any Haj\'os hyperjoin of $H_1$ and $H_2$. 
\end{proof}

\begin{proof}[of Lemma~\ref{lem:digraphs>hypergraphs}]
    Using directed Haj\'os join, it is not hard to construct a digraph $D \in \hk$ that has two induced directed cycles with two common vertices, see Figure~\ref{fig:graph_vs_hypergraph} for an example. By Property~\ref{prop:hypergraph_nul}, $H_D \notin \mathcal H_k$. This proves the second part of the lemma.  

    Let us now prove  the first part of Lemma. Let $H \in \mathcal H_k$. 
    If $H= K_{k+1}$, then  $K_{k+1} = H_{\bid K_{k+1}}$ and since $ \bid K_{k+1} \in \hk$ we are done. 
    
    Assume now that $H$ is the Haj\'os hyperjoin of two hypergraphs $H_1, H_2 \in \mathcal H_k$ with respect to $((e_1, v_1), (e_2, u_1))$. 
    By induction, for $i=1,2$, there is $D_i \in \hk$ such that $H_{D_i} = H_i$. 

    Let $C_1 = v_1 \ra v_2 \ra \dots  \ra  v_{\ell_1} \ra v_1$ be the induced directed cycle of $D_1$ corresponding to $e_1$, and $C_2 = u_1 \ra u_2 \ra \dots  \ra  u_{\ell_2} \ra u_1$ the induced directed cycle of $D_2$ corresponding to $e_2$. 
    Let $D$ be the digraph obtained from $D_{H_1}$ and $D_{H_2}$ by identifying $v_1$ and $u_1$ into a new vertex $v$, and:

\begin{itemize}
    \item if the new hyperedge $e$ of $H$ is $e_1\cup e_2 \cup v \setminus \{v_1,u_1\}$, then delete the arcs $v_{\ell_1}v$, $vu_2$ from $D$ and add the arc $v_{\ell_1}u_{2}$.   
    \item if the new hyperedge $e$ of $H$ is $e_1\cup e_2 \setminus \{v_1,u_1\}$, then delete the arcs $vv_2$, $v_{\ell_1}v$, $vu_2$, $u_{\ell_2}v$ and add the arcs $u_{\ell_2}v_{2}$ and $v_{\ell_1}u_{2}$. 
\end{itemize}

We need to prove that $D \in \hk$ and $H_D=D$. 
We treat the two cases one after the other. 
\smallskip 

Assume we are in the first case. 
Then $D$ is the directed Haj\'os join of $D_1$ and $D_2$ with respect to $(v_{\ell_1}v_1, u_1u_2)$, and thus $D$ is $k$-extremal by Lemma~\ref{lem:extremal_directed_Haj} and thus in $\hk$ by our Theorem~\ref{thm:main_HT}.

Let us now prove that $H_D = H$.  
Let $C = v \ra v_2 \ra \dots \ra v_{\ell_1} \ra u_2 \ra \dots \ra u_{\ell_2} \ra v$ and observe it is an induced directed cycle of $D$. 
We first prove that $E(H) \subseteq E(H_D)$. The newly created edge $e$ of $H$ is in $E(H_D)$ because of $C$. 
Let $f \in E(H) \setminus \{e\}$. 
By Property~\ref{prop:hypergraph_nul},  $f$ does not contain $\{v,v_{\ell_1}\}$ nor $\{v,u_2\}$. Thus $f$ corresponds to an induced directed cycle of $D_1$ or $D_2$ that still exists in $D$, so $f \in E(H_D)$.   

Let us now prove that $E(H_D) \subseteq E(H)$.  
Observe that  $v_{\ell_1}v$ is not a chord of a directed cycle of $D_1$ (because $H_{D_1} \in \mathcal {H}_k$ and  property~\ref{prop:hypergraph_nul}), so deleting it does not create a new induced directed cycle. The same holds for $vu_2$. 
Hence, proving that $E(H_D) \subseteq E(H)$  boils down to proving that the only induced directed cycle going through $v_{\ell_1}u_2$ is $C$. 
Assume it is not the case, and let $C'$ be such an induced directed cycle. 
Then $D_2$ contains an induced directed cycle $C'_2$   with arcs $A(C') \cap A(D_2)$ and $u_1u_2$. 
Then $C_2$ and $C'_2$ are two induced directed cycles of $D_2$ with  
$u_1$ and $u_2$ in common. Since $H_{D_1} = D_1\in \mathcal H_k$, it contradicts Property~\ref{prop:hypergraph_nul}. 
\smallskip 

Assume we are in the second case. 

Let us first prove that $H_D = H$.  The proof  is very similar to that of the first case.  
Let $C = v_2 \ra \dots \ra  v_{\ell_1} \ra u_2 \ra \dots \ra u_{\ell_2} \ra v_2$ and observe it is an induced directed cycle of $D$. 
We first prove that $E(H) \subseteq E(H_D)$. The newly created edge $e$ of $H$ is in $E(H_D)$ because of $C$. 
Let $f \in E(H) \setminus \{e\}$. 
By Property~\ref{prop:hypergraph_nul}, $f$ does not contain $\{v,v_{2}\}$, nor $\{v,v_{\ell_1}\}$, nor $\{v,u_2\}$, nor $\{v,u_{\ell_2}\}$. Thus $f$ corresponds to an induced directed cycle of $D_1$ or $D_2$ that still exist in $D$. 

Let us now prove that $E(H_D) \subseteq E(H)$. 
Similarly to the previous case, the arc $v_1v_2$ ($v_{\ell_1}v$) is not a chord of a directed cycle of $D_1$, so deleting it does not create a new induced directed cycle. The same holds for $vu_2$ and $u_{\ell_2}v$.  
So proving that $E(H_D) \subseteq E(H)$ boils down to proving that  each of $u_{\ell_2}v_2$ and $v_{\ell_1}u_2$ are not contained in any other induced directed cycle than $C$ in $D$. 
Assume for contradiction that $C'$ is an induced directed cycle of $D$ containing $u_{\ell_2}v_2$ (the proof for $v_{\ell_1}u_2$ is similar). 
Than $D_2$ contains an induced directed cycle $C'_2$ with arcs included in $\{A(C') \cap A(D_2) \} \cup \{u_1u_2, u_{\ell_2}u_1\}\}$ that has at least two vertices in $\{u_1, u_2, u_{\ell_2}\}$ in common with $C_2$. A contradiction to the fact that $H_{D_2} \in \mathcal H_k$ and Property~\ref{prop:hypergraph_nul}. 

Let us now prove that $D \in \hk$.  
Observe that $D$ is the bijoin of $D_1$ and $D_2$ with respect to $((v_{\ell_1}, v_1, v_2), (u_{\ell_2}, u_1, u_2))$. Moreover, since  $H_D = H$, $\dic(D) = \chi(H) = k+1$.  Hence, by Lemma~\ref{lem:extremal_bijoin_Haj_only_if}, $D$ is $k$-extremal and thus in $\hk$ thus in $\hk$ by Theorem~\ref{thm:main_HT}.
  
\end{proof}

\begin{figure}[!hbtp]
\begin{center}
        \begin{tikzpicture}[scale = 3]

            \begin{scope}[]
                \vertex (a) at (0,0) {c};
                
                \begin{scope}[shift=(a)]
                    \vertex (b) at (1,0) {b};
                    
                    \begin{scope}[rotate = 60]
                        \vertex (c) at (1,0) {a};
                    \end{scope}

                    \begin{scope}[rotate = 30]
                        \vertex (d) at (0.57735,0) {};
                    \end{scope}

                    \draw[->-, bend left=15] (d) to (a);
                    \draw[->-, bend left=15] (a) to (d);
                    \draw[->-, bend left=15] (d) to (b);
                    \draw[->-, bend left=15] (b) to (d);
                    \draw[->-, bend left=15] (d) to (c);
                    \draw[->-, bend left=15] (c) to (d);    
                \end{scope}
    
                \draw[->-, bend left=15] (c) to (b);
                \draw[->-, bend left=15] (b) to (a);
                \draw[->-, bend left=15] (a) to (c);

                \begin{scope}[shift=(a), rotate = 120]
                    \vertex (u) at (1,0) {d};
                    \begin{scope}[rotate = 60]
                        \vertex (v) at (1,0) {};
                    \end{scope}
                    \draw[->-, bend left=15] (u) to (v);
                    \draw[->-, bend left=15] (v) to (u);
                    \draw[->-, bend left=15] (a) to (v);
                    \draw[->-, bend left=15] (v) to (a);
                    \draw[->-] (u) to (a);
                    \draw[->-, bend right=15] (c) to (u);  

                    \begin{scope}[rotate = 30]
                        \vertex (w) at (0.57735,0) {};
                    \end{scope}

                    \draw[->-, bend left=15] (w) to (a);
                    \draw[->-, bend left=15] (a) to (w);
                    \draw[->-, bend left=15] (w) to (u);
                    \draw[->-, bend left=15] (u) to (w);
                    \draw[->-, bend left=15] (w) to (v);
                    \draw[->-, bend left=15] (v) to (w);    
                \end{scope}
    
                \begin{scope}[shift=(b), rotate = -120]
                    \vertex (u) at (1,0) {};
                    \begin{scope}[rotate = 60]
                        \vertex (v) at (1,0) {};
                    \end{scope}
                    \draw[->-, bend left=15] (u) to (v);
                    \draw[->-, bend left=15] (v) to (u);
                    \draw[->-, bend left=15] (b) to (v);
                    \draw[->-, bend left=15] (v) to (b);
                    \draw[->-] (u) to (b);
                    \draw[->-, bend right=15] (a) to (u);

                    \begin{scope}[rotate = 30]
                        \vertex (w) at (0.57735,0) {};
                    \end{scope}

                    \draw[->-, bend left=15] (w) to (b);
                    \draw[->-, bend left=15] (b) to (w);
                    \draw[->-, bend left=15] (w) to (u);
                    \draw[->-, bend left=15] (u) to (w);
                    \draw[->-, bend left=15] (w) to (v);
                    \draw[->-, bend left=15] (v) to (w);
                \end{scope}
    
                \begin{scope}[shift=(c), rotate = 0]
                    \vertex (u) at (1,0) {};
                    \begin{scope}[rotate = 60]
                        \vertex (v) at (1,0) {};
                    \end{scope}
                    \draw[->-, bend left=15] (u) to (v);
                    \draw[->-, bend left=15] (v) to (u);
                    \draw[->-, bend left=15] (c) to (v);
                    \draw[->-, bend left=15] (v) to (c);
                    \draw[->-] (u) to (c);
                    \draw[->-, bend right=15] (b) to (u);     

                    \begin{scope}[rotate = 30]
                        \vertex (w) at (0.57735,0) {};
                    \end{scope}

                    \draw[->-, bend left=15] (w) to (c);
                    \draw[->-, bend left=15] (c) to (w);
                    \draw[->-, bend left=15] (w) to (u);
                    \draw[->-, bend left=15] (u) to (w);
                    \draw[->-, bend left=15] (w) to (v);
                    \draw[->-, bend left=15] (v) to (w);
                \end{scope}
            \end{scope}

        \end{tikzpicture}
        \end{center}
           \caption{A $3$-extremal digraph. The two induced dicycles $a \rightarrow b \rightarrow c \rightarrow a$ and $a \rightarrow d \rightarrow c \rightarrow a$ share two vertices. Hence its hypergraph of induced dicycles is not $3$-extremal.}\label{fig:graph_vs_hypergraph}
        \end{figure}
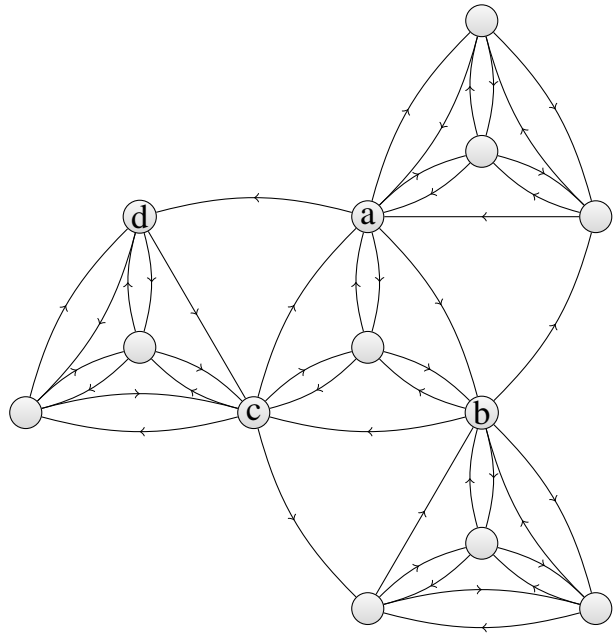



\section{$2$-extremal digraphs}

Similarly to the hypergraph case, the case where $k=2$ seems to be more difficult.  
$\bid K_3$ is of course $2$-extremal. 
A \emph{directed wheel} is a digraph made of a directed cycle plus a vertex linked by a digon to every vertex of the directed cycle. Directed wheels are $2$-extremal. We now give a way to generalize directed wheels to get a simple family of $2$-extremal digraphs that cannot be obtained (at least for some of them) from $\bid K_3$ and directed wheels by applying Haj\'os directed join or Haj\'os tree join. 

\begin{definition}[Generalized directed wheels]\label{def:HTJ2}

A digraph $D$ is a \emph{generalized directed wheel} if it can be obtained from a symmetric rooted tree $T$ on at least $3$ vertices, in which each path from the root to leaf has the same parity (either all even, or all odd) plus a directed cycle $x_1 \ra x_2 \ra \dots \ra x_{\ell} \ra x_1$ where $(x_1, \dots, x_{\ell})$ is a circular ordering of the leaves of $T$ following the natural ordering of an embedding of $T$. 
\end{definition}
See Figure~\ref{fig:2_extremal} for an example of a generalized wheel. Observe that $\bid K_3$ and directed wheels are generalized wheels. Moreover, it is routine work to check that generalized wheels are $2$-extremal.  
Let $\mathcal H_2$ be the smallest class containing generalized wheels and stable by Haj\'os tree join and directed Haj\'os join. 

\begin{conjecture}\label{conj:two_extremal}
A digraph $D$ is $2$-extremal if and only if $D \in \mathcal H_2$.
\end{conjecture}
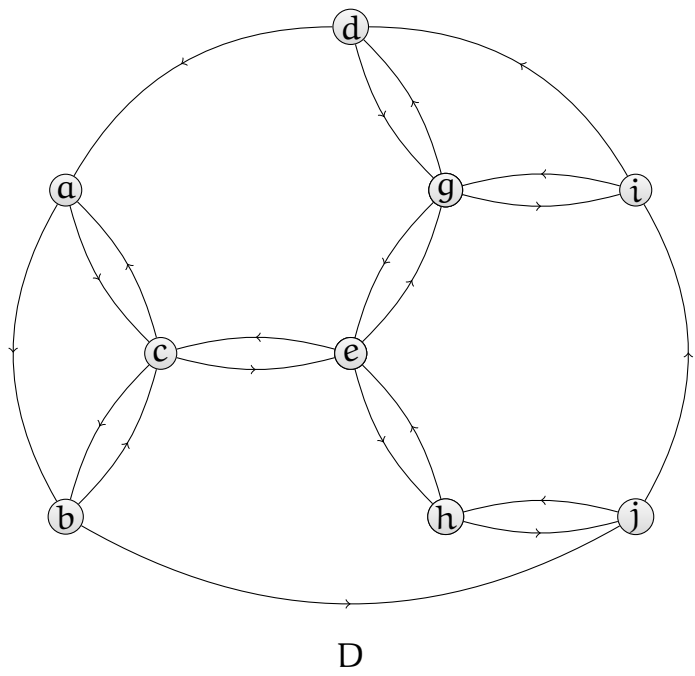
\begin{figure}[!hbtp]
    \begin{center}
        \begin{tikzpicture}[scale=0.5]
            \begin{scope}
                \vertex (c) at (0,0) {$c$};
                \vertex (e) at (5,0) {$e$};
                \draw[->-, bend right = 15] (c) to (e);
                \draw[->-, bend right = 15] (e) to (c);
            \end{scope}

            \begin{scope}[shift=(c), rotate = 120]
                \vertex (a) at (5,0) {$a$};
                \draw[->-, bend right = 15] (c) to (a);
                \draw[->-, bend right = 15] (a) to (c);
            \end{scope}

            \begin{scope}[shift=(c), rotate = -120]
                \vertex (b) at (5,0) {$b$};
                \draw[->-, bend right = 15] (c) to (b);
                \draw[->-, bend right = 15] (b) to (c);
            \end{scope}

            \begin{scope}[shift=(e), rotate = 60]
                \vertex (e) at (0,0) {$e$};
                \vertex (g) at (5,0) {$g$};
                \draw[->-, bend right = 15] (e) to (g);
                \draw[->-, bend right = 15] (g) to (e);
            \end{scope}

            \begin{scope}[shift=(e), rotate = -60]
                \vertex (e) at (0,0) {$e$};
                \vertex (h) at (5,0) {$h$};
                \draw[->-, bend right = 15] (e) to (h);
                \draw[->-, bend right = 15] (h) to (e);
            \end{scope}

            \begin{scope}[shift = (g)]
                \vertex (g) at (0,0) {$g$};
                \vertex (i) at (5,0) {$i$};
                \draw[->-, bend right = 15] (g) to (i);
                \draw[->-, bend right = 15] (i) to (g);
            \end{scope}
            
            \begin{scope}[shift = (g), rotate=120]
                \vertex (g) at (0,0) {$g$};
                \vertex (d) at (5,0) {$d$};
                \draw[->-, bend right = 15] (g) to (d);
                \draw[->-, bend right = 15] (d) to (g);

            \end{scope}

                \begin{scope}[shift = (h), rotate=0]
                \vertex (h) at (0,0) {$h$};
                \vertex (j) at (5,0) {$j$};
                \draw[->-, bend right = 15] (h) to (j);
                \draw[->-, bend right = 15] (j) to (h);
            \end{scope}

            \node () at (e |- 0, -8) {$D$};

            \draw[->-, bend right = 30] (a) to (b);
            \draw[->-, bend right = 30] (b) to (j);
            \draw[->-, bend right = 30] (j) to (i);
            \draw[->-, bend right = 30] (i) to (d);
            \draw[->-, bend right = 30] (d) to (a);

        \end{tikzpicture}
        \end{center}
           \caption{A $2$-extremal digraph.}\label{fig:2_extremal}
        \end{figure}
        

\cleardoublepage 

\ctparttext{\centering In which we try to characterize digraphs which necessarily appear inside digraphs of large dichromatic number.}

\part{Towards a directed analogue of \Gya-Sumner Conjecture} 

\chapter{Towards a directed analogue of \Gya-Sumner conjecture}\label{chpt:gyarfas}

\begin{flushright}{\slshape    
This chapter is built upon a work of Pierre Aboulker, \\
Pierre Charbit and Reza Naserasr, published in \cite{ACN21}}. \\ \medskip
\end{flushright}

\emph{In this chapter, we present a conjecture regarding the induced subdigraphs of digraphs with a large dichromatic number.}

\section{Introduction}

Even though it has been widely studied, there are still a lot of open questions regarding the chromatic number. One of them is the following: which graphs must necessarily appear as induced subgraphs of graphs of large enough chromatic number? This can also be formulated as follows: which classes of graphs $\mc{F}$ are such that the class of graphs not containing any member of $\mc{F}$ as induced subgraphs has a bounded chromatic number? 

As complete graphs have an unbounded chromatic number, and as the only induced subgraphs of complete graphs are complete graphs, such a class $\mc{F}$ must contain a complete graph. On the other hand, \Erd proved in \cite{E59} that there exist graphs of arbitrarily large girth and arbitrarily large chromatic number: this implies that if $\mc{F}$ is finite, it must contain a forest. \Gya and Sumner have conjectured that the reciprocal holds:

\begin{conjecture}
    Given any forest $F$ and complete graph $K$, the class of graphs which contain neither $F$ nor $K$ as induced subgraphs has a bounded chromatic number.
\end{conjecture}

This conjecture is still largely open.

\medskip

There is also a fruitful discussion to be had when $\mc{F}$ is not restricted to be finite. Let $k, \ell$ be two integers. Scott and Seymour proved in \cite{SS16} that if graphs in $\mc{F}$ do not contain $K_k$ nor any odd holes, then $\mc{F}$ has a finite chromatic number. Along with Chudnovsky, they also proved in \cite{CSS17} that forbidding $K_k$ and all holes of size at least $\ell$ also bounds the chromatic number. Chudnovsky, Scott, Seymour and Spirkl then unified these two results, by proving that forbidding $K_k$ and all odd holes of length greater $\ell$ was sufficient to bound the chromatic number of a class of graphs. 

Scott and Seymour then generalized these results with the following one in \cite{SS20brooms}:
\begin{theorem}
    Let $n,r \geq 0$ and $q \geq 1$, the class of graphs not containing $K_n$ nor any hole of length $kq + r$ for any integer $k$ has bounded chromatic number.
\end{theorem}

In both finite and infinite cases, Scott and Seymour have compiled a list of results in their survey \cite{SS20}.

\medskip

In the directed setting, we consider the corresponding problem for dichromatic number: which sets of digraphs $\mc{D}$ are such that $\dic(\F(\mc{D}))$ is finite? Recall that such a set of digraphs is said to be heroic.

In the infinite case, only a few results are known. Notably, Carbonero, Hompe, Moore and Spirkl have shown in \cite{CHMS22} that for any integer $\ell$, the class of digraphs in which all induced directed cycles have length $\ell$ does not have a finite dichromatic number. In the rest of this chapter, we restrict ourselves to finite sets of digraphs. 


\section{The special case of tournaments}

The first step is to look at tournaments and try to characterize heroes, i.e. digraphs that appear in tournaments of large enough dichromatic number. As tournaments are exactly digraphs in $\F(\ova{K_2}, \ovlra{K}_2)$, this is equivalent to finding digraphs $H$ such that $\dic(\F(\ova{K_2},\ovlra{K}_2,H))$ is finite. Clearly, there exist heroes: $\vec C_{3}$ is a hero since in a tournament induced dicycles are isomorphic to $\vec C_{3}$ and, by Ramsey's Theorem, transitive tournaments must be heroes. But are all tournaments heroes? To disprove this, we must exhibit a construction of tournaments with an arbitrarily large dichromatic number. 


Let $(D_i)_{1 \leq i \leq \ell}$ be $\ell$ digraphs. We denote as $\vec C_{\ell}(D_1,\dots,D_{\ell})$ the digraph obtained from their disjoint union by adding all arcs from $D_i$ to $D_{i+1}$, for $1 \leq i \leq \ell - 1$, and all arcs from $D_{\ell}$ to $D_1$. 

The following folklore result is straightforward by induction:

\begin{theorem}\label{thm:tournaments_unbounded}
Let $\ell \geq 3$. Define $F_1 = TT_1$ and, for $k \geq 2$, let $F_k = \vec C_{\ell}(1, F_{k-1}, \dots, F_{k-1})$. Then, for any integer $k$, $\dic(F_k) = k$.
\end{theorem}

\begin{proof}

Let us first prove that $\dic(F_k) \geq k$ for any $k \geq 1$. Let us prove this by induction on $k$. If $k=1$, the statement is clear, thus let $F_k = \vec C_{\ell}(1, F_{k-1}, \dots, F_{k-1})$ and suppose $F_k$ admits a $(k-1)$-dicolouring $\varphi$. Let $V_1, \dots, V_{\ell - 1}$ be the $\ell - 1$ subsets of vertices inducing copies of $F_{k-1}$ in $F_k$, and let $v$ be the remaining vertex. For $1 \leq i \leq \ell - 1$,  $\dic(V_i) = k-1$, and there must be at least one vertex $v_i \in V_i$ with $\varphi(v_i) = \varphi(v)$. But then $v \Ra v_1 \Ra \dots v_{\ell} \Ra v$ is a monochromatic directed cycle, a contradiction.

A $k$-dicolouring of $F_k$ can be obtained recursively by considering any $(k-1)$-dicolouring on each copy of $V_i$, and putting a $k$-th colour on $v$.
\end{proof}

This proves in particular that there exist tournaments of arbitrarily large dichromatic numbers. 

\medskip

Let $D_1$ and $D_2$ be two digraphs. We denote by $D_1 \Ra D_2$ the digraph obtained from the disjoint union of $D_1$ and $D_2$ by adding all arcs from $D_1$ to $D_2$. We can now get to the main result of this section: Berger, Choromanski, Chudnovsky, Fox, Loebl, Scott, Seymour and Thomassé have managed to characterize heroes in tournaments in \cite{hero}.

\begin{theorem}[\cite{hero}]
A digraph is a hero if and only if it can be constructed from the following inductive rules:
\begin{itemize}
    \item $TT_1$ is a hero
    \item If $H_1$ and $H_2$ are heroes, so is $H_1 \Ra H_2$
    \item For every $k \in \mathbb{N}$, if $H$ is a hero, so are $\vec C_{3}(1,k,H)$ and $\vec C_{3}(1,H,k)$
\end{itemize}
\end{theorem}

Note that all classes of digraphs we will study in this part will contain all tournaments as induced subdigraphs: thus all heroes in these classes must also be heroes in tournaments.

\section{The main conjecture}

Let $\mc{D}$ be a (minimal for inclusion) finite heroic class of digraphs. We can prove that $\mc{D}$ must contain the following digraphs:
\begin{itemize}
    \item As symmetric complete graphs have an unbounded dichromatic number and as induced subdigraphs of symmetric complete graphs are symmetric complete graphs, $\mc{D}$ must contain a symmetric complete graph.
    \item The same reasoning for tournaments yields that $\mc{D}$ must contain a tournament. Note that if it contains only one such tournament, it must be a hero.
    \item If $\mc{D}$ does not contain any symmetric forest, then let $k = \max \{|V(D)| \mid D \in \mc{D}\} + 1$. Then $\{\ovlra{G} \mid G\text{ has girth at least $k+1$}\}$ has unbounded dichromatic number, yet is included in $\F(\mc{D})$. Thus $D$ must contain a symmetric forest.
    \item Harutyunyan and Mohar proved in \cite{HM12} that there exist digraphs of arbitrarily large girth and arbitrarily large dichromatic numbers. Thus, a proof similar to the previous case yields that $\mc{D}$ must contain an oriented forest.
\end{itemize}

\medskip

If $|\mc{D}| = 1$, then $\mc{D}$ must contain a digraph that is both a symmetric forest and an oriented forest, which is only possible if $\mc{D} = \{TT_1\}$. If $|\mc{D}| = 2$, then $\mc{D}$ must contain an oriented forest that is also a tournament, and a symmetric forest that is also a symmetric complete graph, which is only possible if $\mc{D} = \{TT_2, \ovlra{K}_2\}$. If $|\mc{D}| = 3$, a similar analysis yields that $\mc{D}$ must be of one of the following forms:
\begin{itemize}
    \item $\{TT_2, \ovlra{K}_k, \ovlra{F}\}$ with $F$ a forest and $k$ an integer
    \item $\{\ova{K_{\alpha}}, \ovlra{K}_k, H\}$ with $\alpha, k \geq 2$ and $H$ a hero
    \item $\{\ovlra{K}_2, H, \ora{F}\}$ with $H$ a hero and $\ora{F}$ an oriented forest
\end{itemize}

\medskip

Note that the first case corresponds to the \Gya-Sumner conjecture, and thus is largely open. The second case is fully solved, using Ramsey's theory. 

\begin{lemma}[\cite{ACN21}]\label{lem:clique_stable_TT}
If $\mc{D} = \{\ova{K_{\alpha}}, \ovlra{K}_k, H\}$ with $\alpha, k \geq 2$ and $H$ a hero, then $\mc{D}$ is heroic if and only if $H$ is a transitive tournament.
\end{lemma}


While the third case remains largely open, Aboulker, Charbit and Naserasr proved that the following constraint on $H$ and $\ora{F}$ holds.

\begin{lemma}[\cite{ACN21}]
If $\{\ovlra{K}_2, H, \ora{F}\}$, with $H$ a hero and $\ora{F}$ an oriented forest, is heroic, then $H$ is a transitive tournament or $\ora{F}$ is an oriented star.
\end{lemma}

Their proof relies on showing that the class of digraphs defined by $H_1 = TT_1$ and $H_k = \vec C_{4}(1,H_{k-1},H_{k-1},H_{k-1})$ for $k \geq 2$ has an unbounded dichromatic number (by Lemma~\ref{thm:tournaments_unbounded}), yet does not contain any induced copy of $\ora{C_3}$ nor of any oriented path on $4$ vertices, thus proving that either $H$ is a transitive tournament, or that $\ora{F}$ is an oriented star. Aboulker, Charbit and Naserasr thus made the following conjecture:

\begin{conjecture}[\cite{ACN21}]
Minimal heroic sets of size $3$ are exactly sets of the form:
\begin{enumerate}
    \item $\{\ovlra{K}_k, \ovlra{F}, TT_2\}$ (equivalent to Gyarfas-Sumner conjecture)
    \item $\{\ovlra{K}_k, \ova{K}_{\alpha}, TT_{\ell} \}$ (proved by Lemma~\ref{lem:clique_stable_TT})
    \item $\{\ovlra{K}_2, \ora{F}, TT_{\ell}\}$
    \item $\{\ovlra{K}_2, \ora{F}_S, H\}$
\end{enumerate}
where $\alpha, k, \ell \in \mathbb{N}$, $\ovlra{F}$ is a symmetric forest, $\ora{F}$ is an oriented forest, $\ora{F_S}$ is an oriented forest of stars and $H$ is a hero.
\end{conjecture}

We disproved the last point in \cite{AACmulti} by proving that $\{\ovlra{K}_2, K_1 + TT_2, \vec C_{3}(1,2,2)\}$ is not heroic. Note that the remaining unsolved cases all concern oriented graphs. Thus in the rest of this work, for any class of digraphs $\mc{D}$, the notation $\F(\mc{D})$ will now be defined to mean $\F(\mc{D} \cup \{\ovlra{K}_2\})$, that is $\F(\mc{D})$ will mean the set of \textbf{oriented graphs} not containing any digraph in $\mc{D}$ as an induced subdigraph.



\section{Solved cases and perspectives}

\medskip

For oriented graphs, the two open cases are concerned with the dichromatic number when forbidding as an induced subdigraph either a hero and an oriented forest of stars, or a transitive tournament and an oriented forest. As transitive tournaments are heroes and stars are trees, a preliminary question is to study the behaviour of oriented graphs with no induced copy of a fixed transitive tournament and no induced copy of a fixed forest of stars. This was solved by Chudnovsky, Scott and Seymour:

\begin{theorem}[\cite{CS19}]
    Let $\ora{S}$ be a forest of star, and $k$ be an integer. $\dic(\F(\ovlra{K_2},\ora{S},TT_k))$ is finite.
\end{theorem}

Note that they proved that this result holds even when only considering colourings of the underlying graph.

\subsection{Forbidding an oriented forest and a transitive tournament}

\begin{conjecture}[\cite{ACN21}]
    Let $\ora{F}$ be an oriented forest and $k$ an integer. $\dic(\F(\ora{F}, TT_{k}))$ is finite.
\end{conjecture}

Note that, since any transitive tournament $T$ is an orientation of a complete graph, and is a subdigraph of any orientation of a complete graph on $2^{V(T)}$ vertices, the above conjecture is equivalent to the following conjecture: 

\begin{conjecture}[\cite{ACN21}]
    Let $\ora{T}$ be an oriented tree and $k$ be an integer. The class of oriented graphs with no induced copy of $\ora{T}$ and clique number at most $k$ has a bounded dichromatic number.
\end{conjecture}

As this problem is largely open, an interesting first step could be to restrict $\ora{T}$ to be an orientation of a path:

\begin{conjecture}[\cite{ACN21}]
    Let $\ora{P}$ be an oriented path, and $k$ be an integer. The class of oriented graphs with no induced copy of $\ora{P}$ and clique number at most $k$ has a bounded dichromatic number.
\end{conjecture}

In \cite{ACN21}, Aboulker, Charbit and Naserasr proved it when $\ora{F}$ is any orientation of a path on $3$ vertices. In \cite{CMPRS22}, Cook, Masar{\'{\i}}k, Pilipczuk, Reinald and Souza proved it for any orientation of a path on $4$ vertices. For larger paths, the only known result is the following one: if $F = \ora{P_6}$, we proved with Pierre Aboulker, Pierre Charbit and Stéphan Thomassé that the conjecture holds for $k = 3$. This is the object of Chapter~\ref{chpt:psix}. Outside of oriented stars and paths, the conjecture is still completely open.






%
%
%



\subsection{Forbidding an oriented forest of stars and a hero}

\begin{conjecture}[\cite{ACN21}]\label{conj:star_hero}
    Let $\ora{S}$ be an oriented forest of stars and $H$ a hero. $\dic(\F(\ora{S}, H))$ is finite.
\end{conjecture}

If $\ora{S} = S_2^+$ and $H = 1 \Ra \ora{C_3}$, we proved that $\dic(\F(\ora{S}, H))$, with Pierre Aboulker and Pierre Charbit. This is the object of Chapter~\ref{chpt:semiround}. Note that this was independently proven by Steiner in \cite{S21}. If $\ora{S} = \ora{P_3}$, we proved with Pierre Aboulker and Pierre Charbit that for any hero $H$,  $\dic(\F(\ora{S}, H))$ is finite. This is the object of Chapter~\ref{chpt:quasi_transitive}.

\medskip

With Pierre Aboulker and Pierre Charbit, we have characterized heroes $\F(K_1 + TT_2)$, which happens to be the class of oriented multipartite complete graphs, and proved that $\vec C_{3}(1,2,2)$ is not a hero in this class, thus disproving this conjecture. This is the object of chapter \ref{chpt:multipartite}. Consequently, we also disproved this conjecture for any oriented forest of stars $\ora{S}$ containing $K_1 + TT_2$, that is for any disconnected oriented forest of stars with at least one arc, as  $\F(K_1 + TT_2, \vec C_{3}(1,2,2)) \subseteq \F(K_1 + TT_2, \vec C_{3}(1,2,2))$. Conjecture~\ref{conj:star_hero} then becomes:

\begin{conjecture}\label{conj:disconnected_star}
    Let $\ora{S}$ be a disconnected oriented forest of stars with at least one arc and $H$ a hero in oriented multipartite complete graphs. $\dic(\F(\ora{S}, H))$ is finite.
\end{conjecture}

In the same paper, we also prove the following statement:
\begin{theorem}[\cite{AACmulti}]
    Let $t \in \mathbb{N}$, $H_1$ and $H_2$ two heroes in $\F(K_1 + kTT_2)$. Then $H_1 \Ra H_2$ is a hero in $\F(K_1 + kTT_2)$.
\end{theorem}

Thus, to completely characterize heroes in $\F(K_1 + kTT_2)$, we only need to prove the following:

\begin{conjecture}
    Let $t \in \mathbb{N}$ and $H$ a hero in $\F(K_1 + kTT_2)$. Then $\vec C_3(1, 1, H)$ is a hero in $\F(K_1 + kTT_2)$.
\end{conjecture}

\medskip

The only disconnected oriented forests of stars not covered by conjecture \ref{conj:disconnected_star} are the stable digraphs $\ova{K}_k$ with $k$ an integer. Harutyunyan, Le, Newman and Thomassé have solved this case \cite{HLNT19}, proving that for any hero in tournaments $H$ and any integer $k$, $\dic(\F(\ova{K}_k, H))$ is finite.

\medskip

Among the remaining open cases, the following two conjectures are particularly interesting, in that if they both turn out to be false, Conjecture \ref{conj:star_hero} would be completely settled whenever $F$ is disconnected:

\begin{conjecture}
    The class of digraphs with no induced copy of $\vec C_3$ nor $TT_2 + K_1 + K_1$ has a bounded dichromatic number.
\end{conjecture}

\begin{conjecture}
    The class of digraphs with no induced copy of $\vec C_3$ nor $TT_2 + TT_2$ has a bounded dichromatic number.
\end{conjecture}




\chapter{Heroes in quasi-transitive oriented graphs}\label{chpt:quasi_transitive}

\begin{flushright}{\slshape    
This chapter is built upon a joint \\
work with Pierre Aboulker and \\
Pierre Charbit, published in \cite{AACmulti}}. \\ \medskip
\end{flushright}

\emph{We give an easy proof that all heroes in tournaments are heroes in $\F( \ora P_3)$.} 

\section{Introduction}

We say that a digraph $G$ is  \emph{quasi-transitive} if for every triple of vertices $x,y,z$, if $xy,yz \in A(G)$, then $xz \in A(G)$ or $zx \in A(G)$ and observe that the class of quasi-transitive digraphs is precisely $\F( \ora P_3)$. 

Given two digraphs $G_1$ and $H_1$ with disjoint vertex sets, a vertex $u \in G_1$, and a digraph $G$, we say that $G$ is obtained by substituting $H_1$ for $u$ in $G_1$,
provided that the following hold: 
\begin{itemize}
    \item $V(G) = (V(G_1) \setminus u) \cup V(H_1)$,
    \item $G[V(G_1) \setminus u] = G_1 \setminus u$,
    \item $G[V(H_1)] = H_1$
    \item for all $v \in  V(G_1) \setminus u$ if $v$  sees (resp. is seen by, resp. is non-adjacent to)  $u$ in $G_1$, then $v$  sees (resp. is seen by, resp. is non-adjacent with) every vertex in $V(G_2)$ in $G$. 
\end{itemize}
Let $\mc T$ be the class of tournaments and $\mc A$ the class of acyclic  digraphs. Let $(\mc A \cup \mc T)^*$ be the closure of $\mc A \cup \mc T$ under taking substitution, that is to say digraphs in $(\mc A \cup \mc T)^*$ are the digraphs obtained from a vertex by repeatedly substituting vertices by digraphs in $\mc A \cup \mc T$.
A classic result of Bang-Jensen and Huang~\cite{BH95} (see also Proposition 8.3.5 in~\cite{BG18}), implies that quasi-transitive digraphs are all in $(\mc A \cup \mc T)^*$. 

\section{Main result}

We can now prove the main result of this chapter.

\begin{theorem}
Heroes in $(\mc A \cup \mc T)^*$ are the same as heroes in tournaments. In particular, heroes in $\F( \ora P_3)$ are the same as heroes in tournaments. 
\end{theorem}

\begin{proof}
Let $H$ be a hero in tournaments and $c$ be the maximum dichromatic number of an $H$-free tournament. 
We prove by induction on the number of vertices that $H$-free digraphs in $(\mc A \cup \mc T)^*$ are also $c$-dicolourable. 
Let $G \in \mc (\mc A \cup \mc T)^*$ on $n \geq 2$ vertices and assume that all digraphs in $\mathcal (\mc A \cup \mc T)^*$ on at most $n-1$ vertices are $c$-dicolourable. 

There exist $G_1, \dots, G_s$, $H_1, \dots, H_{s-1}$ and vertices $v_1 \dots, v_{s-1}$ such that  the $G_i$'s and the $H_i$'s are  digraphs of $\mc A \cup \mc T$ with at least two vertices, $G_1 = K_1$, $G_s = G$, $v_i \in V(G_i)$ and for $i=1, \dots s-1$, $G_{i+1} = G_{i}(v_{i} \la H_{i})$. 

If all $H_i$ are tournaments, then $G$ is a tournament and is thus $c$-dicolourable. 
So we may assume that there exists $1 \leq i \leq s-1$ such that $H_i$ is an acyclic digraph. 
Let $x_1, \dots, x_t$ be the vertices of $H_i$. 
There exist $t$ digraphs $X_1, \dots, X_t$ in $\mathcal (\mc A \cup \mc T)^*$  such that $G$ is obtained from $G_{i+1}$ by substituting $x_1$ by $X_1$, $x_2$ by $X_2$, $\dots,$ $x_t$ by $X_t$  and some vertices of $V(G_{i+1}) \setminus \{x_1, \dots, x_t\}$ by digraphs in $(\mc A \cup \mc T)^*$.  Note that the order in which  these substitutions are performed does not matter. 

Let $X = \cup_{1\leq i \leq  t} V(X_i)$. 
So $V(G) \sm X$ can be partitioned into $3$ sets $S^+$, $S^-$, $S^0$ such that for every $v \in X$, $v$ sees all vertices of $S^+$, is seen by all vertices of $S^-$ and is non-adjacent with all vertices of $S^0$. 

For $i=1, \dots, t$, let $D_i = G[G_i \setminus (X \setminus X_i)]$. 
By induction, the $D_i$'s are $c$-dicolourable. For $i=1, \dots, t$, let $\varphi_i$ be a $c$-dicolouring of $D_i$. 
Assume without loss of generality that $|\varphi_1(X_1)| \geq |\varphi_i(X_i)|$ for $1 \leq i \leq t$. 
In particular $\dic(X_i) \leq |\varphi_1(X_1)|$ for $i=1, \dots, t$. 
Extend $\varphi_1$ to a $c$-dicolouring of $D$ by dicolouring each $X_i$ (independently) with colours from $\varphi_1(X_1)$. We claim that this gives a $c$-dicolouring of $G$. 

Let $C$ be an induced directed cycle of $G$. If $C$ is included in $X$ or $V(G) \sm X$, then $C$ is not monochromatic. 
So we may assume that $C$ intersects both $V(G) \sm X$ and $X$. 
Since vertices in $X$ share the same neighbourhood outside $X$ and $C$ is induced, $C$ must intersect $X$ on exactly one vertex, and this vertex can be chosen to be any vertex of $X$. In particular, we may assume that it is in $X_1$. Hence $C$ is not monochromatic. 
\end{proof}

Note that the proof of the previous theorem actually works for the following stronger statement:
\begin{theorem}
Let $\mc C$ be a class of digraphs closed under taking substitution and let $(\mc A \cup \mc C)^*$ be the closure of $\mc A \cup \mc C$ under taking substitution. Then heroes in $(\mc A \cup \mc C)^*$ are the same as heroes in $\mc C$. 
\end{theorem}

\chapter{Heroes in oriented complete multipartite graphs} \label{chpt:multipartite}

\begin{flushright}{\slshape    
This chapter is built upon a joint \\
work with Pierre Aboulker and \\
Pierre Charbit, published in \cite{AACmulti}}. \\ \medskip
\end{flushright}

\emph{In this chapter, we completely characterize heroes in \omgs.}

\section{Introduction}
Observe that \omgs are precisely the digraphs in $\F( K_1 + TT_2)$.  
The main goal of this chapter is to identify heroes in \omgs.  
As discussed in Chapter~\ref{chpt:gyarfas}, it is conjectured in \cite{ACN21} that heroes in \omgs are the same as heroes in tournaments. We disprove this conjecture by showing  the following:  



\begin{theorem}\label{thm:ce} 
$\vec C_{3}(1,2,2)$ is not a hero in \omgs.
\end{theorem}

\medskip

We actually get a full characterization of heroes in \omgs, by proving that:
\begin{theorem}\label{thm:main_multipartite}
A digraph $H$ is a hero in \omgs  if and only if:
\begin{itemize}
    \item $H= K_1$, 
    \item $H=H_1 \Ra H_2$, where $H_1$ and $H_2$ are heroes in \omgs, or
    \item $H=\vec C_{3}(1, 1, H_1)$ where $H_1$ is a hero in \omgs.
\end{itemize}
\end{theorem}

\section{Heroes in \omgs}

Let $G$ be a digraph.
For two disjoint sets of vertices $X, Y$, we write $X \Rightarrow Y$ to say that for every $x \in X$ and for every $y \in Y$, $xy \in A(G)$, and we write $X \rightarrow Y$ to say that every arc with one end in $X$ and the other one in $Y$ is oriented from $X$ to $Y$ (but some vertices of $X$ might be non-adjacent to some vertices of $Y$). When $X=\{x\}$ we write $x \Rightarrow Y$ and $x  \rightarrow Y$. 


\subsection{Strong components}\label{subsec:strong}

The goal of this subsection is to prove the following:

\begin{theorem}\label{thm:strong}
If $H_1$ and $H_2$ are heroes in $\F( K_1 + TT_2)$, then so is $H_1 \Ra H_2$. 
\end{theorem}

We actually prove the following stronger result:

\begin{theorem}\label{thm:strongg}
Let $H_1$, $H_2$ and $F$ be digraphs such that $H_1 \Ra H_2$ is a hero in $\F( F)$ and $H_1$ and $H_2$ are heroes in $\F( K_1 + F)$. Then $H_1 \Ra H_2$ is a hero in $\F( K_1 + F)$. 
\end{theorem}

To see that Theorem~\ref{thm:strongg} implies Theorem~\ref{thm:strong}, take $F=TT_2$ and observe that $\F( TT_2)$ is the class of digraphs with no arc and thus every digraph is a hero in $\F( TT_2)$.

Note also that by taking $F=K_1$, we have that $\F( F)$ is empty and that $\F( K_1+F)$ is the class of tournaments, so Theorem~\ref{thm:strongg} yields the result of \cite{hero} (see \textbf{(3.1)}) stating that $H$ is a hero in tournaments if and only if all of its strong components are. Then, by induction, we get the same result for the class of digraphs with bounded independence number, reproving Theorem 1.4 of~\cite{HLNT19}. 
\medskip

The rest of this subsection is devoted to the proof of Theorem \ref{thm:strongg}, which is inspired but simpler (we got rid of the intricate notion of \emph{$r$-mountain}) than the analogous result for tournaments in \cite{hero}, even though our result is more general. In \cite{CSS14}, Chudnovsky, Scott and Seymour provided an alternative proof of this result for tournaments, and we figured out it could also similarly get adapted to \omgs.

We start with a few definitions and notations. First, in order to simplify statements of the lemmas, we assume $H_1$, $H_2$ and $F$ are fixed all along the subsection and are as in the statement of Theorem~\ref{thm:strongg}. So there exists constants $c$ and $h$ such that:
\begin{itemize}
\item $H_1$ and $H_2$ have at most $h$ vertices,
\item digraphs in $\F( F, H_{1} \Rightarrow H_{2})$ have dichromatic number at most $c$, 
\item for $i=1,2$, digraphs in $\F( K_1+F, H_i)$ have dichromatic number $c$. 
\end{itemize}

If $G$ is a digraph and $uv \in E$, we set $C_{uv} =v^+ \cap u^-$, that is the of vertices that form a directed triangle with $u$ and $v$. Finally, for $t\in \eN$, we say that a digraph $K$ is a \emph{$t$-cluster} if $\chi(K)\geq t$ and $|V(K)|\leq f(t)$, where $f(t)$ is the function defined recursively by $f(1)=1$ and $f(t)=1+f(t-1)(1+f(t-1))$.

The structure of the proof is very simple, we prove that digraphs in $\F( K_1+F, H_{1} \Rightarrow H_{2})$ that do not contain a $t$-cluster for some $t$ have bounded dichromatic number (Lemma~\ref{lemma:nocluster}), and then that the ones that contain a $t$-cluster for some $t$ also have bounded dichromatic number (Lemma~\ref{lemma:cluster}). 

\begin{lemma}\label{lemma:nocluster}
There exists a function $\varphi$ such that if $t$ is an integer and $G$ is a digraph in $\F( K_1+F, H_{1} \Rightarrow H_{2})$ which contains no $t$-cluster as a subgraph, then $\dic(G)\leq \varphi(c,h,t)$
\end{lemma}

\begin{proof}
We prove this by induction on $t$. For $t=1$ the result is trivial as a $1$-cluster is simply a vertex. Assume the existence of $\varphi(c,h,t-1)$, and assume $G$ is a digraph in $\F( K_1+F, H_{1} \Rightarrow H_{2})$ which contains no $t$-cluster. Say an arc $uv$ is \emph{heavy} if $C_e$ contains a $(t-1)$-cluster, and \emph{light} otherwise. For a vertex $u$ we define $h(u)=\{v\in V(G) \mid \text{ $uv$ or $vu$ is a heavy arc}\}$. 

\begin{claim}
For any vertex $u$, $h(u)$ contains no $(t-1)$-cluster. 
\end{claim}

\begin{subproof}
Assume by contradiction that $K$ is a $(t-1)$-cluster in $h(u)$. By definition of $h(u)$, for every $v \in V(K)$, there exists a $(t-1)$-cluster $K_v$ in $C_{uv}$ or $C_{vu}$ (depending on which of $uv$ or $vu$ is an arc). Let $K'=\{u\}\cup V(K) \cup (\cup_{v\in K} V(K_v))$. We claim that $K'$ is a $t$-cluster. First note that the number of vertices of $K'$ is at most $1+f(t-1) +f(t-1) \cdot f(t-1)=f(t)$. We need to prove that $K'$ is not $(t-1)$-colourable, so let us consider for contradiction a $(t-1)$-colouring of its vertices, and without loss of generality assume $u$ gets colour $1$. Because $K$ is a $(t-1)$-cluster, some vertex $v$ in $K$ must also receive colour $1$, and since $K_v$ is also a $(t-1)$-cluster, some vertex $w$ in $K_v$ must also receive colour $1$, which produces a monochromatic directed triangle. So $K'$ is indeed a $t$-cluster, a contradiction. 
\end{subproof}

\begin{claim}
For any vertex $u$, $\min(\dic(u^{-}),\dic(u^{+})) \leq (h+1) \cdot (\varphi(c,h,t-1)+c)$.
\end{claim}

\begin{subproof}
    Let $u \in V(G)$. By the previous claim and the induction hypothesis, $h(u)$ induces a digraph of dichromatic number at most $\varphi(c,h,t-1)$, so it is enough to prove that one of the sets $u^-_{\ell}:=(u^{-}\setminus h(u))$ and $u^+_{\ell}:=(u^{+} \setminus h(u))$ induces a digraph with a dichromatic number at most $h \cdot \varphi(c,h,t-1)+c\cdot (h+1)$. 
    
    If $u^+_{\ell}$ induces a $H_2$-free digraph, then it has  dichromatic number at most $c< h\cdot\varphi(c,h,t-1)+c\cdot (h+1)$, so we can assume that there exists $V_2 \subseteq u^{+}_{\ell}$ such that $G[V_2] = H_{2}$. We now partition $u^-_{\ell}$ into three sets $A, B, C$, each of which will have bounded dichromatic number.
    
    Let $A = u^{-}_{\ell} \cap (\cup_{v \in V_2} v^+) =u^{-}_{\ell} \cap (\cup_{v \in V_2} C_{uv})$.
        For every $v \in V_2$, $uv \in E$ is light (because $V_2 \subseteq u_l^-$), so $G[C_{uv} \cap A]$ does not contain a $(t-1)$-cluster and is thus $\varphi(c,h,t-1)$-colourable by induction. Now, since $H_2$ contains at most $h$ vertices, we get $\dic(A) \leq h \cdot  \varphi(c,h,t-1)$.
        
    Let $B=u^{-}_{\ell} \cap (\cup_{v \in V_2} v^0)$. Since $G$ is $(K_1+F, H_1 \Ra H_2)$-free, for every $v \in V_2$, $v^0$ is $(F, H_1 \Ra H_2)$-free and thus $\dic(G[v^0]) \leq c$. Hence, $\dic(B) \leq c\cdot h$. 
    
    Finally, consider $C=u^{-}_{\ell}\setminus (A\cup B)$. By the definition of $A$ and $B$, we get $C \Rightarrow V_2$. Since $G$ is $H_1 \Ra H_2$-free, $G[C]$ is $H_1$-free, and therefore $\dic(C) \leq c$.
    
    All together, we get $\dic(x^{-}_{\ell}) \leq h \cdot \varphi(c,h,t-1)+c \cdot (h + 1)$ as desired.
\end{subproof}

By the previous claim, we can partition the set of vertices into the two sets $V^-$ and $V^+$ defined by: 
\begin{eqnarray*}
V^- = \{ u \in V \mid \dic(u^{-}) \leq (h+1) \cdot (c+\varphi(c,h,t-1))\}\\
V^+ = \{ u \in V \mid \dic(u^{+}) \leq (h+1) \cdot (c+\varphi(c,h,t-1))\}
\end{eqnarray*}


If $G[V^-]$ is $H_1$-free and $G[V^+]$ is $H_2$-free, then $\dic(G) \leq 2c<\varphi(c,h,t) $ and we are done. Assume  that there exists $V_1 \subseteq V^-$ such that $G[V_1] = H_{1}$ (the case where $V^+$ contains an induced copy of $H_2$ is symmetrical). 

We now partition $V(G) \setminus V_1$ into three sets of vertices depending on their relation with $V_1$ and prove that each of these set induces a digraph with bounded dichromatic number.
        
        Let $A=\bigcup_{v \in V_1}v^-$. By definition of $V^-$ and since $V_1 \subseteq V^-$, for every $v \in V_1$, $v^-$ has dichromatic number at most $(h+1)(c+\varphi(c,h,t-1))$, and since $H_1$ has $h$ vertices we get that $\dic(A) \leq h \cdot (h+1) \cdot (c+\varphi(c,h,t-1))$.
        
        Let $B=\bigcup_{v \in V_1}v^0$. Since $G$ is $(K_1+F, H_1 \Ra H_2)$-free, for every $v \in V_1$, $v^0$ is $(F, H_1 \Ra H_2)$-free and thus $\dic(G[v^0]) \leq c$. Hence, $\dic(B) \leq c \cdot h$.  
        
        Finally, let $C = V(G) \setminus (A\cup B \cup V_1)$. By definition of $A$ and $B$, we have $V_1 \Ra C$, hence $C$ is $H_2$-free and thus $\dic(C) \leq c$.  

All together, we get that $\dic(G) \leq h + h \cdot (h+1) \cdot (c+\varphi(c,h,t-1)) + ch +c := \varphi(c,h,t)$.

\end{proof}

The proof of the theorem will follow from the second lemma below.

\begin{lemma}\label{lemma:cluster}
If $G\in \mC$ contains a $(3c+1)$-cluster, then $\dic(G)\leq c \cdot 2^{f(3c+1)+1}$. 
\end{lemma}

\begin{proof}
Let $K$ be a $(3c+1)$-cluster in $G$. Assume there exists a vertex $u \in V(G)$ such that $u^-\cap V(K)$ is $H_1$-free and $u^+\cap V(K)$ is $H_2$-free. Since $u^0\cap V(K)$ is by assumption $(F , H_1 \Rightarrow H_2)$-free, we get a partition of $V(K)$ into three sets that induce digraphs with a dichromatic number at most $c$, a contradiction (this still holds if $u\in K$ as we can add it to any of the sets without increasing the dichromatic number).

So, for every $u \in V(G)$, either $u^-\cap V(K)$ contains a copy of $H_1$, or $u^+\cap V(K)$ contains a copy of $H_2$. 
Now for every $V_1\subseteq V(K)$ such that $G[V_1]$ is isomorphic to $H_1$, the set of vertices $u$ such that $V_1 \subset u^-$ is $H_2$-free and therefore has a dichromatic number at most $c$. 
Similarly, for every $V_2\subset V(K)$ such that $G[V_2]$ is isomorphic to $H_2$, the set of vertices $u$ such that $V_2 \subset u^+$ is $H_1$-free and therefore has a dichromatic number at most $c$.
By doing this for every possible copy of  $H_1$ or $H_2$ inside $V(K)$ we can cover every vertex of $V(G)$. 
Moreover, the number of subsets of $V(K)$ that induces a copy of $H_1$ (resp. of $H_2$) is at most $2^{f(3c+1)}$. Hence, we get that $\dic(G)\leq c \cdot 2^{f(3c+1)+1}$. 
\end{proof}

\begin{proof}[of Theorem~\ref{thm:strongg}]
By Lemma~\ref{lemma:nocluster} and Lemma~\ref{lemma:cluster}, we get that every digraph in $\F( K_1+F, H_{1} \Rightarrow H_{2})$ has dichromatic number at most $\max(\varphi(c,h,3c+1),2^{f(3c+1)+1} c)$, which proves Theorem~\ref{thm:strongg}.
\end{proof}

\begin{remark}
Let $K(c,h)$ an integer such that digraphs in $\F( F, H_1 \Ra H_2)$ have a dichromatic number at most $K(c,h)$. From the proof above we can deduce that taking 
$$K(c,h)=\max( (2h \cdot (h+1))^{5c+1} , 2 ^ {2 ^ {2 \cdot 3^{3c + 1}} + 1} \cdot c)$$ 
works (proving as intermediate steps that for every integer $t$, we can take $f(t) \leq 2 ^ {2 \cdot 3^t}$ and $\varphi(c,h,t) \leq (2h \cdot (h+1))^{2c + t}$).
\end{remark}

\subsection{Growing a hero}\label{subsec:growing_heroes}

The goal of this subsection is to prove the following theorem:  
\begin{theorem}\label{thm:delta}
If $H$ is a hero in \omgs, then so is $\vec C_{3}(1, H, 1)$. 
\end{theorem}

The next lemma is proved in~\cite{hero} (see \textbf{(4.2)}) for tournaments but actually holds for every digraphs. 
\begin{lemma}\label{lem:backedge}
Let $G$ a digraph and let $(X_1, \dots, X_n)$ a partition of $V(G)$. Suppose that $d$ is an integer such that:
\begin{itemize}
    \item $\forall\, 1 \leq i \leq n\ \dic(X_i) \leq d$ , %
    \item $\forall\, 1 \leq i < j  \leq n$ , if there is an arc $uv$ with $u \in X_j$ and $v \in X_i$, then $ \dic(X_{i+1} \cup X_{i+2} \cup \dots \cup X_j) \leq d $
\end{itemize}
Then $\dic(G) \leq 2d$. 
\end{lemma}

\begin{proof}
Define a sequence $s_0<s_1<...<s_t=n$ defined recursively as follows: $s_0=0$ and 
$$s_{k}=\max\{ j>s_{k-1} \mid \dic(\bigcup_{s_{k-1} < i \leq j} X_i)\leq  d \}$$
For $k=1\ldots t$,  and  $Y_k=\bigcup_{s_{k-1} < i \leq s_k} X_i$. 
By definition of the sequence $s_k$, $\dic(Y_k)\leq d$ for $k=1, \dots, t$ and $\dic(Y_k\cup X_{s_k+1})>d$ for $k=1, \dots, t-1$, so by the assumption of the lemma, there cannot be an arc from $Y_j$ to $Y_i$ whenever $i \leq j-2$.  Hence, $\bigcup_{i\ even} Y_i$ and $\bigcup_{i\ odd}Y_i$ both have dichromatic number at most $d$, and thus $\dic(G) \leq 2d$. 
\end{proof}

The following is an adaptation of \textbf{(4.4)} in~\cite{hero} with \omgs instead of tournaments  (note also that their proof is concerned with $\vec C_{3}(1, k, H)$ while ours is concerned with $\vec C_{3}(1, 1, H)$).

\begin{lemma}\label{lem:N+N-backedge}
Let $G$ be a $\vec C_{3}(1, 1, H)$-free  \omg given with a partition $(X_1, \dots, X_n)$ of its vertex set $V(G)$. Suppose that $c$ is an integer such that: 
\begin{itemize}
    
   \item $H$-free \omgs have dichromatic number at most $r$, 
   \item $\forall\, 1 \leq i \leq n\ \dic(X_i) \leq r$,
    \item $\forall\, 1 \leq i\leq n \ \forall v \in X_i \ \dic(v^+ \cap (X_1 \cup \dots \cup X_{i-1})) \leq r$,
    \item $\forall\, 1 \leq i\leq n \ \forall v \in X_i \ \dic(v^- \cap (X_{i+1} \cup \dots \cup X_{n})) \leq r$.
\end{itemize}
Then $\dic(G)\leq 8r+4$. 
\end{lemma}

\begin{proof} 
We are going to prove that $G$ satisfies the hypothesis of Lemma~\ref{lem:backedge} with $d = 4r+2$, which implies the result.   
Let $uv$ be an edge such that $u \in X_j$ and $v \in X_i$ where $1 \leq i < j \leq n$.
We want to prove that $ \dic(X_{i+1} \cup X_{i+2} \cup \dots \cup X_j) \leq 4r+2 $. 
Let $W = X_{i+1} \cup \dots \cup X_{j-1}$. Let $Q = v^+ \cap u^- \cap W$.  If $Q$ contains a copy of $H$, then together with $u$ and $v$ it forms a $\vec C_{3}(1, H, 1)$, a contradiction. So $Q$ is $H$-free and thus is $r$-colourable.
Now, each vertex in $W \setminus Q$ is in $u^+ \cup v^- \cup u^o \cup v^o$. By hypothesis, $\dic(v^+ \cap W)$ and $\dic( v^- \cap W)$ are both $r$-colourable, and since $G$ is an \omg, $u^o$ and $v^o$ are stable sets. Finally, by hypothesis, $\dic(X_j) \leq r$. All together, we get that $\dic( X_{i+1} \cup \dots \cup X_j) \leq 4r+2$ as announced. 

\end{proof}

\begin{proof}[of Theorem~\ref{thm:delta}]
Let $H$ be a hero in \omgs and let $h = |V(H)|$. By applying Theorem \ref{thm:strong} with $H_1 = H_2 = H$, we get that $H \Ra H$ is a hero in \omgs. 
So there exists a constant $c$ such that every $(H \Ra H)$-free \omgs have a dichromatic number at most $c$. Note that it also implies that every $H$-free  \omgs have a dichromatic number at most $c$. 

Let $G$ be a $\vec C_{3}(1, 1, H)$-free \omg. 
Set $r = 12c \cdot h^2 + 4c \cdot h + 3c + 18h$. 
We are going to prove that $\dic(G) \leq 8r + 4$ using Lemma~\ref{lem:backedge}

We say that $J\subseteq V(G)$ is a \emph{$H$-jewel} if $G[J]$ is isomorphic to $H\Ra H$. 
The important feature about an $H$-jewel $J$ in an \omg is that, for any vertex $x$ not in $J$, either $x^+ \cap J$ or $x^- \cap J$ contains a copy of $H$, or $x$ has both an in- and an out-neighbour in $J$. 
A \emph{$H$-jewel-chain} of length $n$ is a sequence $(J_1, \dots, J_n)$ of pairwise disjoint $H$-jewels such that for  $i=1, \dots, n-1$, $J_i \Rightarrow J_{i+1}$ , and for  every $1 \leq i < j \leq n$, $J_i \rightarrow J_j$. 
Both notions of $H$-jewel and $H$-jewel-chain exist in~\cite{hero}, the ones we give here are slightly different but are morally similar.

Consider a $H$-jewel-chain $(J_1, \dots, J_n)$ of maximum length $n$. 
Set $J = J_1 \cup \dots \cup J_n$ and $W=V(G) - J$. To simplify statements, we also consider sets $J_i$ for $i\leq 0$ and $i \geq n+1$, which are assumed to be empty. 

The easy but key properties of an $H$-jewel-chain are stated in the following claim.

\begin{claim}\label{claim:jewel-chain}
For every $w\in W$ and $1\leq j \leq n$:
\begin{itemize}
\item $w^+\cap J_j \neq\emptyset \, \Ra \,  w^+\cap J_{j+1} \neq\emptyset$ 
\item $w^-\cap J_j \neq\emptyset\, \Ra \, w^-\cap J_{j-1} \neq\emptyset$. 
\end{itemize}

\end{claim}
\begin{subproof}
Assume $w^+\cap J_j \neq\emptyset$. Then since $J_j\Ra J_{j+1}$, it is not possible that $G[w^-\cap J_{j+1}]$ contains a copy of $H$ for it would create a $\vec C_{3}(1, H, 1)$. Since $G[J_{j+1}]$ is isomorphic to $H\Ra H$, and since $w$ cannot have a non-neighbour in both copies of $H$ (because $G$ is an \omg), this implies that $w$ has at least one out-neighbour in $J_{j+1}$. The proof of the second item is identical up to the reversal of the arcs.
\end{subproof}

For every $w \in W$, let $c(w)$ be the smallest integer $j$ such that $w^+\cap J_j \neq\emptyset$ if such an integer exists, and $c(w)=n+1$ if no such integer exists. 
For $j=1, \dots, n+1$, set $W_j = \{w:  c(w)=j\}$ and $X_j = J_j \cup W_j$. 
Note that, by definition of the $W_j$'s, if $w \in W_j$, then $J_i \ra w$ for every $i \leq j-1$. 
\vspace{0.3cm}

\begin{claim}\label{claim:X}
$\dic(X_j) \leq  4c \cdot h^2+ c + 6h $ for $j=1, \dots, n+1$. 
\end{claim}

\begin{subproof}
Let $1 \leq j \leq n+1$.
We have $\dic(J_j) \leq |J_j| \leq  2h$.  

 For each pair of vertices $a \in J_{j}$ and $b \in J_{j+1}$, 
set $A_{ab} = \{w \in W_j: bw, \, wa \in A(G)\}$. Since $ab \in A(G)$ (because $J_j \Rightarrow J_{j+1}$), and $G$ is $\vec C_{3}(1, H, 1)$-free, $A_{ab}$ must be $H$-free and thus is $c$-colourable for every choice of $a$ and $b$. 
Setting $A = \bigcup_{a,b \in J_j\times J_{j+1}} A_{ab}$, we get that $\dic(A) \leq 4 h^2 \cdot c$. Moreover, since every vertex in $W_j$ has an out-neighbour in $J_j$, we have $A = \{w \in W_j \mid w^- \cap J_{j+1} \neq \emptyset \}$

Let $B = \{w \in W_j: w^o \cap J_{j-1} \neq \emptyset \text{ or }w^o \cap J_{j+1} \neq \emptyset \}$, in other words $B$ is the set of vertices in $W_j$ with at least one non-neighbour in $J_{j-1}$ or $J_{j+1}$. Since $G$ is an \omg, we have $\dic(B) \leq |J_{j-1}| + |J_{j+1}| \leq  4h$. 

Let $C = W_j \setminus (A \cup B)$. By definition of $W_j$, for every $i \leq j-1$, $J_i \rightarrow C$. 
Since $C$ is disjoint from $A$, we have $C \ra J_{j+1}$, and thus, by claim~\ref{claim:jewel-chain} (second bullet), we have $C \ra J_k$ for every $k \geq j+1$. 
Finally, since $C$ is disjoint from $B$, we have furthermore $J_{j-1} \Rightarrow C$ and $C \Rightarrow J_{j+1}$.  
Now, if $C$ contains a $H$-jewel-chain $(J'_1,J'_2)$ of length $2$, then $(J_1, \dots, J_{j-1}, J'_1, J'_2, J_{j+1}, \dots, J_n)$ is a $H$-jewel-chain of size $n+1$, contradicting the maximality of $n$. Hence, $C$ does not contain a jewel chain of size $2$ and thus $\dic(C)\leq c$.

All together, we get that $\dic(X_j) \leq 4c \cdot h^2+ c + 6h$. 
\end{subproof}

\begin{claim}\label{clm:bounded_dichromatic_neighbours_J}
For $j=1, \dots, n$ and for every $u \in J_j$, 
\begin{itemize}
    \item $ \dic \big(u^+ \cap (X_1 \cup \dots \cup X_{j-1}) \big) \leq 4c \cdot h^2+  2c \cdot h + c + 6h $, and
    \item $ u^- \cap (X_{j+1} \cup \dots \cup X_{n+1}) = \emptyset $
\end{itemize}
\end{claim}

\begin{subproof} 
Let $1 \leq j \leq n$ and let $u \in J_j$. 
We first prove the first  bullet.  
By definition of an $H$-jewel-chain, $u$ has no out-neighbour in any $J_i$ for $i \leq j-1$ and by Claim \ref{claim:X}, $\dic(X_{j-1}) \leq  4c \cdot  h^2+c+ 6h$. 
So it is enough to prove that $A =  u^+ \cap (W_1 \cup \dots \cup W_{j-2})$ has a dichromatic number at most $2c \cdot h$. 
By Claim \ref{claim:jewel-chain}, every vertex of $W_1 \cup \dots \cup W_{j-2}$ has an out-neighbour in $J_{j-1}$.  
Moreover, for every $v \in J_{j-1}$, we have $vu \in A(G)$ (because $J_{j-1} \Ra J_j$)  and $v^- \cap A$ is $H$-free, for otherwise a copy of $H$ in $v^- \cap A$ would form, together with $v$ and $u$, a $\vec C_{3}(1, H, 1)$. So $\dic(A) \leq |J_j|\cdot c = 2c \cdot h$ as needed. 

To prove the second bullet, observe that for every $k \geq j+1$, since $J$ is a jewel chain, $u$ has no in-neighbour in $J_k$ and by definition of $W_k$, $u$ has no in-neighbour in $W_k$. 
\end{subproof}

\begin{claim}\label{clm:bounded_dichromatic_neighbours_W}
For $j=1, \dots, n +1$ and for every $w \in W_j$, 
\begin{itemize}
    \item $ \dic\big(w^+ \cap (X_1 \cup \dots \cup 
X_{j-1}) \big) < 8c \cdot h^2+  2c \cdot h + 2c + 12h $, and
\item $ \dic\big(w^- \cap (X_{j+1} \cup \dots \cup X_{n+1}) \big) \leq 8c \cdot h^2 + 2c + 12h$
\end{itemize}
\end{claim}

\begin{subproof}
Let $1 \leq j \leq n+1$ and let $w \in W_j$. 

We first prove the first  bullet. 
By definition of $W_j$, $w$ has no out-neighbour in any of the $J_i$ for $i\leq j-1$ and by Claim \ref{claim:X} $\dic(W_{j-2} \cup W_{j-1}) \leq  8c \cdot h^2 + 2c + 12h $. So it is enough to prove that $A= w^+ \cap \big(W_1 \cup \dots \cup W_{j-3}  \big)$ has  dichromatic number at most $2c \cdot h$. 
Again by definition of $W_j$ we have $J_{j-2}\rightarrow w$ and $J_{j-1} \rightarrow w$, and since $J_{j-2} \cup J_{j-1}$ induces a tournament and $G$ is $(K_1 + TT_2)$-free, $w$ has at most one non-neighbour in $J_{j-2} \cup J_{j-1}$. 
So there exists $s \in \{j-2, j-1\}$ such that $J_s \Rightarrow w$. 
For every $v \in J_s$, if $v^- \cap A$ contains a copy of $H$, then it would form, together with $v$ and $w$, a $\vec C_{3}(1, 1, H)$, a contradiction. So, for every $v \in J_s$, $v^- \cap A$ is $H$-free and is thus $c$-colourable. Finally, by claim \ref{claim:jewel-chain} every vertex in $A$ has an out-neighbour in $J_s$. So we get that $\dic(A) \leq 2c \cdot h$. 

We now prove the second bullet. 
If $j \geq n-1$, then by claim~\ref{claim:X} $\dic(X_n \cup X_{n+1}) \leq 8c \cdot h^2 + 2c + 12h$ and we are done. So we may assume that $j \leq n-2$
By claim~\ref{claim:X}, $\dic(X_{j+1}) \leq 4c \cdot h^2+ 8h + c$, so we may assume that $j\leq n-2$. 
Set $B = w^- \cap \big(X_{j+2} \cup \dots \cup X_{n+1} \big)$. 
By Claim \ref{claim:jewel-chain}, $w$ has an out-neighbour $v \in J_{j+1}$. For $i \geq j+2$, by definition of an $H$-jewel-chain, $v \rightarrow J_i$ and by definition of $W_i$, $v \rightarrow W_i$. So $v \rightarrow B$ and since $G$ is an \omg $B\setminus (v^+\cap B)$ is a stable set. Now, $v^+ \cap B$ is $H$-free, as otherwise $G$ would contain an $\vec C_{3}(1, H, 1)$. So $v^+ \cap B$ is $c$-colourable and thus $\dic(B) \leq c +1$ and thus $ \dic\big(w^- \cap (X_{j+1} \cup \dots \cup X_{n+1}) \big) \leq \dic(X_{j+1}) + c + 1 \leq 4c \cdot h^2 + 3c + 6h + 1$  by claim~\ref{claim:X}. 
\end{subproof}

By Claims~\ref{claim:X}, \ref{clm:bounded_dichromatic_neighbours_J} and ~\ref{clm:bounded_dichromatic_neighbours_W}, we can apply Lemma~\ref{lem:backedge} with $r = 12c \cdot h^2 + 4c \cdot h + 3c + 18h$ to get $\dic(G) \leq 8r+4$. 
\end{proof}


\section{$\vec C_{3}(1, 2, 2)$ is not a hero in \omgs}\label{sec:122}

In~\cite{ARU18} Axenovich et al.\ tried to characterize patterns that must appear in every ordering of the vertices of graphs with a large chromatic number. 
An (undirected) graph $G$ is (what we call) \emph{non-interlaced} if there exists an ordering $(x_1, \dots, x_n)$ on its vertices such that for every $i_1<i_2<i_3<i_4<i_5$, $\{x_{i_1}x_{i_3}, x_{i_3}x_{i_5}, x_{i_2}x_{i_4}\} \subsetneq E(G)$. See Figure~\ref{fig:non-interlaced}. They left as an open question whether non-interlaced graphs have bounded chromatic numbers or not.  In a personal communication, Bartosz Walczak gave us proof of the following result:

\begin{theorem}\label{thm:interlaced}
    The class of non-interlaced graphs has an unbounded chromatic number.
\end{theorem}

The goal of this section is to deduce from this result that $\vec C_{3}(1, 2, 2)$ is not a hero in \omgs. See Theorem~\ref{thm:delta122_not_hero}. 

 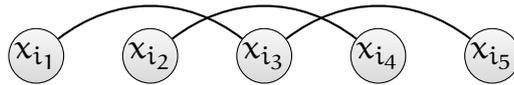
\begin{figure}[ht]
        \begin{center}
        \begin{tikzpicture}[scale=1.5]

        \node (v1) at (1,0) [vertex] {$x_{i_1}$};
        \node (v2) at (2,0) [vertex] {$x_{i_2}$};
        \node (v3) at (3,0) [vertex] {$x_{i_3}$};
        \node (v4) at (4,0) [vertex] {$x_{i_4}$};
        \node (v5) at (5,0) [vertex] {$x_{i_5}$};
        
        \draw[edge, bend right=38, black] (v3) to (v1);
        \draw[edge, bend right=38, black] (v5) to (v3);
        \draw[edge, bend right=38, black] (v4) to (v2);

        \end{tikzpicture}
        \end{center}
        \caption{\label{fig:non-interlaced} A graph is non-interlaced if there is an ordering of its vertices that avoids the above pattern as a subgraph.}
    \end{figure}

Given an \omg D together with an ordering $(V_1, \dots, V_n)$   on its parts, the arcs going from $V_i$ to $V_j$ are called  \emph{forward arcs} if $i<j$, and \emph{backward arcs} otherwise. Moreover, given $i<j$, we say that $u<v$ for every $u \in V_i$ and every $v \in V_j$. 
Finally, we say that an \omg $D$ is \emph{flat} if it admits an ordering $(V_1, \dots, V_n)$ on its parts such that for every vertex $v$ of $D$, the backward arcs going out from (resp. going in) $v$ are included in a single part of $D$. Such an ordering is called a \emph{flat ordering}.


\begin{figure}
\centering
  \centering
  \begin{tikzpicture}[scale=1.2] 
         \node (v3) at (2,2.25) [vertex] {$v_3$};
        \node (v1) at (0.5,1) [vertex] {$v_1$};
        \node (v4) at (1.25,0) [vertex] {$v_4$};
        \node (v2) at (2.75,0) [vertex] {$v_2$};
        \node (v5) at (3.5,1) [vertex] {$v_5$};
        
        \draw[arc] (v3) to (v1);
        \draw[arc] (v3) to (v4);
        
        \draw[arc] (v1) to (v5);
        \draw[arc] (v1) to (v2);
        \draw[arc] (v4) to (v2);
        \draw[arc] (v4) to (v5);
        
        \draw[arc] (v5) to (v3);
        \draw[arc] (v2) to (v3);
        
        \draw[arc] (v1) to (v4);
        \draw[arc] (v2) to (v5);
  \end{tikzpicture}
  \caption{$\vec C_{3}(1, 2, 2)$}
  \label{fig:122}
\end{figure}
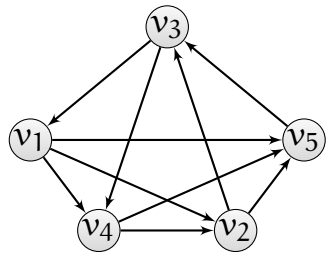
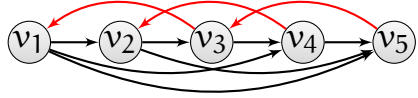
\begin{figure}
  \centering
  \begin{tikzpicture}[scale=1.2] 
          \node (v1) at (1,0) [vertex] {$v_1$};
        \node (v2) at (2,0) [vertex] {$v_2$};
        \node (v3) at (3,0) [vertex] {$v_3$};
        \node (v4) at (4,0) [vertex] {$v_4$};
        \node (v5) at (5,0) [vertex] {$v_5$};
        
        \draw[arc, bend right=38, red] (v3) to (v1);
        \draw[arc, bend right=38, red] (v5) to (v3);
        \draw[arc, bend right=38, red] (v4) to (v2);
        
        \draw[arc] (v3) to (v4);
        
        \draw[arc, bend left = -25] (v1) to (v5);
        \draw[arc] (v1) to (v2);
        \draw[arc] (v4) to (v5);
        
        \draw[arc] (v2) to (v3);
        
        \draw[arc, bend left = -20] (v1) to (v4);
        \draw[arc, bend left = -20] (v2) to (v5);
   \end{tikzpicture}
  \caption{A drawing of $\vec C_{3}(1, 2, 2)$ where the backward arcs (coloured in red) induce the forbidden pattern of non-interlaced graphs.}
  \label{subfig:122_line}
\end{figure}


\begin{lemma}\label{lem:flat_interlaced}
Let $D$ be an \omg with parts $V_1, \dots, V_n$ where $(V_1, \dots, V_n)$ is a flat ordering. If $D$ contains a copy of $\vec C_{3}(1, 2, 2)$, naming its vertices as in Figure~\ref{fig:122}, we must have $v_1 < v_2 < v_3 < v_4 < v_5$. 
\end{lemma}

\begin{proof}
    Suppose that $D$ contains a copy of $\vec C_{3}(1, 2, 2)$ and name its vertices as in Figure~\ref{fig:122}. 
    Since $\vec C_{3}(1, 2, 2)$ is a tournament, $v_i$'s are contained in pairwise distinct parts of $D$, and thus are totally ordered. 
    Since $(V_1, \dots, V_n)$ is a flat ordering, the smallest vertex among $\{v_1, v_2, v_3, v_4, v_5\}$ must have in-degree at most $1$ in $\vec C_{3}(1,2,2)$, and hence must be $v_1$. 
    Similarly, since $v_5$ is the only vertex with out-degree $1$ in $\vec C_{3}(1, 2, 2)$, $v_5$ must be the largest of the $v_i$.  
    If $v_3 < v_2$, then $v_3 < v_2 < v_5$ and the arcs $v_2v_3$ and $v_5v_3$ contradicts the fact that $(V_1, \dots, V_n)$ is a flat ordering, so $v_2 < v_3$. Similarly, if $v_4 < v_3$, then $v_4 < v_3 < v_5$ and the arcs $v_3v_4$ and $v_5v_3$ contradicts the fact that $(V_1, \dots, V_n)$ is a flat ordering, so  $v_3 < v_4$ and thus  $v_1 < v_2 < v_3 < v_4 < v_5$. 
\end{proof}

\begin{theorem}\label{thm:interlaced_hero}
If $\vec C_{3}(1, 2, 2)$ is a hero in \omgs, then every non-interlaced graph has bounded chromatic number.
\end{theorem}

\begin{proof}

Assume that $\vec C_{3}(1, 2, 2)$ is a hero in \omgs. 
Let $\mc F$ be the class of flat $\vec C_{3}(1, 2, 2)$-free \omgs. 
Since $\vec C_{3}(1, 2, 2)$ is a hero in \omgs, there exists a constant $r$ such that every digraph in $\mc F$ has a dichromatic number at most $r$. Let $R \in \mc F$ such that $\dic(R) = r$ and recall that $R$ has a flat ordering. We are going to prove that every non-interlaced graph has a chromatic number at most $2^{2^r}$.

Let $G$ be a non-interlaced (undirected) graph and $(x_1, \dots, x_n)$ the ordering on $V(G)$ given by the definition of non-interlaced graphs (that is an ordering that avoids the pattern in Figure~\ref{fig:non-interlaced}). We construct an \omg  $D'(G)$ as follow. 
 For each $x_i$, we create a stable set $V_i$ in $D'(G)$ of size $n^2$ and we assume the vertices of $V_i$ are organised as an $n \times n$ matrix. The $V_i$ are the parts of $D'(G)$.  Let us now explain how we orient the arcs.  
Given $i<j$, if $x_ix_j \in E(G)$, we orient the arcs from each vertex of the $i^{th}$ line of $V_j$ to each vertex of the $j^{th}$ column of $V_i$.  Every other arc is oriented from $V_j$ to $V_i$. This completes the construction of $D'(G)$. 

Let $v \in V_i$ and assume $v$ is in the $j^{th}$ line and the $k^{th}$ column of $V_i$. 
Then either $x_jx_i \notin E(G)$ and no backward arcs go out from $v$, or $x_jx_i \in E(G)$ and all backward arcs going out from $v$ are included in $V_j$ (more precisely, they goes from $v$ to the vertices of the $j^{th}$ column of $V_j$). Similarly, either $x_kx_i \notin E(G)$ and no backward arc goes in $v$, or $x_ix_k \in E(G)$ and all backward arcs going in $v$ are included in $V_k$ (more precisely, they goes from the $i^{th}$ line of $V_k$ to $v$). Hence, $D'(G)$ is flat and $(V_1,  \dots, V_n)$ is a flat ordering of $D'(G)$.

We now construct another \omg $D(G)$ from $D'(G)$ by introducing, for $j=1, \dots, n-1$, a copy of $R$ between $V_j$ and $V_{j+1}$ that is seen by all vertices in $\cup_{i\leq j}V_j$ and sees all vertices in $\cup_{k \geq i+1}V_k$. This completes the construction of $D(G)$. 

It is clear that $D(G)$ is an \omg and by inserting the flat ordering of each copy of $R$ between each consecutive $V_j$, we get a natural ordering of the parts of  $D(G)$. 
In the rest of the proof, we speak about backward and forward arcs of $D(G)$ with respect to this ordering. 

We are going to prove that $D(G) \in \mc F$ (so $\dic(D(G)) \leq r$) and that $\chi(G) \leq 2^{2^{\dic(D(G))}}$, which together implies the result. 

In order to help in our analysis, we will say that the vertices of $D(G)$ that comes from $D'(G)$ are green. 

The following claim is straightforward by construction. 
\begin{claim}\label{clm:backward_flat}
If $uv$ is a backward arc of $D(G)$, then either both $u$ and $v$ are green, or $u$ and $v$ are both contained in one of the copies of $R$.
\end{claim}

\begin{claim}\label{clm:122-interlaced}
If $v_1, v_2, v_3, v_4, v_5$ are vertices of $D(G)$ such that $v_1 < v_2 < v_3 < v_4 < v_5$, then $\{v_3v_1, v_5v_3, v_4v_2\} \subsetneq A(D(G))$. 
\end{claim}
\begin{subproof}
For otherwise $\{x_1x_3, x_3x_5,x_2x_4\} \subseteq E(G)$, a contradiction. 
\end{subproof}

Let us first prove that $D(G) \in \mc F$. 
By claim~\ref{clm:backward_flat}, $D(G)$ is flat and the ordering we consider is a flat ordering. Assume that $D(G)$ contains a copy of $\vec C_{3}(1, 2, 2)$ and name its vertices as in Figure~\ref{fig:122}. 
By Lemma~\ref{lem:flat_interlaced}, we have that the $v_i$ are in pairwise distinct parts of $D(G)$ and $v_1 < v_2 < v_3 < v_4 < v_5$.  
If $v_{3}$ is in a copy of $R$, since $v_3v_{1}$ and $v_{5}v_3$ are backward arcs of $D(G)$, we get by claim~\ref{clm:backward_flat} that $v_1$ and $v_5$ are in the same copy of $R$ as $v_3$. By construction, since $v_1<v_2<v_3<v_4<v_5$, we get that $v_2$ and $v_4$ are also in this same copy of $R$, a contradiction with the fact that $R$ is $\vec C_{3}(1, 2, 2)$-free. 
So we may assume that $v_{3}$ is green, and so are $v_{1}$ and $v_{5}$ by claim~\ref{clm:backward_flat}. Now, if $v_2$ is in a copy of $R$, then by claim~\ref{clm:backward_flat} $v_4$ is in the same copy of $R$, and since $v_2 < v_3 < v_4$,   $v_3$ must be in that same copy of $R$, a contradiction with the fact that $v_3$ is green.
Hence, $v_2$ is green and by claim \ref{clm:backward_flat} so is $v_4$. Thus, every $v_i$ is green, a contradiction to claim~\ref{clm:122-interlaced}.
This proves that $D(G) \in \mc F$. 
\smallskip

Since  $D(G)$ contains copies of $R$, it has dichromatic number at least $r$, and since $D(G) \in \mc F$, we get that $\dic(D(G)) = r$.  
 Consider a dicolouring $\ora \varphi$ of $D(G)$ with $r$ colours 
We define a colouring $\varphi$ of $V(G)$ from $\ora \varphi$ as follows: for $i=1, \dots, n$, $\varphi(v_i)$ is the set of sets of colours used by each line of $V_i$. 
This gives us a colouring of $V(G)$ with at most $2^{2^r}$ colours. Let us prove that it is a proper colouring of $G$ that is, each colour class is an independent set. 

Assume for contradiction that there exists  $x_ix_j \in E(G)$ such that $\varphi(x_i) = \varphi(x_j)$ and assume without loss of generality that $i<j$. Let us first prove that $D(G)$ has a monochromatic backward arc. 
Consider the set of colours used in the $i^{th}$ line of $V_j$. The same set of colours is used by the vertices of some line of $V_i$, say the $k^{th}$. Now, the $j^{th}$ vertex of the $k^{th}$ line of $V_i$ is seen by every vertex of the $i^{th}$ line of $V_j$, which implies the existence of a monochromatic backward arc as announced.   
Let $uv$ be this monochromatic backward arc, with $v \in V_i$ and $u \in V_j$. Since $i<j$, there is a copy of $R$ between $V_i$ and $V_j$. Since $\dic(R) = r$, one of the vertex $x$ of $R$ is coloured with $\ora \varphi(u)$. By construction of $D(G)$, $ux$ and $xv$ are arcs of $D(G)$ and thus $\{u,x, v\}$ induces a monochromatic directed triangle, a contradiction.   

\end{proof}

This, along with Theorem~\ref{thm:interlaced} implies:

\begin{corollary}\label{thm:delta122_not_hero}
$\vec C_{3}(1,2,2)$ is not a hero in \omgs.
\end{corollary}

\section{An \omg of large dichromatic number}\label{sec:counter_exemple}

The goal of this section is to prove that $\vec C_{3}(1,2, \ora{C_3})$, $\vec C_{3}(1, \ora{C_3}, 2)$, $\vec C_{3}(1,2, 3)$ and $\vec C_{3}(1,3, 2)$ are not heroes in \omgs. This is a direct consequence of Corollary~\ref{thm:delta122_not_hero}, but we give here an ad hoc construction.
Since  reversing all arcs of a $\vec C_{3}(1,2, \ora{C_3})$-free \omg results in a $\vec C_{3}(1, \ora{C_3}, 2)$-free \omg and does not change the dichromatic number, if $\vec C_{3}(1,2, \ora{C_3})$ is not a hero in \omgs then $\vec C_{3}(1,\ora{C_3}, 2)$ is not either. 
Similarly, if $\vec C_{3}(1, 2, 3)$ is not a hero in \omgs then  $\vec C_{3}(1, 3, 2)$ is not either. 
Hence, it is enough to prove that  $\vec C_{3}(1,2, \ora{C_3})$ nor $\vec C_{3}(1,2, 3)$ are heroes in \omgs. This is implied by the existence of  $\{\vec C_{3}(1, 2, \ora{C_3}), \vec C_{3}(1,2, 3)\}$-free \omgs with arbitrarily large dichromatic number. The rest of this section is dedicated to the description of such digraphs. \smallskip

A \emph{feedback arc set} of a given digraph $G$ is a set of arcs $F$ of $G$ such that their deletion from $G$ yields an acyclic digraph. 
The idea of the construction comes from the fact that a feedback arc set of $\vec C_{3}(1,2, \ora{C_3})$ or of $\vec C_{3}(1,2,3)$ must induce a digraph with at least one vertex of in- or out-degree at least $2$. We then describe an \omg with large dichromatic number in which every subtournament has a feedback arc set inducing disjoint directed paths, implying that it does not contain $\vec C_{3}(1,2, \ora{C_3})$ nor $\vec C_{3}(1,2,3)$ by the fact above. 

\smallskip
Let $G$ be a digraph. We denote by $\chi(G)$ the chromatic number of the underlying graph of $G$. 
The (undirected) \emph{line graph} of $G$ is denoted by $L(G)$ and defined as follows: its vertex set is $A(G)$, and two of its vertices $ab, cd \in A(G)$ are adjacent if and only if $b=c$. 

Be aware that the next lemma deals with the chromatic number and not dichromatic number. We think it appears for the first time in~\cite{EH66}.

\begin{lemma}\cite{EH66}\label{lem:line_graph}
For every digraph $G$, we have $\chi(L(G)) \geq \log(\chi(G))$. 
\end{lemma}

\begin{proof}
Let $G$ be a digraph and assume $L(G)$ admits a $k$-colouring.
Observe that a colouring of $L(G)$ is the same as a colouring of the arcs of $G$ in such a way that no $\ora P_3$ is monochromatic. 
Consider the following colouring of $G$: for each $v \in V(G)$, colour $v$ with the set of colours received be the arcs entering in $v$. This is a $2^k$-colouring of $G$ because the colouring of $A(G)$ does not have monochromatic $\ora P_3$. 
\end{proof}

Let $s \geq 3$ be an integer and let us describe the graph $L(L(TT_s))$. Assuming the vertices of $TT_s$ are numbered $v_1, \dots, v_s$ in the topological ordering (that is, for all $1 \le i<j \le s$, we have $v_iv_j \in A(T)$), for any $i<j<k$, $\{v_i,v_j,v_k\}$ induces a $\ora P_3$ in $TT_s$. This way, we get a natural name for the vertices of $L(L(TT_s))$, namely $V(L(L(TT_s))) =  \{(v_i,v_j,v_k) \mid \text{ for every } i <j <k\}$. Moreover, edges of $L(L(TT_s))$ are of the form $(v_i,v_j,v_k)(v_j,v_k,v_{\ell})$ for  every $i<j<k<\ell$. For $2 \leq j \leq s-1$, set $V_j=\{(v_i,v_j,v_k)\}: i<j<k\}$. So $V_j$'s partition the vertices of  $L(L(TT_s))$ into stable sets.

We now define the digraph $D_s$ from $L(L(TT_s))$ as follows. 
The vertices of $D_s$ are the same as the vertices of  $L(L(TT_s))$ and  
$D_s$ is an \omg  with parts $(V_2, V_3, \dots, V_{s-1})$ and we orient the arcs as follow: given $j<k$, the edges of $L(L(TT_s))$ are oriented from $V_j$ to $V_k$ and all the other arcs are oriented from $V_k$ to $V_j$. This completes the description of $D_s$. 

The arcs $v_iv_j$ such that $i<j$ are called the \emph{forward arcs} of $D_s$, and the other arcs the \emph{backward arcs} of $D_s$. Observe that the underlying graph of the graphs induced by the forward arcs of $D_s$ is $L(L(TT_s))$. 

The following remark is the crucial feature of $D_s$. 
\begin{remark}\label{rmk:Ds}
Given a vertex $(v_i,v_j,v_k)$ of $D_s$, the forwards arcs going out $(v_i,v_j,v_k)$ are included in $V_k$ and the forward arcs going in $(v_i,v_j,v_k)$ are included in $V_i$. 
\end{remark}

An \emph{out-star} (resp. \emph{in-star}) is a connected digraph made of one vertex of in-degree $0$ (resp. of out-degree $0$) and vertices of in-degree $1$ (resp. out-degree $1$). Observe that a digraph that does not contain $\ora P_3$ as a subgraph is a disjoint union of in- and out-stars. 

\begin{lemma}\label{lem:Ds_dic}
For every integer $s$, $\dic(D_s) \geq \frac{1}{2}\log(\log(s))$. 
\end{lemma}

\begin{proof}
Let $V_2, \dots, V_{s-1}$ be the partition of $D_s$ as in the definition. Recall that $V(D_s) = \{(v_i,v_j,v_k): 1\leq i<j<k\leq s\}$. 
Denote by $F_s$ the digraph induced by the forward arcs of $D_s$. So the underlying graph of $F_s$ is $L(L(TT_s))$ and by Lemma~\ref{lem:line_graph}, $\chi(F_s) \geq log(log(s))$. 

Let $R$ be an acyclic induced subgraph of $D_s$. 
Observe that a directed path on $3$ vertices in $D_s$ using only arcs in $F_s$ must be of the form $(v_{i_1},v_{i_2},v_{i_3}) \rightarrow (v_{i_2},v_{i_3},v_{i_4}) \rightarrow (v_{i_3},v_{i_4},v_{i_5})$  
where $1\leq i_1<i_2<i_3<i_4<i_5 \leq s$ and is thus contained in a directed triangle of $D_s$ (because $(v_{i_1},v_{i_2},v_{i_3})(v_{i_3},v_{i_4},v_{i_5})$ is not an edge of $L(L(TT_s))$, and thus is not an arc of $F_s$, and thus $(v_{i_3},v_{i_4},v_{i_5})(v_{i_1},v_{i_2},v_{i_3})$ is an arc of $D_s$).  
Hence, $A(R) \cap A(F_s)$ does not contain $\ora P_3$ as a subgraph and is thus a disjoint union of  out- and in-stars. So  $A(R) \cap A(F_s)$ can be partitioned into two stable sets of $F_s$.  
Hence, a $t$-dicolouring of $D_s$ implies a $2t$-(undirected) colouring of $F_s$. As we have that $\chi(F_s) \geq log(log(s))$, the result follows. 
\end{proof}

\begin{lemma}\label{lem:feedback_edge_set}
If $T$ is a tournament contained in $D_s$, then $T$ has a feedback arc set formed by disjoint union of directed paths. 
\end{lemma}

\begin{proof}
Let $T$ be a subgraph of $D_s$ inducing a tournament. Then each vertex of $T$ belongs to a distinct $V_i$ and thus, by Remark~\ref{rmk:Ds}, the forward arcs of $D_s$ that are in $T$ induce a disjoint union of directed paths (i.e. every vertex have in- and out-degree at most $1$) and clearly form a feedback arc set of $T$. 
\end{proof}

\begin{lemma}\label{lem:123}
For every $s \geq 1$, $D_s$ does not contain  $\vec C_{3}(1, 2, \ora{C_3})$ nor $\vec C_{3}(1, 2, 3)$. 
\end{lemma}

\begin{proof}
Observe that the two digraphs $\vec C_{3}(1, 2, \ora{C_3})$ and $\vec C_{3}(1, 2, 3)$ only differ on the orientation of one arc: reversing an arc of the copy of $\ora{C_3}$ in $\vec C_{3}(1,2, \ora{C_3})$ leads to $\vec C_{3}(1,2, 3)$ and reversing an arc of the copy of $TT_3$ in $\vec C_{3}(1,2,3)$ leads to $\vec C_{3}(1,2, \ora{C_3})$. 
Our argument does not make any use of the orientations between the vertices inside this oriented $K_3$. Let $H$ be one of $\vec C_{3}(1, 2, \ora{C_3})$ or $\vec C_{3}(1, 2, 2)$, and let $x$ be the vertex in the copy of $K_1$, and  $y_1$ and $y_2$ the vertices in the copy of $TT_2$. See Figure~\ref{fig:123}.

Thanks to Lemma \ref{lem:feedback_edge_set}, it is enough to prove that in every feedback arc set of $H$, there exists a vertex with in- or out-degree at least $2$. Let $F$ be a feedback arc set of $H$ and assume for contradiction that it induces a disjoint union of directed paths. Then both $xy_1$ and $xy_2$ cannot belong to $F$. So we may assume without loss of generality that $xy_1 \notin F$. But then $F$ must intersect the three disjoint paths of length $2$ that go from $y_1$ to $x$, which necessarily implies that $F$ contains either two arcs coming out of $y_1$ or two arcs coming in $x$.
\end{proof}

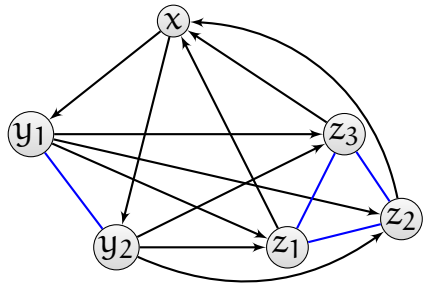
\begin{figure}[ht]
\begin{center}
\begin{tikzpicture}[scale=1.5]

\node (x) at (2,2) [vertex] {$x$};
\node (y1) at (0.75,1) [vertex] {$y_1$};
\node (y2) at (1.5,0) [vertex] {$y_2$};
\node (z1) at (3,0) [vertex] {$z_1$};
\node (z2) at (4,0.25) [vertex] {$z_2$};
\node (z3) at (3.5,1) [vertex] {$z_3$};

\draw[arc] (x) to (y1);
\draw[arc] (x) to (y2);

\draw[arc] (y1) to (z1);
\draw[arc] (y1) to (z2);
\draw[arc] (y1) to (z3);
\draw[arc] (y2) to (z1);
\draw[arc, bend right] (y2) to (z2);
\draw[arc] (y2) to (z3);

\draw[arc] (z1) to (x);
\draw[arc, bend right=38] (z2) to (x);
\draw[arc] (z3) to (x);

\draw[edge] (y1) to (y2);

\draw[edge] (z1) to (z2);
\draw[edge] (z2) to (z3);
\draw[edge] (z3) to (z1);

\end{tikzpicture}
\end{center}
\caption{\label{fig:123} Whatever the orientations of blue edges, $D_s$ does not contain this tournament and hence does not contain $\vec C_{3}(1,2,\ora{C_3})$ nor $\vec C_{3}(1,2,3)$.}

\end{figure}

By Lemma~\ref{lem:Ds_dic} and Lemma~\ref{lem:123}, $\vec C_{3}(1,2,\ora{C_3})$ and $\vec C_{3}(1,2,3)$ are not heroes in \omgs. 

\chapter{Decomposing and dicolouring some locally semicomplete digraphs}\label{chpt:semiround}

\begin{flushright}{\slshape    
This chapter is built upon a joint \\
work with Pierre Aboulker and \\
Pierre Charbit, published in \cite{AAC21}}. \\ \medskip
\end{flushright}

\emph{In this chapter, we give decomposition theorems for locally semicomplete and locally out-transitive digraphs, and use them to prove multiple results.}

\section{Notations}


A vertex $x$ {\em out-dominates} (resp. {\em in-dominates}) a set of vertices $X$ if $X\subseteq x^+$ (resp. $X\subseteq x^-$). A vertex $x$ {\em strictly out-dominates} (resp. {\em strictly in-dominates}) a set of vertices $X$ if $X \subseteq x^+\setminus x^-$ (resp. $X \subseteq x^- \setminus x^+$). A set of vertices $X$ {\em out-dominates} (resp. {\em strictly out-dominates}, {\em in-dominates}, {\em strictly in-dominates}) a set of vertices $Y$ if every vertex of $X$ out-dominates (resp. {\em strictly out-dominates}, {\em in-dominates}, {\em strictly in-dominates}) $Y$.



For a class $\mathcal P$ of digraphs (like semicomplete, tournament, acyclic), a digraph is {\em locally out-$\mathcal P$} (resp. {\em locally in-$\mathcal P$}) if for every vertex $x$, $x^+$ (resp. $x^{-}$) induces a digraph in $\mathcal P$. 
For example, a digraph $D$ is locally out-semicomplete if the out-neighbourhood of every vertex of $D$ induces a semicomplete digraph.
Finally, we will say that a digraph is {\em locally $\mathcal P$} if it is both locally out-$\mathcal P$ and locally in-$\mathcal P$. 
We make one exception for one of the main classes studied in this chapter: the oriented graphs for which the out-neighbourhood of every vertex is a transitive tournament, for which we will use the term "out-transitive oriented graphs" instead of the heavier and possibly confusing "out-transitive tournament oriented graphs".


A {\em linear order} on a digraph $D$ is an order $O =(v_1,v_2,\ldots,v_n)$ of its vertices. Two orders $O_1$ and $O_2$ are {\em equivalent} if $O_1 =(v_1,v_2,\ldots,v_n)$ and $O_2=(v_k,v_{k+1},\ldots,v_n,v_1,v_2,\ldots, v_{k-1})$ for some $k$. An equivalence class for this relation is called a {\em cyclic order} of $D$. Informally this means ordering the vertices of a digraph along a circle. See Figure \ref{fig:inround} for a digraph on $10$ vertices with a cyclic order given by the linear order $(v_1,v_2,\ldots,v_{10})$.

For two vertices $v_i$ and $v_j$ in a linear order $O=(v_1,v_2,\ldots,v_n)$ the \textit{cyclic interval} $[v_i,v_j]$ is defined as follow:
$$
 [v_i,v_j] =     
\begin{cases}
     \{v_k, k\in [i,j]\} \text{ if $i<j$ }\\
     \{v_k, k\not\in\; ]j,i[ \} \text{ if $i\geq j$}
    \end{cases}
$$ 
Note that cyclic intervals really only depend on the cyclic order and not on a linear order chosen as a representative. We also define open and half open intervals $]v_i,v_j[ = [v_i, v_j] \setminus \{v_i, v_j\}$,  $[v_i,v_j[ = [v_i, v_j] \setminus \{ v_j\}$ and  $]v_i,v_j] = [v_i, v_j] \setminus \{v_i\}$.

Given a linear order $(v_1,v_2,\ldots,v_n)$, an arc $v_iv_j$ is {\em forward} if $i<j$ and {\em backward} otherwise. Note that a digraph is acyclic if and only if there exists a linear order for which all arcs are forward arcs.

\section{Introduction}

Semicomplete digraphs are well studied and a natural and fruitful way to extend results on this class is to look at the class of locally semicomplete digraphs.
Introduced in 1990 by Bang-Jensen \cite{B90}, locally semicomplete digraphs have since then been the topic of more than 100 research papers, and a whole chapter in \cite{BG18} is devoted to this class. A particularly nice result in this area is one by Huang that gives a geometric characterization of locally transitive digraphs. We state it here in the particular case of oriented graphs.

\begin{theorem}[Huang, \cite{H92}]\label{thm:round}
If $D=(V,A)$ is a connected oriented graph, then the two conditions below are equivalent 
\begin{enumerate}
    \item for every vertex $x$, both $x^+$ and $x^-$ induce a transitive tournament. 
    \item there exists a cyclic order of the vertices of $D$ such that 
$$\forall xy\in A, \forall z\in ]x,y[, xz\in A \text{ and } zy\in A$$
\end{enumerate}
\end{theorem}

Any oriented graph (strong or not) that satisfies the second condition above  is called a \textit{round} oriented graph. In other words, for every vertex $x$, $x^+$ (resp. $x^-$) consists of a cyclic interval placed just after (resp. before) $x$ in the cyclic order. Note that a round oriented graph is strong if and only if every vertex has at least one in-neighbour. 
This is because the cyclic order given by the theorem then yields a Hamiltonian cycle. By a similar observation, if a round oriented graph is not strong, then it is in fact acyclic.
\medskip

Our first result is a generalization of the theorem above in the particular case of strong oriented graphs. 


\begin{inroundtheorem}

Let $D$ be a strong oriented graph. Then conditions below are equivalent.

\begin{enumerate}
    \item \label{itm:in-round_fst} for every vertex $x$, $x^+$ induces a tournament and $x^-$ induces an acyclic digraph 
    \item \label{itm:in-round_snd} there exists a cyclic order of the vertices of $D$ such that 
$$\forall xy\in A, \forall z\in ]x,y[, zy\in A$$

\end{enumerate}
\end{inroundtheorem}
Again condition \ref{itm:in-round_snd} can be seen as the property that for every vertex $x$, $x^-$ consists in a cyclic interval placed just before $x$ in the order (see Figure \ref{fig:inround}). Following the terminology of \cite{BG18}, any oriented graphs satisfying condition \ref{itm:in-round_snd} of the theorem above will be called {\em in-round}. 

In \cite{LZM08}, we note that the authors prove a similar result with a stronger alternate condition 1: they ask that for every vertex $x$, $x^+$ induces a {\em transitive} tournament and $x^-$ induces an acyclic digraph {\em with a hamiltonian path}. These additional conditions (which are easily seen in fact to be implied by condition 2) are unnecessary, which makes our theorem slightly stronger. Moreover, our proof, exposed in section~\ref{sec:dec}, is also much shorter.

\begin{figure}[htbp]
\begin{center}
\includegraphics[width=0.25\textwidth]{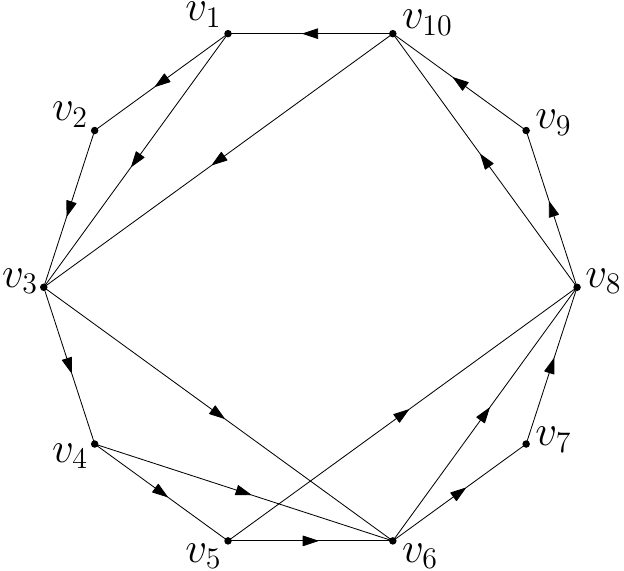}
\caption{An in-round oriented graph that is not round}
\label{fig:inround}
\end{center}
\end{figure}


In fact, Theorem \ref{thm:in-round} is the first step towards our main structural result, which is a decomposition theorem for the class of strong locally out-transitive digraphs. If $H$ is a subdigraph of $D$, we define the {\em contraction} $D/H$ as the digraph obtained by removing all vertices of $H$, then adding a new vertex $h$ such that $xh$ (resp. $hx$) is an arc of $D/H$ if $x^{+}\cap H$ (resp $x^{-}\cap H$) is non empty. 
Beware that in general $D/H$ might contain digons (even if $D$ does not), even though in our case this never happens.

\begin{maintheorem}

If $D$ is a strong locally out-transitive oriented graph, then there exists a partition of its set of vertices into strong subdigraphs $D_{1},\ldots,D_{k}$ such that the digraph obtained by contraction of $D_{1},\ldots,D_{k}$ is a strong in-round oriented graph.

\end{maintheorem}

We give two applications of this theorem. The first one is a dicolouring result: we prove that if $D$ is an out-transitive oriented graph, then it can be partitioned into two acyclic induced subdigraphs. This proves a conjecture of~\cite{ACN21}. Proof together with the context of the conjecture is given in Section~\ref{sec:colouring}. This result was independently proved by Raphael Steiner, see Remark~\ref{Steiner}. 
The second one is a proof that locally in-tournaments satisfy the famous Caccetta-\Hag conjecture. This result is mentioned in  \cite{R13} as an unpublished but not easy result of Paul Seymour but here the idea is to show how our decomposition  theorem more or less directly implies the result. The proof is given in Section~\ref{sec:ch}

Finally, in Section \ref{sec:semicomplete}, we use the techniques developed for the proof of Theorem \ref{thm:struct} to prove a structural theorem, Theorem \ref{thm:k12k21}, for the class of locally semicomplete digraphs. We then apply our Theorem~\ref{thm:k12k21} to give short proofs of two existing results concerning $2$-king and pancyclicity. 
Note that a different structural theorem for locally semicomplete digraph had already been proved in \cite{BGGV} (see Theorem 3.12), but seems to be independent of ours.  

One idea behind this chapter is to promote the idea of finding  a decomposition theorem for classes of digraphs, that is a theorem whose statement is of the kind: either a graph in this class is "basic" (belongs to some simple subclass), or it can be decomposed in some prescribed ways. Such decomposition theorems proved to be very powerful tools in the world of undirected graphs (the most famous example being the celebrated proof of the perfect graph conjecture by \Chu et al in \cite{Perfect}), and there are to our knowledge not so many theorems of this kind in the world of directed graphs, and there is no reason to believe that it could not be as effective in this setting. 

\begin{remark}\label{Steiner}
A week prior to the submission of \cite{AAC21}, R. Steiner published on arXiv a very nice paper \cite{S21} containing another proof of the dicolouring result of locally out-transitive oriented graphs mentioned above (as well as other results about dicolourings). Even though some of the ingredients are in common, the proof is longer and is different as it is an entirely inductive proof whereas ours relies on the structure theorem mentioned above.
\end{remark}


\section{Decomposing locally out-transitive oriented graphs}\label{sec:dec}

We start by proving the theorem mentioned in the introduction about strong oriented graphs that are both locally out-tournament and locally in-acyclic. 
\begin{theorem}\label{thm:in-round}
Let $D$ be a strong oriented graph. Then conditions below are equivalent.
\begin{enumerate}
    
    \item \label{itm:in-round_fst_2}   for every vertex $x$, $x^+$ induces a tournament and $x^-$ induces an acyclic digraph 
    \item \label{itm:in-round_snd_2}  there exists a cyclic order of the vertices of $D$ such that 
$$\forall xy\in A(D), \forall z\in ]x,y[, zy\in A(D)$$

\end{enumerate}
\end{theorem}
We recall that we will use the term in-round for any oriented graph satisfying condition 2.

\begin{proof}[of Theorem \ref{thm:in-round}]
The easy direction is \ref{itm:in-round_snd_2} implies \ref{itm:in-round_fst_2}. Indeed, let $x,y,z$ be such that ${y,z}\subset x^+$ and assume w.l.o.g that $z\in ]x,y[$. Then by \ref{itm:in-round_snd_2} we have that $zy\in A$, so $x^+$ is a tournament. 
Assume now that $x^-$  contains a directed cycle $C$, and let $y$ be the vertex of $C$ such that $C\setminus{y} \subseteq ]y,x[$ (the leftmost vertex of $C$ in the representative of the cyclic order which ends in $x$). Let $z$ be the predecessor of $y$ in $C$. Now we have that $x\in ]z,y[$ and so by \ref{itm:in-round_snd_2} there must be an arc $xy$, which contradicts the fact that $y\in x^-$ (there are no digon here).

Now assume \ref{itm:in-round_fst_2}. For every vertex $x$, $x^-$ induces a non-empty acyclic oriented graph, and hence contains a vertex $y$ such that $y^+\cap x^- = \emptyset$ (take the last vertex in a topological ordering of $x^-$). 
For every $x$ we arbitrarily choose one such vertex and denote it by $f(x)$. If $z \in f(x)^+ \setminus \{x\}$, then since $f(x)^+$ induces a tournament, $z$ and $x$ must be connected by an arc, and this cannot be $zx$ by definition of $f(x)$, so it must be  $xz$. So $x$ out-dominates $f(x)^+ \setminus \{x\}$ for every vertex $x$. 

Now let $H$ be the oriented graph induced by the arcs $f(x)x$. Each vertex of $H$ has in-degree exactly $1$, so $H$ contains a cycle $C$. Set $V(C) = \{v_1, \dots, v_n\}$ where $v_iv_{i+1} \in A(C)$ for $i=1, \dots, n-1$ and $v_nv_1 \in A(C)$. So $f(v_{i+1}) = v_i$ for $i=1, \dots, n-1$ and $f(v_1)=v_n$. 
If $C$ does not span all vertices of $D$, then, since $D$ is strong, there exists an arc $xy \in A(D)$ such that $x\in V(C)$ and $y\not\in V(C)$. Assume without loss of generality that $x=v_1$. Since $f(v_2) = v_1$, $v_2$ out-dominates $v_1^+ \setminus \{v_2\}$, so $v_2y \in A(D)$. Similarly, $y$ in-dominates $V(C)$, which contradicts \ref{itm:in-round_fst_2}. 

So $C$ consists of a Hamiltonian cycle of $D$. Let $v_i,v_j \in V(C)$ such that $v_iv_j \in A(D)$. We consider subscripts modulo $n$.  We have $f(v_{i+1}) = v_i$, so $v_{i+1}$ out-dominates $v_{i}^+ \setminus \{v_i\}$. Hence $v_{i+1}v_j \in A(D)$. Applying the same reasoning to $f(v_{i+2}) = v_{i+1} $, we get that $v_{i+2}v_j \in A(D)$. Similarly, for every $k$ such that $i < k < j$, we have $v_kv_j \in A(D)$, so  \ref{itm:in-round_snd_2} holds.
 
\end{proof}

We now recall the statement of the main structural theorem mentioned in the introduction.

\begin{theorem}\label{thm:struct}
If $D$ is a strong locally out-transitive oriented graph, then there exists a partition of its set of vertices into strong subdigraphs $D_{1},\ldots,D_{k}$ such that the digraph obtained by contraction of $D_{1},\ldots,D_{k}$ is a strong in-round oriented graph.
\end{theorem}


Let us start with a lemma about locally out-semicomplete digraphs (be aware that the rest of this section is about oriented graphs and not digraphs, but we state it here for digraphs as we will make use of it in Section~\ref{sec:semicomplete}, where we study digraphs).

\begin{lemma}\label{lem:base}
Let $D$ be a locally out-semicomplete digraph. Let $H$ be a strong induced subdigraph of $D$ and let $z\in V(D) \setminus V(H)$. If $z^-\cap V(H) \neq \emptyset$ and $z^+\cap V(H) = \emptyset$, then $z$ in-dominates $V(H)$.
\end{lemma}
\begin{proof}
Let $h$ and $h'$ be vertices of $H$ such that $hz$ and $hh'$ are arcs of $D$. Because $D$ is locally out-semicomplete and $z^+\cap V(H) = \emptyset$, $h'z$ must be an arc. Since $H$ is strong, we get that $z$ in-dominates $V(H)$.
\end{proof}

\begin{proof}[of Theorem ~\ref{thm:struct} ]

From now on $D$ will denote a strong  locally out-transitive oriented graph. 

We define a {\em hub}  to be a subset of vertices $H$ of $D$ such that 
\begin{itemize}
\item $H$ induces a strong oriented graph,
\item there exists $x\notin H$ such that $x$ in-dominates $H$.
\end{itemize}

Note that a hub is necessarily a strict subset of $V(D)$.
A hub is trivial if it consists of a single vertex.

By Theorem~\ref{thm:in-round}, an oriented graph is in-round if the in-neighbourhood of each vertex induces an acyclic oriented graph, and the out-neighbourhood induces a tournament. Hence, since a cycle that is out-dominated by a vertex is a hub, $D$ is in-round if and only if there is no non-trivial hub. 

Assume $D$ is not in-round, and consider a hub $H$ that is (inclusion-wise) maximal. We want to prove that $D/H$ is a locally out-transitive oriented graph, i.e. $D/H$ does not contain digon and is  locally out-transitive. 
Assume first by contradiction that $D/H$ contains a digon, so there exists a vertex $z\notin H$ that is {\em mixed} for $H$, that is such that there exists $\{h,h'\}\subset H$ with $(hz,zh')\in A(D)^{2}$.
Since $H$ is strong, there is a directed path from $h$ to $h'$.  Consider such $h$ and $h'$ with a directed path of minimal length joining them in $H$. 
Since $h^+$ induces a  tournament, the vertex following $h$ in this directed path must be adjacent with $z$, implying that $hh' \in A(D)$ by the choice of $h$ and $h'$. 
There exists $x$ in $V(D)$ such that $H\subset x^{-}$, and since $D$ is a locally out-transitive oriented graph, there must be an arc between $x$ and $z$. 
It cannot be $xz$ because in that case, $h^+$ would contain $\{x,z,h'\}$ which induces a directed triangle, a contradiction to the fact that $D$ is locally out-transitive. 
So $zx$ must be an arc. But then $H\cup\{z\}$ is a hub that contradicts the maximality of $H$. This proves that $D/H$ does not contain digon.  

We now prove that $D/H$ is locally out-transitive.
First observe that by Lemma \ref{lem:base},  there are three types of vertices outside $H$: the one that have no arc to or from $H$, the one that has out-neighbours in $H$ but no in-neighbour, and the ones that in-dominates $H$. 
 Let $h$ be the vertex of $D/H$ obtained after having contracted $H$. 
Let $x$ be an in-neighbour of $h$ in $D/H$, which means that in $D$, $x$ is an in-neighbour of some vertex $y \in H$. Let $z$ be another out-neighbour of $x$ in $D/H$. As $D$ is locally out-transitive, $y$ and $z$ must be adjacent in $D$, and thus $z$ is adjacent with $h$ in $D/H$. Thus the out-neighbourhood of $x$ in $D/H$ is a tournament.  
Assume now for contradiction that $x^+$ contains a directed triangle $hz_1$, $z_1z_2$ $z_2h$ in $D/H$. 
So $z_1$ has an in-neighbour $h_1$ in $H$ and $z_2$ has an out-neighbour in $H$. 
Since $D$ is locally out-transitive and $h_1$ and $z_2$ are out-neighbours of $x$, $z_2$ and $h_1$ must be adjacent, and since $z_2$ has an out-neighbour in $H$, it has no in-neighbour in $H$ and thus $z_2h_1 \in A(D)$. 
But then $\{z_1, z_2, h_1\}$ induces a directed triangle in the out-neighbourhood of $x$ in $D$, a contradiction. 
It remains to show that $h^+$ induces a transitive tournament in $G/H$, which is implied by the fact that if a vertex $z$ has an in-neighbour in $H$, then it in-dominates $H$. 
\medskip

Consider now two distinct maximal hubs $H_{1}$ and $H_{2}$. We want to prove that their intersection is empty. For $i=1,2$, let $h_1$ be a vertex that in-dominates $H_i$. Note that it is not possible that $h_1\in H_2$ and $h_2\in H_1$ simultaneously, as this would imply a digon between the two vertices. So without loss of generality, assume $h_1\notin H_2$. But now if $H_1\cap H_2 $ is non-empty, then $h_1$ has an in-neighbour in $H_2$ and by Lemma \ref{lem:base} and since we proved that no vertex can be mixed for $H_2$, we have that $h_1$ in-dominates $H_{1}\cup H_{2}$, which is, therefore, a hub (it is strongly connected because the intersection is non-empty). This contradicts the maximality of $H_1$ and $H_2$ 

This proves that maximal hubs define a partition of $V(D)$ (recall that every vertex belongs to a maximal hub since every singleton is a hub). Moreover, as argued above, if there is an arc $xy$ from a maximal hub $H_{1}$ to a maximal hub $H_{2}$, then $y$ in-dominates $H_{1}$. We summarize this with the following claim (the transitive tournament fact is due to the graph being locally out-transitive).

\begin{claim} Maximal hubs form a partition of $V(D)$. Moreover, if $H_{1}$ and $H_{2}$ are two maximal hubs, then either:
\begin{itemize}
\item  there is no arc between $H_1$ and $H_2$, or
\item  there are all arcs from $H_{1}$ to a subset of $H_{2}$ inducing a transitive tournament, and no other arc, or
\item  there are all arcs from $H_{2}$ to a subset of $H_{1}$ inducing a transitive tournament and no other arc. 
\end{itemize}

\end{claim}

Let $D'$ be the digraph obtained by contracting every maximal hub. If $D'$ contains a non-trivial hub $H'$, then by the claim above it is clear that the set $H$ of vertices of $D$ that are mapped to vertices in $H'$ by the contraction form a hub that would contradict the maximality of the hubs that were contracted to obtain $D'$. So $D'$ contains no non-trivial hub, so no in-dominated cycle, and hence by Theorem \ref{thm:in-round}, it is in-round. This concludes the proof of Theorem \ref{thm:struct}.
\end{proof}

\section{Applications of Theorem \ref{thm:struct}}
\subsection{Dicolourings}\label{sec:colouring}
We recall the following conjecture.

\begin{conjecture}[\cite{ACN21}]\label{conj:oriented}
Let $H$ be a hero and let $F$ be an oriented forest. The set $\{  H, F\}$ is heroic if and only if:
\begin{itemize}
\item either $F$ is the disjoint union of oriented stars,
\item or $H$ is a transitive tournament.
\end{itemize}
\end{conjecture}

The first case that was left in \cite{ACN21} is the case $F=S^+_2$, and $H=\ora{C_{3}}$ and it was conjectured that digraphs 
without any induced subgraph in $\{  \ora{C_3}, S^+_2\}$
 have a dichromatic number at most two. 
As mentioned in the introduction, we now prove this result, and in fact a stronger result. 



\begin{theorem}\label{thm:2dic}
Every locally out-transitive oriented graph has a dichromatic number at most $2$.
\end{theorem}

Note that it indeed extends the question mentioned above as it amounts to forbidding $\ovlra{K_2}$, $S^+_2$ and the tournament on $4$ vertices built by taking a directed triangle $\ora{C_{3}}$ and adding a vertex with an arc going to the three other vertices. As already said, forbidding $S^+_2$ implies that every out-neighbourhood induces a tournament, and forbidding this 4-vertex tournament implies that every out neighbourhood must induce a $\ora{C_3}$-free tournament, hence acyclic (every tournament containing a directed cycle must contain a $\ora{C_3}$).

To ease the induction proof we will prove the following stronger result. 

\begin{theorem}\label{thm:main}
Let $D$ be a locally out-transitive oriented graph and $T$ a subset of $V(D)$ inducing a transitive tournament. Then $D$ admits a proper $2$-dicolouring such that all vertices of $T$ receive the same colour.
\end{theorem}

The first step towards the proof of this theorem is to prove the exact same statement for the class of in-round oriented graphs (which we recall can also be seen as locally out-transitive and locally in-acyclic oriented graphs).
\begin{proposition}\label{prop:colin-round}
Every in-round oriented graph has a dichromatic number at most $2$. More precisely, for every vertex $x$, there exists a proper $2$-dicolouring such that $\{x\}\cup x^{+}$ is monochromatic.
\end{proposition}
\begin{proof}
We may assume that the oriented graph is strong since otherwise a $2$-dicolouring of each strong component yields a proper $2$-dicolouring of the whole oriented graph. 
Consider the cyclic order given by the definition of in-round and pick any vertex $x$. Let $y$ be the vertex such that $xy$ is a longest arc, that is the arc such that the interval $[x,y]$ contains the maximum number of vertices. 
This implies that in the linear order given by the interval $]y,x[$, all arcs are forward arcs since a backward arc $x'y'$ would force the arc $xy'$, contradicting the maximality of $xy$. Hence, $]y,x[$ induces an acyclic oriented graph. 
Moreover, $[x,y]$ induces an acyclic oriented graph since it is included in $y^{-}$ and by definition of the in-round cyclic order. This concludes the proof.
\end{proof}

We are now ready to prove Theorem \ref{thm:main}.

\begin{proof}[of Theorem \ref{thm:main}]
Let $D$ be a locally out-transitive oriented graph and $T$ a subset of $V(D)$ inducing a transitive tournament. 
Again we can assume that $D$ is strong, and we proceed by induction on the number of vertices of $D$. 
 We consider the decomposition of $V(D)$ into maximal hubs given by Theorem \ref{thm:struct} and label the hubs $H_{1},\ldots,H_{n}$ so that the order corresponds to the cyclic order $h_{1},\ldots,h_{n}$ of the in-round oriented graph $D'$ obtained after contracting $H_1, \dots, H_n$. 
 Since the theorem guarantees that $D'$ contains no digon, there can exist arcs in only one direction between two distinct $H_i$. 
Moreover, Lemma \ref{lem:base} implies that if there is an arc $xy$ from $x\in H_{i}$ to $y\in H_{j}$, then $y$ in-dominates $H_i$. Moreover, by the in-round property of $D'$, we also get that $y$ in-dominates $H_k$ for every  $k\in [i,j[$. 

We define $T_{i}$ for $i=1, \ldots, n$ to be the set of vertices in $H_{i}$ that have an in-neighbour out of $H_{i}$. $T$ is non empty since $D$ is strong, and by the previous discussion $T_i=x^{+}\cap H_{i}$ for every $x\in H_{i-1}$, which in particular implies that $T_i$ induces a transitive tournament. Let $s$ be the source in the transitive tournament $T$, and without loss of generality assume it belongs to $H_{1}$. Note that because every other vertex in $T$ is an out-neighbour of $s$, we have that $T\cap H_{i}\subseteq T_{i}$ for every $i\geq 2$.  Finally, we observe that if $C$ is a cycle that intersects $H_i$, then either $C$ intersects $T_{i}$ or $C$ is entirely included in $H_i$.

\begin{figure}[htbp]
\begin{center}
\includegraphics[width=0.8\textwidth]{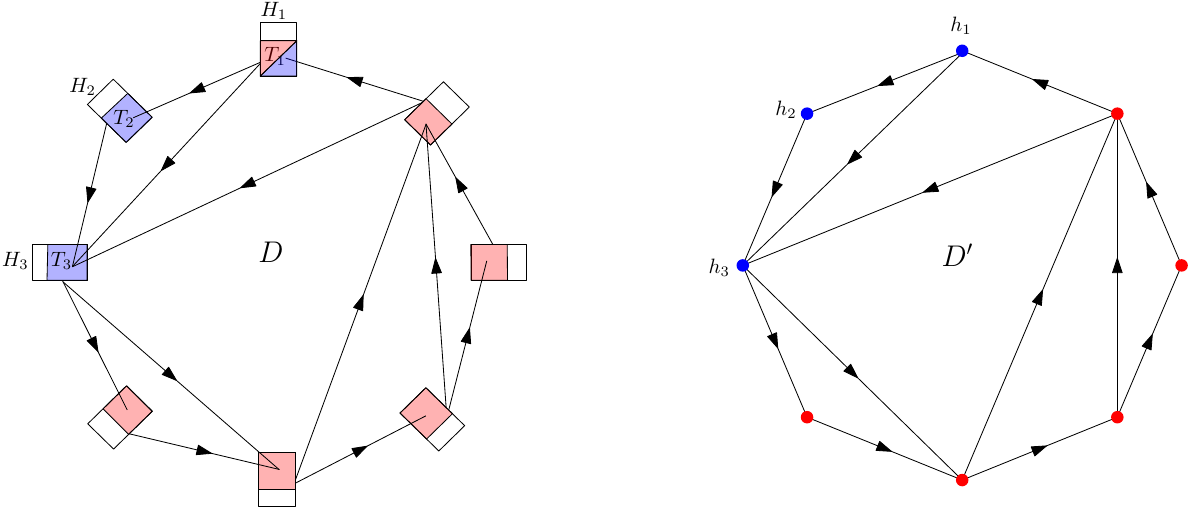}
\caption{The digraphs $D$ and $D'$ in the proof of \ref{thm:main}}
\label{fig:2color}
\end{center}
\end{figure}

Now we are ready to define our dicolouring. First, we consider the in-round oriented graph $D'$ and denote by $T'$ the subtournament (transitive) in $D'$ consisting in all $h_{i}$ such that $T\cap H_{i}$ is non-empty. By Proposition \ref{prop:colin-round}, there exists a proper $2$-dicolouring of $D'$ such that the out-neighbourhood of $h_{1}$ (which contains $T'$) is monochromatic, say coloured $1$. Now by induction, we can ask for every $i\geq 2$ for a dicolouring of $H_{i}$ such that all vertices of $T_{i}$ get the colour of $h_i$ in the dicolouring of $D'$ defined above. For $i=1$ we ask by induction for a dicolouring of $H_{1}$ such that every vertex of $T\cap H_{1}$ gets the colour of $h_1$, that is $1$. On Figure \ref{fig:2color}, we pictured a case where $T$ intersects the hubs $H_1$,$H_2$ and $H_3$, colour $1$ is blue, and colour $2$ is red.

First, note that in this dicolouring, every vertex of $T$ gets colour $1$ because of the assumption on the dicolouring of $D'$. We now need to prove  that this dicolouring is a proper  $2$-dicolouring of $D$. Assume by contradiction that there exists a monochromatic cycle and consider a minimal such cycle $C$. Since the dicolouring is by induction proper in every hub, $C$ is not included in any $H_{i}$. Since the vertices in each $H_{i}$ have the same out-neighbours out of $H_{i}$, the minimality of $C$ implies that $C$ contains at most one vertex from each hub $H_{i}$ and this vertex must belong to $T_{i}$. The only case where this does not yield directly a monochromatic cycle in the contracted digraph $D'$ (and hence a contradiction) is if $C$ is coloured $2$ and contains a vertex $x$ in $T_{1}\setminus T$. Let $y$ be its successor on the cycle. Then $y$ belongs to some $T_{j}$ where $j$ is such that $h_{1}h_{j}$ is an arc of $D'$. But by assumption on the dicolouring of $D'$ this implies that $h_{j}$ gets colour $1$ and therefore $y$ must get colour $1$, which is a contradiction that finishes the proof of Theorem \ref{thm:main}.

\end{proof}


\subsection{A Special Case of the Caccetta-\Hag Conjecture}\label{sec:ch}
A beautiful and famous conjecture due to Caccetta and \Hag states the following. 
\begin{conjecture}[Caccetta-\Hag]
Let $k\geq 2$ be an integer. Every digraph $D$ on $n$ vertices with no directed circuits of length at most $k$ contains a vertex of out-degree less than $n/k$.
\end{conjecture}
The case $k=2$ is trivial but the case $k=3$ is still widely open and has attracted a lot of attention. In \cite{R13} (see page 3), it is mentioned that for $k=3$, while adding the hypothesis that the graph has no $S^+_2$ makes it very easy, the dual case of forbidding $S^-_2$ was proven by Seymour but is "substantially more difficult". 

Here we prove that this comes as an easy consequence of Theorem~\ref{thm:struct} and Theorem~\ref{thm:in-round}, for any value $k\geq 3$.

\begin{theorem}\label{thm:CH}
Let $D$ be locally in-tournament oriented graph on $n$ vertices with no directed cycle of length at most $k$. Then $D$ contains a vertex of out-degree less than $n/k$.
\end{theorem}

Theorem~\ref{thm:struct} and Theorem~\ref{thm:in-round} were designed for locally out-tournament but of course by reversing the arcs we get an equivalent statement for locally in-tournament (here an out-round digraph is a digraph obtained by an in-round digraph by reversing all the arcs). We combine both to get this corollary.

\begin{corollary}\label{cor:locain}
If $D$ is a  
strong locally in-tournament that does not contain any $\ora{C_3}$, 
there exists a partition of its set of vertices into strong subdigraphs $D_{1},\ldots,D_{k}$ such that the digraph $D'$ obtained by contraction of every $D_{i}$ is a strong oriented graph which is out-round, that is admits a cyclic order on its vertices such that 
$$\forall x,y,z\in V(D')\  (xy\in A(D') \land  z\in ]x,y[  )\ \Ra\ xz\in A(D')
$$
\end{corollary}

\begin{proof}[of Theorem \ref{thm:CH}]
Let $D$ be a digraph as in the statement of the theorem and denote by $n$ the number of its vertices. First observe that we can assume $D$ to be strong, since one can apply it to a terminal strong component. Since $k\geq 3$, notice that $D$ does not contain $\ora{C_3}$ so we can apply Corollary \ref{cor:locain} above. 

We begin by proving a weighted version for out-round oriented graphs.

\begin{lemma}
Let $D$ be a strong out-round oriented graph with no dicycle $\ora{C_3}$ and let $w$ be a positive weight function on the vertices of $D$. Denote by $W$ the sum of all weights. Then there exists a vertex $u$ such that 
$$\sum_{v\in u^+} w(v) < \frac{W-w(u)}{k} $$
\end{lemma}

Let us show first that this lemma indeed implies the theorem. We can apply Corollary \ref{cor:locain} to $D$ to obtain a strong out-round oriented graph $D'$. We assign to each vertex of $D'$ a weight equal to the order of the corresponding contracted subdigraph of $D$ and apply the above lemma to get a vertex $u$ satisfying the claim. Let $D_i$ be the subdigraph of $D$ whose contraction yielded $u$. Now by applying induction to $D_i$, one finds a vertex in $D_i$ with an out-degree (strictly) less than $|D_i|/k$, which combined with the lemma gives the desired result.

Let us now prove the lemma. Consider the cyclic order as in Corollary~\ref{cor:locain}. For every vertex $x$, denote by $f(x)$ the last of its out-neighbour along the cyclic order and by $\phi(x)$ the quantity $\sum_{y\in x^+} w(y)$. Observe that due to the out-round structure, $\phi(x)$ is exactly the sum of weights in the interval $]x,f(x)]$. We denote also by $f^{(i)}(x)=f(f^{(i-1)}(x))$ with the convention $f^{(0)}(x)=x$.

Consider a path on $k$ vertices of the form $xf(x)f^{(2)}(x)\ldots f^{(k-1)}(x)$. Assume by contradiction that there exists integers $i$ and $j$ such that the intervals $]f^{i}(x), f^{(i+1)}(x)[$ and $]f^{j}(x), f^{(j+1)}(x)[$ overlap, that is $f^{j}(x)\in ]f^{i}(x), f^{(i+1)}(x)[$ and $f^{(i+1)}(x)\in ]f^{j}(x), f^{(j+1)}(x)[$. Because of the out-round property, this implies the presence of the arcs $f^i(x)f^j(x)$ and  $f^j(x)f^{(i+1)}(x)$. If $j<i$ the first arc along with the path from $f^j(x)$ to $f^i(x)$ gives a cycle of length strictly less than $k$ and similarly if $j<i$ with the second arc with the path from $f^{(i+1)}(x)$ to $f^j(x)$. In both cases, we, therefore, have a contradiction.
Hence the intervals $]f^{i}(x), f^{(i+1)}(x)[$ are disjoint and consecutive, and we have thus proven that for every vertex $x$ 
$$\sum_{i=0}^{k-1} \phi(f^{(i)}(x)) \leq W-w(x)$$
Now consider some cycle $C$ of the form $yf(y)f^{(2)}(y)\ldots f^{(p-1)}(y)=y$ (such a cycle must exist). By summing the above inequality for all vertices in the cycle one gets
$$k\sum_{x\in V(C)} \phi(x) \leq Wp-\sum_{x\in V(C)} w(x)$$
So there must exist $x\in V(C)$ such that:
$$ k\phi(x)+w(x)\leq W $$
which is exactly the assertion of the lemma.

\end{proof}



%
%
%


\section{Structure of locally semicomplete digraphs} \label{sec:semicomplete}

In this section, digons are allowed and we focus on locally semicomplete digraphs, which we recall is the class of digraphs such that the in-neighbourhood and out-neighbourhood of any vertex is semicomplete.

A statement in the flavour of Conjecture \ref{conj:oriented} is easy to prove. The proof was independently obtained by Raphael Steiner~\cite{S21}. 

\begin{theorem}
For every hero $H$, locally tournaments oriented graphs with no induced copy of $H$  have bounded dichromatic number. More precisely, it is bounded by at most twice the maximum dichromatic number of a $H$-free tournament. 
\end{theorem}

\begin{proof}[]
Let $k$ be the maximum dichromatic number of an $H$-free tournament. 
Let $x$ be any vertex. By induction on the number of vertices, we know that the oriented graph induced by the set of non-neighbours $N=V\setminus (\{x\}\cup x^+ \cup x^-)$ is $2k$ colourable so we first properly colour $N$ with colours in $\{1,\ldots,2k\}$. Now since $x^+$ and $x^-$ are tournaments, they are $k$-colourable by hypothesis. We can thus use colours $\{1,\ldots,k\}$ for a proper dicolouring of $x^-$, and colours $\{k+1,\ldots,2k\}$ for a proper dicolouring of $x^+$, give $x$ any colour, and it is not difficult to check that this gives a valid $2k$-dicolouring of $D$ (the main fact to observe is that there are no arcs from $x^-$ to $N$, and from $N$ to $x^+$).
\end{proof}

The purpose of this section is to use our techniques to give a proof of the following structural theorem for the class of locally semicomplete digraphs. As already mentioned in the introduction, such a theorem is given in \cite{BGGV} (see Theorem 3.12), but our theorem is much simpler to state and might have some interesting applications. In fact, the theorem of \cite{BGGV} has the same first two cases below, but their third case (called {\em evil} in Chapter 6 of the monograph \cite{BG01}) has a more complicated description that seems to us less easy to handle for applications. We propose some illustrations after the proof of the theorem.

Remember that a {\em universal vertex} is a vertex $x$ such that $x^+=x^-=V(D)\setminus \{x\}$. 

\begin{theorem}\label{thm:k12k21}
Let $D$ be a connected locally semicomplete digraph, then either:
\begin{itemize}
    \item $D$ is semicomplete with a universal vertex.
    \item There exists a partition of $V(D)$ into $k\geq 2$ subsets each inducing strong connected semicomplete digraphs such that the digraph obtained by contracting every member of the partition is a round oriented graph.
    \item there exists a partition of $V(D)$ into four sets $E$, $F$, $G$ and $H$ such that:
        \begin{itemize}
            \item $F$ and $H$ are non empty, and one of $E$ and $G$ is non empty.
            \item $D[E]$, $D[F]$, $D[G]$ and $D[H]$ are semicomplete;
            \item $E$ strictly out-dominates $F$, $F$ out-dominates $G$, $G$ out-dominates $H$ and $H$ strictly out-dominates $E$.
            \item $\forall x\in G,\ x^+\cap E\neq\emptyset$  and $x^-\cap E\neq\emptyset$
        \end{itemize}
\end{itemize}
\end{theorem}

\begin{figure}[htbp]
\begin{center}
\includegraphics[width=0.3\textwidth]{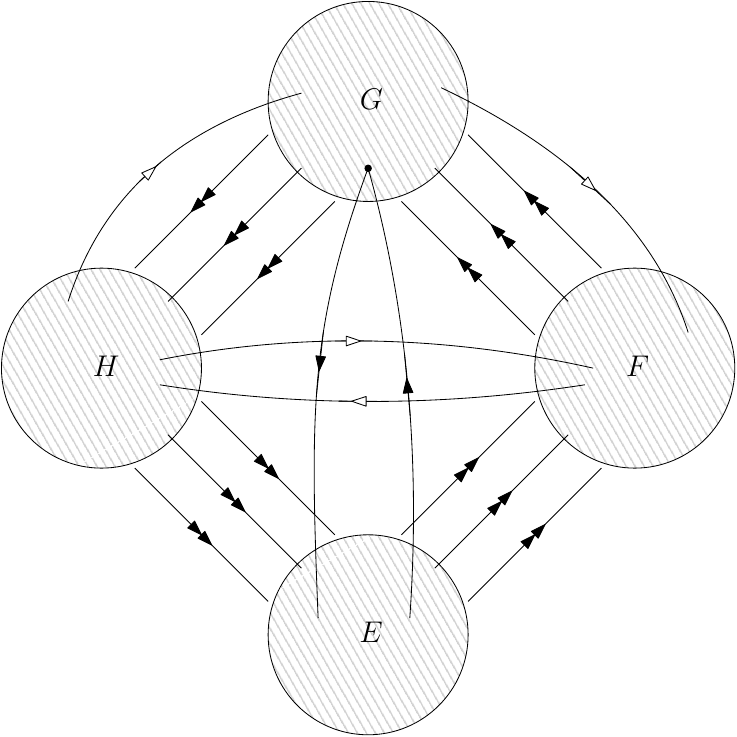}
\caption{The structure in the third case of Theorem \ref{thm:k12k21}}
\label{fig:4blobs}
\end{center}
\end{figure}

For this proof, we need to relax the notion of hub of the previous section: a set $X \subseteq V$ is called a {\em weak hub} if there exists a vertex which strictly in-dominates {\em or} strictly out-dominates $X$. The "strictly" part of this definition was implicit before for hubs since we were in the context of oriented graphs. A weak hub $X$ is said to be {\em mixed} if there exist vertices $x \notin X$ and $u,v \in X$ such that both  $xu$, $vx$ are arcs.  

We split the proof of Theorem \ref{thm:k12k21} into two parts depending on the existence of a maximal weak hub that is mixed. When there are no such subsets, the proof will have many similarities with our proof of Theorem \ref{thm:struct}.
\begin{lemma}\label{lem:no_mixed_weak hub}
Let $D$ be a connected locally semicomplete digraph such that no maximal weak hub is mixed. Then either $D$ is semicomplete with a universal vertex, or there exist $k \geq 2$ disjoint sets $X_1, X_2 \dots X_k$ of vertices such that $D[X_i]$ is strong for $i=1, \dots, k$, and  the digraph obtained by contracting every $X_i$ is a round oriented graph. 
\end{lemma}

\begin{proof}[of Lemma \ref{lem:no_mixed_weak hub}]
Let $D$ be a connected locally semicomplete digraph and assume $D$ is not semicomplete with a universal vertex. We first prove that maximal weak hubs form a partition of $V(D)$.

\begin{claim}\label{clm:weak hub_intersection}
Let $X$ and $Y$ be two distinct maximal weak hubs. Then $X \cap Y = \emptyset$.
\end{claim}

{\bf Proof:}
Assume $X$ and $Y$ contradict the claim. If there exists $x \notin Y$ which strictly in-dominates $X$, then there is at least one arc to $x$ from a vertex of $X \cap Y \neq \emptyset$ and since $Y$ is not mixed, there is no arc from $x$ to $Y$. As $D[Y]$ is strongly connected, Lemma \ref{lem:base} implies that $x$ strictly in-dominates $Y$. But then $X \cup Y$ would be a weak hub, contradicting the maximality of $X$.
If now there exists $x\in Y$ which strictly in-dominates $X$, then consider any arc $ab$ where $a\in X \cap Y$ and $b\in X\setminus Y$, which must exist since $D[X]$ is strong. But now $Y$ is mixed because of the arcs $ab$ and $bx$, a contradiction. 
So $X$ is not strictly in-dominated and is similarly it cannot be strictly out-dominated. 
\qed
\medskip
 
\begin{claim}\label{clm:weakhubcover}
Every vertex belongs to some maximal weak hub. 
\end{claim}
{\bf Proof:}
Let $u$ be a vertex of $D$. It suffices to prove  $\{u\}$ is a weak hub. Assume by contradiction that this is not the case. By the definition of weak hub this forces $u^+=u^-$. If $u^{+} = V \setminus \{u\}$, $D$ is semicomplete with a universal vertex, a contradiction. If not, as $D$ is connected, there exist vertices $v \in u^{+}$ and $w \notin u^{+}$ such that $vw \in A$ or $wv \in A$. But then, as $u^{-} = u^{+}$, $u$ and $w$ are both  in-neighbours or both out-neighbours of $v$,  implying $w \in u^{-} = u^{+}$, a contradiction. 
\qed

The next claim describes the structure of the arcs linking maximal weak hubs.
\begin{claim}\label{clm:weakhubinteract}
Let $X$ and $Y$ be two distinct maximal weak hubs. Then either $X$ strictly in-dominates $Y$, or $X$ strictly out-dominates $Y$, or there is no arc between $X$ and $Y$.
\end{claim}
{\bf Proof:}
Assume that there exists $x$ in $X$ and $y$ in $Y$ such that $xy\in A(D)$. By Lemma \ref{lem:base} applied to $H=X$ and $z=y$, and because $X$ is not mixed by hypothesis, we have that $y$ strictly in-dominates $X$. But now by applying the same Lemma to the digraph obtained from $D$ by reversing all arcs (which is still locally out-semicomplete since $D$ is locally semicomplete), we get that $Y$ strictly in-dominates $X$.
\qed

Let us now consider a partition of $V$ into maximal weak hubs $X_{1}, X_{2} \dots X_{k}$ and let $D'$ the digraph obtained by contracting every set in the partition. By Claim \ref{clm:weakhubinteract}, $D'$ is a locally tournament oriented graph. Moreover, the in-neighbourhood (resp. the out-neighbourhood) of any vertex $x \in V(D')$ is acyclic. Indeed, if there were any such cycle $C$ on vertices of $D'$  corresponding to weak hubs $X_{i_{1}}, X_{i_{2}} \dots X_{i_{\ell}}$ with $\ell \geq 2 $, then $\cup_{j=1}^{\ell} X_{i_j}$ would induce a weak hub in-dominated (resp. out-dominated) by every vertex of the weak hub corresponding to $x$ in $D$, a contradiction to the maximality of the $X_i$. Due to Theorem \ref{thm:round}, $D'$ is a round digraph. 
\end{proof}

To prove Theorem \ref{thm:k12k21} it remains to prove the following lemma which deals with the case where there exists a maximal weak hub that is mixed.
\begin{lemma}\label{lem:mixed_weak hub}
Let $D$ be a connected locally semicomplete digraph such that there exists a maximal weak hub which is mixed. Then there exists a partition of $V(D)$ into four sets $E$, $F$, $G$ and $H$ such that:
        \begin{itemize}
            \item $E$ and $G$ are non empty, and one of $F$ and $H$ is non empty.
            \item $D[E]$, $D[F]$, $D[G]$ and $D[H]$ are semicomplete;
            \item $E$ strictly out-dominates $F$, $F$ out-dominates $G$, $G$ out-dominates $H$ and $H$ strictly out-dominates $E$.
                        \item For all $ x\in G,\ x^+\cap E\neq\emptyset$  and $x^-\cap E\neq\emptyset$

        \end{itemize}
\end{lemma}

\begin{proof}[of Lemma \ref{lem:mixed_weak hub}]
Let $X$ be a maximal weak hub which is mixed. Set $E:=X$ and:
\begin{eqnarray*}
G:= X^{M} &=& \{u \notin X \ \mid \exists a, b \in X \text{ such that } ua \in A, bu \in A \} \\
H:=X^{IN} &=& \{u \notin X \ \mid \exists a \in X \text{ such that } ua \in A \}\setminus X^M\\
F:=X^{OUT} &=& \{u \notin X\ \mid \exists a \in X \text{ such that } au \in A \}\setminus X^M \\
X^{NO} &=&V \setminus (X \cup X^{IN} \cup X^{OUT} \cup X^{M}) = \{u \in V \mid \forall v \in X, uv \notin A\text{ and } vu \notin A\}. 
\end{eqnarray*}

We are going to prove that $X^{NO} = \emptyset$ and that $E,F,G,H$ satisfy the output of the theorem. 

Since $X$ is strong, by Lemma \ref{lem:base}, $X$ strictly out-dominates $X^{OUT}$  and due to the same lemma applied to the digraph obtained by reversing all arcs of $D$, we have that $X^{IN}$ strictly out-dominates $X$. 

Since $X$ is assumed to be a weak hub that is mixed, we have that $X^{M}$ and $X$ are non-empty and one of $X^{IN}$ or $X^{OUT}$ is non-empty.

Let  $u \in X^{OUT}$ and $v \in X^{M}$ and let us prove that $uv \in A(D)$. By definition of $X^M$, there exists $w \in X$ such that $wv \in A(D)$. As $wu \in A(D)$, there must be an arc between $u$ and $v$. If $vu \in A(D)$,  then $X \cup \{v\}$ is a weak hub, contradicting the maximality of $X$. So $uv \in A(D)$ and thus $X^{OUT}$ out-dominates $X^M$. Similarly, $X^{IN}$ in-dominates $X^{M}$. 

Now, since $D$ is semicomplete and each of $X$, $X_{IN}$, $X^{OUT}$ and $X^M$ are included in the out-neighbourhood  of some vertex, we get that, $D[X]$, $D[X^{IN}]$, $D[X^{OUT}]$ and $D[X^M]$ are semicomplete. 

It remains to show that $X^{NO} = \emptyset$. In order to do so, we are going to prove that there is no arc between $X^{NO}$ and $X \cup X^{IN} \cup X^{OUT} \cup X^M$, contradicting the fact that $D$ is connected.  
There cannot be any arc $uv$ with $u \in X^{M} \cup X^{IN}$ and  $v \in X^{NO}$ as any such $u$ has an out-neighbour in $X$ and this would create an induced $S^+_2$. Similarly, there cannot be an arc $uv$ with $u \in X^{NO}$ and $v \in X^{M} \cup X^{OUT}$ as such a $v$ has an in-neighbour $w \in X$ and this would create an induced $S^-_2$. 
Let $w \in X^{M}$. 
If there is an arc $uv$ with $u \in X^{OUT}$ and $v \in X^{NO}$, then $\{u,v,w\}$ induces $S_2^+$. If there is an arc $uv$ with $u \in X^{NO}$ and $v \in X^{IN}$, then $\{u,v,w\}$ induces $S^-_2$. 
Altogether, we get that there is no arc between $X^{NO}$ and $X \cup X^{IN} \cup X^{OUT} \cup X^M$ as announced. 

\end{proof}

\subsection*{Two short applications of Theorem \ref{thm:k12k21}}
Let us mention here some applications of the previous theorem. 

A {\em 2-king} in a digraph is a vertex that can reach every other vertex by a directed path of length at most $2$. In \cite{WYW} it is proved that a locally semicomplete digraph that is not a blow-up of a round oriented graph (that is, not in the second case of Theorem \ref{thm:k12k21}) has a $2$-king. As we show now, this is a direct consequence of Theorem \ref{thm:2dic}. Indeed, when $D$ is semicomplete, it is easy to see that a vertex $x$ of maximum out-degree satisfies that every vertex in $x^-$ is either in $x^+$ or has an in-neighbour in $x^+$, and therefore $x$ is a $2$-king. And if $D$ is described by the third case, we distinguish two cases. 
\begin{itemize}
\item If $H$ is non empty, then we claim that a vertex $x$ in $G$ that is a $2$-king in $D[G]$ (which exists by the previous argument) is a $2$-king in $D$. Indeed, every vertex in $H$ is in $x^+$, and since there is all arcs from $H$ to $E$, there is a directed path of length at most $2$ from $x$ to each vertex of $E$. Finally, since $x$ has at least one out-neighbour in $E$ and there is all arcs from $E$ to $F$, there is a directed path of length $2$ from $x$ to each vertex in $F$. 
\item If $H$ is empty, then $F$ is not and in that case we take any vertex that is a $2$-king in $E$, and it is clearly a $2$-king of $D$. 
\end{itemize}

Another topic is pancyclicity, that is the property that a digraph contains a directed cycle for all possible lengths between $3$ and the number of vertices. We do not write the proof here but Theorem \ref{thm:k12k21} implies the result of \cite{BGGV} characterizing pancyclicity for locally semicomplete digraphs, since again the only non-easy case is when the digraph is neither semicomplete nor the blowup of  a round digraph, in which case it follows without too much effort because of the simplicity of this third case compared to the one of Theorem 3.2 in \cite{BGGV}.

\section{Perspectives}
In the context of Conjecture  \ref{conj:oriented}, it would be interesting to prove that $\{  H, S^+_2\}$ is heroic for every hero $H$. 
We prove it in the case where the hero consists of a $\ora{C_3}$ plus a vertex dominating it. In order to extend this partial result, one idea could be to use the structure theorem for the heroes of Berger et al. mentioned in Chapter~\ref{chpt:gyarfas}. Let us restate this theorem.


\begin{theorem}[\cite{hero}]\label{thm:hero}
A tournament is a hero if and only if it can be constructed by the following inductive rules:
\begin{itemize}
	\item $K_1$ is a hero.
	\item If $H_1$ and $H_2$ are heroes, then $H_1 \Rightarrow H_2$ is also a hero.
	\item If $H$ is a hero, then for every $k \ge 1$, the tournaments $\vec C_{3}(H, TT_k, K_1)$ and $\vec C_{3}(TT_k, H, K_1)$ both are heroes. 
\end{itemize}

\end{theorem}

In the light of this theorem, a first step to prove that $\{  H, S^+_2\}$ is heroic for every hero $H$ would be to show that, given two heroes $H_1$ and $H_2$, if $\{  H_i, \SD\}$ is heroic for $i=1,2$, then $\{  H_1\Ra H_2, \SD\}$ is also heroic. 
The case $H_1=K_1$ is solved in the prepublication \cite{S21} mentioned in the introduction, but we think that  a decomposition theorem for locally out-tournament oriented graphs, in the spirit of Theorem \ref{thm:k12k21} would  be of great help for the general case


Finally, we wonder whether Theorem \ref{thm:k12k21} can be applied to prove some open problems about locally semicomplete digraphs or locally tournament oriented graphs. One such problem is the famous Second Neighbourhood Conjecture due to Seymour, which is known for tournaments but not for locally tournament oriented graphs. The first two cases of Corollary \ref{thm:k12k21} can be dealt with without too much difficulty, but we were alas not able to deal with the last one.
\medskip 


\chapter{$(\dP6, \text{ triangle})$-free digraphs have bounded dichromatic number}\label{chpt:psix}

\begin{flushright}{\slshape    
This chapter is built upon a joint work \\
with Pierre Aboulker, Pierre Charbit \\
and Stéphan Thomassé, published in \cite{AACT22}}. \\ \medskip
\end{flushright}

\emph{In this chapter, we prove that $\dP6$ is a hero in triangle-free digraphs.}

\section{Introduction}

In this chapter, we are interested in the following conjecture:

\begin{conjecture}[\cite{ACN21}]
    Let $\ora{F}$ be an oriented forest and $k$ an integer. $\dic(\F(\ora{F}, TT_{k}))$ is finite.
\end{conjecture}

Which is equivalent to the following one:

\begin{conjecture}[\cite{ACN21}]
    Let $\ora{T}$ be an oriented tree and $k$ be an integer. The class of oriented graphs with no induced copy of $\ora{T}$ and clique number at most $k$ has bounded dichromatic number.
\end{conjecture}

We prove the following result, which corresponds to the above conjecture when $\ora{T} = \dP6$ and $k = 2$. 

\begin{theorem}\label{thm:main_thm}
For every $D \in \F(\dP6)$ with $\omega(D) \geq 2$, $ \dic(D) \leq 382$.
\end{theorem}

Note that we did not try to optimize the bound.

\section{Preliminaries}

Let $D$ be a digraph. 
For $X\subset V(D)$ we define $N^+(X)=\{y\in V(D)\setminus X, \exists x\in X \text{ such that } xy\in A(D)\}$ and $
N^-(X)=\{y\in V(D)\setminus X, \exists x\in X \text{ such that } yx\in A(D)\}$.
We say that $D$ is  {\em triangle-free} if $\omega(D) \leq 2$. 

A {\em trail} of a digraph $D$ is a sequence of vertices $x_1x_2\ldots x_p$ such that $x_ix_{i+1}\in A(D)$ for each $i<p$ and each arc is used once (but vertices can be used several times). 
It is \emph{closed} if $x_1 = x_p$ and its {\em length} is its number of arcs. We say {\em odd closed trail} for a closed trail of odd length. 




A set of vertices $X$ is \emph{dipolar} if for every $x \in X$, $N^+(x) \subseteq X$ or $N^-(x) \subseteq X$. 
This notion was first introduced in \cite{ACN21} under the name "nice set" and has been renamed "dipolar set" in \cite{CMPRS22}. 
The main tool using dipolar sets is the following lemma. We include its proof because it is short and enlightening for people unfamiliar with the dichromatic number. 

\begin{lemma}[Lemma 17 in~\cite{ACN21}]\label{lem:dipolar}
Let $\mc C$ be a class of digraphs closed under taking induced subdigraph. Suppose that there exists a constant $c$ such that each digraph $D \in \mathcal C$ has a dipolar set $S$ such that $\dic(S) \leq c$. Then $\dic(\mc C) \leq 2c$. 
\end{lemma}

\begin{proof}
Let $D \in \mc C$ be a minimal counter example, that is: $\dic(D)=2c+1$ and for every proper  subdigraph $H$ of $D$, $\dic(H)\le 2c$. 
By the hypothesis, $D$ admits a dipolar set $S$,  such that $\dic(S) \leq c$. 
Set $S^+ = \{x \in S \mid N^-(x) \subseteq S\}$ and $S^- = \{x \in S \mid N^+(x) \subseteq S\}$. By definition of a dipolar set, $S= S^+ \cup S^-$. 

The key observation is that any directed cycle that intersects $S$ and $V(D) \sm S$  intersects  both $S^+$ and $S^-$. 
Hence, by the minimality of $D$, we can dicolour  $V(D) \sm S$ with $2c$ colours. We can then extend this dicolouring to $D$ by using colours $1, \dots, c$ for $S^+$ and $c+1, \dots, 2c$ for $S^- \setminus S^+$. 
\end{proof}

The strategy to prove our result is to show that every digraph in our class has a dipolar set with a dichromatic number at most $191$ and then apply Lemma~\ref{lem:dipolar}. 
The next two results give simple techniques to bound the dichromatic number of a digraph, they  will  be extensively used to prove that the dichromatic number of some dipolar set is bounded. The first one is probably well known but we don't have any reference for it, the proof is very short.

\begin{lemma}\label{lem:odd_cycle}
If a digraph $D$ does not contain odd directed cycles as subdigraphs, then $\dic(D) \leq 2$. 
\end{lemma}

\begin{proof}
Let $D $ be a digraph with no odd directed cycle and since the dichromatic number of a digraph is the maximum of the dichromatic number of its strong components, we can assume without loss of generality that $D$ is strongly connected. In that case, we prove that the underlying graph $G$ of $D$ is in fact bipartite. Assume by contradiction $G$ contains an odd cycle $C = c_1\ra c_2\ra  \dots\ra  c_{2k+1}\ra c_1$. For $i = 1, \dots, {2k + 1}$, let $P_i$ be a shortest directed path from $c_i$ to $c_{i + 1}$ (indices being taken modulo $2k + 1$). Observe that either $P_i = c_ic_{i+1}$, or $c_{i+1}c_i \in A(D)$, in which case $P_i$ has odd length, for otherwise $P_i \cup \{c_{i+1}c_{i}\}$ is an odd directed cycle.  Hence the union of the $P_i$ for $i = 1 \dots 2k + 1$ forms a closed odd trail, which contains an odd directed cycle, a contradiction. 
\end{proof}

The next result is the dichromatic version of the celebrated Gallai-Roy-Vitaver theorem asserting that the chromatic number is upper-bounded by the largest length of a directed path. In a nutshell: the dichromatic number is upper-bounded by the largest length of a directed path of some feedback arc set.

\begin{proposition}\label{prop:noPk}
Let $D$ be a digraph. Given a total ordering of the vertices of $D$, we say that an arc $xy$ is forward if $x$ precedes $y$ in this ordering, and backwards otherwise. The two following propositions are equivalent 
\begin{itemize}
    \item $\dic(D)\leq k$
    \item There exists an ordering of the vertices of $D$ such that there exists no directed path on $k+1$ vertices consisting only of backward arcs.
\end{itemize}
\end{proposition}

\begin{proof}

One direction is easy: if $\dic(D)\leq k$ then there exists a partition $(C_1,C_2,\ldots C_k)$ of $V(D)$ with $C_i$ inducing an acyclic digraph. We construct an order on $V(D)$ by putting all vertices of $C_i$ before all vertices of $C_{i+1}$ for each $i$ and within each class we use a topological sort. It is clear that in the resulting order, there can be no patch on more than $k$ vertices where all  arcs go backwards since a backward arc goes from one class to a previous one.

For the converse direction, assume that $D$ has an ordering on its vertices such that there exists no directed path on $k+1$ vertices consisting only of backward arcs and let us prove that $D$ is $k$-dicolourable. 
For every $x \in V(D)$, define $f(x)$ the maximum number of vertices in a path consisting only of backward arcs and ending in $x$. By definition $1\leq f(x)\leq k$. Define $C_i=f^{-1}(i)$ and let us prove that $C_i$ does not contain any backward arc. Assume by contradiction $xy$ is such an arc. Then there exists a path on $i$ vertices ending in $x$ consisting only of backward arcs, which implies that $f(y)\geq i+1$, a contradiction. So each $C_i$ induces an acyclic digraph, and thus $\dic(D)\leq k$.
\end{proof}

The last lemma of this section is used to find induced directed paths.

\begin{lemma}\label{lem:two_push}
Let $D$ be a triangle-free digraph, $C$ a (not necessarily induced) odd directed cycle of $D$ and $a \in N(C)$. Then there exists consecutive vertices $b\rightarrow c\rightarrow d$ of $C$ such that
\begin{itemize}
    \item either $a\ra b\ra c\ra d$ is an induced $\dP4$,
    \item or $b\ra c\ra d\ra a$  is an induced $\dP4$,
    \item or $a\ra b\ra c\ra d$ is a $C_4$ (in particular, $a\in N^{+}(C)\cap N^{-}(C)).$
\end{itemize}
\end{lemma}

\begin{proof} Assume $a\in N^-(C)$. Let us denote by $x_1,\ldots,x_{2k+1}$ the vertices of $C$ (i.e. $\forall i\leq 2k$, $x_ix_{i+1}\in A(D)$ and $x_{2k+1}x_1\in A(D)$). Assume without loss of generality that $ax_{1}\in A(D)$. 
Let $1 \leq p \leq k$ be the maximum integer such that $ax_{2p+1} \in A(D)$.
Since the digraph is triangle-free, $ax_{2k+1}\notin A(D)$, so $p \leq k$. 
 It is straightforward to see that $b=x_{2p+1}$, $c=x_{2p+2}$, $d=x_{2p+3}$  satisfies either the first or third item of the lemma. By reversing the arcs of the digraph, the same proof works if $a\in N^+(C)$. 
\end{proof}

We will often use this lemma the following way: if $a\in N^+(C)\setminus N^-(C)$ (resp. $a\in N^-(C)\setminus N^+(C)$), then the first (resp. the second) output holds.


\section{Proof of Theorem~\ref{thm:main_thm}}
For a subset $X$ of vertices, we define recursively the sets $N_k^+(X)$, $N_k^{-}(X)$ and $N_k(X)$ by $N_{0}^+(X)=N_{0}^-(X)=N_0(X)=X$, and for $k\geq 1$:
\begin{alignat*}{2}
&N_{k}^+(X)&&=N^+(N_{k-1}^+(X))\setminus \bigcup_{i<k} N_i(X) \\
&N_{k}^-(X)&&=N^-(N_{k-1}^-(X))\setminus \bigcup_{i<k} N_i(X) \\
&N_{k}(X)&&=N_{k}^+(X)\cup N_k^-(X)
\end{alignat*}
We gather in the following claim several straightforward facts that we will use in the proof.
\begin{claim}\label{clm:basicsets} For any $X\subset V$, the following hold
\begin{enumerate}
    \item\label{it:1} $N^{+}_1(X) = N^+(X)$, $N^{-}_1(X) = N^-(X)$ and $N_1(X) = N^+(X)\cup N^-(X)$
    \item\label{it:2} There are no arcs between $X$ and  $N_k(X)$ for $k>1$.
    \item\label{it:3} If $x\in N_{k-1}(X)$, then either $N^+(x)\subseteq \bigcup_{i\leq k} N_i(X)$ or $N^-(x)\subseteq \bigcup_{i\leq k} N_i(X)$.  
 
    \item\label{it:4} If $x\in N_{k}^+(X)$ (resp $ N_{k}^-(X)$), there exists a directed path $x_0x_1\ldots x_k$ (resp. $x_kx_{k-1}\ldots x_0$)  such that $x_k=x$ and $x_i\in N_{i}^+(X)$ for every $i\geq 0$.
    
\end{enumerate}
\end{claim}
Items 1), 2) and 3) follow from the definition and 4) is easy to prove by induction on $k$.\\

Let now $D$ be a triangle-free digraph in $\F(\dP6)$. 
Let $C=x_1x_2\dots x_{2k+1}x_1$ be a  (not necessarily induced) odd directed cycle of $D$ of minimum length (we may assume it exists by Lemma~\ref{lem:odd_cycle}). 
During the proof, for simplicity, we write $C$ for $V(C)$, $D[C]$ for $D[V(C)]$ and $N_{k}(C)$ for $N_{k}(V(C))$.

We are going to prove that the set 
$$S= C \cup N(C) \cup N_{2}(C) \cup N_{3}(C)$$ 
is dipolar and has dichromatic number at most $191$, which implies Theorem~\ref{thm:main_thm} by Lemma~\ref{lem:dipolar}. 

\begin{claim}\label{clm:Sdipo}
$S$ is dipolar. Moreover, $\dic(N_{3}(C)) \leq 2$. 
\end{claim}

\begin{subproof}
To prove that $S$ is dipolar, we need to prove that for every vertex $x$ in $S$, either $N^+(x)$ or $N^-(x)$ is contained in $S$. Note that by Claim \ref{clm:basicsets} item 3, this is trivial if $x\in C\cup N_1(C) \cup N_2(C)$. 

Assume now that  $x \in N^{+}_3(C)$ and let us prove that $N^+(x) \subseteq N(C) \cup N_{2}(C)$, which will imply both parts of the claim since this proves that $N^+_3(C)$ is an independent set.  

By Claim \ref{clm:basicsets} item 4, there exists a directed path $x_0\ra x_1\ra x_2\ra x_3$, where $x_3=x$ and $x_i\in N^+_i(C)$. If $x_1 \in N^{+}(C) \setminus N^{-}(C)$, then, by Lemma~\ref{lem:two_push}, there exists $a,b,c \in C$ such that $abcx_1$ is an induced $\dP4$. Since there is no arc between $C$ and $N_{2}(C) \cup N_{3}(C)$ (by Claim \ref{clm:basicsets} item \ref{it:2}) and $D$ is triangle-free, $a\ra b\ra c\ra x_1\ra x_2\ra x_3$ is an induced $\dP6$, a contradiction. 

So we can assume $x_1 \in N^{+}(C) \cap N^{-}(C)$. Consider $y\in N^+(x)$, and let us prove that $y\in N(C) \cup N_2(C)$. Let $t$ be an in-neighbour of $x_0$ in $C$ and observe that $t\ra x_0\ra x_1\ra x_2\ra x_3\ra y$ is a $\dP6$ and the only way for it not to be induced (because of (Claim \ref{clm:basicsets} item \ref{it:2})) is that $y$ is adjacent with one of $\{t,x_0,x_1\}$. If $y$ is adjacent with $t$ or $x_0$, then $y \in N(C)$. If $y$ is adjacent with $x_1$, and since $x_1 \in N^{+}(C) \cap N^{-}(C)$, we get that $y \in  N_{2}(C)$. We thus have proven that $y\in N(C) \cup N_2(C)$. Similarly, if $x \in N^{-}_3(C)$, then $N^-(x) \subseteq N(C) \cup N_{2}(C)$, which concludes the proof of this claim.

\end{subproof}

\begin{claim}\label{clm:C_bounded}
$\dic(D[C]) \leq 3$. 
\end{claim}

\begin{subproof}
By minimality of $C$, removing a vertex from $C$ yields a digraph with no odd directed cycle, which thus has a dichromatic number at most $2$ by Lemma~\ref{lem:odd_cycle}.
\end{subproof}

\begin{claim} \label{clm:N+moinsN-_bounded}
$\dic(N^{+}(C) \setminus N^{-}(C)) \leq 4$ and  $\dic(N^{-}(C) \setminus N^{+}(C)) \leq 4$. 
\end{claim}

\begin{subproof}
Let us prove that $\dic(N^{+}(C) \setminus N^{-}(C)) \leq 4$.  
We first prove that $N^+(x_1) \cup N^+(x_2)$ intersects all odd directed  cycles of $N^{+}(C) \setminus N^{-}(C)$. Suppose that it is not the case, and let $C'$ be such an odd directed  cycle. 
Let $i\geq 3$ be minimum such that $x_i$ has an out-neighbour in $C'$ (so that $x_1, \dots, x_{i-1}$ don't). Since $C'\subset N^+(C)\setminus N^-(C)$, $x_i$ does not have an in-neighbour in $C'$, so by Lemma~\ref{lem:two_push} applied to $C'$, there are $3$ consecutive vertices $a, b, c$ of $C'$, such that $x_i \ra a \ra b \ra c$ is an induced $\dP4$. By the choice of $i$, we then have that $x_{i-2} \ra x_{i-1} \ra x_i \ra a \ra b \ra c$ is an induced $\dP6$, a contradiction.
Now, $N^{+}(C) \setminus N^{-}(C)$ can be partitioned into two stable sets and a digraph with no odd directed cycle, and thus be $4$-dicoloured.  
By directional duality,  $\dic(N^{-}(C) \setminus N^{+}(C)) \leq 4$.
\end{subproof}

\begin{claim}\label{clm:N2+moinsN2-_bounded}
   $\dic(N^{+}_2(C) \setminus N^{-}_2(C)) \leq 2$ and $\dic(N^{-}_2(C) \setminus N^{+}_2(C)) \leq 2$. 
\end{claim}

\begin{subproof}
We prove that $\dic(N^{+}_2(C) \setminus N^{-}_2(C)) \leq 2$. Assume by contradiction this is not the case so that by Lemma~\ref{lem:odd_cycle} we get an odd directed cycle $C'$ in $N^{+}_2(C) \setminus N^{-}_2(C)$ .  Let $u$ be a vertex in $N^{+}(C) \cap N^-(C')$, which is non empty by definition of $N^{+}_2(C)$.

If $u \in N^{+}(C) \setminus N^{-}(C)$, then by Lemma~\ref{lem:two_push}, there exist $a,b,c \in C$ such that $a \ra b \ra c \ra u$ is an induced $\dP4$, which along with a vertex $v \in N^+(u) \cap V(C')$ and the out-neighbour of $v$ in $V(C')$ forms an induced $\dP6$, a contradiction (remember that by Claim \ref{clm:basicsets} Item \ref{it:2}, there is no arc between $C$ and $C'$). 

Thus $u \in N^{+}(C) \cap N^{-}(C)$ and since $V(C')$ is disjoint from $N^{-}_2(C)$, $u$ has no in-neighbour in $V(C')$. Hence, by Lemma~\ref{lem:two_push} applied on $C'$, there exist $a,b,c \in V(C')$ such that $u\ra a \ra b \ra c$ is an induced $\dP4$, which along with any $v \in N^-(u) \cap C$ and the in-neighbour of $v$ in $C$ forms an induced $\dP6$, a contradiction.
\end{subproof}

\begin{claim}\label{clm:hard}
$\dic(N^{+}(C) \cap N^{-}(C)) \leq 30$. Moreover, if for every $x \in C$, both $N_2^+(x)$ and $N_2^-(x)$ are  stable sets, then $\dic(N_2^+(C) \cap N_2^-(C)) \leq 30$. 
\end{claim}

\begin{subproof} 
The same proof works for the two assertions of the claim. Let $\ell \in \{1,2\}$ and observe that, by hypotheses (triangle-free for $\ell=1$, or the assumption of the second sentence for $\ell=2$), for every $x \in C$, both $N^{\ell +}(x)$ and $N^{\ell -}(x)$ are stable sets.

Let $X = (N^{\ell+}(C) \cap N^{\ell-}(C)) \setminus N^{\ell}(\{x_1, \dots, x_6\})$. It is enough to prove that $\dic(X) \leq 30-12= 18$. 


For each vertex $v  \in X$, choose (arbitrarily) a vertex $x_i$ (resp. $x_j$) in $C$ such that there is a directed path of length $l$ from $v$ to $x_i$ (resp. from $x_j$ to $v$). Set $out(v)=i$ and $in(v)=j$ so that we define two functions $out$ and $in$ from $X$ to $\{1,\ldots,2k+1\}$.

 In the case where $\ell = 2$, let  $p^{+}_{v}$ (resp. $p^{-}_{v}$) be a vertex such that $v \ra p^+_v \ra x_{out(v)}$ (resp. $x_{in(v)} \ra p^-_v \ra v$). In the rest of the proof, $v \ra p^+_v \ra x_{out(v)}$ is understood as $v \ra x_{out(v)}$ in the case where $\ell = 1$.

 For $i \in [0,5]$, let  $X_{i} = \{v \in X \mid out(v) = i\mod 6 \}$ and then define $X_{i,\geq} = \{v \in X_{i} \mid out(v) \geq in(v)\}$ and $X_{i,<} = \{v \in X_{i} \mid out(v) < in(v)\}$. It is enough to prove that $\dic(X_{i,\geq}) \leq 2$ and $\dic(X_{i,<})\leq 1$ for $i=0, \dots, 5$. 

So now $i$ is fixed and we define a total order $\prec$ on $X_i$ the  following way: we say first that $u \prec v$ when $out(u) < out(v)$ and then extend arbitrarily this partial ordering to a total ordering of $X_i$. 

We first prove that $\dic(X_{i,\geq}) \leq 2$ using Proposition~\ref{prop:noPk} applied to the reversal of $\prec$ defined above. Suppose then by contradiction that there exist $a,b,c \in X_{i,\geq}$ such that $a \prec b \prec c$ and $ab, bc \in A(D)$. Since $N^{\ell -}(x)$ is a stable set for every $x \in C$, $out(a)\neq out(b)$ and $out(b) \neq out(c)$ and thus
$$out(c) \geq 6 + out(b) \geq 12 +out(a) \geq 12 + in(a)$$
If $in(a)$ has the same parity as $out(a)$ (and thus as $out(b)$ and $out(c)$), then 
$x_{1} \ra x_2   \rightarrow \dots \rightarrow x_{in(a)} \rightarrow p^{-}_{a} \rightarrow a \rightarrow b \rightarrow c \rightarrow p^{+}_{c} \rightarrow x_{out(c)} \rightarrow \dots \rightarrow x_{2k+1} \rightarrow x_{1}$ 
is an odd  closed trail (it does need to be a directed cycle because $p^-_a = p^+_c$ is possible)  and otherwise, 
$x_{1} \ra x_2  \rightarrow \dots \rightarrow x_{in(a)} \rightarrow p^{-}_{a} \rightarrow a \rightarrow b \rightarrow p^{+}_{b} \rightarrow x_{out(b)} \rightarrow \dots \rightarrow x_{2k+1} \rightarrow x_{1}$ 
is an odd directed cycle. In both cases, we get an odd directed trail that has strictly fewer vertices than $C$, and since an odd  closed trail contains an odd directed cycle, we get our contradiction. Thus $\dic(X_{i,\geq}) \leq 2$.

We now prove that $\dic(X_{i,<})\leq 1$. 
Suppose that there exists $a,b \in X_{i,<}$ such that $b \prec a$ and $ab \in A(D)$. 
Thus $out(b)+6 \leq out(a) < in(a) $. 
If $out(a)$ and $in(a)$ do not have the same parity, 
then $x_{out(a)} \rightarrow  x_{out(a)+1} \rightarrow \dots \rightarrow x_{in(a)} \rightarrow p^+_a \rightarrow a \rightarrow p^-_a \rightarrow x_{out(a)}$ is an odd  closed trail.  
Otherwise $out(a)$ and thus $out(b)$ have the same parity as $in(a)$, and then $x_{out(b)} \rightarrow  \dots \rightarrow x_{in(a)} \rightarrow p^-_a \rightarrow a \rightarrow b \rightarrow p^+_b \rightarrow x_{out(b)} $ is an odd directed cycle.  
In both cases, it has strictly fewer vertices than $C$, a contradiction.  
Thus $\dic(X_{i,<}) \leq 1$ by Proposition~\ref{prop:noPk}. 
\end{subproof}

Let $\dC{3,2}$ be the digraph with vertices $u,v_1, v_2, w_1, w_2$ and arcs $uv_1, v_1v_2, v_2w_2, uw_1, w_1w_2$. Observe that if $G \in \F(\dC{3,2})$, then for every $x \in V(G)$, $N^{+}_2(x)$ and $N^{-}_2(x)$ are stable sets. Hence, by the previous claims (all of them), we get that for every triangle-free digraph $G \in \F(\dP6, \dC{3,2})$,
the set $Q \cup N(Q) \cup N_{2}(Q) \cup N_{3}(Q)$, where $Q$ is an odd directed cycle of $G$ of minimum length, is dipolar and has a dichromatic number at most $3 + 4 + 4+ 2 + 2 +1 + 1 + 30 + 30=77$. Hence, by Lemma~\ref{lem:dipolar} we get that:

\begin{claim}\label{clm:C32_bounded}
Triangle-free digraphs in $\F(\dP6, \dC{3,2})$ have dichromatic number at most $144$. 
\end{claim}

We are now able to prove the last bit of the proof. 

\begin{claim}\label{clm:N2+capN2-}
$\dic(N^{+}_2(C) \cap N^{-}_2(C)) \leq 144$.
\end{claim}

\begin{subproof}
By Claim~\ref{clm:C32_bounded}, we may assume that $N^{+}_2(C) \cap N^{-}_2(C)$ contains  $\dC{3,2}$ as an induced subdigraph. 
Thus there exists $u,v_1,v_2,w_1,w_2 \in N^{+}_2(C) \cap N^{-}_2(C)$ such that $uv_1, uw_1, v_1v_2, w_1w_2, v_2w_2 \in A(D)$. 
Moreover, there exists $r,s \in C$, and $t \in N^+(C)$ such that $rs,st,tu \in A(D)$. Now, since $r \ra s \ra t \ra u \ra v_1 \ra v_2$ is not induced, $t$ and $v_2$ are adjacent, and since $r \ra s \ra t \ra u \ra w_1 \ra w_2$ is not induced, $t$ and $w_2$ are adjacent. Hence $t,v_2,w_2$ forms a triangle, a contradiction. 
\end{subproof}

Altogether, we get that $\dic(S) \leq 3 + 4 + 4 + 30 +2 + 2 + 144 +1 +1 = 191$, and thus $\dic(D) \leq 382$. 

\chapter{Heroes in orientations of chordal graphs}\label{chpt:chordal}

\begin{flushright}{\slshape    
This chapter is built upon a joint work with Pierre \\
Aboulker and Rapha\"el Steiner, published in \cite{AAS22}}. \\ \medskip
\end{flushright}

\emph{In this chapter, we characterize heroes in orientations of chordal graphs.}

\section{Introduction}
A classical theorem of Dirac~\cite{D61} states that all chordal graphs can be obtained by iteratively glueing some complete graphs along cliques (see Section~\ref{sec:chordal} for a formal statement). 
 This implies for undirected graph colouring that chordal graphs are perfect graphs, and thus their chromatic numbers and colouring properties are determined solely by the (largest) cliques contained in them.
 It is then natural to ask whether also for the dichromatic number of \ocgs important characteristics are determined by the largest dichromatic numbers of their subtournaments. In particular, it is a natural problem to characterise the heroes in \ocgs and to see whether they are the same as for tournaments.
 
 In this chapter, we find surprising answers to the above questions. First, there are very few heroes in \ocgs and as our main contribution, we completely describe these digraphs, as follows. 

\begin{theorem}\label{thm:main_chordal}
A digraph $H$ is a hero in \ocgs if and only if $H$ is a transitive tournament or isomorphic to $\vec C_{3}(1,1,k)$ for some integer $k \ge 1$.
\end{theorem} 

Secondly, our constructions in the proof of the above characterisation exhibit \ocgs with arbitrarily large dichromatic numbers all of whose subtournaments are $2$-colourable, showing that in contrast to chromatic number the dichromatic number of an \ocg heavily depends on its global structure and not only on the cliques contained in it.

We denote by $\vec C_3$ the directed cycle on three vertices, also called \emph{directed triangle} (observe that $\vec C_3 = \vec C_{3}(1,1,1)$). 
It is easy to see that a hero in tournaments is either a transitive tournament, or isomorphic to  $\vec C_{3}(1,1,k)$ for some integer $k \ge 1$, or it contains one of the heroes $\vec C_{3}(1,2,2)$, $K_1 \Ra \vec C_3$ or $\vec C_3 \Ra K_1$ as a subtournament. Moreover, since  reversing all arcs of a $(\vec C_3 \Ra K_1)$-free \ocg results in a $ (K_1 \Ra \vec C_3)$-free \ocg and does not change the dichromatic number, proving that $\vec C_3 \Ra K_1$ is not a hero in \ocgs implies that $K_1 \Ra \vec C_3$ is not either. 
Hence, to prove Theorem~\ref{thm:main_chordal}, it will be enough to prove the following: 
\begin{itemize}
    \item Transitive tournaments and $\vec C_{3}(1,1,k)$ for $k \ge 1$ are heroes in \ocgs. This is done in Section~\ref{sec:heroes}.
    \item  $\vec C_{3}(1,2,2)$ and $\vec C_3 \Ra K_1$ are not heroes in \ocgs. This is respectively done in subsections~\ref{subsec:122} and~\ref{subsec:1C3}. 
\end{itemize}

\medskip

\section{Proofs}

\subsection{A few words on chordal graphs}\label{sec:chordal}

Chordal graphs have been studied for the first time in the pioneer work of Dirac~\cite{D61} who proved that every chordal graph $G$ is either a complete graph, or contains a clique $S$ such that $G \setminus S$ is disconnected. This easily implies that all chordal graphs can be obtained by glueing complete graphs along cliques. From this point of view, it is natural to try to generalize results on tournaments to orientations of chordal graphs. 

In this chapter, we will use the two following well-known properties of chordal graphs. The first one formalizes the notion of `glueing along a clique'.

\begin{lemma}\cite{D61}\label{lem:chordal_intersection_on_clique}
Let $G_1$ and $G_2$ be two chordal graphs such that $V(G_1) \cap V(G_2)$ induces a complete graph both in $G_1$ and $G_2$. Then their union is a chordal graph.
\end{lemma}

A vertex is \emph{simplicial} if its neighbourhood induces a complete graph. 
\begin{lemma}\cite{D61}\label{lem:simplicial}
Every chordal graph has a simplicial vertex. 
\end{lemma}

\subsection{$\vec C_{3}(1,1,k)$ and transitive tournaments are heroes in \ocgs}\label{sec:heroes}

\begin{theorem}[Stearns, \cite{stearns}]\label{thm:stearns}
For each integer $n \geq 1$, a tournament with at least $2^{n-1}$ vertices contains a transitive tournament with $n$ vertices.
\end{theorem}

In the following, we define the \emph{triangle degree} of a vertex $x$ in a digraph $G$ as the maximum size of a collection of directed triangles that pairwise share the common vertex $x$ but no further vertices.

\begin{lemma}\label{lem:bounded_triangle_degree}
Every vertex of a $\vec C_{3}(1,1,k)$-free tournament has triangle degree less than $2^{2k-2}$. 
\end{lemma}

\begin{proof}
Let $G$ be a $\vec C_{3}(1,1,k)$-free tournament and $x$ a vertex of $G$. 
Assume for contradiction that $x$ has triangle degree at least $2^{2k-2}$, that is, there exist pairwise distinct vertices $a_1, b_1, \dots, a_{2^{2k-2}}, b_{2^{2k-2}}$ such that $x \ra a_i \ra b_i \ra x$. 
By Theorem~\ref{thm:stearns} we can find a transitive tournament $T$ in $G[\{a_1, \dots, a_{2^{2k-2}}\}]$ of order at least $2k-1$. Up to renaming the vertices, we may assume that $T=G[\{a_1, \dots, a_{2k-1}\}]$ and that $a_1, \dots, a_{2k-1}$ is the topological ordering of $T$.   
Then look at $b_{2k-1}$. Set $b_{2k-1}^+ \cap T  = T^+$ and $b_{2k-1}^- \cap T = T^-$ and observe that $V(T) = T^+ \cup T^-$ since we are in a tournament. 
If $|T^+| \geq k$, then $T^+$ together with $b_{2k-1}$ and $a_{2k-1}$ contains a  $\vec C_{3}(1,1,k)$, a contradiction. So $|T^+|\leq k-1$. 
If $|T^-| \geq k$, then $T^-$ together with $b_{2k-1}$ and $x$ contains $\vec C_{3}(1,1,k)$, a contradiction. So $|T^+|\leq k-1$. Hence, $|V(T)| \leq 2k-2$, a contradiction. 
\end{proof}

\begin{theorem}
Transitive tournaments and $\vec C_{3}(1,1,k)$ are heroes in \ocgs. 
More precisely, $TT_k$-free \ocgs have dichromatic number at most $2^{k-1}-1$ and  $\vec C_{3}(1,1,k)$-free \ocgs have dichromatic number at most $2^{2k-2}$.  
\end{theorem}

\begin{proof}
A $TT_k$-free \ocg has no subtournament of order at least $2^{k-1}-1$ by Theorem~\ref{thm:stearns}, and since chordal graphs are perfect graphs, its underlying graph has chromatic number at most $2^{k-1}-1$ and thus dichromatic number at most $2^{k-1}-1$.

We now prove that $\vec C_{3}(1,1,k)$-free \ocgs have dichromatic number at most $2^{2k-2}$.  We proceed by induction on the number of vertices. 
Let $G$ be a $\vec C_{3}(1,1,k)$-free \ocg. Let $x$ be a simplicial vertex of the underlying graph of $G$.  Note that the triangle degree of $x$ in $G$ is equal to the triangle degree of $x$ in the subtournament $G[\{x\} \cup x^+ \cup x^-]$, which by Lemma~\ref{lem:bounded_triangle_degree} is less than $2^{2k-2}$.

We can then find a dicolouring of $G \setminus x$ with $2^{2k-2}$ colours by induction,  and since the triangle degree of $x$ in $G$ is less than $2^{2k-2}$, there is a colour $i \in \{1, \dots, 2^{2k-2}\}$ such that assigning $i$ to $x$ does not produce a monochromatic directed triangle. We thus get a dicolouring of $G$: if there existed a monochromatic directed cycle in this dicolouring of $G$, there would also have to exist an \emph{induced} monochromatic directed cycle, and since all induced cycles in $G$ have length $3$, this cycle would have to be a monochromatic directed triangle. However, such a triangle does not exist, neither through $x$ nor in $G\setminus x$ (by inductive assumption). 
\end{proof}

\subsection{Constructions} \label{sec:nonheroes}

\subsubsection{$\vec C_{3}(1,2,2)$ is not a hero in orientations of chordal graphs}\label{subsec:122}

In this subsection, we present a construction of \ocgs with arbitrarily large dichromatic number but containing no copy of $\vec C_{3}(1,2,2)$. 

\begin{theorem}\label{thm:ce_122}
$\vec C_{3}(1,2,2)$ is not a hero in \ocgs.
\end{theorem}

\begin{proof}
We inductively construct a sequence $(G_k)_{k \in \mathbb{N}}$ of digraphs such that for each $k \ge 1$, the digraph $G_k$ is an orientation of a chordal graph with no copy of $\vec C_{3}(1,2,2)$ satisfying $\vec{\chi}(G_k)=k$.

Let $G_1$ be the digraph on one vertex, and having defined $G_{k}$, define $G_{k+1}$ as follows. 
Start with a copy $T$ of $TT_{k+1}$, and for each arc $e=uv$ of $T$, create a distinct copy $G_{k}^e$ of $G_{k}$ (vertex-disjoint for different choices of the arc $e \in A(T)$, and all vertex-disjoint from $T$). Next, for each $e=uv \in A(T)$, we add all the arcs $vy$ and $yu$ for every $y \in V(G_k^e)$.  
This completes the description of the digraph $G_{k+1}$.

For every arc $e=uv \in A(T)$, consider the underlying graph of $G_{k+1}[\{u,v\} \cup V(G_k^e)]$. By definition, this graph is obtained from the chordal underlying graph of $G_k^e$ by adding an adjacent pair of universal vertices. Since the addition of universal vertices preservers the chordality of a graph, we can see that the underlying graph of $G_{k+1}[\{u,v\} \cup V(G_k^e)]$ is chordal, for every choice of $e$. Since $T$ and $G_{k+1}[\{u,v\} \cup V(G_k^e)]$ intersect in  the clique $\{u,v\}$, we may now repeatedly apply Lemma~\ref{lem:chordal_intersection_on_clique} to see that $G_{k+1}$ is still an \ocg.

Next, let us prove that $G_{k+1}$ does not contain $\vec C_{3}(1,2,2)$. Assume towards a contradiction that $G_{k+1}$ contains a copy of $\vec C_{3}(1,2,2)$, induced by the set of vertices $A \subseteq V(G_{k+1})$. Since the copies $G_{k}^e, e \in A(T)$ of $G_k$ are vertex disjoint and have no connecting arcs, and since $A$ induces a tournament, $A$ intersects at most one of the vertex sets of these copies. Let $f=xy \in A(T)$ be a fixed edge such that $A \subseteq V(T) \cup V(G_k^f)$. 

Since $G_k^f$ is $\vec C_{3}(1,2,2)$-free by inductive assumption, it follows that $A$ intersects $V(T)$ in at least one vertex. 
As $\vec C_{3}(1,2,2)$ is not acyclic, $A$ is also not fully contained in $V(T)$, and thus $A \cap V(G_k^f) \neq \emptyset$. 

The argument above implies that $A \cap V(T) \subseteq \{x,y\}$, as $x$ and $y$ are the only vertices in $V(T)$ whose neighbourhoods in $G_{k+1}$ intersect $V(G_k^f)$. 
In fact, we must have $A \cap V(T)=\{x,y\}$, for if $|A\cap V(T)|=1$ then either $x$ would form a sink in $G_{k+1}[A]$ or $y$ would form a source in $G_{k+1}[A]$, both of which are impossible, since $G_{k+1}[A] \simeq \vec C_{3}(1,2,2)$ is strongly connected. Note that by definition of $G_{k+1}$, every vertex in $A \setminus \{x,y\} \subseteq V(G_k^f)$ must form a directed triangle together with the arc $xy$. 

But $A$ induces $\vec C_{3}(1,2,2)$ in $G_{k}$ and there is no arc in $\vec C_{3}(1,2,2)$ forming a directed triangle with every other vertex, as there is no arc from the only vertex of $\vec C_{3}(1,2,2)$ of outdegree $1$ to the only vertex of $\vec C_{3}(1,2,2)$ of indegree $1$, a contradiction.  This shows that $G_{k+1}$ is indeed $\vec C_{3}(1,2,2)$-free.

Finally, let us prove that $\dic(G_k)=k+1$. A $(k+1)$-dicolouring of $G_k$ can easily be obtained by piecing together individual $k$-dicolourings of the copies $G_k^e, e \in A(T)$ of $G_k$ and assigning to all vertices in the transitive tournament $T$ a new $(k+1)^{th}$ colour not appearing in the copies. To show that $\dic(G_{k+1})>k$, assume towards a contradiction that $G_k$ admits a $k$-dicolouring $c:V(G_{k+1}) \rightarrow \{1,\ldots,k\}$. 
Then, since $T$ is a clique on $k+1$ vertices, there exists a monochromatic arc $e=uv$. Let $i \in \{1,\ldots,k\}$ be such that $c(u)=c(v)=i$. Then since $\dic(G_k)=k$, the copy $G_k^e$ of $G_k$ glued to $uv$ must use all $k$ colours in the dicolouring induced on it by $c$, and in particular, there exists some $w \in V(G_k^e)$ such that $c(w)=i$. Now, however, the directed triangle $x \rightarrow y \rightarrow w \rightarrow x$ is monochromatic, a contradiction to our choice of $c$. This completes the proof that $\dic(G_{k+1})=k+1$, and hence the proof of the theorem. 
\end{proof}



\subsubsection{$\vec C_3 \Ra K_1$ is not a hero in orientations of chordal graphs}\label{subsec:1C3}

All along this subsection, we denote by $\mc C$ the class of $(\vec C_3 \Ra K_1)$-free \ocgs. The goal of this subsection is to construct digraphs in $\mc C$ with arbitrarily large dichromatic number.

\begin{lemma}\label{lem:directedsum}
Let $G, F \in \mc C$ and let $T$ be a transitive subtournament of $G$. Then the digraph $K$ obtained from $G$ and $F$ by adding every arc from $T$ to $F$ is in $\mc C$. 
\end{lemma}

\begin{proof}
 Given a graph $G$, the graph obtained by adding a vertex $v$ adjacent with every vertex of $G$ results in a chordal graph  as, if $v$ lies in an induced cycle, there is an arc between $v$ and every other vertex of this cycle, which is thus a triangle.
 Thus, adding vertices of $T$ to $F$ one by one, together with all arcs from $T$ to $F$, returns a chordal graph $F'$. The intersection of $V(F')$ and  $V(G)$ is $T$, which is a tournament. Hence, by Lemma~\ref{lem:chordal_intersection_on_clique}, the union of $G$ and $F'$, that is $K$, is an \ocg.

Suppose for contradiction that $K$ contains a subgraph $H$ isomorphic to  $\vec C_3 \Ra K_1$.
Since $G, F \in \mathcal C$, $H$ must intersect both $G$ and $F$ and since $H$ is a tournament, it must be included in $T \cup F$. 
Since there is no arc from $F$ to $T$, the directed triangle of $H$ cannot intersect both $T$ and $F$, and hence must be included in $F$ (as $T$ is a transitive tournament and thus have no directed triangle). The fourth vertex of $H$ contains the directed triangle in its in-neighbourhood, and thus must also be in $F$, a contradiction.  

\end{proof}

\begin{lemma}\label{lem:add_vertex_TT}
Let $G \in \mc C$ and let $T$ be a transitive subtournament of $G$ on vertices $\{v_1, \dots, v_n\}$ such that $v_1, \dots, v_n$  is the topological ordering of $T$. 
Then for every $j \in \{1, \dots, n-1\}$, the digraph $F$ obtained from $G$ by adding a vertex $x$ that sees $v_1, \dots, v_j$ and is seen by $v_{j+1}, \dots, v_n$ is in $\mc C$. 
\end{lemma}

\begin{proof}
By Lemma~\ref{lem:chordal_intersection_on_clique}, $F$ is an \ocg. Assume for contradiction that $F$ contains a copy $H$ of $K_1 \Ra \vec C_3$. Since $G \in \mc C$, $H$ must contain $x$ and thus be included in $G[K]$ where $K=V(T) \cup \{x\}$. Now, observe that $x^- \cap K$, $v_i^- \cap K$ for $i=1, \dots, j$ and $v_k^- \cap K$ for $k=j+1, \dots, n$ are transitive tournaments. Thus $G[K]$ cannot contain $H$, since one vertex in $H$ includes a directed triangle in its in-neighbourhood.  
\end{proof}

In the following, given a $k$-colouring $c:V(F) \rightarrow \{1,\ldots,k\}$ of a digraph $F$, we say that a subdigraph of $F$ is \emph{rainbow} (with respect to $c$), if its vertices are assigned pairwise distinct colours. 

\begin{lemma}\label{lem:rainbow}
Let $G \in \mc C$ such that $\dic(G)=k$. There exists a digraph $F=F(G) \in \mc C$ with $\dic(F)=k$ satisfying the following property: For every $k$-dicolouring of $F$, there exists a rainbow transitive tournament of order $k$ contained in $F$.
\end{lemma}
\begin{proof}
We prove the lemma by showing the following statement using induction on $i$ (the lemma then follows by setting $F:=F^{(k)}$). 

\medskip

$(\star)$ For every $i \in \{1,\ldots,k\}$, there exists a digraph $F^{(i)} \in \mc C$ such that $\dic(F^{(i)})=k$, and for every $k$-dicolouring of $F^{(i)}$, there exists a copy of $TT_i$ contained in $F^{(i)}$ which is rainbow.

\medskip

The statement of $(\star)$ is trivially true for $i=1$, since we may put $F^{(1)}:=G$, and in every $k$-dicolouring of $F^{(1)}$ any single vertex forms a rainbow $TT_1$. 

For the inductive step, let $i \in \{1,\ldots,k-1\}$ and suppose we have established the existence of a digraph $F^{(i)} \in \mc C$ of dichromatic number $k$ such that every $k$-dicolouring of $F^{(i)}$ contains a rainbow copy of $TT_i$. 

We now construct a digraph $F^{(i+1)}$ from $F^{(i)}$ as follows: Let $\mathcal{X}$ denote the set of all $X \subseteq F^{(i)}$ such that $X$ induces a $TT_i$ in $F^{(i)}$. Now, for every $X \in \mathcal{X}$ create a distinct copy $G_X$ of the digraph $G$ (pairwise vertex-disjoint for different choices of $X$, and all vertex-disjoint from $F^{(i)}$). Finally, for every $X\in \mathcal{X}$, add all the arcs $xy$ with $x \in X$ and $y \in V(G_X)$. Since $F^{(i)} \in \mc C$ and $G_X \in \mc C$ for every $X \in \mathcal{X}$, we can repeatedly apply Lemma~\ref{lem:directedsum} to find that the resulting digraph, which we call $F^{(i+1)}$, is still contained in $\mc C$. 

Note that by construction, no directed cycle in $F^{(i+1)}$ intersects more than one of the vertex-disjoint subdigraphs $F^{(i)}$ and  $(G_X|X \in \mathcal{X})$ of $F^{(i+1)}$, and hence, these digraphs may be coloured independently in every dicolouring of $F^{(i+1)}$. This immediately implies $\dic(F^{(i+1)})=\max\{\dic(F^{(i)}),\dic(G)\}=k$. 

To prove the inductive claim, consider any $k$-dicolouring $c:V(F^{(i+1)}) \rightarrow \{1,\ldots,k\}$ of $F^{(i+1)}$. Then by inductive assumption, there exists a rainbow copy of $TT_i$ contained in the subdigraph of $F^{(i+1)}$ isomorphic to $F^{(i)}$. Let $X$ denote its vertex set, and let $I \subseteq \{1,\ldots,k\}$ be the set of $i$ distinct colours used on $X$. Since $i<k$ and $\dic(G_X)=k$, there exists a vertex $v \in V(G_X)$ such that $c(v) \notin I$. Now, the vertex-set $X \cup \{v\}$ induces a rainbow $TT_{i+1}$ contained in $F^{(i+1)}$, as desired. This proves $(\star)$ and thus the lemma. 

\end{proof}

\begin{theorem}
The digraph $\vec C_3 \Ra K_1$ is not a hero in \ocgs. 
\end{theorem}
\begin{proof}
We construct a sequence of digraphs $(G_k)_{k \in \mathbb N}$ such that $\dic(G_k)=k$ and $G_k \in \mc C$.  
Let $G_1$ be the one-vertex-digraph and, having defined $G_k$, define $G_{k+1}$ as follows. Let $F_k:=F(G_k) \in \mc C$ be the digraph given by Lemma~\ref{lem:rainbow}, such that $\dic(F_k)=k$ and such that every $k$-dicolouring of $F_k$ contains a rainbow copy of $TT_k$. 

Let $\mathcal{T}$ denote the set of transitive tournaments which are subdigraphs of $F_k$. Now, for each transitive subtournament $T \in \mathcal{T}$, add a copy $F^T_k$ of $F_k$ (vertex-disjoint for different choices of $T$, and all vertex-disjoint from $F_k$). Next, for every $T \in \mathcal{T}$, add all the arcs $xy$ with $x \in V(T)$ and $y \in V(F^T_k)$. Finally, for every choice of $T \in \mathcal{T}$ and every transitive subtournament $T'$ of $F^T_k$, add a vertex $x_{T,T'}$ that is seen by every vertex of $T'$ and that sees every vertex of $T$.  
This completes the description of the digraph $G_{k+1}$.

By repeatedly applying Lemma~\ref{lem:directedsum} and Lemma~\ref{lem:rainbow}, we can see that all of the operations performed to construct $G_{k+1}$ preserve containment in $\mc C$, and hence, since $F_k \in \mc C$, we also must have $G_{k+1} \in \mc C$. 

Let us now prove that $\dic(G_{k+1}) = k+1$. A $(k+1)$-dicolouring can be achieved by piecing together individual $k$-dicolourings of $F_k$ and its copies $F_k^T, T \in \mathcal{T}$, and assigning to all vertices of the form $x_{T,T'}$ (which form a stable set in $G_{k+1}$) a distinct $(k+1)$-th colour. 

Finally, to prove that $\dic(G_{k+1})>k$, assume towards a contradiction that $G_{k+1}$ admits a dicolouring using colours from $\{1,\ldots,k\}$. Then by Lemma~\ref{lem:rainbow}, in this dicolouring $F_k$ contains a rainbow transitive tournament $T$ of order~$k$. Again by Lemma~\ref{lem:rainbow}, also $F^T_k$ contains a rainbow transitive subtournament $T'$ of order $k$. Now consider the vertex $x_{T,T'}$ in $G$, and let $i \in \{1,\ldots,k\}$ denote its colour. Since both $T$ and $T'$ contain all $k$ colours, there exist vertices $t_1 \in V(T)$ and $t_2 \in V(T')$ which are both assigned colour $i$. Finally, this yields a contradiction, since now the directed triangle $t_1 \rightarrow t_2 \rightarrow x_{T, T'} \rightarrow t_1$ in $G_{k+1}$ is monochromatic. 
\end{proof}

\section{Perspectives}

After characterising heroes in \ocgs, it is natural to ask what are the heroes in orientations of subclasses or superclasses of chordal graphs. 

Concerning superclasses of chordal graphs, consider the following construction (already mentioned in~\cite{ACN21} and Chapter~\ref{chpt:gyarfas}). Let $G_1$ be the graph on $1$ vertex, and having defined $G_{k-1}$ inductively, define $G_k = \vec C_{3}(1, G_{k-1}, G_{k-1}, G_{k-1})$. 
It is then easy to see that $\dic(G_k)=k$ and that the underlying graph of $G_k$ does not contain an induced path of length $4$. Hence, the underlying graphs of the $G_k$'s are perfect graphs, and even co-graphs, which implies that $\vec C_3$ is not a hero in orientations of perfect graphs. So the only possible heroes are transitive tournaments, which are trivial, since transitive tournaments are heroes in any orientations of graphs in $\mc C$, whenever $\mc C$ is a $\chi$-bounded class of graphs.

Regarding subclasses of chordal graphs, orientations of interval graphs seem to be an intriguing case. On one hand, we were not able to decide whether or not $\vec C_{3}(1,2,2)$ or $\vec C_3 \Ra K_1$ are heroes in this class,  and our attempts have not led us to a strong opinion as to the answer. 
On the other hand, we can prove the following. A \emph{unit interval graph} is an interval graph that admits an interval representation in which every interval has unit length. 

\begin{theorem}
Heroes in orientations of unit interval graphs are the same as heroes in tournaments.
\end{theorem}

\begin{proof}
Since complete graphs are unit interval graphs, the set of heroes in  orientations of proper interval graphs is a subset of the set of heroes in tournaments. 

We are going to prove the following, which easily implies that every hero in tournaments is a hero in orientations of unit interval graphs. \smallskip

$(\star)$ For every integer $C$, if $G$ is an orientation of a unit interval graph in which every subtournament has dichromatic number at most $C$, then $G$ is $2C$-dicolourable. 
\smallskip

Let $G$ be an orientation of a unit interval graph and $C$ an integer such that every subtournament of $G$ has dichromatic number at most $C$. 
Consider an interval representation of $G$ where each interval has length $1$ and assume without loss of generality that the endpoints of each interval are not integers. 
 For every integer $k$, let $K_k$ be the set of vertices of $G$ whose associated interval contains $k$. So each $K_k$ induces a subtournament of $G$, and by hypothesis, $G[K_k]$ is $C$-dicolourable. Moreover, since each interval has length $1$ and their extremities are not integers, the $K_k$'s  partition the vertices of $G$ and  there is no arc between $K_i$ and $K_j$ whenever $|i-j| \geq 2$.  Hence, piecing together dicolourings of $G[K_k]$ with colours from $\{1, \dots, C\}$ when $k$ is odd, and from $\{C+1, \dots, 2C\}$ when $k$ is even, results in a $2C$-dicolouring of $G$. 
\end{proof}


We say that a digraph is \emph{$t$-local} if the out-neighbourhood of each of its vertices induces a digraph with dichromatic number at most $t$. A class of digraphs $\mc C$ has the \emph{local to global property} if, for every integer $t$, $t$-local digraphs in $\mc C$ have bounded dichromatic numbers. 
It is proved in~\cite{HLTW19} that tournaments have the local to global property, and this result was generalised to the class of digraphs with bounded independence number in~\cite{HLNT19}.  
Since $K_1 \Ra  \vec C_3$ is not a hero in \ocgs, we get that the class of \ocgs does not have the local to global property, and that even $1$-local \ocgs can have an arbitrarily large dichromatic number. 
We wonder if other interesting classes of digraphs have it.
\bigskip



\cleardoublepage 

\ctparttext{\centering In which we study an edge-colouring problem on multigraphs.}
\part{Edge-defective colouring}

\chapter{Vizing's and Shannon's Theorems for defective edge colouring}\label{chpt:defective}

\begin{flushright}{\slshape    
This chapter is built upon a joint \\
work with Pierre Aboulker and \\
Chien-Chung Huang, published in \cite{AAH22}}. \\ \medskip
\end{flushright}

\emph{In this chapter, we study a generalization of the edge-colouring problem on multigraphs, find a tight upper bound on the number of colours needed and discuss algorithms in the case of simple graphs.}

\section{Introduction}
As this chapter is mainly concerned with multigraphs, graphs will be called \emph{simple graphs} in this chapter. 
Let $G$ be a multigraph. 
An \emph{edge colouring of $G$ with defect $d$} is a colouring of its edges so  that each vertex is incident with at most $d$ edges of the same colour. We say that $G$ is \emph{$k$-edge colourable with defect $d$}, or simply \emph{$(k,d)$-edge colourable}, if $G$ admits an edge colouring with defect $d$ using (at most) $k$ colours. In other words, the edge set can be partitioned into at most $k$ submultigraphs of maximum degree at most $d$. 
 The \emph{$d$-defective chromatic index} of $G$ is the minimum $k$ such that $G$ is $(k,d)$-edge colourable and is denoted by $\chi^{'}_{d}(G)$. 
So $\chi'_1(G)$ is the usual chromatic index. 

This notion is called \textit{frugal edge colouring} in~\cite{amini:inria-00144318} and \textit{improper edge colouring} in~\cite{HILTON2001253}. We follow the vocabulary of the analogous concept of defective vertex colouring, a now well-established notion. See~\cite{W18} for a nice dynamic survey on defective vertex colouring.

Our first result is the following.

\begin{theorem}\label{thm:main_theorem}
Let $d, \Delta \geq 1$ and let $G$ a multigraph with maximum degree $\Delta$. 
If $d$ is even, then $\chi'_{d}(G) = \lceil \frac{\Delta}{d} \rceil$, and if $d$ is odd, then $\chi'_{d}(G) \leq \lceil \frac{3\Delta - 1}{3d - 1} \rceil$, and this bound is tight for all $\Delta$ and $d$. 
\end{theorem}

The case $d=1$ corresponds to the classic result of Shannon~\cite{S49} on chromatic index  stating that for every multigraph $G$, $\chi'_1(G) \leq \lfloor \frac{3\Delta(G)}{2} \rfloor$ (observe that $\lceil \frac{3\Delta - 1}{2} \rceil = \lfloor \frac{3\Delta}{2} \rfloor$ for all $\Delta$). 
When $d$ is even, the result is almost trivial in our context (see Theorem~\ref{thm:deven}), and was already known in the more general context of list colouring~\cite{HILTON2001253, amini:inria-00144318}. 
When $d$ is odd, a proof that $\chi'_d(G) \leq \lceil \frac{3\Delta}{3d - 1} \rceil$ in the context of list colouring is announced in~\cite{amini:inria-00144318}, but seems to contain a flaw and actually holds only in the case where $\Delta$ is divisible by $3k-1$. See Section~\ref{sec:further} for more on the list colouring context. 

\medskip

Vizing's celebrated theorem on edge colouring~\cite{V64} states that for every simple graph $G$, $\chi'_1(G) \in \{\Delta(G), \Delta(G) +1\}$, and Holyer~\cite{H81}, and Leven and Galil~\cite{LZ83} proved that deciding if $\chi'_1(G) =\Delta(G)$ is NP-complete even restricted to $d$-regular simple graphs as soon as $d \geq 3$. 
We generalize both results by proving that for every simple graph $G$, $\chi'_{d}(G) \in \{ \lceil \frac{\Delta}{d} \rceil, \lceil \frac{\Delta+1}{d} \rceil \}$ (which is easily implied by Vizing's Theorem) and we characterize the values of $\Delta$ and $d$ for which the problem is NP-complete. 
More precisely, we prove that, for given $\Delta$ and $d$, the problem of determining $\chi'_d(G)$ for a $\Delta$-regular simple graph is NP-complete if and only if $d$ is odd and $\Delta = kd$ for some integer $k \geq 3$. See Theorems~\ref{thm:simple_polynomial} and~\ref{thm:NP}. 
\medskip

We give some definitions and preliminary results in Section~\ref{sec:intro}.
We prove the generalization of Shannon's Theorem in Section~\ref{sec:multi} and the proof of the generalization of Vizing's Theorem in Section~\ref{sec:simple}. Finally, in Section~\ref{sec:further}, we conjecture a generalisation of Theorem~\ref{thm:main_theorem} for list colouring and a generalisation of the Goldberg-Seymour Conjecture.
\medskip


\section{Preliminaries} \label{sec:intro}

The following gives a trivial lower bound on the $d$-defective chromatic index that turns out to be tight whenever $d$ is even (see Theorem~\ref{thm:deven}). 

\begin{lemma}\label{lem:trivialbound}
For every multigraph $G$, $\chi'_{d}(G) \geq \lceil \frac{\Delta(G)}{d} \rceil$.
\end{lemma}

\begin{proof}
At least $\lceil \frac{\Delta(G)}{d} \rceil$ colours are needed to colour the edges incident to a vertex of degree $\Delta(G)$.
\end{proof}

\begin{lemma}\label{lem:chi_increasing}
Let $k, d, \Delta$ be integers. 
If every $(\Delta + 1)$-regular multigraph is $(k,d)$-edge colourable, then every $\Delta$-regular multigraph is also  $(k,d)$-edge colourable. 
\end{lemma}

\begin{proof}
Let $G$ be a $\Delta$-regular multigraph. 
Take two disjoint copies $G'$ and $G''$ of $G$ and add an edge between each vertex $v \in V(G')$ and its copy in $G''$. The obtained multigraph $H$ is $(\Delta+1)$-regular and contains $G$ as a submultigraph, so $\chi'_d(G) \leq \chi'_d(H) \leq k$. 
\end{proof}


\subsubsection*{Factors in multigraphs}

A \emph{$k$-factor} of $G$ is a $k$-regular spanning submultigraph of $G$. 
We sometimes consider a $k$-factor $F$ as its edge set $E(F)$. 
We recall this theorem from Petersen \cite{P91}, one of the very first fundamental results in multigraph theory:

\begin{theorem}[\cite{P91}\label{thm:petersen}]
Let $\Delta$ be an even integer. A $\Delta$-regular multigraph admits a $k$-factor for every even integer $k \leq \Delta$. 
\end{theorem}

An \emph{Euler tour} of a multigraph $G$ is a closed walk in $G$ that traverses every edge of $G$ exactly once.   
It is a well-known fact that a multigraph admits an Euler tour if and only if it is connected and all its vertices have even degrees. 
The next two lemmas use this fact to prove the existence of factors. This idea was already used by Petersen to prove his theorem. 

\begin{lemma}\label{lem:delta_even_kV_even}
Let $G$ be a connected $2k$-regular multigraph with an even number of edges. Then the edges of $G$ can be partitioned into two $k$-factors. 
\end{lemma}

\begin{proof}
We number the edges $e_1, e_2, \dots, e_{2t}$ of $G$ along an Euler tour $C$ and we let $A= \{e_1, e_3, \dots, e_{2t-1}\}$ and $B = \{e_2, e_4, \dots, e_{2t}\}$. 
Since consecutive edges of $C$ are numbered with different parities and its first and last edges  have distinct parities, $A$ and $B$ are both $k$-regular. 
\end{proof}

\begin{lemma}\label{lem:delta_even_kV_odd}
Let $G$ be a connected $2k$-regular multigraph with an odd number of edges, and let $e \in E(G)$. There exist two multigraphs $G_{A} = (V,A)$ and $G_{B} = (V,B)$ such that $E(G) = A \cup B \cup \{e\}$, $\Delta(G_{A}) \leq k$ and $\Delta(G_{B}) \leq k$.
\end{lemma}

\begin{proof}
The proof is the same as for the previous Lemma, except that we do not assign the last edge of the Euler tour, and we choose $e$ to be this last edge. 
\end{proof}

The next theorem roughly says that, in a $\Delta$-regular multigraph, one can find a $k$-factor  as soon as $k$ is even and is relatively small compared to $\Delta$. It was first proved in \cite{K84}. See also Theorem 3.10 $(v)$ in \cite{AK11}. The version stated here is a simplified version of the original theorem.

\begin{theorem}[\cite{K84}\label{lem:factor2/3}\label{lem:delta_odd_k_even_factor_including}]
Let $\Delta$ be an odd integer and $G$ a $2$-edge connected $\Delta$-regular multigraph.  Let $e \in E(G)$. Let $k$ be an even integer with $k \leq \frac{2\Delta}{3}$. Then $G$ has a $k$-factor containing $e$. 
\end{theorem}



\subsubsection*{Shannon multigraphs}

Given an integer $k$, the \emph{Shannon multigraph} $Sh(k)$ is the multigraph made of three vertices connected by $\lfloor \frac{k}{2} \rfloor$, $\lfloor \frac{k}{2} \rfloor$ and $\lceil \frac{k}{2} \rceil$ edges respectively. See Figure~\ref{fig:shanon}. Observe that 
\begin{itemize}
    \item $\Delta(Sh(k))=k$,
    \item  when $k$ is even, $Sh(k)$ is $k$-regular and has $\frac{3k}{2}$ edges and,
    \item when $k$ is odd, $Sh(k)$ has two vertices of degree $k$ and one vertex of degree $k-1$ and has $\frac{3k-1}{2}$ edges. 
\end{itemize}

    \begin{figure}[H]
    \centering
    \begin{tikzpicture}
        \begin{scope}
        \vertex (1) at (0,1) {};
        \vertex (2) at (-0.866,-0.5) {};
        \vertex (3) at (0.866,-0.5) {};
        \node (4) at (-0.866,0.5) {$\lfloor \frac{k}{2} \rfloor$};
        \node (5) at (0.866,0.5) {$\lfloor \frac{k}{2} \rfloor$};
        \node (6) at (0,-1) {$\lceil \frac{k}{2} \rceil$};
        \draw (1) -- (2);
        \draw (2) -- (3);
        \draw (3) -- (1);
        \end{scope}
    \end{tikzpicture}
    \caption{The Shannon multigraph $Sh(k)$} \label{fig:shanon}
    \end{figure}
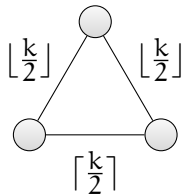

\begin{lemma}\label{lem:shannon_tight}
Let $k,d \geq 1$ with $d$ odd. Then $\chi'_d(Sh(k)) = \lceil\frac{3k-1}{3d-1}\rceil$. 
\end{lemma}

\begin{proof}
Consider an ordering $(e_i)_{1 \leq i \leq |E(Sh(k))|}$ of the edges of $Sh(k)$ such that for any $1 \leq i \leq |E(Sh(k))| - 2$, $e_i$, $e_{i+1}$ and $e_{i+2}$ form a triangle. 
Such an ordering can be obtained by setting $e_{1}$ to be any edge with both extremities of degree $k$ and then setting, for $i=2, \dots, |E(Sh(k))| - 1$, $e_{i+1}$ to be any unnumbered edge coming right after $e_{i}$ in clockwise order. 
The following statement is easily proven using induction:
\emph{
    For every odd integer $\ell$ such that $1 \leq \ell \leq \frac{2|E(Sh(k)|-1}{3}$, every contiguous subsequence of $(e_{i})_{1 \leq i \leq |E(Sh(k))|}$ of length $\frac{3\ell - 1}{2}$ induces a multigraph of maximum degree $\ell$.
}

Thus, colouring the first $\frac{3d-1}{2}$ edges of $(e_i)_{1 \leq i \leq |E(Sh(k))|}$ in one colour, the following $\frac{3d-1}{2}$ in a second colour and so on, yields a colouring with at most $\lceil \frac{|E(Sh(k))|}{\frac{3d-1}{2}} \rceil$ colours such that each colour class induces a submultigraph with maximum degree at most $d$, and each colour class except at most one has $\frac{3d-1}{2}$ edges. Since every submultigraph of $Sh(k)$ with maximum degree $d$ (recall that $d$ is odd) has at most $\frac{3d-1}{2}$ edges, this colouring is an optimal d-defective edge colouring and thus: 
\[ \chi'_d(Sh(k)) = \Bigl\lceil \frac{|E(Sh(k))|}{\frac{3d-1}{2}} \Bigr\rceil = 
  \begin{cases}
     \lceil \frac{3k}{3d-1} \rceil = \lceil \frac{3k-1}{3d-1} \rceil   & \text{ if } k \text{ is even,} \\
     \lceil \frac{3k-1}{3d-1} \rceil & \text{ if } k \text{ is odd.}
  \end{cases}
\]
\end{proof}


\section{Generalization of Shannon's Theorem} \label{sec:multi}


The goal of this section is to prove Theorem~\ref{thm:main_theorem}. In view of Lemma~\ref{lem:trivialbound}, for even $d$, it suffices to prove the upper bound $\chi'_d(G) \leq \lceil \frac{\Delta}{d} \rceil$. Moreover, for both even and odd $d$, it is enough to prove the result for $\Delta$-regular multigraphs. Indeed, if $G$ is not $\Delta$-regular, then we can build a $\Delta$-regular multigraph $G'$ containing $G$ as a submultigraph  as follows: take two copies of $G$, and for each vertex $v$ of $G$, add $\Delta -d(v)$ edges between the two copies of $v$. Then $\chi'_{d}(G) \leq \chi'_{d}(G')$. So it suffices to prove the following two results. The case where $d$ is even was already known, but we give the proof anyway for completeness. 

\begin{theorem}[\cite{HILTON2001253, amini:inria-00144318}\label{thm:deven}]
Let $d, \Delta \geq 1$ with $d$ even. For every $\Delta$-regular multigraph $G$, $\chi^{'}_{d}(G) = \lceil \frac{\Delta}{d} \rceil$.
\end{theorem}

\begin{proof}
If $\Delta$ is even, then $G$ has a $\min\{d,\Delta\}$-factor by Theorem~\ref{thm:petersen}, and it follows inductively that $\chi'_d(G) \leq \lceil \frac{\Delta(G)}{d} \rceil$.  
If $\Delta$ is odd, then $\Delta+1$ is even and (by the previous sentence) every $(\Delta+1)$-regular multigraph is $(k,d)$-edge colourable, where $k = \lceil \frac{\Delta+1}{d} \rceil = \lceil \frac{\Delta}{d} \rceil$; hence $\chi'_d(G) \leq \lceil \frac{\Delta}{d} \rceil$ by Lemma~\ref{lem:chi_increasing}. Equality holds in both cases by Lemma~\ref{lem:trivialbound}.
\end{proof}

\begin{theorem}
Let $d, \Delta \geq 1$ with $d$ odd. 
For every $\Delta$-regular multigraph $G$, $\chi'_d(G) \leq \lceil \frac{3\Delta - 1}{3d - 1} \rceil$. 
\end{theorem}

\begin{proof}
If $d = 1$, then the result follows from the classic result of Shannon, and so we may assume that $d \geq 3$. 
By Lemma~\ref{lem:chi_increasing}, it is enough to prove it for values of $\Delta$ such that $\lceil \frac{3\Delta - 1}{3d - 1} \rceil < \lceil \frac{3(\Delta +1) - 1}{3d - 1} \rceil$, that is, for $\Delta \in \{(i+1)d - \lceil \frac{i}{3} \rceil \mid i \geq 0\} = \{d, 2d-1, 3d-1, 4d-1, 5d-1, 6d-2, \dots\}$. We call such integers \textit{special}. In particular, we have $\Delta \geq d \geq 3$.

Let $G$ be a counterexample that minimizes  $\Delta$ and has minimum order. That is, $\Delta$ is special, $G$ is $\Delta$-regular, $\chi'_d(G) = \lceil \frac{3\Delta - 1}{3d - 1} \rceil +1$, every $\Delta$-regular multigraph with fewer vertices than $G$ is $(\lceil \frac{3\Delta - 1}{3d - 1} \rceil, d)$-edge colourable, and for every special integer $\Delta' < \Delta$, every $\Delta'$-regular multigraph is $( \lceil \frac{3\Delta' - 1}{3d - 1} \rceil, d)$-edge colourable. 
\medskip

\begin{claim}\label{lem:extremity_shannon}
If $G$ has a cut edge $e$, then at least one connected component of $G - e$ is isomorphic to $Sh(\Delta)$. 
\end{claim}

\begin{proofclaim}
Set $e = ab$ and let $A$ and $B$ be the two connected components of $G - e$ containing $a$ and $b$ respectively. 
Assume for contradiction that neither $A$ nor $B$ is isomorphic to $Sh(\Delta)$. 
Vertices of $A$ have degree $\Delta$ in $A$ except for $a$, which has degree $\Delta -1$; hence, $\Delta$ is odd.  
If $|V(A)| = 1$, then $a$ has degree $1$, a contradiction with the fact that $\Delta \geq 3$.  
If $|V(A)| = 3$, then $A$ is isomorphic to $Sh(\Delta)$, a contradiction. 
We can thus assume $|V(A)| \geq 5$.    

Let $G_A$ be the multigraph obtained from $G$ by replacing $A$ by $Sh(\Delta)$ as in  Figure~\ref{fig:bridge}. $G_A$ is $\Delta$-regular (because $\Delta$ is odd) and has strictly fewer vertices than $G$. 
Hence, by minimality of $G$, $G_A$ admits an edge colouring $c_A$ with defect $d$ using at most  $\lceil \frac{3\Delta - 1}{3d - 1} \rceil$ colours. 
We define symmetrically $G_B$ and $c_B$. We may assume, by properly permuting colours in $G_B$, that $c_B(e) = c_A(e)$. 
 We can now obtain an edge colouring of $G$ with defect $d$ using at most $\lceil \frac{3\Delta - 1}{3d - 1} \rceil$ colours by assigning colour $c_A(e)$ to $e$, colour $c_{A}(e')$ to any edge $e'$ in $B$, and colour $c_{B}(e')$ to any edge $e'$ in $A$, a contradiction.  
\end{proofclaim}

    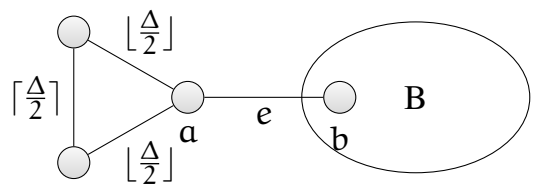
\begin{figure}[H]
    \centering
    \begin{tikzpicture}
        \vertex[label=below:$a$] (1) at (1,0) {};
        \vertex (2) at (-0.5,-0.866) {};
        \vertex (3) at (-0.5,0.866) {};
        \node (4) at (0.5,-0.866) {$\lfloor \frac{\Delta}{2} \rfloor$};
        \node (5) at (0.5,0.866) {$\lfloor \frac{\Delta}{2} \rfloor$};
        \node (6) at (-1,0) {$\lceil \frac{\Delta}{2} \rceil$};
        \node (8) at (4,0) {B};
        \draw [rotate around={90:(8)}] (8) ellipse (1cm and 1.5cm);
        \vertex[label=below:$b$] (7) at (3,0) {};
        \node (8) at (2,-0.25) {$e$};
        \draw (1) -- (2);
        \draw (2) -- (3);
        \draw (3) -- (1);
        \draw (7) -- (1);
    \end{tikzpicture}
    \caption{the multigraph $G_A$} \label{fig:bridge}
    \end{figure}
    
Observe that, if a $\Delta$-regular multigraph has a cut edge, then $\Delta$ must be odd. Moreover, if $\Delta$ is odd, then for every $(\frac{3\Delta-1}{3d-1}, d)$ edge colouring  of $Sh(\Delta)$,  there is a colour $c$ such that the (unique) vertex of $Sh(\Delta)$ with degree $\Delta -1$ is incident with at most $d-1$ edges coloured with $c$. This simple observation is used in the proof of the following claim.

\begin{claim}\label{lem:at_most_one_bridge}
$G$ has at most one cut edge.
\end{claim}

\begin{proofclaim}
Suppose for contradiction that $G$ has two cut edges $uv$ and $u'v'$.  
By Claim~\ref{lem:extremity_shannon}, we may assume that $G$ is made of two disjoint copies of $Sh(\Delta)$ plus a multigraph $A$ as in Figure~\ref{fig:two_bridges}. Note that $u=u'$ is possible. 

Assume first that $u \neq u'$. Then $A + uu'$ is $\Delta$-regular and has strictly fewer vertices than $G$. So, by minimality of $G$,  $A + uu'$ admits a $(\lceil \frac{3\Delta - 1}{3d - 1} \rceil,d)$-edge colouring  $c_A$.  
We can extend this colouring to $G$ by giving colour $c_A(uu')$ to $uv$ and $u'v'$ and then extending this colouring to the two copies of $Sh(\Delta)$ without any new colour  (this is possible by the observation stated right before the claim). This leads to a  $(\lceil \frac{3\Delta - 1}{3d - 1} \rceil,d)$-edge colouring of $G$, a contradiction.  

    \begin{figure}[H]
    \centering
    \begin{tikzpicture}
        \begin{scope}
        \vertex[label=below:$v$] (1) at (1,0) {};
        \vertex (2) at (-0.5,-0.866) {};
        \vertex (3) at (-0.5,0.866) {};
        \node (4) at (0.5,-0.866) {$\lfloor \frac{\Delta}{2} \rfloor$};
        \node (5) at (0.5,0.866) {$\lfloor \frac{\Delta}{2} \rfloor$};
        \node (6) at (-1,0) {$\lceil \frac{\Delta}{2} \rceil$};
        \node (8) at (4,0) {A};
        \draw [rotate around={90:(8)}] (8) ellipse (1cm and 1.5cm);
        \vertex[label=below:$u$] (7) at (3,0) {};
        
        \vertex[label=below:$v'$] (9) at (7,0) {};
        \vertex (10) at (8.5,-0.866) {};
        \vertex (11) at (8.5,0.866) {};
        \node (12) at (7.5,-0.866) {$\lfloor \frac{\Delta}{2} \rfloor$};
        \node (12) at (7.5,0.866) {$\lfloor \frac{\Delta}{2} \rfloor$};
        \node (13) at (9,0) {$\lceil \frac{\Delta}{2} \rceil$};
        \vertex[label=below:$u'$] (14) at (5,0) {};
        
        \draw (9) -- (10);
        \draw (10) -- (11);
        \draw (11) -- (9);
        \draw (14) -- (9);

        \draw (1) -- (2);
        \draw (2) -- (3);
        \draw (3) -- (1);
        \draw (7) -- (1);
        \end{scope}

    \end{tikzpicture}
    \caption{the multigraph $G$ when $u \neq u'$}
    \label{fig:two_bridges}
    \end{figure}
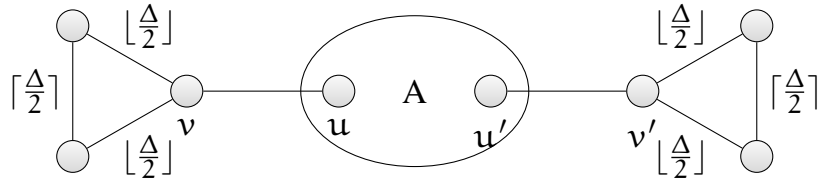

Assume now that $u = u'$. We consider the multigraph $G'$ obtained by replacing the two copies of $Sh(\Delta)$ with four new vertices $w,\ x,\ y,\ z$ as in Figure~\ref{fig:two_bridges2}.  
It is easy to check that $G'$ is $\Delta$-regular and since $G'$ has two vertices less then $G$,  it is $(\lceil \frac{3 \Delta - 1}{3d - 1} \rceil,d)$-edge colourable. 
This gives us a  $(\lceil \frac{3 \Delta - 1}{3d - 1} \rceil,d)$-edge colouring of $A$ that can easily be extended to the two copies of $Sh(\Delta)$ without any new colour (this is again possible by the observation  stated right before the claim), leading to a  $(\lceil \frac{3\Delta - 1}{3d - 1} \rceil,d)$-edge colouring of $G$, a contradiction. 
\end{proofclaim}

    \begin{figure}[H]
    \centering
    \begin{tikzpicture}[scale=0.8]
        \begin{scope}
        \vertex[label=below:$v$] (1) at (7,1.25) {};
        \vertex (2) at (8.5,0.383) {};
        \vertex (3) at (8.5,2.116) {};
        \node (4) at (7.6,2) {$\lfloor \frac{\Delta}{2} \rfloor$};
        \node (5) at (7.6,0.5) {$\lfloor \frac{\Delta}{2} \rfloor$};
        \node (6) at (9,1.25) {$\lceil \frac{\Delta}{2} \rceil$};

        \node (8) at (4,0) {A};
        \draw [rotate around={90:(8)}] (8) ellipse (1cm and 1.5cm);
        \vertex[label=below:$u$] (7) at (5,0) {};
        
        \vertex[label=below:$v'$] (9) at (7,-1.25) {};
        \vertex (10) at (8.5,-2.116) {};
        \vertex (11) at (8.5,-0.383) {};
        \node (12) at (7.6,-2) {$\lfloor \frac{\Delta}{2} \rfloor$};
        \node (12) at (7.6,-0.5) {$\lfloor \frac{\Delta}{2} \rfloor$};
        \node (13) at (9,-1.25) {$\lceil \frac{\Delta}{2} \rceil$};
        
        \draw (9) -- (10);
        \draw (10) -- (11);
        \draw (11) -- (9);
        \draw (7) -- (9);

        \draw (1) -- (2);
        \draw (2) -- (3);
        \draw (3) -- (1);
        \draw (1) -- (7);
        \end{scope}

        \begin{scope}[xshift = 8cm]
        
        \vertex[label=below:$u$] (1) at (5,0) {};
        \node (2) at (4,0) {A};
        \draw [rotate around={90:(2)}] (2) ellipse (1cm and 1.5cm);
        \vertex[label=below:$w$] (3) at (7,0) {};
        \vertex[label=above:$x$] (4) at (9,2) {};
        \vertex[label=below:$y$] (5) at (9,-2) {};
        \vertex[label=below:$z$] (6) at (11,0) {};
        
        \node[label=above:$\lfloor \frac{\Delta}{2} \rfloor$-$1$] () at ($(3)!0.3!(4)$) {};
        \node[label=below:$\lfloor \frac{\Delta}{2} \rfloor$-$1$] () at ($(3)!0.3!(5)$) {};
        \node[label=above:$\lfloor \frac{\Delta}{2} \rfloor$] () at ($(6)!0.4!(4)$) {};
        \node[label=below:$\lfloor \frac{\Delta}{2} \rfloor$] () at ($(6)!0.4!(5)$) {};

        \draw (1) to[bend left = 13] (3);
        \draw (1) to[bend right = 13] (3);
        
        \draw (3) -- (4);
        \draw (3) -- (5);
        \draw (3) -- (6);
        \draw (4) -- (6);
        \draw (5) -- (6);
        \draw (4) to[bend left = 15] (5);
        \draw (4) to[bend right = 15] (5);
        
        \end{scope}
    \end{tikzpicture}
    \caption{On the left: the multigraph $G$ when $u=u'$, on the right: the multigraph $G'$.} \label{fig:two_bridges2}
    \end{figure}
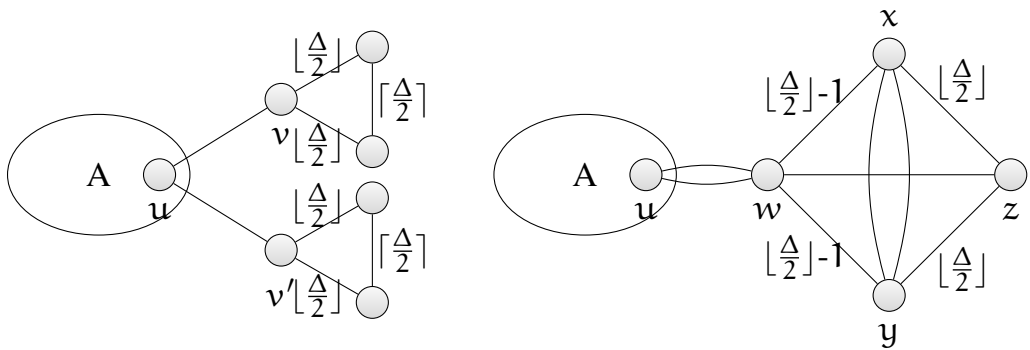

\begin{claim}\label{lem:delta_odd_k_even_extract_factor}
$G$ has a $k$-factor for every even integer $k \leq \frac{2\Delta}{3}$. 
\end{claim}

\begin{proofclaim}
Let $k \leq \frac{2 \Delta}{3}$ be an even integer. 
If $\Delta$ is even, the result holds by Theorem~\ref{thm:petersen}. So we may assume that $\Delta$ is odd. 
If $G$ is $2$-edge connected, then we are done by Theorem~\ref{lem:factor2/3}. 
So assume $G$ has a cut edge $uv$. Let $A, B$ be the two connected components of $G \setminus uv$ with $u \in V(A)$ and $v \in V(B)$. By Claim~\ref{lem:at_most_one_bridge}, $G$ has no other cut edges and thus $A$ and $B$ are both $2$-edge-connected. 
By Claim~\ref{lem:extremity_shannon}, either $A$ or $B$ is isomorphic to $Sh(\Delta)$.  Without loss of generality, we suppose that it is $B$. 
Let $w$ and $x$ be the two other vertices of $B$. Let $y$ be a neighbour of $u$ in $A$. 
Consider $G' = G + uv + yw - uy - vw$ (see Figure~\ref{fig:yeah}). It is easy to check that $G'$ is $\Delta$-regular and $2$-edge-connected (recall that $\Delta \geq 3$ and thus $\lfloor \frac{\Delta}{2} \rfloor \geq 1$). Applying  Theorem~\ref{lem:factor2/3} on $G'$ with $e = wy$, $G'$ has a $k$-factor $F$ containing the edge $wy$. 
There exists an integer $s \leq k-1$ such that $F$ contains $s$ edges $wx$, and $k-s-1$ edges $wv$. So $F$ must contain $k-s$ edges $vx$ and thus $F$ contains exactly one edge $uv$. 
Hence, $F - uv - yw + uy + vw$ is a $k$-factor of $G$.
\end{proofclaim}

    \begin{figure}[H]
    \centering
    \begin{tikzpicture}
        \begin{scope}[xshift = 0cm]
        \vertex[label=below:$v$] (1) at (-1,0) {};
        \vertex[label=below:$x$] (2) at (0.5,-0.866) {};
        \vertex[label=above:$w$] (3) at (0.5,0.866) {};
        \node (4) at (-0.4,-0.75) {\small $\lfloor \frac{\Delta}{2} \rfloor$};
        \node (5) at (-0.4,0.75) {\small $\lfloor \frac{\Delta}{2} \rfloor$};
        \node (6) at (1,0) {\small $\lceil \frac{\Delta}{2} \rceil$};
        \node (8) at (-4,0) {A};
        \draw [rotate around={90:(8)}] (8) ellipse (1cm and 1.5cm);
        \vertex[label=below:$u$] (7) at (-3,0) {};
        \vertex[label=below:$y$] (9) at (-5,0) {};
        \draw (1) -- (2);
        \draw (2) -- (3);
        \draw (3) -- (1);
        \draw (7) -- (1);
        \draw (7) to[bend right = 25] (9);
        \end{scope}
        
        \begin{scope}[xshift = 8cm]
        \vertex[label=below:$v$] (1) at (-1,0) {};
        \vertex[label=below:$x$] (2) at (0.5,-0.866) {};
        \vertex[label=above:$w$] (3) at (0.5,0.866) {};
        \node (4) at (-0.4,-0.75) {\small $\lfloor \frac{\Delta}{2} \rfloor$};
        \node (5) at (-0.4,0.7) {\small $\lfloor \frac{\Delta}{2} \rfloor \textcolor{red}{ - 1}$};
        \node (6) at (1,0) {\small $\lceil \frac{\Delta}{2} \rceil$};
        \node (8) at (-4,0) {A};
        \draw [rotate around={90:(8)}] (8) ellipse (1cm and 1.5cm);
        \vertex[label=below:$u$] (7) at (-3,0) {};
        \vertex[label=below:$y$] (9) at (-5,0) {};
        \draw (1) -- (2);
        \draw (2) -- (3);
        \draw (3) -- (1);
        
        \draw[dashed, red] (3) to[bend left = 25] (1);
        
        \draw (7) to[bend right = 25] (1);
        \draw[red] (7) to[bend left = 25] (1);
        \draw[red] (9) to[bend left = 35] (3);
        \draw[dashed,red] (7) to[bend right = 25] (9);
        \end{scope}
        
    \end{tikzpicture}
    \caption{The multigraphs $G$ and $G'$} \label{fig:yeah}
    \end{figure}
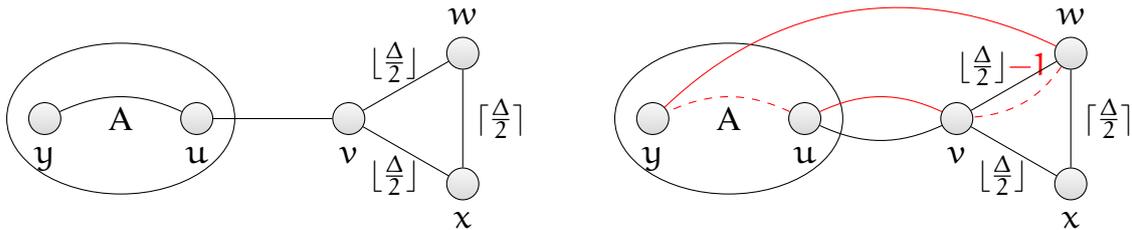
    
We are now ready to prove the theorem. We distinguish cases with respect to the value of $\Delta$ and the corresponding value of $\lceil \frac{3\Delta - 1}{3d - 1} \rceil$. Recall that $\Delta$ is a special integer, that is $\Delta \in \{(i+1)d - \lceil \frac{i}{3} \rceil \mid i \geq 0\} = \{d, 2d-1, 3d-1, 4d-1, 5d-2, 6d-2, \dots\}$.
\smallskip

\textbf{Case 1:} $\lceil \frac{3\Delta - 1}{3d - 1} \rceil = 1$, $\Delta = d$. The result holds trivially. 
\smallskip

\textbf{Case 2:}\label{case:delta_two_d_minus_one} 
$\lceil \frac{3\Delta - 1}{3d - 1} \rceil = 2$, $\Delta = 2d - 1$. 
Since $d$ is odd, $d-1$ is even, and $d-1 < \frac{4d-2}{3}=\frac{2\Delta}{3}$. 
So, by Claim~\ref{lem:delta_odd_k_even_extract_factor}, $G$ has a $(d-1)$-factor, say $F$.  
Now, $G-F$ is $d$-regular, 
and thus $\chi'_{d}(G) \leq 2$. This proves case 2. 
\smallskip

\textbf{Case 3:}\label{case:delta_three_d_minus_one} 
$\lceil \frac{3\Delta - 1}{3d - 1} \rceil = 3$, $\Delta = 3d - 1$. 
Since $d$ is odd, $\Delta$ is even. 
By Theorem~\ref{thm:petersen}, $G$ has a $2d$-factor $F$. 
By applying Lemma~\ref{lem:delta_even_kV_even} on connected components of even size of $F$ and Lemma~\ref{lem:delta_even_kV_odd} on connected components of odd size, we can extract two multigraphs  $G_{A}$ and $G_{B}$ along with a matching $M$ such that $E(F) = E(G_A) \cup E(G_B) \cup M$, $\Delta(G_{A}) \leq d$ and $\Delta(G_{B}) \leq d$. 
Now, $E(G)$ can be partitioned into $E(G_A)$, $E(G_B)$ and $E(G) \setminus (E(F) \setminus M)$. 
Since the multigraph induced by $E(G) \setminus (E(F) \setminus M)$ has a maximum degree at most $3d-1 - 2d + 1 = d$,  each of these sets  induces a multigraph with maximum degree at most $d$. This proves case 3. 
\smallskip

\textbf{Case 4:} 
$\lceil \frac{3\Delta - 1}{3d - 1} \rceil = 4$, $\Delta = 4d - 1$. 
Since $d \geq 3$, we have $2d < \frac{8d-2}{3}= \frac{2\Delta}{3}$. 
So, by Claim~\ref{lem:delta_odd_k_even_extract_factor}, $G$ has a $2d$-factor, say $A$, and $B = G-A$ is a $(2d-1)$-factor of $G$.  
By applying Lemma~\ref{lem:delta_even_kV_even} on connected components of $A$ of even size and Lemma~\ref{lem:delta_even_kV_odd} on connected components of $A$ of odd size, we get a partition of $E(A)$ into three sets $A_1$, $A_2$ and $M$ such that $\Delta(A_{1}) \leq d$, $\Delta(A_{2}) \leq d$ and $M$ is a matching. 

It is now enough to prove that $\chi'_{d}(B \cup M) \leq 2$. Let $C$ be a connected component of $B \cup M$. 
If every vertex of $C$ is incident with an edge of $M$, then $C$ has an even number of vertices and is $2d$-regular, so its number of edges is $d$ times its number of vertices, which is even, and thus $\chi'_d(C) = 2$ by Lemma~\ref{lem:delta_even_kV_even}. 
Assume now that there exists a vertex of $C$ that is not incident with an edge of $M$. Take two copies of $C$, and add an edge between the copies of each vertex of $C$ not incident with an edge of $M$. The obtained multigraph has an even number of vertices and is $2d$-regular, so it is $(2,d)$-edge colourable by Lemma~\ref{lem:delta_even_kV_even} and thus so is $C$. So each connected component of $B \cup M$ is $(2,d)$-edge colourable, and thus so is $B \cup M$. This proves case 4.
\smallskip

\textbf{Case 5:} 
$\lceil \frac{3\Delta - 1}{3d - 1} \rceil \geq 5$, $\Delta \geq  5d - 2$. 
Note that $3d-1$ is even since $d$ is odd. Also, since $d \geq 3$, $3d-1 = \frac{9d-3}{3} <  \frac{10d-4}{3} \leq \frac{2\Delta}{3}$. 
So, by Claim~\ref{lem:delta_odd_k_even_extract_factor}, $G$ has a $(3d-1)$-factor, say $F$. By Case 3, $F$ is $(3,d)$-edge colourable. As $G-F$ is $(\Delta-(3d-1))$-regular, and as $\Delta-(3d-1)$ is less than at least one special integer less than $\Delta$, it follows from minimality of $\Delta$ that 
$$\chi'_{d}(G-F) \leq \lceil \frac{3(\Delta - (3d - 1)) - 1}{3d - 1} \rceil,$$
and thus 
$$\chi'_{d}(G) \leq 3 + \lceil \frac{3(\Delta - (3d - 1)) - 1}{3d - 1} \rceil = \lceil \frac{3\Delta - 9d + 3 - 1 + 9d - 3}{3d-1} \rceil = \lceil \frac{3\Delta - 1}{3d - 1} \rceil.$$
This proves case 5 and the theorem. 
\end{proof}


\section{Simple graphs: Vizing's Theorem and NP-completeness}\label{sec:simple}


In this section, we will only consider simple graphs. Vizing~\cite{V64} proved the following theorem:

\begin{theorem}[Vizing's Theorem,  \cite{V64}] \label{thm:vizing}
For every simple graph $G$ with maximum degree $\Delta$, $\chi'_{1}(G) \in \{ \Delta, \Delta + 1 \}$.
\end{theorem}

While there are only $2$ possibilities, deciding between them was proven to be NP-complete even for regular simple graphs. 

\begin{theorem}[ Holyer~\cite{H81}, Leven and Galil~\cite{LZ83}]
\label{thm:chromatic_index_NP_complete}
For every $\Delta \geq 3$, it is NP-complete to decide if a $\Delta$-regular simple graph $G$ is $\Delta$-edge colourable. 
\end{theorem}

Vizing's theorem easily implies its following generalization to $d$-defective edge colouring. 
\begin{corollary}\label{coro:vizing}
For every $d \geq 1$ and every simple graph $G$ with maximum degree $\Delta$, $\chi'_{d}(G) \in \{ \lceil \frac{\Delta}{d} \rceil, \lceil \frac{\Delta+1}{d} \rceil \}$.
\end{corollary}

\begin{proof}
The lower bound holds by Lemma~\ref{lem:trivialbound}. For the upper bound,  consider an edge colouring of $G$ with $\Delta(G) + 1$ colours (it exists by Vizing's Theorem) and let $M_1 , \dots , M_{\Delta(G) + 1}$ be the classes of colours. By assigning colour $1$ to $M_1 \cup \dots \cup M_{d}$, colour $2$ to $M_{d + 1} \cup \dots \cup M_{2d}$,  etc,  we obtain a $(\lceil \frac{\Delta+1}{d} \rceil,d)$ edge colouring of $G$.
\end{proof}

We point out that Vizing~\cite{V64} also proved that for every (not necessarily simple) multigraph $G$ with maximum degree $\Delta$ and edge multiplicity $\mu$, $\chi'_{1}(G) \leq \Delta + \mu$ where the edge multiplicity is the maximum number of edges between two vertices. This directly implies that $\chi'_{d}(G) \leq \lceil \frac{\Delta + \mu}{d} \rceil$.
\medskip


In the following cases, one can distinguish between the two possibilities in Corollary~\ref{coro:vizing}.

\begin{theorem}\label{thm:simple_polynomial}
Let $d, \Delta \geq 1$ and let $G$ be a simple graph with maximum degree $\Delta$. Then:

\begin{itemize}
    \item[(a)] $\chi'_d(G) = \lceil \frac{\Delta}{d} \rceil$ if $(i)$ $d$ does not divide $\Delta$, or $(ii)$ $d$ is even, or $(iii)$ $\Delta = d$.
    
    \item[(b)] If $d$ is odd and $\Delta = 2d$, then $\chi'_d(G) = \lceil \frac{\Delta}{d} \rceil = 2$ if and only if every $2d$-regular connected component of $G$ has an even number of vertices; otherwise $\chi'_d(G) = \lceil \frac{\Delta +1}{d} \rceil=  3$.
\end{itemize}

\end{theorem}

\begin{proof}
In $(a)$, $(i)$ follows from Corollary~\ref{coro:vizing}, since if $d$ does not divide $\Delta$, then $\lceil \frac{\Delta}{d} \rceil = \lceil \frac{\Delta + 1}{d} \rceil$, $(ii)$ is contained in Theorem~\ref{thm:main_theorem} (even if $G$ is not simple), and $(iii)$ is obvious.

To prove $(b)$, note first that a $2d$-regular component $C$ of $G$ with $n$ vertices has $dn$ edges, and $d$ is odd; so the order and size of $C$ are either both even or both odd. 

Suppose first that every $2d$-regular component has an even order and an even size. Take two disjoint copies of $G$ and, for each vertex $v$ of $G$, add $2d - d(v)$ edges between the two copies of $v$. The resulting (not necessarily simple) multigraph $G'$ is $2d$-regular, and each of its connected components has even order and size (as the components of $G$ of odd order were not $2d$-regular, they are included in components of even order in $G'$). Now, by Lemma~\ref{lem:chi_increasing}, $\chi'_d(G) \leq \chi'_d(G') = 2$, and so $\chi'_d(G) = 2$.

Assume now that $G$ has a $2d$-regular component $C$ of odd order and size. 
Since $d$ is odd,  $C$ does not admit a $d$-factor, and so $C$ cannot be $(2,d)$-edge coloured. So, by Corollary~\ref{coro:vizing}, $\chi'_d(G) = \lceil \frac{2d+1}{d} \rceil= 3$. 

\end{proof}

We now prove a generalization of Theorem~\ref{thm:chromatic_index_NP_complete} in the context of defective edge colouring. Before that, we need the following construction.

For all integers $k,d \geq 1$, we construct a simple graph $G_{kd, d}$ such that $G$ is $kd$-regular and $\chi'_{d}(G) = k$. 
We can set $G_{d,d} = K_{d+1}$. Inductively, having defined $G_{kd,d}$, let $G_{(k+1)d,d}$ be the simple graph obtained by taking two disjoint copies of $G_{kd,d}$  and adding the edges of any $d$-regular bipartite simple graph between these two copies\footnote{For example, naming $u_1, \dots, u_n$ and $v_1, \dots, v_n$ the vertices of the two copies of $G_{kd,d}$, add the edges $u_iv_i, u_iv_{i+1}, \dots, u_iv_{i+d}$ for $i=1, \dots, n$, subscripts being taken modulo $n$. It gives a $d$-regular bipartite simple graph as soon as $n \geq d$.}. 
The obtained simple graph is clearly $(k+1)d$-regular, and we can $(k+1, d)$-edge colour it by taking a $(k,d)$-edge colouring for the two copies of $G_{kd,d}$ and add a new colour for the added edges, and finally  by Lemma~\ref{lem:trivialbound} it does not admit a $(k,d)$-edge colouring. Hence  $\chi'_d(G_{(k+1)d,d}) = k+1$.

Now assume $d$ is odd and let $H$ be obtained from $G_{kd,d}$ by subdividing one edge $ab$ with a new vertex $v$ of degree $2$. 

\begin{lemma}\label{lem:H_av_bv_same_colour}
$H$ has a $(k,d)$-edge colouring, and in every such colouring the edges $av$ and $bv$ have the same colour.  
\end{lemma}

\begin{proof}
Let $|V(G_{kd,d})|=n$. 
Let us first prove that $n$ is even. 
Since $G_{kd,d}$ is $kd$-regular, in a $(k,d)$-edge colouring of $G_{kd,d}$, each vertex must be incident with exactly $d$ edges of each colour. So every colour occurs on exactly $\frac{dn}{2}$ edges. Hence, $\frac{dn}{2}$ is an integer, and since $d$ is odd, $n$ is even.  

Clearly $H$ is $(k,d)$-edge colourable since $G_{kd,d}$ is. 
Every vertex of $H$ except $v$ has degree $kd$, so in a $(k,d)$-edge colouring of $H$,  every vertex is incident with exactly $d$ edges of each colour. 
Assume for contradiction that $v$ is incident with edges of two different  colours, and let $c$ be one of these colours. 
In $H\setminus \{v\}$, $c$ occurs on $\frac{d(n-1) + (d-1)}{2}= \frac{dn-1}{2} $ edges, and thus on $\frac{dn+1}{2}$ edges of $H$. But since $n$ is even, $\frac{dn+1}{2}$ is not an integer, a contradiction. 
Thus $av$ and $bv$ must have the same colour. 
\end{proof}

In the proof of the following theorem, we will use many copies of $H$, all with the same values of $k$ and $d$. We will use subscripts consistently: if $H_{u,i}$ is a copy of $H$, then it will contain a vertex $v_{u,i}$ of degree $2$ with neighbours $a_{u,i}$ and $b_{u,i}$.

Note that the pairs $d, \Delta$ in the following theorem are precisely those that are not covered by Theorem~\ref{thm:simple_polynomial}.

\begin{theorem}\label{thm:NP}
Let $d$ be an odd integer and $\Delta = kd$ for some integer $k \geq 3$. Then it is NP-complete to decide if a $\Delta$-regular simple graph is $(k,d)$-edge colourable.
\end{theorem}

\begin{proof}
The problem is clearly in NP. The case $d = 1$ is Theorem~\ref{thm:chromatic_index_NP_complete}, and so we may assume that $d \geq 3$. We perform a reduction from the case $d = 1$. Let $G$ be a $k$-regular simple graph.

We construct a simple graph $G'$ containing $G$ as follows: starting with $G$, for each vertex $u$ of $G$ add $\frac{k(d-1)}{2}$ disjoint copies $H_{u,i}$ of $H$ for $i = 1,2, \dots, \frac{k(d-1)}{2}$ and identify each vertex $v_{u,i}$ with $u$. The graph $G'$ is clearly simple and $kd$-regular. We will prove that $\chi'_{1}(G) = k$ if and only if $\chi'_d(G') = k$, and this will prove the theorem.

Suppose first that $\chi'_1(G) = k$. Starting with a $(k,1)$-edge colouring of $G$, we extend it to $G'$ as follows: for each vertex $u \in V(G)$ and colour $c \in \{ 1, 2, \dots, k\}$, give colour $c$ to all edges $va_{u,i}$ and $vb_{u,i}$ with $\frac{(c-1)(d-1)}{2} + 1 \leq i \leq \frac{c(d-1)}{2}$, and extend this to a $(k,d)$-edge colouring of $H_{u,i}$, which is possible by Lemma~\ref{lem:H_av_bv_same_colour}. 
Now, for each colour $c \in \{1, 2, \dots, k\}$, $u$ is incident to $d$ edges coloured $c$: $d-1$ edges in $E(G') \setminus E(G)$ and one edge of $G$. So we have constructed a $(k,d)$-edge colouring of $G'$. Hence $\chi'_d(G') = k$.

Suppose now that $\chi'_d(G') = k$, and fix a $(k,d)$-edge colouring of $G'$. By Lemma~\ref{lem:H_av_bv_same_colour}, for each vertex $u \in V(G)$ and colour $c \in \{1,2,\dots,k\}$, $u$ is incident to an even number of edges in $E(G') \setminus E(G)$ with colour $c$, and so (since $d$ is odd) $u$ must be incident to an odd number of edges of $G$ with colour $c$. Since there are $k$ colours and $G$ is $k$-regular, $u$ must be adjacent to exactly one edge of each colour, and so the $(k,d)$-edge colouring of $G'$ contains a $(k,1)$-edge colouring of $G$. Hence $\chi_1(G) = k$. This completes the proof.

\end{proof}



\section{Perspectives}\label{sec:further}

Recall that multigraphs are allowed to have multiple edges. 

\subsection*{List colouring}
The \textit{$d$-defective list chromatic index} of a multigraph $G$, denoted by $ch'_d(G)$, is defined as the minimum $k$ such that, for any choice of a list of $k$ integers given to each edge, there is an edge colouring with defect $d$ such that each edge receives a colour from its list.  So $ch'_1(G)$ is the usual list chromatic index. 

Borodin et al.~\cite{10.1006/jctb.1997.1780} proved that Shannon bound holds for the list chromatic index, that is, for every multigraph $G$, 
$ch_1'(G) \leq \lceil \frac{3 \Delta(G)}{2} \rceil$. It is then natural to ask if Theorem~\ref{thm:main_theorem} extends to defective list edge colouring. 
As mention in the introduction, when $d$ is even, it is proved in~\cite{HILTON2001253} (and a simpler proof is given in~\cite{amini:inria-00144318}) that for every multigraph $G$,  $ch'_d(G) = \lceil \frac{\Delta(G)}{d} \rceil$. 
When $d$ is odd, a proof that $ch'_{d}(G) \leq \lceil \frac{3\Delta}{3d - 1} \rceil$ is announced in~\cite{amini:inria-00144318} but seems to have a flaw and actually holds only in the case where $\Delta$ is divisible by $3k-1$. 

\begin{conjecture}
For every odd integer $d$ and for every  multigraph $G$, $ch'_d(G) \leq  \lceil \frac{3\Delta - 1}{3d-1} \rceil$
\end{conjecture}

We finally mention the following stronger conjecture that corresponds to the infamous list edge colouring conjecture for $d=1$ and is proved for bipartite graph in~\cite{HILTON2001253}. 

\begin{conjecture}[\cite{HSS98}]
For every multigraph $G$ and every integer $d$, $ch'_d(G) = \chi'_d(G)$.
\end{conjecture}

\subsection*{The Goldberg-Seymour Conjecture}
Let $d \geq 1$ and $G$ a multigraph. Observe that in any edge colouring of $G$ with defect $d$, and for any $X \subseteq V(G)$, each colour class contains at most $\lfloor \frac{d|X|}{2} \rfloor$ edges, which leads to the following lower bound on the $d$-defective edge chromatic number of any multigraph $G$: 

$$
    \chi'_d(G) \leq \Gamma_d(G) = 
    \max 
    \Big\{ \Bigl\lceil \frac{|E(G[X])|}{\lfloor \frac{d|X|}{2}\rfloor} \Bigr\rceil \mid X \subseteq V(G),\ |X| \geq 2 \Big\}.
$$
The following was known as the Goldberg-Seymour Conjecture~\cite{G73, S79} for almost 50 years. Recently, Chen, Jing and Zang~\cite{CJZ19} announced a proof (the paper is still under revision).

\begin{theorem}[Golberg-Seymour~\cite{G73, S79}]\label{thm:goldberg}  
For every multigraph $G$,\\ $\chi'_1(G) \leq \max\{\Gamma_1(G), \Delta(G) + 1\}$. 
\end{theorem}

We think that the following generalization could hold. 

\begin{conjecture}\label{conj:Golderg_gen}
Every multigraph $G$ satisfies $\chi'_d(G) \leq \max\{\Gamma_d(G), \lceil \frac{\Delta(G) +1}{d} \rceil\}$.
\end{conjecture}

An easy proof of the conjecture could start as follows. Let $G$ be a counter-example to Conjecture~\ref{conj:Golderg_gen}, that is $\chi'_d(G) > \max\{\Gamma_d(G), \lceil \frac{\Delta(G) +1}{d} \rceil\}$ for some $d \geq 3$.  
By Theorem~\ref{thm:goldberg}, $\chi'_1(G) \leq \max\{\Gamma_1(G), \Delta(G) + 1\}$. 
As $\chi'_d(G) \leq \lceil \frac{\chi'_1(G)}{d} \rceil$, if $\chi'_1(G) \leq \Delta(G) +1$, then $\chi'_d(G) \leq \frac{\Delta(G) +1}{d}$, a contradiction.  
So may assume that $\Delta(G) +1 < \chi'_1(G) = \Gamma_1(G)$. 
This implies that $\chi'_d(G) \leq \lceil \frac{\Gamma_1(G)}{d} \rceil$. 
So it is enough to prove that $\lceil \frac{\Gamma_1(G)}{d} \rceil \leq \max\{\Gamma_d(G), \lceil \frac{\Delta(G) +1}{d} \rceil\}$. 

Unfortunately, this last inequality does not hold, for example in the following simple example.  
Consider the multigraph $G$ made of three vertices connected by respectively $7$, $7$ and $2$ edges. So $\Delta(G) + 1 = 15$, $\Gamma_1(G) = \max\{ \frac{2}{1},\frac{7}{1},\frac{16}{1}\} = 16$ and $\chi'_1(G) = 16$.
Moreover, $\Gamma_3(G) = \max\{ \frac{2}{3}, \frac{7}{3}, \frac{16}{4}\} = 4$. 
Hence, 
$$6 = \Big\lceil \frac{\Gamma_1(G)}{3} \Big\rceil >  \max \Big\{\Gamma_3(G), \Big\lceil \frac{\Delta(G) +1}{3} \Big\rceil \Big\} = \max \Big\{4, \Big\lceil \frac{15}{3} \Big\rceil \Big\} =5. $$

\subsection*{The degree Ramsey number of stars}

In this subsection, we briefly describe the link between the degree Ramsey number of stars and defective edge colouring. We are thankful to Ross Kang for bringing this to our attention.    

Let $H$, $G$ be simple graphs. Let $H \rightarrow_s G$ mean that every colouring of $E(H)$ with $s$ colours produces a monochromatic copy of $H$. 
The \textit{degree Ramsey number} of a simple graph $G$ is $R_{\Delta}(G;s) = \min\{\Delta(H): H \rightarrow_s G\}$. 
Observe that $H \rightarrow_s K_{1,d+1}$ means that $\chi'_d(H) \geq s+1$. Hence,  $R_{\Delta}(K_{1,d+1};s) = \min\{\Delta(H): \chi'_d(H) \geq s+1\}$. 

It can be proved (with a little brain gymnastic) that the following result of Kinnersley, Milans and West is equivalent to corollary~\ref{coro:vizing}. 

\begin{theorem}[\cite{KMW12}]
If $s \geq 2$, then 
$
R_{\Delta}(K_{1,d+1}; s) = \left\{
    \begin{array}{ll}
       s\cdot d & \mbox{if } d \mbox{ is odd,} \\
        s\cdot d +1 &  \mbox{if } d \mbox{ is even.}
    \end{array}
\right.
$
\end{theorem}

It could be of interest to look at the degree Ramsey number of multigraphs.






\cleardoublepage

\label{app:bibliography} 

\manualmark 
\markboth{\spacedlowsmallcaps{\bibname}}{\spacedlowsmallcaps{\bibname}} 
\refstepcounter{dummy}

\addtocontents{toc}{\protect\vspace{\beforebibskip}} 
\addcontentsline{toc}{chapter}{\tocEntry{\bibname}}

\printbibliography 




\end{document}